\documentclass[10pt]{article}
\usepackage{amsmath,amscd}
\usepackage{amssymb,latexsym,amsthm}
\usepackage{color}
\usepackage[spanish,english]{babel}
\usepackage[pdftex]{hyperref}
\usepackage{amssymb}
\usepackage{color}
\usepackage{amsmath,amsthm,amscd}
\usepackage[latin1]{inputenc}
\usepackage{lscape}
\usepackage{fancyhdr}
\usepackage{amsfonts}
\usepackage{pb-diagram}

\numberwithin{equation}{section}
\newtheorem{theorem}{Theorem}[section]

\newtheorem{definition}[theorem]{Definition}
\newtheorem{proposition}[theorem]{Proposition}

\newtheorem{conjecture}[theorem]{Conjecture}
\newtheorem{corollary}[theorem]{Corollary}

\theoremstyle{definition}

\newtheorem{example}[theorem]{Example}

\newtheorem{remark}[theorem]{Remark}

\deactivatequoting

\title{\textbf{The Dixmier problem\\ for skew $PBW$ extensions and rings}}
\author{William Fajardo\\
\texttt{wafajardoc@udistrital.edu.co}\\
Universidad Distrital Francisco José de Caldas, Bogotá\\
Oswaldo Lezama\\
\texttt{jolezamas@unal.edu.co}
\\Seminario de Álgebra Constructiva - SAC$^2$\\ Departamento de Matemáticas\\ Universidad Nacional de
Colombia, Sede Bogot\'a\\}

\date{}
\begin{document}
\maketitle
\begin{abstract}
\noindent In this paper we discuss for skew $PBW$ extensions the famous Dixmier problem formulated by Jacques Dixmier in 1968. The skew $PBW$ extensions are noncommutative rings of polynomial type and covers several algebras and rings arising in mathematical physics and noncommutative algebraic geometry. For this purpose, we introduce the \textbf{Dixmier algebras and rings}, and we will study the Dixmier problem for algebras over commutative rings, in particular, for $\mathbb{Z}$-algebras, i.e., for arbitrary rings. The results are focused on the investigation of the Dixmier problem for matrix algebras, product of algebras, tensor product of algebras and also on the Dixmier question for the following particular key skew $PBW$ extension: Let $K$ be a field of characteristic zero and let $\mathcal{CSD}_n(K)$ be the $K$-algebra generated by $n\geq 2$ elements $x_1,\dots,x_n$ subject to relations
\begin{center}
	$x_jx_i=x_ix_j+d_{ij}$, for all $1\leq i<j\leq n$, with $d_{ij}\in K-\{0\}$.
\end{center}
We prove that the algebra $\mathcal{CSD}_n(K)$ is central and simple. In the last section we present a matrix-computational  approach to the problem formulated by Jacques Dixmier and also we compute some concrete nontrivial examples of automorphisms of the first Weyl algebra $A_1(K)$ and $\mathcal{CSD}_n(K)$ using the \texttt{MAPLE} library \textbf{SPBWE} developed for the first author. We compute the inverses of these automorphisms, and for $A_1(K)$, its factorization through some elementary automorphisms. For $n$ odd, we found some endomorphisms of $\mathcal{CSD}_n(K)$ that are not automorphisms. We conjecture that $\mathcal{CSD}_n(K)$ is Dixmier when $n$ is even.
\bigskip


\bigskip

\end{abstract}

\tableofcontents

\section{Introduction}

In this section we recall the Dixmier question formulated by Jacques Dixmier in 1968 as well as some key conjectures close related to the Dixmier problem and the relationship between them.

\subsection{The Dixmier conjecture}\label{subsection1.1}

Let $K$ be a field and $A_1(K)$ be the \textbf{\textit{first Weyl algebra}} defined as the quotient algebra
\begin{center}
$A_1(K):=K\{t,x\}/\langle xt-tx-1\rangle$,
\end{center}
where $K\{t,x\}$ is the free $K$-algebra generated by $t$ and $x$, and $\langle xt-tx-1\rangle$ is the two-sided ideal of $K\{t,x\}$ generated by $xt-tx-1$. Thus, $A_1(K)$ is the associative $K$-algebra generated by two elements $t,x$ that satisfy the relation
\begin{center}
$xt=tx+1$.
\end{center}
In 1968, Jacques Dixmier in his paper \cite{Dixmier} set the following question when $char(K)=0$:
\begin{center}
	It is every endomorphism of $A_1(K)$ an automorphism?
\end{center}
Here an endomorphism should be understood as a $K$-linear ring homomorphism, i.e., an algebra homomorphism. Observe that the original problem is not a conjecture but it is a question. However, in the mathematical literature this problem is known as a conjecture in the following way.
\begin{conjecture}[\textbf{Dixmier conjecture}]
Let $K$ be a field of characteristic zero.
\begin{center}
Every endomorphism of $A_1(K)$ is an automorphism.
\end{center}
\end{conjecture}
In his original paper \cite{Dixmier} Dixmier wrote: ``A. A Kirillov informed me that the Moscow school also
considered this problem". Thus, the problem could be named as the Dixmier-Kirillov conjecture.

The Dixmier conjecture can be formulated in general for the \textbf{\textit{$n$-th Weyl algebra}} $A_n(K)$, $n\geq 1$ (\cite{Essen}). This algebra is generated by $2n$ elements $t_1,\dots,t_n,x_1,\dots,x_n$ that satisfy the following relations:
\begin{center}
$[x_i, t_j]=\delta_{ij}$, $[x_i,x_j]=0=[t_i,t_j]$, $1\leq i,j\leq n$.
\end{center}

\begin{conjecture}[\textbf{Generalized Dixmier conjecture}]
	Let $K$ be a field of characteristic zero and $n\geq 1$.
	\begin{center}
		Every endomorphism of $A_n(K)$ is an automorphism.
	\end{center}
\end{conjecture}

For $char(K)=p>0$, Bavula in \cite{Bavula4} presents a monomorphism of $A_1(K)$ that is not an automorphism:

\begin{center}
$\alpha: A_1(K)\to A_1(K)$, $t\mapsto t+t^p$, $x\mapsto x$.
\end{center}
$\alpha$ is not an automorphism since its restriction to the center $Z(A_1(K))=K[t^p,x^p]$ is not an automorphism:
\begin{align}\label{equ1.1}
	\alpha\mid_{Z(A_1(K))}: Z(A_1(K))\to Z(A_1(K)), \ \ t^p\mapsto t^p+t^{p^2}, x^p\mapsto x^p.
\end{align}

The Dixmier conjecture invites to investigate the question for other algebras. For example, in \cite{Launois-Kitchin} Kitchin and Launois proved that every endomorphism of a simple quantum generalized Weyl algebra over a commutative Laurent polynomial ring in one variable is an automorphism. The quantum generalized Weyl algebras were introduced by Bavula and Jordan in \cite{Bavula.Jordan}, a particular example of this type of algebras is the algebra $A_{\alpha,q}$ defined in the following way (see \cite{Launois-Kitchin}): Let $K$ be a field and $0\neq q,\alpha\in K$, then $A_{\alpha,q}$ is generated by three variables $e_1,e_2,e_3$ subject to the following relations:
\begin{center}
$e_1e_3=q^{-2}e_3e_1$, $e_2e_3=q^2e_3e_2+\alpha$, $e_2e_1=q^{-2}e_1e_2-q^{-2}e_3$, $e_2
^2+(q^4-1)e_3e_1e_2+\alpha q^2(q^2+1)e_1=0$.
\end{center}
In \cite{Launois-Kitchin} is proved that every endomorphism of $A_{\alpha,q}$ is an automorphism when $q$ is not a
root of unity and $\alpha$ nonzero.

In \cite{Bavula1} the Dixmier problem was studied for the algebra of polynomial integro-differential operators over a field of characteristic zero: \textit{Let $\mathbb{I}_1$ be the \textbf{algebra of polynomial integro-differential operators} over a field $K$ of characteristic zero. Then, each algebra endomorphism of $\mathbb{I}_1$ is an automorphism}. Recall (see \cite{Bavula1}) that $\mathbb{I}_1$ is the associative $K$-algebra generated by three elements $\partial,\int, H$ that satisfy the relations
\begin{center}
	$\partial \int=1$, $[H,\int]=\int$, $[H,\partial]=-\partial$, $H(1-\int\partial)=(1-\int\partial)H=1-\int\partial$,
\end{center}
where $\partial,\int, H$ are the operators defined on the commutative polynomial algebra $K[t]$ by
\begin{align*}
K[t] & \xrightarrow{\partial} K[t]  & K[t] & \xrightarrow{\int} K[t] & K[t] & \xrightarrow{H:=\partial\circ t} K[t]\\
p & \mapsto \frac{dp}{dt}  & t^n & \mapsto \frac{t^{n+1}}{n+1}& p & \mapsto t\frac{dp}{dt}+p
\end{align*}
In \cite{Bavula3} were investigated several ring and homological properties not only for the algebra $\mathbb{I}_1$, but also for its generalization $\mathbb{I}_n$, $n\geq 1$. $\mathbb{I}_n$ is the $K$-algebra generated by $3n$ variables $\partial_i,\int_i,H_i$, $1\leq i\leq n$, with relations as in $\mathbb{I}_1$ for every $i$, and in addition, $a_ia_j=a_ja_i$ for every $i,j$ with $a_k\in \{\partial_k,\int_k,H_k\}_{k=1}^n$. In \cite{Bavula3} was proved that the algebra $\mathbb{I}_n$ is a non-simple, non-Noetherian algebra with trivial center which is not a domain. In \cite{Bavula1} is conjectured that every endomorphism of $\mathbb{I}_n$ is an automorphism.

The Dixmier problem has been studied even for some non-associative algebras (\cite{Bavula2}):

\textit{Let $K$ be a field of characteristic zero and $n\geq 2$. Then, every monomorphism of the \textbf{Lie algebra $\boldsymbol{u_n}$ of triangular derivations of the
polynomial algebra} $K[T]:= K[t_1, . . . , t_n]$ is an automorphism}.

$u_n$ is defined as a Lie subalgebra of the Lie algebra $Der_K(K[T])$ of $K$-derivations of $K[T]$:
\begin{center}
$u_n:=K\partial_1+P_1\partial_2+\cdots+P_{n-1}\partial_n$, where $\partial_i:=\frac{\partial}{\partial t_i}$ and $P_i:=K[t_1,\dots,t_{i}]$, $1\leq i\leq n$.
\end{center}
According to (2) in \cite{Bavula2}, $u_n$ is in fact a Lie subalgebra of $Der_K(K[T])$.
Nothing is said in \cite{Bavula2} about the Dixmier question for $char(K)\neq 0$.

The Dixmier problem has been investigated also for some quantum algebras and its localizations. In \cite{Backelin} was proved that there are monomorphisms of the complex \textbf{\textit{quantized Weyl algebra}} $A_1^q(\mathbb{C})$ that are not automorphisms, where $A_1^q(\mathbb{C})$ is the $\mathbb{C}$-algebra generated by two elements $t$ and $x$ such that
\begin{center}
$xt-qtx=1$, with $q\in \mathbb{C}^*:=\mathbb{C}-\{0,1\}$.
\end{center}
If $q$ is not a root of unity, later was proved that every algebra endomorphism of a simple localization of  $A_1^q(\mathbb{C})\otimes \cdots \otimes A_1^q(\mathbb{C})$ is an automorphism (see  page 3 in \cite{Tang} and \cite{Launois}). In particular, every algebra endomorphism of the simple localization $A_1^q(\mathbb{C})_{\mathcal{Z}}$ of $A_1^q(\mathbb{C})$ is an automorphism, where ${\mathcal{Z}}$ is the Ore set of $A_1^q(\mathbb{C})$ defined by $\mathcal{Z}:=\{z^k\mid k\geq 0\}$, with $z:=xt-tx$ (\cite{Tang}). Another interesting case of localization was presented in Theorem 1.6 of \cite{Bavula5}, there in was proved that the localization of the first Weyl algebra $A_1(K)$ by the Ore multiplicatively closed subset generated by $S:=\{\partial t+i\mid i\in \mathbb{Z}\}$ does not satisfies the Dixmier conjecture, i.e., not every endomorphism of $A_1(K)_S$ is an automorphism.

Some results on the Dixmier conjecture have been published recently. In particular, in \cite{Zheglov} Alexander Zheglov from Lomonosov Moscow State University claims that the Dixmier conjecture is true.
\begin{theorem}[Theorem 1.1, \cite{Zheglov}]\label{Theorem1.3}
Let $K$ be a field of characteristic zero. The Dixmier conjecture for the first Weyl algebra is true, i.e., $End_K(A_1(K))=Aut_K(A_1(K))$.
\end{theorem}

\begin{remark}\label{remark1.4}The proof by Zheglov is based on the theory of normal forms for ordinary differential operators (\cite{Guo}), however, according to AI (Gemini and ChatGPT), the preprint \cite{Zheglov} (updated in 2026) has not been published up to now (october 26, 2026) in any journal, it is only a preprint. James Zhang (University of Washington) in January 2026 (see \cite{JamesZhang}), in a paper about Relative Dixmier property for Poisson algebras, aserves that the Dixmier conjecture still remains wide open. Paper \cite{Zheglov} and the first version of our paper (see arXiv:2506.09285v1) were forgotten by their authors to include them in \cite{JamesZhang}.
\end{remark}	

\textbf{In this work we are interested only in a pure algebraic matrix-constructive and computational approach to the Dixmier problem}. One of the purposes of this paper is to discuss the Dixmier problem for skew $PBW$ extensions. This class of noncommutative rings (Definition \ref{gpbwextension}) covers several algebras and rings arising in mathematical physics and noncommutative algebraic geometry. The Weyl algebra and the quantized Weyl algebra are particular examples of skew $PBW$ extensions. A second purpose of the paper is to consider the Dixmier question not only for algebras over fields but also for algebras over commutative rings, in particular, for $\mathbb{Z}$-algebras, i.e., for arbitrary associative rings with unit. The results are focused on the Dixmier problem for matrix algebras, product of algebras, tensor product of algebras and also on the Dixmier problem for a key particular skew $PBW$ extension that includes the Weyl algebra $A_1(K)$ in characteristic zero, see Definition \ref{definition4.21}. Our third purpose consists in presenting a matrix approach to the problem formulated by Dixmier and also to compute concrete nontrivial examples of automorphisms of $A_1(K)$ and $\mathcal{CSD}_n(K)$ using the \texttt{MAPLE} library \textbf{SPBWE} developed for the first author in \cite{Fajardo2} and \cite{Fajardo3} (see also \cite{Lezama-sigmaPBW} and \cite{Lezama-algebraic}). We will compute the inverses of these automorphisms, and for $A_1(K)$, its factorization through some elementary automorphisms. For $n$ odd, we found some endomorphisms of $\mathcal{CSD}_n(K)$ that are not automorphisms.

The paper is organized in the following way: In this introductory section we recall some key conjectures close related to the Dixmier conjecture and the relationship between them. In the second, we will review some known results on the conjecture, including some recent interesting \textbf{pure algebraic advances of the problem}. In the third section we will introduce the Dixmier algebras and rings as a preliminary material for the investigation of the Dixmier problem for skew $PBW$ extensions. The fourth section is dedicated to study the Dixmier question for skew $PBW$ extensions. In the last section we will present a matrix approach to the original problem formulated by Jacques Dixmier and also we will compute concrete nontrivial examples of automorphisms of $A_1(K)$ and $\mathcal{CSD}_n(K)$, the inverses of these automorphisms, and for $A_1(K)$, \textbf{its factorization through some elementary automorphisms}, using the \texttt{MAPLE} library \textbf{SPBWE}. For $n$ odd, we found some endomorphisms of $\mathcal{CSD}_n(K)$ that are not automorphisms. The novelty of the paper and the main results are concentrated in Subsections \ref{subsection3.1}, \ref{subsection3.2}, \ref{subsection4.2} and Section \ref{section5}.

In this paper, a ring means an associative ring  with unit non necessarily commutative. If $A$ is a ring, then $A^*$ is the group of invertible elements of $A$. If $S\subseteq A$, then the left ideal of $A$ generated by $S$ is denoted by $\left\langle S\right\rbrace $. If $a,b\in A$, then $[a,b]:=ab-ba$. The center of $A$ is denoted by $Z(A)$.  An element $a\in A$ is \textbf{\textit{normal}} if $aA=Aa$. For $n\geq 1$, the ring of square matrices of size $n\times n$ over the ring $A$ will be denoted by $M_n(A)$. If nothing contrary is assumed, $K$ denotes a field.

\subsection{Some close related conjectures}

In this subsection we recall some famous conjectures close related to the Dixmier problem (see \cite{Essen}, \cite{Essen2} and \cite{Essen3}). We start with the Jacobian conjecture formulated in \cite{Tsuchimoto} in the following way (see also \cite{Bavula5}).

\begin{conjecture}[\textbf{Jacobian conjecture}]
Let $K$ be a field of characteristic zero and let $K[T] = K[t_1, \dots, t_n]$ be the polynomial algebra, $n\geq 1$. Let $\sigma:K[T]\to K[T]$ be an algebra endomorphism of $K[T]$. If the \textbf{Jacobian} $J(\sigma):=\det[\frac{\partial \sigma(t_i)}{\partial t_j}]\in K^*$, then $\sigma$ is an automorphism.
\end{conjecture}
From a geometric point of view, the conjecture has been formulated in the mathematical literature in the following way. Let $F:=(F_1,\dots,F_n):K^n\to K^n$ be a \textbf{\textit{polynomial function}}, i.e.,
\begin{center}
	$z:=(z_1,\dots,z_n)\mapsto (F_1(z_1,\dots,z_n),\dots,F_n(z_1,\dots,z_n))$,
\end{center}
for some polynomials $F_i\in K[T]$, $1\leq i\leq n$. For $z\in K^n$, let $F'(z):=\det (JF(z))$, where $JF:=[\frac{\partial F_i}{\partial t_j}]_{1\leq i,j\leq
	n}$ is the \textit{\textbf{Jacobian matrix}} of $F$. Since there exists a correspondence between the algebra endomorphisms of $K[T]$ and the polynomial functions on $K^n$, the Jacobian conjecture can be formulated in the following equivalently way:

\begin{center}
 \textit{If $F'(z)\neq 0$ for every $z\in K^n$ $($or equivalently, $\det(JF)\in K^*$ $)$, then $F$ is invertible $($i.e., $F$ has an inverse which is also a polynomial function$)$}.
\end{center}
The Jacobian conjecture was first formulated in 1939 by O. Keller in \cite{Keller} for $n=2$ and polynomials with integer coefficients. Some researchers refer to the Jacobian conjecture as the Keller problem. Perhaps the most complete review of the Jacobian conjecture is the monograph \cite{Essen} of Essen, Kuroda, and Crachiloa.

Another interesting conjecture related to the previous, and hence, to de Dixmier question, is a conjecture about the kernel of some special derivation (see \cite{Essen2} and \cite{Essen3}). We preserve the previous notation.

\begin{conjecture}[\textbf{Kernel conjecture}]\label{KC}
	Let $K$ be a field of characteristic zero and $F:=(F_1,\dots,F_n):K^n\to K^n$ be a polynomial function. Assume that
	$\det(JF)\in K^*$ and consider an $n$-tuple of derivations on $K[T]$, denoted by
	$\frac{\partial}{\partial F_i}$, $1\leq i\leq n$, as follows:
	\begin{center}
		$
		\begin{bmatrix}
		\frac{\partial}{\partial F_1}\\
		\vdots \\
		\frac{\partial}{\partial F_n}
		\end{bmatrix}:=((JF)^{-1})^T\begin{bmatrix}
		\frac{\partial}{\partial t_1}\\
		\vdots \\
		\frac{\partial}{\partial t_n}
		\end{bmatrix}.
		$
	\end{center}
Then, $\ker(\frac{\partial}{\partial
	F_n})=K[F_1,\dots,F_{n-1}]$.

\end{conjecture}

For the next conjecture (see \cite{Essen2} and \cite{Essen3}) we need to recall the following notion: A derivation $D$ of a ring $A$ is \textbf{\textit{locally nilpotent}} if given $a\in A$ there exists an integer $m(a)$ such that $D^{m(a)}a=0$.

\begin{conjecture}[\textbf{Second Kernel Conjecture}]
	
	Let $K$ be a field of characteristic zero and let $D$ be a locally nilpotent derivation on
$K[T]$ such that there exists a polynomial $p\in K[T]$ with $Dp=1$. Then, $\ker(D)=K[F_1,\dots,F_{n-1}]$, for some $F_i\in K[T]$ algebraically independent
over $K$.
\end{conjecture}

The Zariski cancellation problem arises in commutative algebra and can be formulated as follows.

\begin{conjecture}[\textbf{Zariski cancellation problem}]
Let $K$ be a field and $A:=K[T]=K[t_1,\dots,t_n]$ be the polynomial algebra and $B$ be a commutative $K$-algebra,
\begin{center}
	if $A[t]\cong B[t]$, then $A\cong B$?
\end{center}	
\end{conjecture}
Here the isomorphisms should be understood as isomorphisms of $K$-algebras. Abhyankar-Eakin-Heinzer (1972, \cite{Abhyankar}) proved that $K[t_1]$ is cancellative. Fujita (1979, \cite{Fujita2}) and Miyanishi-Sugie (1980, \cite{Miyanishi}) proved that if $char K=0$, then $K[t_1,t_2]$ is cancellative. If $char K\neq 0$, Russell in 1981 proved that $K[t_1,t_2]$ is cancellative (\cite{Russell}). In 2014, Gupta proved that if $n\geq 3$ and $char K\neq 0$ then $K[t_1,\dots,t_n]$ is not cancellative (\cite{Gupta}, \cite{Gupta2}). The problem remains open for $n\geq 3$ and $char K=0$. Recently the problem has been considered for noncommutative algebras (\cite{BellZhang}), and even for arbitrary rings (\cite{Lezama-sigmaPBW}, Chapter 20). Let $A$ be an arbitrary $K$-algebra, it is said that $A$ is \textbf{\textit{cancellative}} if for every $K$-algebra $B$,
\begin{center}
$A[t]\cong B[t]\Rightarrow A\cong B$.
\end{center}

\subsection{Relationship between conjectures}\label{subsection1.3}

In this subsection we present some well-known results about the relationship between the conjectures presented before. For this,
we denote the conjectures in the following way:

\begin{itemize}
	\item[]The Generalized Dixmier Conjecture: $DC_n$
	\item[]The Jacobian Conjecture: $JC_n$
	\item[]The Kernel Conjecture: $KC_n$
	\item[]The Second Kernel Conjecture: $2KC_n$
	\item[]The Zariski Cancellation Problem: $ZCP_n$
\end{itemize}
 The subscript $n$ indicates the number of variables either of $K[T]$ or of the Weyl algebra.

We include only the proof of that $DC_n$ implies $JC_n$. For the others relations we indicate the references.

\begin{proposition}[\cite{Essen2}, Proposition 3.28]\label{proposition1.7}
$DC_n$ $\Rightarrow$ $JC_n$.
\end{proposition}
\begin{proof}
Let $\sigma:K[T]\to K[T]$ be an algebra endomorphism of $K[T]$ such that $\det[\frac{\partial \sigma(t_i)}{\partial t_j}]\in K^*$ and let $F_i:=\sigma(t_i)$, $1\leq i\leq n$. Then $\det(JF)\in K^*$ and we can define
\begin{center}
	$\varphi:A_n(K)\to A_n(K)$, $\varphi(t_i):=F_i$, $\varphi(x_i):=\frac{\partial}{\partial F_i}$, $1\leq i\leq n$,
\end{center}
where $\frac{\partial}{\partial F_i}$ is as in Conjecture \ref{KC}. Observe that $\varphi$ is a well-defined algebra homomorphism since
\begin{center}
	$[\varphi(x_i), \varphi(t_j)]=\delta_{ij}$, $[\varphi(x_i),\varphi(x_j)]=0=[\varphi(t_i),\varphi(t_j)]$, $1\leq i,j\leq n$.
\end{center}
Assuming $DC_n$ we get that $\varphi$ is surjective. Let $g\in K[T]\subset A_n(K)$, then there exists $p\in A_n(K)$ such that $\varphi(p)=g$. Hence, $g$ can be written as a polynomial in the variables $\frac{\partial}{\partial F_i}$ with coefficients in $K[F_1,\dots,F_n]$, $g=\sum_{\alpha}c_{\alpha}(\frac{\partial}{\partial F})^{\alpha}$, with $\alpha:=(\alpha_1,\dots,\alpha_n)\in \mathbb{N}^n$, $(\frac{\partial}{\partial F})^{\alpha}:=(\frac{\partial}{\partial F_1})^{\alpha_1}\cdots (\frac{\partial}{\partial F_n})^{\alpha_n}$ and $c_{\alpha}\in K[F_1,\dots,F_n]$. Considering $g$ as a differential operator we can apply $g$ to $1\in K[T]$, so $g\cdot 1=g1=g=c_0\in K[F_1,\dots,F_n]$, whence $K[T]\subseteq K[F_1,\dots,F_n]\subseteq K[T]$, i.e., $K[F_1,\dots,F_n]=K[T]$. Since $Im(\sigma)=K[F_1,\dots,F_n]$, then $\sigma$ is surjective . Finally, the condition $\det(JF)\in K^*$ implies that $\sigma$ is injective (see \cite{Bavula2}, page 554).
\end{proof}

\begin{itemize}
	\item[]$JC_{2n}$	$\Rightarrow$ $DC_n$ (See \cite{Tsuchimoto}, Corollary 7.3; \cite{Kontsevich}, Theorem 1; \cite{Bavula6}, Theorem 3)
	\item[]$JC_n$ $\Rightarrow$ $KC_n$ (See \cite{Essen2}, page 63)
	\item[]$KC_{n+1} $ $\Rightarrow$ $JC_n$ (See \cite{Essen2}, Proposition 2.5)
	\item[]$ZCP_n$ $\Leftrightarrow$ $2KC_n$ (See \cite{Essen2}, Proposition 3.8)
	\end{itemize}

\begin{remark}\label{remark1.6}
	Recently (July 21, 2026), Levent Alpöge published a preprint that aserves that the Jacobian conjecture is not true for $n\geq 3$, hence, from Proposition \ref{proposition1.7}, the Generalized Dixmier conjecture is not true. The preprint was announced in the social media $X$ and the blog of research \textit{What's new}, but, according to AI (Gemini and ChatGPT), up to now (october 26, 2026) it has not been published in any indexed prestige journal. However, according to AI, the global mathematical community has accepted the result of Levent Alpöge. \textbf{In this paper we accept the result of Levent Alpöge}.  
\end{remark}

\section{Some recent pure algebraic results on the Dixmier conjecture}

In \cite{Zheglov} Alexander Zheglov from Lomonosov Moscow State University claims that the Dixmier conjecture is true (see Theorem \ref{Theorem1.3}). The proof given by Zheglov is based on the theory of normal forms for ordinary differential operators (see \cite{Guo}). However, take in account Remark \ref{remark1.4}. and the works \cite{Guccione} and \cite{Guccione2} about the shape of possible counterexamples to the Dixmier conjecture.

In this section we review some recent interesting \textbf{pure algebraic} results on the classical Dixmier conjecture. Inspired in this, in the next sections we will introduce the Dixmier algebras and rings and we will investigate the Dixmier problem for skew $PBW$ extensions.

\subsection{Two recent results using graduations and bimodules over $A_n(K)$}

In 2018 V.V. Bavula and V. Levandovskyy proved a theorem, published in 2020 in \cite{Levandovsky}, about the Dixmier conjecture in a particular situation. We present next this advance on the conjecture that involves an interesting graduation of $A_1(K)$. In \cite{Levandovsky}, $A:=A_1(K)$ is described as a $\mathbb{Z}$-graded algebra with graduation $A=\oplus_{i\in \mathbb{Z}}A_{i}$, where $A_{0}:=K[h]$, with $h:=xt$ and, for $i\geq 1$, $A_{i}:=K[h]t^i$ and $A_{-i}:=K[h]x^i$ (we have adapted to our notation the notation used in \cite{Levandovsky}). Given a non-zero polynomial $p\in A$, the number of non-zero homogeneous components of $p$ is called the \textbf{\textit{mass}} of $p$, denoted by $m(p)$.

\begin{proposition}[\cite{Levandovsky}, Theorem 1.1]\label{theorem2.1}
Let $p,q$ be elements of the first Weyl algebra $A_1(K)$ with $m(p)\leq 2$ and $m(q)\leq 2$. If $[p,q]=1$, then $p=\phi(x)$ and $q=\phi(t)$, for some automorphism $\phi\in Aut_K(A_1(K))$.
\end{proposition}

The next proposition gives an equivalent form of $DC_1$. Recall that $char(K)=0$ and hence $A_1(K)$ is a simple algebra.
\begin{proposition}\label{proposition2.2}
The following conditions are equivalent:
\begin{enumerate}
\item[\rm (i)]$DC_1$.
\item[\rm (ii)]If $[p,q]=1$ for some $p,q\in A_1(K)$, then $p,q$ generate $A_1(K)$ as $K$-algebra.
\end{enumerate}
\end{proposition}
\begin{proof}
$\rm (i)\Rightarrow \rm (ii)$:
Let $p,q\in A_1(K)$ such that $[p,q]=1$, then $\phi: A_1(K)\to A_1(K)$ defined by $\phi(x):=p$ and $\phi(t):=q$ is a well-defined algebra endomorphism of $A_1(K)$, so $\phi$ is an automorphism, in particular, $Im(\phi)=A_1(K)$, but $Im(\phi)$ is the subalgebra of $A_1(K)$ generated by $\phi(x)$ and $\phi(t)$, i.e., $p,q$ generate $A_1(K)$ as $K$-algebra.

$\rm (ii)\Rightarrow \rm (i)$: Let $\phi: A_1(K)\to A_1(K)$ be an algebra endomorphism, and let $p:=\phi(x)$ and $q:=\phi(t)$, then $[p,q]=pq-qp=\phi(xt-tx)=\phi(1)=1$ and $Im(\phi)$ is generated by $p,q$. By the hypothesis, $\phi$ is surjective, i.e., $\phi$ is an automorphism since $\ker(\phi)=0$.
\end{proof}

Thus, Proposition \ref{theorem2.1} says that the Dixmier conjecture holds if there exist elements $p,q\in A_1(K)$ that satisfy the following conditions: (a) $[p,q]=1$ (b) $m(p)\leq 2$ and $m(q)\leq 2$.

\begin{remark}
In a recent paper Gang Han and Bowen Tan (see \cite{Han}) consider the following $\mathbb{Z}$-graduation for $A_1(K)$: Take the inner derivation $ad_{xt}$ of $A_1(K)$, $ad_{xt}(f):=xtf-fxt$, for every $f\in A_1(K)$. Then the spectrum of this inner derivation is $\mathbb{Z}$ and for every $i\in \mathbb{Z}$, the $i$-homogeneous component of $A_1(K)$ is the $i$-eigenspace of $ad_{xt}$ denoted $D_i$. Thus, $A_1(K)$ has the $\mathbb{Z}$-graduation, $A_1(K)=\bigoplus_{i\in \mathbb{Z}}D_i$. Theorem 3.15 in \cite{Han} says that if $z,w\in A_1(K)$ are such that $[z,w]=1$ and $z$ is a sum of not more than $2$ homogeneous elements of $A_1(K)$, then $z$ and $w$ generate $A_1(K)$. This theorem of Han and Tan improves the result of Bavula and Levandovskyy.
\end{remark}

A second interesting recent advance on the Dixmier conjecture is presented by Niels Lauritzen and Jesper Funch Thomsen in 2019, and published in 2021 in the beautiful paper \cite{Lauritzen}. This advance involves bimodules over $A_n(K)$. We present next the main results of \cite{Lauritzen} related to $DC_n$ and the most important tools needed. We have included some proofs for a better understanding of the concepts involved.

\begin{itemize}
\item Let $K$ be a field and $A$ be a $K$-algebra.
\item A \textbf{\textit{bimodule}} $M$ over $A$ is a left and right $A$-module such that $(a_1m)a_2=a_1(ma_2)$, for every $a_1,a_2\in A$ and $m\in M$; an \textbf{\textit{homomorphism of bimodules}} $f:M\to N$ is a homomorphism of left and right modules. If $f:A\to A$ is a $K$-algebra endomorphism and $M$ is a bimodule over $A$, then $\boldsymbol{M^f}$ denotes the bimodule over $A$ defined by $a_1ma_2:=a_1mf(a_2)$, for $a_1,a_2\in A$ and $m\in M$. Similarly is defined $\boldsymbol{^fM}$.
\item The \textbf{\textit{graph}} of an algebra homomorphism $f:S\to A$ of $K$-algebras $S$ and $A$ is the $A-S$-bimodule $\boldsymbol{A^f}$ defined by $axs:=axf(s)$, for $a,x\in A$ and $s\in S$. In particular, we have the graph $A^f$ of an algebra endomorphism $f:A\to A$. The \textbf{\textit{dual graph}} of $f$ is $^fA$ (see Proposition \ref{proposition2.3} below).
\item Let $f:A\to A$ be an algebra endomorphism. Then, $(^fA)^f=\,^f(A^f)$. This bimodule is denoted by $^fA^f$ and we have the isomorphism of bimodules $^fA\otimes_A A^f\cong\, ^fA^f$ given by $a\otimes b\mapsto ab$, for $a,b\in A$.
\item The \textbf{\textit{enveloping algebra}} $A^e$ of the algebra $A$ is defined by $A^e:=A\otimes_K A^{\circ}$, where $A^{\circ}$ is the \textbf{\textit{opposite algebra}} of $A$ (in $A^{\circ}$ the product is given by $a*b:=ba$, for $a,b\in A$). If $M$ is a bimodule over $A$, then $M$ is a left $A^e$-module through $(a\otimes b)m:=amb$, for $a,b\in A$ and $m\in M$.
\item If $M$ y $N$ are bimodules over $A$, then
\begin{center}
$M\otimes N:=M\otimes_A N$ and $Hom(M,N):=Hom_A(_AM,_AN)$
\end{center}
are bimodules over $A$, with products
\begin{center}
$a(m\otimes n)b:=am\otimes nb$, $(afb)(x):=f(xa)b$, for $a,b\in A$ and $m\in M, n\in N$.
\end{center}
\item A bimodule $P$ is \textbf{\textit{invertible}} if there exists a bimodule $Q$ and bimodule isomorphisms
\begin{center}
	$\alpha: P\otimes Q\to A$ and $\beta:Q\otimes P\to A$
\end{center}
such that
\begin{equation}\label{equation2.1}
\alpha(p\otimes q)p'=p\beta(q\otimes p')\ \text{and} \ \beta(q\otimes p)q'=q\alpha(p\otimes q'), \ \text{for all} \ p,p'\in P, q,q'\in Q.
\end{equation}

Observe that if $P$ in invertible, then $Q\cong Hom(P,A)$.
\begin{proof}
We adapt to our notation the proof given in \cite{Curtis}, Theorem 3.54, page 60.
We define
\begin{center}
	$\phi:Q\to Hom(P,A)$,  $q\mapsto \phi_q:P\to A$,
	
	$\phi_q(p):=\alpha(p\otimes q)$, for $p\in P$.
\end{center}
Note that $\phi_q$ is a homomorphism of left $A$-modules. If $\phi_q=0$, then $\alpha(p\otimes q)=0$ for every $p\in P$, let $1=\beta(q_1\otimes p_1+\cdots+q_n\otimes p_n)$, so $q=1q=\beta(q_1\otimes p_1+\cdots+q_n\otimes p_n)q=q_1\alpha(p_1\otimes q)+\cdots+q_n\alpha(p_n\otimes q)=0$, i.e., $\phi$ is injective. Now let $h\in Hom(P,A)$, notice that $\phi(q_1h(p_1)+\cdots+q_nh(p_n))=h$. In fact, let $q:=q_1h(p_1)+\cdots+q_nh(p_n)\in Q$, then for $p\in P$,
\begin{center}
$\phi_q(p)=\alpha(p\otimes (q_1h(p_1)+\cdots+q_nh(p_n)))=\alpha(p\otimes q_1h(p_1))+\cdots+\alpha(p\otimes q_nh(p_n))=\alpha(p\otimes q_1)h(p_1)+\cdots+\alpha(p\otimes q_n)h(p_n)=h(\alpha(p\otimes q_1)p_1)+\cdots+h(\alpha(p\otimes q_n)p_n)=h(\alpha(p\otimes q_1)p_1+\cdots+\alpha(p\otimes q_n)p_n)=h(p\beta(q_1\otimes p_1)+\cdots+p\beta(q_n\otimes p_n))=h(p\beta(q_1\otimes p_1+\cdots+q_n\otimes p_n))=h(p1)=h(p)$.
\end{center}
Thus, $\phi_q(p)=h(p)$, for every $p\in P$, i.e., $\phi(q)=\phi_q=h$. This proves that $\phi$ is surjective.
\end{proof}

\item For an arbitrary field $K$ and $n\geq 1$, $A_n(K)\otimes _KA_n(K)\cong A_{2n}(K)$,

where the algebra isomorphism is defined by
\begin{center}
	$t_i\otimes 1\mapsto t_i$,\ \  $1\otimes t_i\mapsto t_{n+i}$,
	
		$x_i\otimes 1\mapsto x_i$,\ \  $1\otimes x_i\mapsto x_{n+i}$,
\end{center}
for $1\leq i\leq n$. Moreover, $A_n(K)\cong A_n(K)^{\circ}$, where the algebra isomorphism is defined by
\begin{center}
	$t_i\mapsto t_i$,\ \  $x_i\mapsto -x_i$, for $1\leq i\leq n$.
\end{center}
These algebra isomorphisms imply that $M$ is an $A_n(K)$-bimodule if and only if $M$ is a left $A_n(K)^e=A_n(K)\otimes_KA_n(K)^{\circ}\cong A_{2n}(K)$-module.
\end{itemize}

\begin{proposition}[\cite{Lauritzen}, Proposition 1.1]\label{proposition2.3}
Let $K$ be a field and $A$ be a $K$-algebra. Let $f:A\to A$ be an endomorphism of $A$ and $M$ be a bimodule over $A$. Then,
\begin{enumerate}
	\item[\rm (i)]$M\otimes A^f\cong M^f$.
	\item[\rm (ii)]$Hom(A^f,A)\cong\, ^fA$.
	\item[\rm (iii)]Assume that $A^*=K^*$. $f$ is invertible if and only if $A^f$ is invertible. In this case, $(A^f)^{-1}\cong\, ^fA$.
	\item[\rm (iv)]If $f$ is invertible, then $Ann_{A^e}(^{f^{-1}}A)=Ann_{A^e}(A^f)$.
\end{enumerate}
\end{proposition}
\begin{proof}
	(i) The isomorphism of bimodules is given by $m\otimes a\mapsto ma$, for $m\in M$ and $a\in A^f$.
	
	(ii) In this case the isomorphism of bimodules is given by $\phi \mapsto \phi(1)$ , for $\phi\in Hom(A^f,A)$.
	
	(iii) $\Rightarrow)$: According to (i), $A^f\otimes A^{f^{-1}}\cong (A^f)^{f^{-1}}$ ($a\otimes b\mapsto ab$), but $(A^f)^{f^{-1}}\cong A$ ($x\mapsto x$), thus $A^f\otimes A^{f^{-1}}\cong A$. Similarly, $A^{f^{-1}}\otimes A^f\cong (A^{f^{-1}})^f\cong A$. These isomorphisms trivially satisfy (\ref{equation2.1}). This means that $A^f$ is invertible with inverse $A^{f^{-1}}$. But $A^{f^{-1}}\cong \, ^fA$, where the isomorphism is namely $f$. Thus, $(A^f)^{-1}=\,^fA$.
	
	$\Leftarrow)$: Assume that $A^f$ is invertible, then there exists $Q$ such that $A^f\otimes Q\cong A$ and $Q\otimes A^f\cong A$, moreover, $Q\cong Hom(A^f,A)\cong\, ^fA$. Hence, $(A^f)^{-1}=\,^fA$. Thus, we have a bimodule isomorphism $A\cong\, ^fA\otimes A^f$, but observe the bimodule isomorphism $^fA\otimes A^f\cong \, ^fA^f$ given by $a\otimes b\mapsto ab$. Therefore, we have a bimodule isomorphism $\phi: A\to\, ^fA^f$ such that for every $a\in A$, $\phi(a)=\phi(a1)=a\cdot \phi(1)=f(a)\phi(1)=\phi(1a)=\phi(1)\cdot a=\phi(1)f(a)$. Since $\phi$ is injective then $f$ is injective. In order to prove that $\phi$ is surjective observe first that $\phi(1)\in A^*$: Indeed, since $\phi$ is surjective there exists $x\in A$ such that $\phi(x)=1$, so $1=f(x)\phi(1)=\phi(1)f(x)$. Now, let $a\in A$, then there exists $b\in A$ such that $\phi(b)=a$, hence $a=\phi(b)=\phi(1)f(b)$, but $\phi(1)\in A^*=K^*$ and $f$ is $K$-linear, so $a=f(\phi(1)b)$.
	
	(iv) $a_1\otimes b_1+\cdots+a_n\otimes b_n\in Ann_{A^e}(^{f^{-1}}A)$ if and only if $(a_1\otimes b_1+\cdots+a_n\otimes b_n)1=0$ if and only if $f^{-1}(a_1)b_1+\cdots+f^{-1}(a_n)b_n=0$ if and only if $a_1f(b_1)+\cdots+a_nf(b_n)=0$ if and only if $(a_1\otimes b_1+\cdots+a_n\otimes b_n)1=0$ if and only if $a_1\otimes b_1+\cdots+a_n\otimes b_n\in Ann_{A^e}(A^f)$.
\end{proof}

\begin{proposition}[\cite{Lauritzen}, Corollary 2.6]
Let $K$ be a field of characteristic zero and let $A:=A_n(K)$. If $f:A\to A$ is an algebra endomorphism, then the graph $A^f$ and the dual graph $^fA$ are simple bimodules.
\end{proposition}

\begin{proposition}[\cite{Lauritzen}, page 168]\label{proposition2.5}
	Let $K$ be a field of characteristic zero and let $A:=A_n(K)$. The following conditions are equivalent:
\begin{enumerate}
	\item[\rm (i)]$DC_n$.
	\item[\rm (ii)]If $f:A\to A$ is an algebra endomorphism, then the bimodule $^fA\otimes_A A^f$ is simple.
\end{enumerate}
\end{proposition}
\begin{proof}
Recall first the bimodule isomorphism $^fA\otimes_A A^f\cong \,^fA^f$.

(i)$\Rightarrow$(ii): Let $N\neq 0$ be a bisubmodule of $^fA^f$, then observe that $N$ is a two-sided ideal of $A$: In fact, let $a,b\in A$ and $n\in N$, then since $f$ is surjective there exist $a',b'\in A$ such that $f(a')=a$ and $f(b')=b$, so $anb=f(a')nf(b')=a'nb'\in N$. Therefore, $N=\,^fA^f$.

(ii)$\Rightarrow$(i): $Im(f)$ is a nonzero bisubmodule of $^fA^f$, so $Im(f)=A$, hence $f$ is an automorphism.
\end{proof}

\begin{proposition}[\cite{Lauritzen}, Theorem 3.9]\label{proposition2.6}
Let $K$ be an arbitrary field and let $A:=A_n(K)$. Let $f:A\to A$ be an algebra endomorphism of $A$. Consider the left ideal $J$ of $A^e$ defined by
\begin{center}
$J:=\left\langle 1\otimes t_1-f(t_1)\otimes 1, \dots, 1\otimes t_n-f(t_n)\otimes 1,1\otimes \partial_1-f(\partial_1)\otimes 1, \dots, 1\otimes \partial_n-f(\partial_n)\otimes 1\right\rbrace $.
\end{center}
Then,
\begin{enumerate}
\item[\rm (i)]$J=Ann_{A^e}(A^f)$.
\item[\rm (ii)]If $f$ is invertible, then
\begin{center}
	{\small
	$J=\left\langle t_1\otimes 1-1\otimes f^{-1}(t_1), \dots, t_n\otimes 1-1\otimes f^{-1}(t_n),\partial_1\otimes 1-1\otimes f^{-1}(\partial_1), \dots, \partial_n\otimes 1-1\otimes f^{-1}(\partial_n)\right\rbrace $.}
\end{center}
\item[\rm (iii)]If
\begin{center}
	$J=\left\langle t_1\otimes 1-1\otimes q_1, \dots, t_n\otimes 1-1\otimes q_n,\partial_1\otimes 1-1\otimes p_1, \dots, \partial_n\otimes 1-1\otimes p_n\right\rbrace $,
\end{center}
	for some $q_i,p_i\in A^{\circ}$, $1\leq i\leq n$, then $f$ is invertible and $q_i=f^{-1}(t_i)$, $p_i=f^{-1}(\partial_i)$.
\end{enumerate}
\end{proposition}

\subsection{The quantum analogue of the Dixmier conjecture}

In this subsection we present the study given in \cite{Tang} about the Dixmier conjecture for some simple localizations of two quantum algebras, namely, for the ``symmetric" multiparameter quantized Weyl algebra $\mathcal{A}_n^{\overline{q},\Lambda}(K)$ and the Maltsiniotis multiparameter quantized Weyl algebra $A_n^{\overline{q},\Gamma}(K)$. The Dixmier problem for localizations of quantum algebras defined by the submonoid generated by some normal elements is known in the literature as \textbf{\textit{the quantum analogue of the Dixmier conjecture}} (see \cite{Tang} and \cite{Tang2}). We start recalling the definition of algebras $\mathcal{A}_n^{\overline{q},\Lambda}(K)$ and $A_n^{\overline{q},\Gamma}(K)$ given in \cite{Tang} (the algebra of Maltsiniotis was introduced first in \cite{Maltsiniotis}). We remark that these two algebras are particular examples of the skew $PBW$ extensions that we will consider in Section \ref{Section4}.

\begin{definition}[\cite{Tang}]
Let $K$ be an arbitrary field, $n\geq 1$ and $\overline{q}:=(q_1,\dots,q_n)\in (K^*)^n$.
\begin{enumerate}
\item[\rm (i)]Let $\Gamma:=[\gamma_{ij}]\in M_n(K^*)$ be a multiplicatively skew-symmetric matrix. The \textbf{Maltsiniotis multiparameter quantized Weyl algebra} $A_n^{\overline{q},\Gamma}(K)$ is the $K$-algebra generated by $x_1,y_1,\dots,x_n,y_n$ subject to the following relations:
\begin{center}
$y_iy_j=\gamma_{ij}y_jy_i$, $1\leq i,j\leq n$;

$x_ix_j=q_i\gamma_{ij}x_jx_i$, $1\leq i<j\leq n$;

$x_iy_j=\gamma_{ji}y_jx_i$, $1\leq i<j\leq n$;

$x_iy_j=q_i\gamma_{ji}y_jx_i$, $1\leq j<i\leq n$;

$x_iy_i-q_iy_ix_i=1+\sum_{k=1}^{i-1}(q_k-1)y_kx_k$, $1\leq i\leq n$.
\end{center}
\item[\rm (ii)]Let $\Lambda:=[\lambda_{ij}]\in M_n(K^*)$ with $\lambda_{ij}=\lambda_{ji}^{-1}$ and $\lambda_{ii}=1$. The \textbf{``symmetric" multiparameter quantized Weyl algebra} $\mathcal{A}_n^{\overline{q},\Lambda}(K)$ is the $K$-algebra generated by $x_1,y_1,\dots,x_n,y_n$ subject to the following relations:
\begin{center}
	$y_jy_i=\lambda_{ji}y_iy_j$, $1\leq i<j\leq n$;
	
	$x_ix_j=\lambda_{ij}x_jx_i$, $1\leq i<j\leq n$;
	
	$x_iy_j=\lambda_{ji}y_jx_i$, $1\leq i<j\leq n$;
	
	$x_jy_i=\lambda_{ij}y_ix_j$, $1\leq i<j\leq n$;
	
	$x_iy_i-q_iy_ix_i=1$, $1\leq i\leq n$.
\end{center}
\end{enumerate}  	
\end{definition}

\begin{remark}
We recall next some remarks on the previous definitions (see \cite{Tang}).
\\
(a) We have copied the definitions as were given in \cite{Tang}, however observe that in (ii) the conditions on $\Lambda$ also means that $\Lambda$ is multiplicatively skew-symmetric.
\\
(b) $A_n^{\overline{q},\Gamma}(K)$ and $\mathcal{A}_n^{\overline{q},\Lambda}(K)$ are non simple algebras.
\\
(c) $Z(A_n^{\overline{q},\Gamma}(K))=K=Z(\mathcal{A}_n^{\overline{q},\Lambda}(K))$ when none of parameters $q_i$ is a root of unity.
\\
(d) When $n=1$ and $q=1$, both $\mathcal{A}_n^{\overline{q},\Lambda}(K)$ and $\mathcal{A}_n^{\overline{q},\Gamma}(K)$ coincides with $A_1(K)$.
\\
(e) When $n=1$, both $\mathcal{A}_n^{\overline{q},\Lambda}(K)$ and $\mathcal{A}_n^{\overline{q},\Gamma}(K)$ coincides with $A_1^q(K)$.
\\
(f) Related with the isomorphism
\begin{center}
$A_n(K)\cong A_1(K)^{\otimes n}:=A_1(K)\otimes \cdots \otimes A_1(K)$
\end{center}
(see \cite{Levandovsky}), in our context we have that if $\lambda_{ij}=1$ for every $1\leq i,j\leq n$, then
\begin{center}
$\mathcal{A}_n^{\overline{q},\Lambda}(K)\cong A_1^{q_1}(K)\otimes \cdots \otimes A_1^{q_n}(K)$.
\end{center}
\end{remark}

Next we present the main results of \cite{Tang} related to the the quantum analogue of the Dixmier conjecture.

\begin{proposition}\label{proposition2.10}
{\rm (a)} The algebra $\mathcal{A}_n^{\overline{q},\Lambda}(K)$ has the following properties:
\begin{enumerate}
	\item[\rm (i)]For every $1\leq i\leq n$, $z_i:=x_iy_i-y_ix_i$ is normal.
	\item[\rm (ii)]Let $\mathcal{Z}$ be the submonoid of $\mathcal{A}_n^{\overline{q},\Lambda}(K)$ generated by the elements $z_1,\dots,z_n$. Then $\mathcal{Z}$ is an Ore set of $A$. If $q_i$ is not a root of unity for any $1\leq i\leq n$, then the localization $\mathcal{A}_n^{\overline{q},\Lambda}(K)_\mathcal{Z}$ is a simple algebra.
	\item[\rm (iii)]If $q_i$ is not a root of unity for any $1\leq i\leq n$, then the set $N$ of all normal elements of $\mathcal{A}_n^{\overline{q},\Lambda}(K)$ is given by
	\begin{center}
		$N=\{az_1^{l_1}\cdots z_n^{l_n}\mid a\in K, l_i\geq 0, 1\leq i\leq n\}$.
	\end{center}
	\item[\rm (iv)] $($\textbf{Quantum analogue of the Dixmier conjecture for} $\boldsymbol{\mathcal{A}_n^{\overline{q},\Lambda}(K)}$$)$ Assume that $q_1^{i_1}\cdots q_n^{i_n}=1$ implies that $i_1=\cdots=i_n=0$. Then every algebra endomorphism of  $\mathcal{A}_n^{\overline{q},\Lambda}(K)_{\mathcal{Z}}$ is an automorphism.
	\end{enumerate}
{\rm (b)} $($\textbf{Quantum analogue of the Dixmier conjecture for} $\boldsymbol{A_n^{\overline{q},\Gamma}(K)}$$)$ Assume that $q_1^{i_1}\cdots q_n^{i_n}=1$ implies that $i_1=\cdots=i_n=0$. Then every algebra endomorphism of  $\boldsymbol{A_n^{\overline{q},\Gamma}(K)}_{\mathcal{Z'}}$ is an automorphism.
\end{proposition}
\begin{proof}
(i) We will prove that $z_i\mathcal{A}_n^{\overline{q},\Lambda}(K)\subseteq \mathcal{A}_n^{\overline{q},\Lambda}(K)z_i$. Similarly can be proved that $\mathcal{A}_n^{\overline{q},\Lambda}(K)z_i\subseteq z_i\mathcal{A}_n^{\overline{q},\Lambda}(K)$.

For $a\in K$, it is clear that $z_ia=az_i$. Now,
\begin{center}
$z_ix_i=(x_iy_i-y_ix_i)x_i=x_iy_ix_i-y_ix_i^2=x_i(q_i^{-1}x_iy_i-q_i^{-1})-(q_i^{-1}x_iy_i-q_i^{-1})x_i=q_i^{-1}x_i(x_iy_i-1)-q_i^{-1}(x_iy_i-1)x_i=
q_i^{-1}x_i(x_iy_i-1)-q_i^{-1}x_iy_ix_i+q_i^{-1}x_i=q_i^{-1}x_i(x_iy_i-y_ix_i)=q_i^{-1}x_iz_i\in \mathcal{A}_n^{\overline{q},\Lambda}(K)z_i$;

\smallskip

for $i<j$, $z_ix_j=(x_iy_i-y_ix_i)x_j=x_iy_ix_j-y_ix_ix_j=x_i\lambda_{ji}x_jy_i-y_i\lambda_{ij}x_jx_i=\lambda_{ji}x_ix_jy_i-\lambda_{ij}y_ix_jx_i=
\lambda_{ji}\lambda_{ij}x_jx_iy_i-\lambda_{ij}\lambda_{ji}x_jy_ix_i=x_jz_i\in \mathcal{A}_n^{\overline{q},\Lambda}(K)z_i$;

\smallskip

for $i>j$, $z_ix_j=(x_iy_i-y_ix_i)x_j=x_iy_ix_j-y_ix_ix_j=x_i\lambda_{ji}x_jy_i-y_i\lambda_{ij}x_jx_i=\lambda_{ji}x_ix_jy_i-\lambda_{ij}y_ix_jx_i=
\lambda_{ji}\lambda_{ij}x_jx_iy_i-\lambda_{ij}\lambda_{ji}x_jy_ix_i=x_jz_i\in \mathcal{A}_n^{\overline{q},\Lambda}(K)z_i$;

\smallskip

$z_iy_i=(x_iy_i-y_ix_i)y_i=x_iy_i^2-y_ix_iy_i=q_iy_ix_iy_i-y_iq_iy_ix_i=q_iy_iz_i\in \mathcal{A}_n^{\overline{q},\Lambda}(K)z_i$;

\smallskip

for $i<j$, $z_iy_j=(x_iy_i-y_ix_i)y_j=x_iy_iy_j-y_ix_iy_j=x_i\lambda_{ij}y_jy_i-y_i\lambda_{ji}y_jx_i=\lambda_{ij}\lambda_{ji}y_jx_iy_i-\lambda_{ji}\lambda_{ij}y_jy_ix_i=y_jz_i\in \mathcal{A}_n^{\overline{q},\Lambda}(K)z_i$;

\smallskip

for $i>j$, $z_iy_j=(x_iy_i-y_ix_i)y_j=x_iy_iy_j-y_ix_iy_j=x_i\lambda_{ij}y_jy_i-y_i\lambda_{ji}y_jx_i=\lambda_{ij}\lambda_{ji}y_jx_iy_i-\lambda_{ji}\lambda_{ij}y_jy_ix_i=y_jz_i\in \mathcal{A}_n^{\overline{q},\Lambda}(K)z_i$.
\end{center}
From these computations we get that if $p\in \mathcal{A}_n^{\overline{q},\Lambda}(K)$, then $z_ip\in \mathcal{A}_n^{\overline{q},\Lambda}(K)z_i$, so $z_i\mathcal{A}_n^{\overline{q},\Lambda}(K)\subseteq \mathcal{A}_n^{\overline{q},\Lambda}(K)z_i$.

(ii) The proof of this part is due to M. Akhavizadegan and D. A. Jordan. In fact, for every $1\leq i\leq n$, let $z_i':=x_iy_i-y_ix_i\in A_n^{\overline{q},\Gamma}(K)$ and let $\mathcal{Z}'$ the submonoid of $A_n^{\overline{q},\Gamma}(K)$ generated by $z_1',\dots,z_n'$. In \cite{Jordan} was proved the algebra isomorphism $\mathcal{A}_n^{\overline{q},\Lambda}(K)_{\mathcal{Z}}\cong A_n^{\overline{q},\Gamma}(K)_{\mathcal{Z}'}$. Theorem 3.2 in \cite{Jordan2} says that $A_n^{\overline{q},\Gamma}(K)_{\mathcal{Z}'}$ is a simple algebra when $q_i$ is not a root of unity for any $1\leq i\leq n$. Thus, $\mathcal{A}_n^{\overline{q},\Lambda}(K)_\mathcal{Z}$ is simple.

(iii) For the proof of this part see Corollary 1.1 of \cite{Tang}.

(iv) See the proof of Theorem 3.1 of \cite{Tang}.

(b) This follows from (a)-(iv) and the isomorphism in the proof of (a)-(ii).
\end{proof}

\section{Dixmier algebras and rings}

As was observed in the previous section, in the mathematical literature the Dixmier conjecture has been investigated for algebras over fields, both for simple algebras and for non-simple algebras. In the simple case every endomorphism is a monomorphism. In this section we will study the Dixmier problem for algebras over commutative rings, in particular, for $\mathbb{Z}$-algebras, i.e., for rings. For this, we introduce next two classes of algebras and rings induced by the Dixmier problem. A more general situation that includes algebras and rings is also defined. We will use this general notion in the next section.

\subsection{Definitions and examples}\label{subsection3.1}

\begin{definition}\label{definition3.1}
Let $R$ and $A$ be rings such that $A$ is a left $R$-module. We say that $A$ is \textbf{\textit{weakly Dixmier $(\boldsymbol{WD})$ with respect to $\boldsymbol{R}$}}, if every $R$-linear ring monomorphism $\phi:A\to A$ is an automorphism. $A$ is \textbf{\textit{Dixmier $(\boldsymbol{D})$ with respect to $\boldsymbol{R}$}}, if every $R$-linear ring endomorphism $\phi$ of $A$ is an automorphism. In particular, if $R=Z$ is a commutative ring and $A$ is an $Z$-algebra, we say that $A$ is a \textbf{\textit{weakly Dixmier algebra}} if every algebra monomorphism $\phi:A\to A$ is an automorphism. $A$ is a \textbf{\textit{Dixmier algebra}} if every algebra endomorphism $\phi $ of $A$ is an automorphism. If $R=\mathbb{Z}$ we say that $A$ is a \textbf{weakly Dixmier ring} and a \textbf{Dixmier ring}, respectively.
\end{definition}

It is clear that

\begin{center}
$D\Rightarrow WD$.
\end{center}

Next we present some trivial examples and counterxamples of $D$ and $WD$ algebras and rings.
\begin{example}
	(i) Any finite ring is $WD$, however, the following example shows a finite ring that is not $D$: Consider the ring of \textbf{\textit{upper triangular matrices}} over $\mathbb{Z}_2$
	\begin{center}
		$T_2(\mathbb{Z}_2):=\{\begin{bmatrix} a & b\\
		0 & c
		\end{bmatrix}\mid a,b,c\in \mathbb{Z}_2\}$
	\end{center}
and the ring endomorphism given by
	\begin{align*}
T_2(\mathbb{Z}_2) & \xrightarrow{\phi} T_2(\mathbb{Z}_2)\\
\begin{bmatrix} a & b\\
0 & c
\end{bmatrix} & \mapsto \begin{bmatrix} a & 0\\
0 & c
\end{bmatrix}.
	\end{align*}
Clearly $\phi$ is not injective and also not surjective.
	
	(ii) The ring $\mathbb{Z}$ of integers is $D$ since the only ring endomorphism of $\mathbb{Z}$ is the identical homomorphism; the same is true for the ring $\mathbb{Z}_n$ of integers modulo $n\geq 2$, the field $\mathbb{Q}$ of rational numbers and the field $\mathbb{R}$ of real numbers.
	
	(iii) The field $\mathbb{C}$ of complex numbers is not a $WD$ ring. In fact, it is well-known that there are ring monomorphisms(=endomorphisms) of $\mathbb{C}$ that are not surjective, see \cite{Yale}, Section 6. However, considering $\mathbb{C}$ as a $\mathbb{C}$-algebra, $\mathbb{C}$ is $D$ since the only algebra endomorphism of $\mathbb{C}$ is the identical homomorphism. In addition, Theorem 3 of \cite{Yale} says that $\mathbb{C}$ as $\mathbb{R}$-algebra is $D$ since the only $\mathbb{R}$-algebra endomorphisms of $\mathbb{C}$ are the identical homomorphism and the complex conjugation.
	
	(iv) The division ring $\mathbb{H}$ of quaternions is not a $WD$ ring: By (iii), let $\phi:\mathbb{C}\to \mathbb{C}$ be a ring monomorphism of $\mathbb{C}$ that is not surjective, then $\phi$ induces the non surjective ring monomorphism $\widetilde{\phi}$ given by
	\begin{align*}
	\mathbb{H} & \xrightarrow{\widetilde{\phi}} \mathbb{H}\\
	\begin{bmatrix} z & w\\
	-\overline{w} & \overline{z}
	\end{bmatrix} & \mapsto \begin{bmatrix} \phi(z) & \phi(w)\\
	-\overline{\phi(w)} & \overline{\phi(z)}
	\end{bmatrix}.
	\end{align*}
	
	(v) Let $R$ and $A$ be rings such that $A$ is a left $R$-module. Observe that
	\begin{center}
	if $A$ is a $D$ ($WD$) ring, then $A$ is a $D$ ($WD$) with respect to $R$,
	\end{center}
	
	but as we just noticed in (iii), the converse is not true.
	
	(vi) As was observed in Subsection \ref{subsection1.1}, if $char(K)=p>0$, then $A_1(K)$ is not a $WD$ algebra, and hence, $A_1(K)$ is not a $WD$ ring. The same is true for $A_1^q(\mathbb{C})$, with $q\neq 1$.
	
	(vii) The field $\mathbb{C}$ of complex numbers shows that the following well-known classes of rings, in general, are not Dixmier (for a precise definition of these rings see \cite{Lam1}, \cite{McConnell} or also \cite{Lezama2}): artinian, noetherian, simple, semisimple, local, semilocal, prime, semiprime, regular, von Neumann regular, Goldie, $FBN$, hereditary, semihereditary, perfect, semiperfect, and $PI$.   
	\end{example}

Next we examine the Dixmier condition for the most elementary algebraic constructions: Let $A$ be a $Z$-algebra that is $D$ ($WD$). Are $D$ ($WD$) the following algebras?
\begin{enumerate}
\item[(i)] The quoient algebra $A/I$, where $I$ is a proper two-sided ideal of $A$.
\item[(ii)] $B$ a proper subalgebra of $A$.
\item[(iii)] $C$ a proper extension of $A$.
\item[(iv)] The polynomial algebra $A[t]$.
\item[(v)] The free ring $A\{X\}$, where $X\neq \emptyset$ is an alphabet. The group algebra $A[G]$, where $G$ is a group.
\item[(vi)] The algebra of series $A[[t]]$.
\item[(vii)] The Ore extension $A[x;\sigma, \delta]$ ($\sigma$ is an algebra endomorphism of $A$ and $\delta$ is a $\boldsymbol{\sigma}$-\textbf{\textit{derivation}} of $A$, i.e., $\delta(a+b)=\delta(a)+\delta(b)$, $\delta(z\cdot a)=z\cdot \delta(a)$, $ \delta(ab)=\sigma(a)\delta(b)+\delta(a)b$, for $a,b\in A$ and $z\in Z$, \cite{Ore}).
\item[(viii)] The algebra $M_n(A)$ of square matrices over $A$, $n\geq 2$.
\item[(ix)] $A\times B$, where $B$ is another $D$ ($WD$) algebra.
\item[(x)] $A\otimes_Z B$, where $B$ is another $D$ ($WD$) algebra.
\item[(xi)] The localization $A_S:=AS^{-1}$, where $S$ is a right Ore subset of $A$.
\end{enumerate}

\begin{example}\label{example3.3}
In this example we will answer the above questions.

(i) False. As was noticed in Subsection \ref{subsection1.1}, $\mathbb{I}_1$ is a $D$ algebra, but in the remark of page 239 of \cite{Bavula1} was proved that there exits a monomorphism of the quotient algebra $\mathbb{I}_1/F$ that is not an automorphism, hence, $\mathbb{I}_1/F$ is not a $WD$ algebra, where $F$ is the only non trivial two-sided ideal of $\mathbb{I}_1$ (for a precise definition of $F$ see page 241 of \cite{Bavula1}).

(ii) False. We noticed in Subsection \ref{subsection1.1} that if $q$ is not a root of unity, then the simple localization $A_1^q(\mathbb{C})_\mathcal{Z}$ of $A_1^q(\mathbb{C})$ is a $D$ algebra, but $A_1^q(\mathbb{C})\subset A_1^q(\mathbb{C})_\mathcal{Z}$  and $A_1^q(\mathbb{C})$ is not $WD$.

(iii) False. $\mathbb{Z}$ is $D$, but $\mathbb{Z}\subset \mathbb{Z}[t]$ and $\mathbb{Z}[t]$ is not $WD$: the monomorphism $\mathbb{Z}[t]\to \mathbb{Z}[t]$ defined by $t\mapsto t^2$ is not an automorphism. Another counterexample, but $\mathbb{C}$-algebras, is $\mathbb{C}\subset A_1^q(\mathbb{C})$, with $q\neq 1$.

(iv) False. The first counterexample of (iii) applies. More general, observe that for any algebra $A$, $A[t]$ is not $WD$.

(v) False. $A\{X\}$ is a $Z$-algebra that is not $D$ since the endomorphism defined by $p(X)\mapsto p(0)$ is not surjective. Now, let $G$ be an abelian group. Notice that for any algebra $A$, $A[G]$ is not $WD$: the monomorphism $A[G]\to A[G]$ defined by $\sum a_g\cdot g\mapsto \sum a_g\cdot g^{2}$ is not an automorphism (recall that $A[G]$ is a free left $A$-module with basis $G$ and $a_g=0$ for almost every $g\in G$). 

(vi) False. $\mathbb{Z}$ is $D$, but $\mathbb{Z}\subset \mathbb{Z}[[t]]$ and $\mathbb{Z}[[t]]$ is not $WD$: the monomorphism $\mathbb{Z}[[t]]\to \mathbb{Z}[[t]]$ defined by $\sum a_it^i\mapsto \sum a_it^{2i}$ is not an automorphism. In general, observe that for any algebra $A$, $A[[t]]$ is not $WD$.

(vii) False. $\mathbb{Z}$ is $D$, but $\mathbb{Z}[x;\sigma,\delta]$ is not $D$ ($WD$), where $\sigma=i_\mathbb{Z}$ and $\delta=0$. Thus, the same example of (iii) applies in this case. A less trivial counterexample can be constructed from the results in \cite{Lam}. Let $K$ be a division ring and $\phi:K[x;\sigma,\delta]\to K[x;\sigma,\delta]$ be a ring homomorphism; recall that $\phi$ is uniquely determined by a polynomial $p:=\phi(x)\in K[x;\sigma,\delta]$ such that $pa=\sigma(a)p+\delta(a)$, for every $a\in K$; moreover, $\phi(a):=a$ for every $a\in K$ (universal property of Ore extensions, see \cite{Ore}, \cite{McConnell}, or also \cite{Lezama-sigmaPBW} and \cite{Lezama2}). If $p\notin K$, then $\phi$ is injective: Let $q:=q_0+q_1x+\cdots+q_mx^m\in K[x;\sigma,\delta]$ such that $\phi(q)=0=q_0+q_1\phi(x)+\cdots+q_m\phi(x)^m=q_0+q_1p+\cdots+q_mp^m$; if $m=0$, then $q=0$; let $m\geq 1$ and let $p:=p_0+p_1x+\cdots+p_rx^r$, with $r\geq 1$, $p_r\neq 0$, then considering the leader term of $q_mp^m$ we get that $q_mp_r\sigma^r(p_r)\sigma^{2r}(p_r)\cdots\sigma^{(m-1)r}(p_r)=0$, but since $K$ is a division ring, then $\sigma$ is injective, so $q_m=0$. Thus, the leader term of $q$ is zero, i.e., $q=0$. Now, $Im(\phi)$ is the subring $K[p]$ of $K[x;\sigma,\delta]$ generated by $K$ and $p$. In general, $K[p]\neq K[x;\sigma,\delta]$ as the following example shows: Let $K=\mathbb{C}$, $\sigma(a):=\overline{a}$, for $a\in \mathbb{C}$, $\delta:=0$ and $p:=x^3$. Observe that $pa=\sigma(a)p+\delta(a)$, for every $a\in \mathbb{C}$. Then, $\phi:\mathbb{C}[x;\sigma,\delta]\to \mathbb{C}[x;\sigma,\delta]$ given by $\phi(x):=p$, $\phi(a):=a$, $a\in \mathbb{C}$, is an $\mathbb{R}$-algebra monomorphism. Notice that $Im(\phi)\neq \mathbb{C}[x;\sigma,\delta]$ since $x\notin Im(\phi)$. Thus, $\mathbb{C}$ is a Dixmier $\mathbb{R}$-algebra, but $\mathbb{C}[x;\sigma,\delta]$ is an $\mathbb{R}$-algebra that is not $WD$.

(viii) See Theorem  \ref{proposition3.4} below.

(ix) False. $\mathbb{Z}$ is $D$ but $\mathbb{Z}\times \mathbb{Z}$ is not $D$: The map $\phi: \mathbb{Z}\times \mathbb{Z}\to \mathbb{Z}\times \mathbb{Z}$ given by $\phi[(n,m)]:=(n,n)$ is an endomorphism of $ \mathbb{Z}\times \mathbb{Z}$ that is not an automorphism. However, if the product of two algebras is $D$ ($WD$), then each factor is $D$ ($WD$), see Theorem \ref{proposition3.5} below.

(x) See Theorem \ref{theorem3.6} and Remark \ref{reamrk3.9} below.

(xi) With respect to the localization construction we can make the following remarks. As was observed in Subsection \ref{subsection1.1}, there exists a localization of the first Weyl algebra $A_1(K)$ that is not $D$. So, if $DC_1$ is true (see \cite{Zheglov}), then the answer to (xi) is False. On the other hand, the quantized Weyl algebra $A_1^q(\mathbb{C})$ is not $WD$ (whence not $D$), but if $q$ is not a root of unity, then $A_1^q(\mathbb{C})_\mathcal{Z}$ is $D$ ($WD$). Thus, localizations of algebras can or not give examples of $D$ ($WD$) algebras, even if the algebras are or not $D$ ($WD$). For a final remark to this construction see Theorem \ref{proposition3.7} below.

\end{example}

\subsection{Some elementary results on Dixmier rings and algebras}\label{subsection3.2}

Related to some of the algebraic constructions considered in the previous subsection, we have the following easy facts.

\begin{theorem}\label{proposition3.4}
Let $R$ and $A$ be rings such that $A$ is a left $R$-module, and let $n\geq 1$. If $M_n(A)$ is $D$ $(WD)$ with respect to $R$, then $A$ is $D$ $(WD)$ with respect to $R$.
\end{theorem}
\begin{proof}
Assume that $A$ is not $D$, then $M_n(A)$ is not $D$ since every $R$-linear ring endomorphism $\phi $ of $A$ induces an $R$-linear ring endomorphism $\Phi$ of $M_n(A)$ given by $[x_{ij}]\mapsto [\phi(x_{ij})]$. Similarly for $WD$.
\end{proof}

With respect to the converse of the previous theorem we have the following easy and partial result.

\begin{theorem}\label{theorem3.5}
Let $n\geq 1$ and $A$ be a ring that satisfies the following conditions:
\begin{enumerate}
\item[\rm (i)]$A$ is $D$ $(WD)$.
\item[\rm (ii)]$A$ is simple.
\item[\rm (iii)]There exist in $A$ elements $e_{ij}, 1\leq i,j\leq n$, such that
$$1=e_{11}+\cdots+e_{nn}$$
and
\begin{center}
	$e_{ij}e_{lk}=\left\{
	\begin{array}{cc}
	0, & \text{if }j\neq l\text{ \ } \\
	e_{ik}, & \text{if }j=l.
	\end{array}%
	\right. $
\end{center}
Then $M_n(A)$ is $D$ $(WD)$.
\end{enumerate} 
\end{theorem}
\begin{proof}
We divide the proof in two steps.

\textit{Step 1}. \textit{Universal property of ring of matrices}. Let $A$ be a ring and $n\geq 1$. Let $A_0$ be a ring that satisfies the following conditions: 
\begin{enumerate}
	\item[\rm (a)]There exists a ring homomorphism $g:A\to A_0$.
	\item[\rm (b)]There exist in $A_0$ elements $h_{ij}$, $1\leq i,j\leq
	n$, such that:
	\begin{enumerate}
		\item[\rm (a)]$1=h_{11}+\cdots+h_{nn}$.
		\item[\rm (b)]For any $x,y\in A$,
		\begin{center}
			$h_{ij}g(x)h_{lk}g(y)=\left\{
			\begin{array}{cc}
			0, & \text{if }j\neq l\text{ \ } \\
			h_{ik}g(x)g(y), & \text{if }j=l.
			\end{array}%
			\right. $
		\end{center}
	\end{enumerate}
\end{enumerate}
Then, there exists an unique ring homomorphism 
$\overline{g}:M_n(A)\to A_0$ such that $\overline{g}d=g$ and
$\overline{g}(E_{ij})=h_{ij}$, $1\leq i,j\leq n$, with $d:A\to
M_n(A)$ the ring homorphism defined by 
$d(x):=E_{11}x+\cdots+E_{nn}x$, where the $E_{ij}$ are the elementary matrices of $M_n(A)$:
\[
\begin{diagram}
\node{} \node{A} \arrow{s,t}{g} \arrow{e,t}{d} \node{M_n(A)} \arrow{sw,r,..}{\overline{g}}\\
\node{} \node{A_0}
\end{diagram}
\]
Moreover, if $g$ is surjective, then $\overline{g}$ is surjective.

In fact, any element $F=[f_{ij}]\in M_n(A)$ can represented an unique way as $F=\sum_{i,j}^n E_{ij}f_{ij}$, so
we define $\overline{g}(F):=\sum_{i,j}^n h_{ij}g(f_{ij})$. It is clear that 
$\overline{g}$ is additive, so, by (a),
$\overline{g}(E)=h_{11}+\cdots+h_{nn}=1$. From (b), and since $\overline{g}$ is additive and multiplicative, we get that $\overline{g}$ is multiplicative.
Now, let $x\in A$, then
$\overline{g}d(x)=\overline{g}(E_{11}x+\cdots+E_{nn}x)=h_{11}g(x)+\cdots+h_{nn}g(x)=(h_{11}+\cdots+h_{nn})g(x)=g(x)$,
i.e., $\overline{g}d=g$. Let $h:M_n(A)\to
A_0$ be another homomorphism such that $hd=g$ and $h(E_{ij})=h_{ij}$, $1\leq
i,j\leq n$, then
$h(E_{ij}f_{ij})=h(E_{ij}(E_{11}f_{ij}+\cdots+E_{nn}f_{ij}))=h(E_{ij})h(E_{11}f_{ij}+\cdots+E_{nn}f_{ij})=
h(E_{ij})hd(f_{ij})=h_{ij}g(f_{ij})=\overline{g}(E_{ij}f_{ij})$,
whence $h(F)=\overline{g}(F)$, i.e., $h=\overline{g}$.

Finally, given $z\in A_0$, there exists $x\in A$ such that $g(x)=z$,
whence $z=\overline{g}(d(x))$, i.e., $\overline{g}$ is surjective.

\textit{Step 2}. Let $\Phi:M_n(A)\to M_n(A)$ be a ring homomorphism and the $e_{ij}$ as in (iii). From (ii), $\Phi$ is injective. We have to prove that $\Phi$ is surjective. From \textit{Step 1}, with $g=i_A$ and $h_{ij}=e_{ij}$, there exists a ring homomorphism $\overline{g}:M_n(A)\to A$ such that $\overline{g}\circ d=i_A$ and $\overline{g}(E_{ij})=e_{ij}$, $1\leq i,j\leq n$:
\[
\begin{diagram}
\node{} \node{A} \arrow{s,t}{i_A} \arrow{e,t}{d} \node{M_n(A)} \arrow{sw,r,..}{\overline{g}}\\
\node{} \node{A}
\end{diagram}
\] 
We define $\Phi_A:=\overline{g}\circ \Phi\circ d$. Then, $\Phi_A$ is an endomorphism of $A$, and by (i), $\Phi_A$ is bijective. Let $F\in M_n(A)$, then $\overline{g}(F)\in A$ and there exists $a\in A$ such that $\Phi_A(a)=\overline{g}(F)$. Thereby, $\overline{g}(F)=\Phi_A(a)=\overline{g}\circ \Phi\circ d(a)=\overline{g}(\Phi\circ d(a))$, but from (ii), $\overline{g}$ is injective, so $\Phi (d(a))=F$. Thus, $\Phi$ is surjective.  
\end{proof} 

\begin{corollary}
Let $n\geq 1$ and $A$ be as in Theorem \ref{theorem3.5}. Then, $M_{2n}(A)$ is $D$.  
\end{corollary}
\begin{proof}
This follows from the previous theorem and the isomorphism $M_{2n}(A)\cong M_2(M_n(A))$ defined by
$M_2(M_n(A)) \xrightarrow{\Psi} M_{2n}(A)$, $\Psi(F)=\Psi(\begin{bmatrix}F_{11} & F_{12}\\ F_{21} & F_{22}\end{bmatrix})=\begin{bmatrix}F_{11}^{(1)}\cdots F_{11}^{(n)} &
F_{12}^{(1)}\cdots F_{12}^{(n)} \\
F_{21}^{(1)}\cdots F_{21}^{(n)} & F_{22}^{(1)}\cdots F_{22}^{(n)}\end{bmatrix}$, where $F_{ij}^{(k)}$ is the $k$-th column of $F_{ij}\in M_n(A), 1\leq k\leq n, \ 1\leq i,j\leq 2$.
\end{proof} 

\begin{theorem}\label{proposition3.5}
	Let $R$, $A$ and $B$ be rings such that $A$ and $B$ are left $R$-modules. If $A\times B$ is $D$ $(WD)$ with respect to $R$, then $A$ and $B$ are $D$ $(WD)$ with respect to $R$.
\end{theorem}
\begin{proof}
	If $A$ is not $D$, then $A\times B$ is not $D$ since every $R$-linear ring endomorphism $\phi $ of $A$ induces an $R$-linear ring endomorphism $\widetilde{\phi}$ of $A\times B$ given by $(a,b)\mapsto (\phi(a),b)$. Similarly for $B$ and for $WD$.
	\end{proof}
The converse of the previous theorem is not true as we observed in (ix) of Example \ref{example3.3}.

\begin{theorem}\label{theorem3.6}
Let $Z$ be a commutative domain and $A,B$ be $Z$-algebras such that both are free $Z$-modules. Then,
\begin{enumerate}
\item[\rm (i)] If $X$ is a $Z$-basis of $A$ and $Y$ is a $Z$-basis of $B$, then $X\otimes Y:=\{x\otimes y | x\in X, y\in Y\}$ is a $Z$-basis of $A\otimes_Z B$.
\item[\rm (ii)] Assume that $1_A\in X$ and $1_B\in Y$. If $A\otimes_Z B$ is $D$ $(WD)$, then $A$ and $B$ are $D$ $(WD)$.
\end{enumerate}
\end{theorem}
\begin{proof}
(i) This fact it is well-known (see \cite{Lezama1} and \cite{Lezama2}), anyway we will include the proof for completeness. We have to prove that for every $Z$-module $N$ and every map
	$\psi :$ $X\otimes Y$ $\longrightarrow $ $N$,
there exists an unique $Z$-homomorphism
	$\overline{\psi }:$ $A\otimes_Z B$ $\longrightarrow N$
such that $\overline{\psi }\left( x\otimes y\right) =\psi \left( x\otimes y\right)$, for every $x\otimes y\in X\otimes Y$ (\cite{Lezama1}). For this, we apply the universal property of the tensor product (\cite{Lezama2}), thus, let $H':A\times B\to N$ be the bilinear map, (i.e., $H'$ is linear in each argument and $Z$-balanced), defined by $H'(a,b):=\sum_{i=1,j=1}^{k,m}z_iz_j'\psi(x_i\otimes y_j)$, where $a=\sum_{i=1}^kz_ix_i$ and $b=\sum_{j=1}^mz_j'y_j$ are the unique representations of $a$ and $b$ in the $Z$-bases of $A$ and $B$, respectively, whit $z_i,z_j'\in Z$ and $x_i\in X, y_j\in Y$. Then, there exists an unique $Z$-homomorphism $H:A\otimes_Z B\to N$ such that $H\circ t=H'$, where $H(a\otimes b):=H'(a,b)$ and $t:A\times B\to A\otimes_Z B$ with $t(a,b):=a\otimes b$. We define $\overline{\psi }:=H$, and observe that $\overline{\psi }(x\otimes y)=H(x\otimes y)=H\circ t(x,y)=H'(x,y)=\psi (x\otimes y)$, for every $x\otimes y\in X\otimes Y$. The uniqueness of $\overline{\psi }$ follows from the uniqueness of $H$. This completes the proof of (i).

(ii) \textit{Step 1}. Let $a\in A$ and $b\in B$. Then, $a\otimes b=0$ if and only if $a=0$ or $b=0$. In fact, it is clear that if $a=0$ or $b=0$, then $a\otimes b=0$. Now, assume that $a\otimes b=0$ with $a\neq 0$ and $b\neq 0$. From (i), there exist $x_1,\dots,x_k\in X$ and $y_1,\dots,y_m\in Y$ such that $a=z_1x_1+\cdots+z_kx_k$ and $b=z_1'y_1+\cdots+z_m'y_m$, for some $z_i,z_j'\in Z$, with $1\leq i\leq k$, $1\leq j\leq m$. Since $a\neq 0$ and $b\neq 0$, there exist $0\neq z_{i_1}\in \{z_1,\dots,z_k\}$ and $0\neq z_{j_1}'\in \{z_1',\dots,z_m'\}$. So, $0=a\otimes b=\sum_{i=1,j=1}^{k,m}z_iz_j'(x_i\otimes y_j)$, with $z_{i_1}z_{j_1}'\neq 0$, and this contradices (i).

\textit{Step 2}. Let  $\phi$ be an endomorphism of $A$. We have to prove that $\phi$ is bijective.

$\phi$ is injective: Let $a\in A$ such that $\phi(a)=0$. Then, for the endomorphism $\phi\otimes i_B:A\otimes_Z B\to A\otimes_Z B$, we have $(\phi\otimes i_B)(a\otimes 1_B)=\phi(a)\otimes 1_B=0\otimes 1_B=0$. But since $A\otimes_Z B$ is $D$, then $\phi\otimes i_B$ is bijective, so $a\otimes 1_B=0$, and from \textit{Step 1}, $a=0$. Thus, $\phi$ is injective.

$\phi$ is surjective: Let $a\in A$, since $A\otimes_Z B$ is $D$, then
\begin{center}
$a\otimes 1_B=(\phi\otimes i_B)(a_1\otimes b_1+\cdots+a_l\otimes b_l)$, with $a_i\in A$ and $b_i\in B$, for $1\leq i\leq l$.
\end{center}
Whence, $a\otimes 1_B=\phi(a_1)\otimes b_1+\cdots+\phi(a_l)\otimes b_l$. Let $a_i=\sum_{r=1}^{k_i}z_{ir}x_{ir}$ and $b_i=\sum_{s=1}^{m_i}z_{is}'y_{is}$, with $x_{i1},\dots,x_{ik_i}\in X$, $y_{i1},\dots,y_{im_i}\in Y$ and $z_{i1},\dots,z_{ik_i}, z_{i1}',\dots,z_{im_i}'\in Z$. Therefore,
\begin{center}
$a\otimes 1_B=\sum_{i=1}^l\phi(\sum_{r=1}^{k_i}z_{ir}x_{ir})\otimes (\sum_{s=1}^{m_i}z_{is}'y_{is})=\sum_{i=1}^l[\sum_{r=1}^{k_i}z_{ir}\phi(x_{ir})]\otimes (\sum_{s=1}^{m_i}z_{is}'y_{is})=\sum_{i=1}^l[\sum_{r=1,s=1}^{k_i,m_i}z_{ir}z_{is}'(\phi(x_{ir})\otimes y_{is})]=\sum_{i=1}^l(\sum_{r=1,s=1}^{k_i,m_i}z_{ir}z_{is}')(\phi(x_{ir})\otimes y_{is})=[\sum_{r=1}^{k_1}\phi(z_{1r}z_{11}'x_{1r})]\otimes y_{11}+\cdots+[\sum_{r=1}^{k_1}\phi(z_{1r}z_{1m_1}'x_{1r})]\otimes y_{1m_1}+\cdots+[\sum_{r=1}^{k_l}\phi(z_{lr}z_{l1}'x_{lr})]\otimes y_{l1}+\cdots+[\sum_{r=1}^{k_l}\phi(z_{lr}z_{lm_l}'x_{lr})]\otimes y_{lm_l}=\phi[\sum_{r=1}^{k_1}(z_{1r}z_{11}'x_{1r})]\otimes y_{11}+\cdots+\phi[\sum_{r=1}^{k_1}(z_{1r}z_{1m_1}'x_{1r})]\otimes y_{1m_1}+\cdots+\phi[\sum_{r=1}^{k_l}(z_{lr}z_{l1}'x_{lr})]\otimes y_{l1}+\cdots+\phi[\sum_{r=1}^{k_l}\phi(z_{lr}z_{lm_l}x_{lr})]\otimes y_{lm_l}$.
\end{center}
Consider in $Im(\phi)\subseteq A$ the following elements:
\begin{center}
$u_{11}:=\phi[\sum_{r=1}^{k_1}(z_{1r}z_{11}'x_{1r})],\dots,u_{1m_1}:=\phi[\sum_{r=1}^{k_1}(z_{1r}z_{1m_1}'x_{1r})],$

$\vdots$

$u_{l1}:=\phi[\sum_{r=1}^{k_l}(z_{lr}z_{l1}'x_{lr})],\dots,u_{lm_l}:=\phi[\sum_{r=1}^{k_l}\phi(z_{lr}z_{lm_l}x_{lr})]$.
\end{center}
Now, recall that $A\otimes_Z B$ is an $A-B$-bimodule, moreover $A\otimes_Z B$ is a free left $A$-module with basis $1_A\otimes Y:=\{1_A\otimes y| y\in Y\}$ (in fact, see \cite{Lezama2}, $A\otimes_Z B\cong A\otimes_Z Z^{(Y)}\cong (A\otimes _Z Z)^{(Y)}\cong A^{(Y)}$, and in this $A$-isomorfism $u_{11}\otimes y_{11}+\cdots+u_{lm_l}\otimes y_{lm_l}\mapsto (\dots, 0,u_{11},0\dots,0,u_{lm_l},0,\dots,)$; similarly, $A\otimes_Z B$ is a free  right-$B$-module with basis $\{X\otimes 1_B\}$). Thus,

\begin{center}
$a\cdot (1_A\otimes 1_B)=u_{11}\cdot (1_A\otimes y_{11})+\cdots +u_{1m_1}\cdot (1_A\otimes y_{1m_1})+\cdots +u_{l1}\cdot (1_A\otimes y_{l1}) +\cdots+u_{lm_l}\cdot (1_A\otimes y_{lm_l})$,
\end{center}
and considering the different $y^{,}s$, we can write
\begin{center}
	$a\cdot (1_A\otimes 1_B)=u_1\cdot (1_A\otimes y_1)+\cdots+u_w\cdot (1_A\otimes y_w)$, with $u_1,\dots ,u_w\in Im(\phi)$ and $y_1,\dots,y_w\in Y$.
\end{center}
Since $1_B\in Y$, then we have two possibilities: either $1_B\notin \{y_1,\dots, y_w\}$ or $1_B\in \{y_1,\dots, y_w\}$; in the first case $a=0$ and in the second case there exists exactly one basis element $y_j$ in $\{y_1,\dots,y_w\}$ such that $y_j=1_B$; in the first case $a\in Im(\phi)$, and in the second case we can assume that $j=1$ and then
\begin{center}
	$(a-u_1)\cdot (1_A\otimes 1_B)=u_2\cdot (1_A\otimes y_2)+\cdots+u_w\cdot (1_A\otimes y_w)$,
\end{center}
whence $a=u_1\in Im(\phi)$. Thereby, $\phi$ is surjective.

Thus, $A$ is $D$. Similarly for $B$ and for $WD$.
\end{proof}	
\begin{theorem}\label{proposition3.7}
Let $Z$ be a commutative ring and $A$ be a $Z$-algebra. If $A$ is a domain and $S$ is a right Ore subset of $A$ such that $AS^{-1}$ is $D$ $(WD)$, then every endomorphism $\phi$ of $A$, with $\phi(S)\subseteq S$, is injective.  	
\end{theorem}
\begin{proof}
$\phi$ induces an endomorphism $\widetilde{\phi}$ of $AS^{-1}$ given by $\widetilde{\phi}(\frac{a}{s}):=\frac{\phi(a)}{\phi(s)}$, for $\frac{a}{s}\in AS^{-1}$. Notice that if $\phi$ is not injective, then $\widetilde{\phi}$ is not injective. Thus, since $AS^{-1}$ is $D$ ($WD$), then $\phi$ is injective.	    	
\end{proof}

Another classical algebraic construction is the graded ring $Gr(A)$ associated to a $\mathbb{N}$-filtered ring $A$ (see \cite{McConnell}, \cite{Oystaeyen}, \cite{Nastasescu2} or also \cite{Lezama2} for the general theory of graded rings, algebras and modules). The notions introduced in Definition \ref{definition3.1} can be formulated for the graded and filtered cases changing $\phi$ for graded and filtered homomorphisms, respectively. For the filtered-graded construction we have the following result.

\begin{theorem}\label{proposition3.8}
Let $R$ and $A$ be $\mathbb{N}$-filtered rings such that $A$ is a $\mathbb{N}$-filtered left $R$-module. If the graded ring $Gr(A)$ is graded Dixmier with respect to $Gr(R)$, then $A$ is filtered Dixmier with respect to $R$.
\end{theorem}
\begin{proof}
Let $\phi:A\to A$ be an $R$-linear ring homomorphism that is filtered as homomorphism of filtered $R$-modules and filtered as homomorphism of filtered rings. $\phi$ induces  $Gr(\phi):Gr(A)\to Gr(A)$ that is a $Gr(R)$-linear graded ring homomorphism. By the hypothesis, $Gr(\phi)$ is an automorphism, hence $\phi$ is an automorphism.
\end{proof}

Another easy remark for the graded Dixmier condition is as follows. There are graded algebras over fields for which all algebra automorphisms are graded, i.e., $Aut_{Gr}(A)=Aut(A)$ (see \cite{ZhangJ4}). This is the case for the \textbf{\textit{quantum plane}} $K_q[x,y]$, when $q$ is not a root of unity; recall that $K_q[x,y]$ is defined by the relation $yx=qxy$, with $q\in K^*$  (for the automorphisms of the quantum plane see \cite{Artamonov5}, Theorem 4). Another example is $\mathcal{A}_n^{\overline{q},\Lambda}(K)$
(see Theorem 2.1 in \cite{Tang} for the automorphisms of $\mathcal{A}_n^{\overline{q},\Lambda}(K)$ when none of $q_1,\dots,q_n$ is a root of unity). Recall that for $n=1$, $\mathcal{A}_n^{\overline{q},\Lambda}(K)=A_1^q(K)$. For such algebras the Dixmier condition implies the graded one.

\begin{theorem}\label{proposition3.9}
	Let $K$ be a field and $A$ be a graded $K$-algebra such that  $Aut_{Gr}(A)=Aut(A)$. If $A$ is Dixmier, then $A$ is graded Dixmier. In addition, if $End_{Gr}(A)=End(A)$, then $A$ is Dixmier if and only if $A$ is graded Dixmier.	
\end{theorem}
\begin{proof}
	Let $\phi:A\to A$ be a graded algebra endomorphism of $A$, then $\phi$ is an algebra endomorphism, but since $A$ is Dixmier, then $\phi$ is an automorphism of $K$-algebras, so, by the hypothesis, $\phi$ is an automorphism of graded algebras. This means that $A$ is graded Dixmier. From this, the second assertion is trivial.
\end{proof}

\begin{remark}\label{reamrk3.9}
(i) With respect to the converse of (ii) in Theorem \ref{theorem3.6}, let $Z,A,X,B$ and $Y$ be as in Theorem \ref{theorem3.6}, with $1_A\in X$ and $1_B\in Y$. If $A$ and $B$ are $D$ $(WD)$, then $A\otimes_Z B$ is $D$ $(WD)$?

(ii) For $A_1(K)$ we have $Z=K$, $A=B=A_1(K)$, $1\in X=Y={\rm Mon}(A)$, with ${\rm Mon}(A)$ as in Definition \ref{gpbwextension}, so, if the answer to (i) is yes, we would get from the result of Zheglov (see \cite{Zheglov}) that $A_1(K)\otimes A_1(K)\cong A_2(K)$ is $D$. Taking now $A=A_2(K)$ and $B=A_1(K)$, with $X={\rm Mon}(A)$ and $Y={\rm Mon}(B)$, we would get that $A_3(K)$ is $D$. By induction, we would get that the Generalized Dixmier conjecture is also true since $A_n(K)\cong A_1(K)^{\otimes n}:=A_1(K)\otimes \cdots \otimes A_1(K)$, and hence, from Proposition \ref{proposition1.7}, the Jacobian conjecture is true. However take in account Remarks \ref{remark1.4} and \ref{remark1.6}. Thus, since the Alpöge's result is true, then the Generalized Dixmier conjecture is not true for $n\geq 3$, \textbf{and hence, in general, the tensor product of Dixmier algebras and rings is not Dixmier}.  

(iii) According to (ii), we could not assert that the conjecture of Bavula about $\mathbb{I}_n$ is true.
\end{remark}

We conclude this section with an easy proposition related to the converse of (ii) in Theorem \ref{theorem3.6}.   

\begin{proposition}
Let $K$ be a field, $A$ and $B$ be simple $K$-algebras such that $Z(A)=K$ or $Z(B)=K$. Let $X$ be a $K$-basis of $A$ and $Y$ be a $K$-basis of $B$ such that $1_A\in X$ and $1_B\in Y$. Assume that $A$ and $B$ are $D$. Then, if $\Phi$ is an endomorphism of $A\otimes_K B$ such that $\Phi(A\otimes 1_B)\subseteq A\otimes 1_B$  and $\Phi(1_A\otimes B)\subseteq 1_A\otimes B$, then $\Phi$ is an automorphism.
\end{proposition} 
\begin{proof}
According to Corollary III.1.8 in \cite{Artin}, $A\otimes_K B$ is a simple $K$-algebra, so $\Phi$ is injective. Assume that $\Phi$ is not surjective, then there exists $x\otimes y\in X\otimes Y$ such that $x\otimes y\notin Im(\Phi)$ (recall that $X\otimes Y$ is a $K$-basis of $A\otimes_K B$, Theorem \ref{theorem3.6}). Therefore, $x\otimes 1_B\notin Im(\Phi)$ or $1_A\otimes y\notin Im(\Phi)$. Since $A$ is $D$ and $A\otimes 1_B\cong A$, then the condition $\Phi(A\otimes 1_B)\subseteq A\otimes 1_B$ induces an automorphism $\Phi_A: A\otimes 1_B\to A\otimes 1_B$, where $\Phi_A$ is the restriction of $\Phi$ to $A\otimes 1_B$. Thus, there exists $a\otimes 1_B\in A\otimes 1_B$ such that $x\otimes 1_B=\Phi_A(a\otimes 1_B)=
\Phi(a\otimes 1_B)$, i.e., $x\otimes 1_B\in Im(\Phi)$, a contradiction. Similarly for $1_A\otimes y\notin Im(\Phi)$.
\end{proof}

\newpage

\section{The Dixmier problem for skew $PBW$ extensions}\label{Section4}

In this section we discuss the famous Dixmier question for skew $PBW$ extensions introduced by Gallego and Lezama in \cite{LezamaGallego} (see also \cite{Lezama-sigmaPBW}). We start recalling the definition and some key properties of this class of noncommutative rings of polynomial type. A complete study of the skew $PBW$ extensions can be found in \cite{Lezama-sigmaPBW}, including its Gröbner bases and applications to noncommutative algebraic geometry.

\subsection{Skew $PBW$ extensions}

We start recalling some basic facts about the skew $PBW$ extensions.

\begin{definition}[\cite{Lezama-sigmaPBW}, Chapter 1]\label{gpbwextension}
	Let $R$ and $A$ be rings. We say that $A$ is a \textit{\textbf{skew $PBW$
			extension of $R$}} $($also called a $\sigma-PBW$ extension of
	$R$$)$ if the following conditions hold:
	\begin{enumerate}
		\item[\rm (i)]$R\subseteq A$.
		\item[\rm (ii)]There exist finitely many elements $x_1,\dots ,x_n\in A$ such $A$ is an $R$-free left module with basis
		\begin{center}
			${\rm Mon}(A):= \{x^{\alpha}=x_1^{\alpha_1}\cdots
			x_n^{\alpha_n}\mid \alpha=(\alpha_1,\dots ,\alpha_n)\in
			\mathbb{N}^n\}$, with $\mathbb{N}:=\{0,1,2,\dots\}$.
		\end{center}
		In this case we say that $A$ is a \textbf{ring of left polynomial type} over $R$ with respect to
		$\{x_1,\dots,x_n\}$. The set ${\rm Mon}(A)$ is called the set of \textbf{standard monomials} of
		$A$.
		\item[\rm (iii)]For every $1\leq i\leq n$ and $r\in R-\{0\}$ there exists $c_{i,r}\in R-\{0\}$ such that
		\begin{equation}\label{sigmadefinicion1}
		x_ir-c_{i,r}x_i\in R.
		\end{equation}
		\item[\rm (iv)]For every $1\leq i,j\leq n$ there exists $c_{i,j}\in R-\{0\}$ such that
		\begin{equation}\label{sigmadefinicion2}
		x_jx_i-c_{i,j}x_ix_j\in R+Rx_1+\cdots +Rx_n.
		\end{equation}
		Under these conditions we will write $A:=\sigma(R)\langle
		x_1,\dots ,x_n\rangle$.
	\end{enumerate}
\end{definition}

\begin{remark}\label{notesondefsigampbw}
	(i) In general, for $i\neq j$ the elements $x_i$ and $x_j$ do not commute. Since $\mathrm{Mon}(A)$ is an
	$R$-basis for $A$, in the above definition the elements $c_{i,r}$ and $c_{i,j}$ are unique. Note
	that for $i=j$, $c_{i,i}=1$: in fact, $x_i^2-c_{i,i}x_i^2=r_0+r_1x_1+\cdots+r_nx_n$, with $r_k\in
	R$ for $0\leq k\leq n$, hence $r_k=0$ and $c_{i,i}=1$.
	
	(ii) If $r=0$, then we define $c_{i,0}=0$: indeed, $0=x_i0=c_{i,0}x_i+r'$, with $r'\in R$, but
	since $\mathrm{Mon}(A)$ is an $R$-basis, then $r'=0$ and $c_{i,0}=0$.
	
	(iii) Condition (iv) in Definition \ref{gpbwextension} is equivalent to the following: for every
	$1\leq i<j\leq n$ there exists left invertible $c_{i,j}\in R$ such that
	\begin{equation}\label{equation1.1.4}
	x_jx_i-c_{i,j}x_ix_j\in R+Rx_1+\cdots +Rx_n.
	\end{equation}
	In fact, from (\ref{sigmadefinicion2}) there exist $c_{j,i},c_{i,j}\in R$ such that
	$x_ix_j-c_{j,i}x_jx_i\in R+Rx_1+\cdots +Rx_n$ and $x_jx_i-c_{i,j}x_ix_j\in R+Rx_1+\cdots +Rx_n$,
	but since $\mathrm{Mon}(A)$ is an $R$-basis then $1=c_{j,i}c_{i,j}$, whence, for every $1\leq i<j\leq n$,
	$c_{i,j}$ is left invertible. Conversely, assuming that $c_{i,j}$ is left invertible for $1\leq
	i<j\leq n$, let $c_{i,j}'\in R$ such that $c_{i,j}'c_{i,j}=1$, so from (\ref{equation1.1.4}),
	$x_ix_j-c_{i,j}'x_jx_i\in R+Rx_1+\cdots +Rx_n$, hence $c_{j,i}:=c_{i,j}'\neq 0$ and condition (iv) in
	Definition \ref{gpbwextension} holds.
	
	(iv) The elements of $\mathrm{Mon}(A)$ will also be denoted by capital letters, thus, $x^\alpha\in \mathrm{Mon}(A)$
	will be represented also as $X$ if it is not important to highlight the exponents
	$\alpha_1,\dots,\alpha_n$ in $x^{\alpha}$.
	
	(v) Each element $f\in A-\{0\}$ has a unique representation in the form $f=c_1X_1+\cdots+c_tX_t$,
	with $c_i\in R-\{0\}$ and $X_i\in \mathrm{Mon}(A)$, $1\leq i\leq t$.
\end{remark}

Associated to a skew $PBW$ extension $A=\sigma(R)\langle x_1,\dots
,x_n\rangle$ there are $n$ injective endomorphisms
$\sigma_1,\dots,\sigma_n$ of $R$ and $\sigma_i$-derivations, as
the following proposition shows.

\begin{proposition}[\cite{Lezama-sigmaPBW}, Chapter 1]\label{sigmadefinition}
	Let $A$ be a skew $PBW$ extension of $R$. Then, for every $1\leq
	i\leq n$, there exist an injective ring endomorphism
	$\sigma_i:R\rightarrow R$ and a $\sigma_i$-derivation
	$\delta_i:R\rightarrow R$ such that
	\begin{center}
		$x_ir=\sigma_i(r)x_i+\delta_i(r)$,
	\end{center}
	for each $r\in R$.
\end{proposition}

Two remarkable particular cases of skew $PBW$ extensions are recalled next.

\begin{definition}[\cite{Lezama-sigmaPBW}, Chapter 1]\label{sigmapbwderivationtype}
	Let $A$ be a skew $PBW$ extension.
	\begin{enumerate}
		\item[\rm (a)]
		$A$ is \textbf{quasi-commutative}\index{quasi-commutative} if conditions {\rm(}iii{\rm)} and {\rm(}iv{\rm)} in Definition
		\ref{gpbwextension} are replaced by
		\begin{enumerate}
			\item[\rm ($iii'$)]For every $1\leq i\leq n$ and $r\in R-\{0\}$ there exists a $c_{i,r}\in R-\{0\}$ such that
			\begin{equation}
			x_ir=c_{i,r}x_i.
			\end{equation}
			\item[\rm ($iv'$)]For every $1\leq i,j\leq n$ there exists $c_{i,j}\in R-\{0\}$ such that
			\begin{equation}
			x_jx_i=c_{i,j}x_ix_j.
			\end{equation}
		\end{enumerate}
		\item[\rm (b)]$A$ is \textbf{bijective}\index{bijective} if $\sigma_i$ is bijective for
		every $1\leq i\leq n$ and $c_{i,j}$ is invertible for any $1\leq i,j\leq n$.
	\end{enumerate}
\end{definition}

If $A$ is quasi-commutative, then $\delta_k=0$ for every $1\leq k\leq n$ and $p_{\alpha,r},p_{\alpha, \beta}=0$ in Proposition \ref{coefficientes}.

Many important algebras and rings coming from mathematical physics
are particular examples of skew $PBW$ extensions: Habitual ring of
	polynomials in several variables, \textbf{Weyl algebras}, enveloping
algebras of finite dimensional Lie algebras, algebra of
$q$-differential operators, many important types of Ore algebras,
algebras of diffusion type, additive and multiplicative analogues
of the Weyl algebra, dispin algebra $\mathcal{U}(osp(1,2))$,
quantum algebra $\mathcal{U}'(so(3,K))$, Woronowicz algebra
$\mathcal{W}_{\nu}(\mathfrak{sl}(2,K))$, Manin algebra
$\mathcal{O}_q(M_2(K))$, coordinate algebra of the quantum group
$SL_q(2)$, $q$-Heisenberg algebra \textbf{H}$_n(q)$, Hayashi
algebra $W_q(J)$, differential operators on a quantum space
$D_{\textbf{q}}(S_{\textbf{q}})$, Witten's deformation of
$\mathcal{U}(\mathfrak{sl}(2,K))$, \textbf{quantum Weyl algebra of Maltsiniotis}, multiparameter Weyl algebra
$A_n^{Q,\Gamma}(K)$, quantum symplectic space
$\mathcal{O}_q(\mathfrak{sp}(K^{2n}))$, some quadratic algebras in
3 variables, some 3-dimensional skew polynomial algebras,
particular types of Sklyanin algebras, homogenized enveloping
algebra $\mathcal{A}(\mathcal{G})$, Sridharan enveloping algebra
of 3-dimensional Lie algebra $\mathcal{G}$, among many others. For
a precise definition of any of these rings and algebras see \cite{LezamaGallego},
\cite{lezamareyes1} and \cite{Lezama-sigmaPBW}. The skew $PBW$ has been intensively studied in the last years (see \cite{Lezama-sigmaPBW}).

Next we will fix some notation and a monomial order in $A$ (see \cite{Lezama-sigmaPBW}, Chapter 1).

\begin{definition}[\cite{Lezama-sigmaPBW}, Chapter 1]\label{1.1.6}
	Let $A$ be a skew $PBW$ extension of $R$ with endomorphisms $\sigma_i$ as in Proposition
	\ref{sigmadefinition}, $1\leq i\leq n$.
	\begin{enumerate}
		\item[\rm (i)]For $\alpha=(\alpha_1,\dots,\alpha_n)\in \mathbb{N}^n$,
		$\boldsymbol{\sigma^{\alpha}}:=\sigma_1^{\alpha_1}\cdots \sigma_n^{\alpha_n}$,
		$\boldsymbol{|\alpha|}:=\alpha_1+\cdots+\alpha_n$. If $\beta=(\beta_1,\dots,\beta_n)\in \mathbb{N}^n$, then
		$\boldsymbol{\alpha+\beta}:=(\alpha_1+\beta_1,\dots,\alpha_n+\beta_n)$.
		\item[\rm (ii)]For $X=x^{\alpha}\in \mathrm{Mon}(A)$,
		$\boldsymbol{\exp(X)}:=\alpha$ and $\boldsymbol{\deg(X)}:=|\alpha|$.
		\item[\rm (iii)]Let $0\neq f\in A$. If $t(f)$ is the finite
		set of \textbf{\textit{terms}} that conform $f$, i.e., if $f=c_1X_1+\cdots +c_tX_t$, with $X_i\in \mathrm{Mon}(A)$ and $c_i\in
		R-\{0\}$, then $\boldsymbol{t(f)}:=\{c_1X_1,\dots,c_tX_t\}$.
		\item[\rm (iv)]Let $f$ be as in {\rm(iii)}, then $\boldsymbol{\deg(f)}:=\max\{\deg(X_i)\}_{i=1}^t.$
	\end{enumerate}
\end{definition}

In $\mathrm{Mon}(A)$ we define
\begin{center}
	$x^{\alpha}\succeq x^{\beta}\Longleftrightarrow
	\begin{cases}
	x^{\alpha}=x^{\beta}\\
	\text{or} & \\
	x^{\alpha}\neq x^{\beta}\, \text{but} \, |\alpha|> |\beta| & \\
	\text{or} & \\
	x^{\alpha}\neq x^{\beta},|\alpha|=|\beta|\, \text{but $\exists$ $i$ with} &
	\alpha_1=\beta_1,\dots,\alpha_{i-1}=\beta_{i-1},\alpha_i>\beta_i.
	\end{cases}$
\end{center}
It is clear that this is a total order on $\mathrm{Mon}(A)$, called \textit{\textbf{deglex}} order. If
$x^{\alpha}\succeq x^{\beta}$ but $x^{\alpha}\neq x^{\beta}$, we write $x^{\alpha}\succ x^{\beta}$.
Each element $f\in A-\{0\}$ can be represented in a unique way as $f=c_1x^{\alpha_1}+\cdots
+c_tx^{\alpha_t}$, with $c_i\in R-\{0\}$, $1\leq i\leq t$, and $x^{\alpha_1}\succ \cdots \succ
x^{\alpha_t}$. We say that $x^{\alpha_1}$ is the \textit{\textbf{leading monomial}} of $f$ and we write
$lm(f):=x^{\alpha_1}$; $c_1$ is the \textit{\textbf{leading coefficient}} of $f$, $lc(f):=c_1$, and
$c_1x^{\alpha_1}$ is the \textit{\textbf{leading term}} of $f$ denoted by $lt(f):=c_1x^{\alpha_1}$. We say that $f$ is \textit{\textbf{monic}} if $lc(f):=1$. If $f=0$,
we define $lm(0):=0$, $lc(0):=0$, $lt(0):=0$, and we set $X\succ 0$ for any $X\in \mathrm{Mon}(A)$. We observe that
\begin{center}
	$x^{\alpha}\succ x^{\beta}\Rightarrow lm(x^{\gamma}x^{\alpha}x^{\lambda})\succ
	lm(x^{\gamma}x^{\beta}x^{\lambda})$, for every $x^{\gamma},x^{\lambda}\in \mathrm{Mon}(A)$.
\end{center}

The next proposition complements Definition \ref{gpbwextension}.

\begin{proposition}[\cite{Lezama-sigmaPBW}, Chapter 1]\label{coefficientes}
	Let $A$ be a ring of a left polynomial type over $R$ w.r.t.\ $\{x_1,\dots,x_n\}$. $A$ is a skew
	$PBW$ extension of $R$ if and only if the following conditions hold:
	\begin{enumerate}
		\item[\rm (a)]For every $x^{\alpha}\in \mathrm{Mon}(A)$ and every $0\neq
		r\in R$ there exist unique elements $r_{\alpha}:=\sigma^{\alpha}(r)\in R-\{0\}$ and $p_{\alpha
			,r}\in A$ such that
		\begin{equation}\label{611}
		x^{\alpha}r=r_{\alpha}x^{\alpha}+p_{\alpha , r},
		\end{equation}
		where $p_{\alpha ,r}=0$ or $\deg(p_{\alpha ,r})<|\alpha|$ if $p_{\alpha , r}\neq 0$. Moreover, if
		$r$ is left invertible, then $r_\alpha$ is left invertible.
		
		\item[\rm (b)]For every $x^{\alpha},x^{\beta}\in \mathrm{Mon}(A)$ there
		exist unique elements $c_{\alpha,\beta}\in R$ and $p_{\alpha,\beta}\in A$ such that
		\begin{equation}\label{612}
		x^{\alpha}x^{\beta}=c_{\alpha,\beta}x^{\alpha+\beta}+p_{\alpha,\beta},
		\end{equation}
		where $c_{\alpha,\beta}$ is left invertible, $p_{\alpha,\beta}=0$ or
		$\deg(p_{\alpha,\beta})<|\alpha+\beta|$ if $p_{\alpha,\beta}\neq 0$.
	\end{enumerate}
\end{proposition}

For the investigation of the Dixmier and the weak Dixmier conditions for the skew $PBW$ extensions, it will be very important the universal property that we present next.

If $A=\sigma(R)\langle x_1,\dots,x_n\rangle$ is a skew $PBW$
extension of the ring $R$, then, as was observed in Proposition
\ref{sigmadefinition}, $A$ induces injective endomorphisms
$\sigma_k:R\to R$ and $\sigma_k$-derivations $\delta_k:R\to R$,
$1\leq k\leq n$. Moreover, from the Definition
\ref{gpbwextension} and Remark \ref{notesondefsigampbw}, there exists a unique finite set of constants
$c_{ij}, d_{ij}, a_{ij}^{(k)}\in R$, $c_{ij}\neq 0$, such that
\begin{equation}\label{equation1.2.1}
x_jx_i=c_{ij}x_ix_j+a_{ij}^{(1)}x_1+\cdots+a_{ij}^{(n)}x_n+d_{ij},
\ \text{for every}\  1\leq i<j\leq n.
\end{equation}

\begin{definition}[\cite{Lezama-sigmaPBW}, Chapter 1]\label{definition1.2.1}
	Let $A=\sigma(R)\langle x_1,\dots,x_n\rangle$ be a skew $PBW$ extension. $R$, $n$,
	$\sigma_k,\delta_k, c_{i,j}$, $d_{ij}, a_{ij}^{(k)}$, with $1\leq i<j\leq n$, $1\leq k\leq n$,
	defined as before, are called the \textbf{parameters} of $A$.
\end{definition}

\begin{proposition}[\textbf{Universal property}; \cite{Lezama-sigmaPBW}, Chapter 1]\label{122}\index{universal property}
	Let $A=\sigma(R)\langle x_1,\dots,x_n\rangle$ be a skew $PBW$ extension with parameters $R, n,
	\sigma_k,\delta_k, c_{i,j}, d_{ij}, a_{ij}^{(k)}$, $1\leq i<j\leq n$, $1\leq k\leq n$. Let $B$ be a
	ring with a homomorphism $\varphi:R\to B$ and elements $y_1,\dots,y_n\in B$ such that
	\begin{enumerate}
		\item[\rm (i)]$y_k\varphi(r)=\varphi(\sigma_k(r))y_k+\varphi(\delta_k(r))$, for every $r\in R$, $1\leq k\leq
		n$.
		\item[\rm (ii)]$y_jy_i=\varphi(c_{i,j})y_iy_j+\varphi(a_{ij}^{(1)})y_1+\cdots
		+\varphi(a_{ij}^{(n)})y_n+\varphi(d_{ij})$, $1\leq i<j\leq n$.
	\end{enumerate}
	Then, there exists a unique ring homomorphism $\widetilde{\varphi}:A\to B$ such that
	$\widetilde{\varphi}\iota=\varphi$ and $\widetilde{\varphi}(x_i)=y_i$, where $\iota$ is the
	inclusion of $R$ in $A$, $1\leq i\leq n$.
\end{proposition}

\begin{proposition}[\textbf{Hilbert's basis theorem for skew $PBW$ extensions}; \cite{Lezama-sigmaPBW}, Theorem 3.1.5]\label{1.3.4}
Let $A$ be a bijective skew $PBW$ extension of $R$. If $R$ is a left {\rm(}right{\rm)} noetherian
ring then $A$ is also a left {\rm(}right{\rm)} noetherian ring.
\end{proposition}

\begin{proposition}[\cite{Lezama-sigmaPBW}, Proposition 3.2.1]\label{1.1.10}
Let $A$ be a skew $PBW$ extension of $R$. If $R$ is a domain, then $A$ is a domain.
\end{proposition}

\begin{proposition}[\cite{Lezama-sigmaPBW}, Corollary 3.2.2]\label{corollary3.2.2}
Let $A$ be a skew $PBW$ extension of $R$. If $R$ is a domain, then $A^*=R^*$.
\end{proposition}

\begin{remark}\label{reamrk3.12}
	We remark that the Gröbner theory of left ideals and modules of skew $PBW$ extensions and some of its important applications in homological algebra have been developed in \cite{Lezama-sigmaPBW} and implemented in \texttt{MAPLE} in \cite{Fajardo2} and \cite{Fajardo3}. This implementation is based
	on the library \textbf{SPBWE} specialized for working with bijective skew $PBW$
	extensions. The library has utilities to calculate Gröbner bases, and
	it includes some functions that compute the module of syzygies,
	free resolutions and left inverses of matrices, among other things.	
\end{remark}

\subsection{Main results}\label{subsection4.2}

In this subsection we present the main results of the present work. We start with a theorem induced by Proposition \ref{proposition2.2} and that gives a characterizations of the Dixmier condition (Definition \ref{definition3.1}) for skew $PBW$ extensions that are simple rings.

\begin{theorem}\label{theorem4.15}
Let $A=\sigma(R)\langle x_1,\dots,x_n\rangle$ be a skew $PBW$ extension with parameters $R, n,
\sigma_k,\delta_k, c_{i,j}$, $d_{ij}, a_{ij}^{(k)}$, $1\leq i<j\leq n$, $1\leq k\leq n$. If $A$ is a simple ring, then the following conditions are equivalent:
	\begin{enumerate}
		\item[\rm (i)]$A$ is $D$ with respect to $R$.
		\item[\rm (ii)]If $y_1,\dots,y_n\in A$ are such that
		\begin{enumerate}
			\item[\rm (a)]$y_k r=\sigma_k(r)y_k+\delta_k(r)$, for every $r\in R$, $1\leq k\leq n$,
			\item[\rm (b)]$y_jy_i=c_{i,j}y_iy_j+a_{ij}^{(1)}y_1+\cdots
			+a_{ij}^{(n)}y_n+d_{ij}$, $1\leq i<j\leq n$,
		\end{enumerate}
			\end{enumerate}
		then the subring generated by $R$ and $y_1,\dots,y_n$ coincides with $A$.
\end{theorem}
\begin{proof}
$\rm (i)\Rightarrow \rm (ii)$:
Let $y_1,\dots,y_n\in A$ that satisfy (a) and (b), then from Theorem \ref{122}, $\phi: A\to A$ given by $\phi(x_i):=y_i$, $1\leq i\leq n$, and $\phi(r):=r$, for every $r\in R$, is a well-defined $R$-linear ring endomorphism of $A$ , so $\phi$ is an automorphism, in particular, $Im(\phi)=A$, but $Im(\phi)$ is the subring of $A$ generated by $\phi(R)=R$ and $\phi(x_i)$, $1\leq i\leq n$, i.e., $R$ and $y_1,\dots,y_n$ generate $A$.

$\rm (ii)\Rightarrow \rm (i)$: Let $\phi: A\to A$ be an $R$-linear ring endomorphism. Let $y_i:=\phi(x_i)$, for $1\leq i\leq n$, then $Im(\phi)$ is generated by $R$ and $y_1,\dots,y_n$, so, by the hypothesis, $\phi$ is surjective. But since $A$ is simple, then $\ker(\phi)=0$ and hence $\phi$ is an automorphism.
\end{proof}

The simplicity of noncommutative rings of polynomial type has been investigated in several papers (see for example \cite{Jordan3}, \cite{Lam}, \cite{Silvestrov} and the references therein). In the context of the paper the most remarkable example of simple skew $PBW$ extension is the first Weyl algebra $A_1(K)$ when $char(K)=0$. In the next theorem we give another interesting example that covers the first Weyl algebra $A_1(K)$.

\begin{theorem}\label{theorem4.16}
	Let $A=\sigma(R)\langle x_1,\dots,x_n\rangle$ be a skew $PBW$ extension with parameters $R, n,
	\sigma_k,\delta_k, c_{i,j}$, $d_{ij}, a_{ij}^{(k)}$, $1\leq i<j\leq n$, $1\leq k\leq n$, that satisfy the following conditions:
	\begin{enumerate}
		\item[\rm (i)]$R=K$ is a field of characteristic zero.
		\item[\rm (ii)]$n\geq 2$.
		\item[\rm (iii)]$\sigma_k=i_K$ and $\delta_k=0$ for $1\leq k\leq n$.
		\item[\rm (iv)]$c_{ij}=1$, $a_{ij}^{(k)}=0$ and $d_{ij}\neq 0$ for all $1\leq i<j\leq n$ and all $1\leq k\leq n$.
			\end{enumerate}
Then,
\begin{enumerate}
\item[\rm (a)]$A$ is a noetherian domain.
\item[\rm (b)]$A^*=K^*$.
\item[\rm (c)]$Z(A)=K$ and $A$ is cancellative.
\item[\rm (d)]$A$ is simple.
\end{enumerate}
\end{theorem}
\begin{proof}
Notice first that condition (iii) means that $A$ is a $K$-algebra, so the coefficients of $K$ commute with the variables $x_1,\dots, x_n$. Thus, $A$ is the $K$-algebra generated by $x_1,\dots, x_n$ subject to relations
\begin{center}
$x_jx_i=x_ix_j+d_{ij}$, for all $1\leq i<j\leq n$, with $d_{ij}\neq 0$.
\end{center}
In the proof we will consider the deglex order in $\mathrm{Mon}(A)$.

(a) Since $K$ is a noetherian domain, this follows from Propositions \ref{1.3.4} and \ref{1.1.10}.

(b) Since $K$ is a domain, this follows from Proposition \ref{corollary3.2.2}.

(c) $Z(A)=K$: Let $0\neq f=c_1X_1+\cdots +c_tX_t\in Z(A)$, with $X_l\in \mathrm{Mon}(A)$, $c_l\in
K^*$ for $1\leq l\leq t$ and $X_1\succ \cdots \succ X_t$. We will show that $t=1$ and $X_1=1$. If so, then $Z(A)\subseteq K\subseteq Z(A)$, i.e., $Z(A)=K$. We divide the proof in three steps.

Step 1. Observe first that if $1\neq X=x_{i_1}^{\alpha _1}\cdots x_{i_k}^{\alpha_k}\in {\rm Mon}(A)$, with $k\geq 2$,  $\alpha_j\geq 1$, $1\leq j\leq k$ and $i_1<i_2<\cdots<i_k$, then a direct computation, using the defining relations of $A$, shows that
\begin{center}
	$Xx_{i_1}=x_{i_1}^{\alpha_1+1}x_{i_2}^{\alpha_2}\cdots x_{i_k}^{\alpha_k}+\alpha_2d_{i_1i_2}x_{i_1}^{\alpha_1}x_{i_2}^{\alpha_2-1}x_{i_3}^{\alpha_3}\cdots x_{i_{k-1}}^{\alpha_{k-1}}x_{i_k}^{\alpha_{k}}+\cdots+\alpha_{k-1}d_{i_1i_{k-1}}x_{i_1}^{\alpha_1}x_{i_2}^{\alpha_2}x_{i_3}^{\alpha_3}\cdots x_{i_{k-1}}^{\alpha_{k-1}-1}x_{i_k}^{\alpha_{k}}+\alpha_kd_{i_1i_k}x_{i_1}^{\alpha_1}x_{i_2}^{\alpha_2}x_{i_3}^{\alpha_3}\cdots x_{i_{k-1}}^{\alpha_{k-1}}x_{i_k}^{\alpha_{k}-1}$.
\end{center}

Step 2. Assume that $t\geq 2$. Then $X_1\neq 1$. We will consider the two possible cases for $X_1$.

Case 1. $X_1$ involves only one variable. Let $X_1=x_{i_1}^{\alpha_1}$, with $\alpha_1\geq 1$. If $X_2=1$, then $t=2$ and $f=c_1x_{i_1}^{\alpha_1}+c_2$. By (ii), there exists $i_2\neq i_1$ and, without lost of generality, we can assume that $i_1<i_2$. Then, using the relations of $A$ we have
\begin{center}
$x_{i_2}f=c_1x_{i_1}^{\alpha_1}x_{i_2}+c_1\alpha_1d_{i_1i_2}x_{i_1}^{\alpha_1-1}+c_2x_{i_2}=fx_{i_2}=c_1x_{i_1}^{\alpha_1}x_{i_2}+c_2x_{i_2}$.
\end{center}
From this we get that $c_1\alpha_1d_{i_1i_2}x_{i_1}^{\alpha_1-1}=0$. By (i), $\alpha_1=0$, a contradiction. Thus, $X_2\neq 1$. For $X_2$ two cases arise.

Case 1.1. Assume first that $X_2$ involves only one variable. Let $X_2=x_{r_1}^{\beta_1}$, with $\beta_1\geq 1$. Thus, $f=c_1x_{i_1}^{\alpha_1}+c_2x_{r_1}^{\beta_1}+c_3X_3+\cdots+c_tX_t$. Since $X_1\succ X_2$, then $i_1\leq r_1$. Two cases arise.

Case 1.1.1. Assume first that $i_1<r_1$, then
\begin{center}
	$0=x_{i_1}f-fx_{i_1}=c_1x_{i_1}^{\alpha_1+1}+c_2x_{i_1}x_{r_1}^{\beta_1}+c_3x_{i_1}X_3+\cdots+c_tx_{i_1}X_t-(c_1x_{i_1}^{\alpha_1+1}+
	c_2x_{i_1}x_{r_1}^{\beta_1}+c_2\beta_1d_{i_1r_1}x_{r_1}^{\beta_{1}-1}+c_3X_3x_{i_1}+\cdots+c_tX_tx_{i_1})$.
\end{center}
Assume that $\beta_1\geq 2$. After cancelling the similar terms we get that the greatest monomial is $x_{r_1}^{\beta_{1}-1}$, so $c_2\beta_1d_{i_1r_1}=0$. By (i), $\beta_1=0$, a contradiction. So, $\beta_1=1$ and hence each of the other monomials $X_3,X_4,\dots,X_t$ of $f$ involves only one variable, thus $f$ has one of the following forms:
\begin{center}
	
$f=c_1x_{i_1}^{\alpha_1}+c_2x_{r_1}+c_3x_{j_3}+\cdots+c_{t-1}x_{j_{t-1}}+c_{t}$, or,

$f=c_1x_{i_1}^{\alpha_1}+c_2x_{r_1}+c_3x_{j_3}+\cdots+c_{t}x_{j_t}$,
\end{center}
where $r_1<j_3<\cdots<j_t$. If $\alpha_1\geq 2$, then from $0=x_{r_1}f-fx_{r_1}$ we get that $c_1\alpha_1d_{i_1r_1}=0$, so by (i), $\alpha_1=0$, a contradiction.
Then, $\alpha_1=1$ and $f$ is a linear polynomial:
\begin{center}
	
	$f=c_1x_{i_1}+c_2x_{r_1}+c_3x_{j_3}+\cdots+c_{t-1}x_{j_{t-1}}+c_{t}$, or,
	
	$f=c_1x_{i_1}+c_2x_{r_1}+c_3x_{j_3}+\cdots+c_{t}x_{j_t}$.
\end{center}
In both cases consider $f^2\in Z(A)$, from (a), $f^2\neq 0$. Then,
\begin{center}
$0=x_if^2-f^2x_{i_1}=x_{i_1}(c_1x_{i_1}^2+c_2x_{i_r}^2+c_3x_{j_3}^2+\cdots)-(c_1x_{i_1}^2+c_2x_{i_r}^2+c_3x_{j_3}^2+\cdots)x_{i_1}$.
\end{center}
After cancelling the similar terms and reducing we get that $2c_2^2d_{i_1r_1}x_{r_1}=0$, so $2c_2^2d_{i_1r_1}=0$, and by (i), $c_2^2d_{i_1r_1}=0$, a contradiction.

Case 1.1.2. Now assume that $i_1=r_1$. Then $f=c_1x_{i_1}^{\alpha_1}+c_2x_{i_1}^{\beta_1}+c_3X_3+\cdots+c_tX_t$, whence $\alpha_1>\beta_1$. Since $\beta_1\geq 1$, then $\alpha_1\geq 2$. By (ii), there exists $i_2\neq i_1$ and, without lost of generality, we can assume that $i_1<i_2$. Then
\begin{center}
	$0=x_{i_2}f-fx_{i_2}=c_1x_{i_1}^{\alpha_1}x_{i_2}+c_1\alpha_1d_{i_1i_2}x_{i_1}^{\alpha_1-1}+c_2x_{i_1}^{\beta_1}x_{i_2}+c_2\beta_1d_{i_1i_2}x_{i_1}^{\beta_1-1}+c_3x_{i_2}X_3+\cdots+c_tx_{i_2}X_t-(c_1x_{i_1}^{\alpha_1}x_{i_2}+c_2x_{i_1}^{\beta_1}x_{i_2}+c_3X_3x_{i_2}+\cdots+c_tX_tx_{i_2})$.
\end{center}
After cancelling the similar terms we get that the greatest monomial is $x_{i_1}^{\alpha_1-1}$, so $c_1\alpha_1d_{i_1i_2}=0$. By (i), $\alpha_1=0$, a contradiction.

Case 1.2. $X_2$ involves $l\geq 2$ variables. Let $X_2=x_{r_1}^{\beta_1}\cdots x_{r_l}^{\beta_l}$, with $l\geq 2$, $\beta_u\geq 1$, $1\leq u\leq l$, and $r_1<r_2<\cdots<r_l$. Since $X_1\succ X_2$, then $i_1\leq r_1$. We have
\begin{center}
	$0=x_{i_1}f-fx_{i_1}=c_1x_{i_1}^{\alpha_1+1}+c_2x_{i_1}x_{r_1}^{\beta_1}\cdots x_{r_l}^{\beta_l}+c_3x_{i_1}X_3+\cdots+c_tx_{i_1}X_t-(c_1x_{i_1}^{\alpha_1+1}+
	c_2x_{i_1}x_{r_1}^{\beta_1}\cdots x_{r_l}^{\beta_l}+c_2\beta_2d_{i_1r_2}x_{r_1}^{\beta_1}x_{r_2}^{\beta_2-1}x_{r_3}^{\beta_3}\cdots x_{r_{l-1}}^{\beta_{l-1}}x_{r_l}^{\beta_{l}}+\cdots+c_2\beta_{l-1}d_{i_1r_{l-1}}x_{r_1}^{\beta_1}x_{r_2}^{\beta_2}x_{r_3}^{\beta_3}\cdots x_{r_{l-1}}^{\beta_{l-1}-1}x_{i_l}^{\beta_{l}}+c_2\beta_ld_{i_1r_l}x_{r_1}^{\beta_1}x_{r_2}^{\beta_2}x_{r_3}^{\beta_3}\cdots x_{r_{l-1}}^{\beta_{l-1}}x_{r_l}^{\beta_{l}-1}+c_3X_3x_{i_1}+\cdots+c_tX_tx_{i_1})$.
\end{center}
After cancelling the similar terms we get that the greatest monomial is $x_{r_1}^{\beta_1}x_{r_2}^{\beta_2}x_{r_3}^{\beta_3}\cdots x_{r_{l-1}}^{\beta_{l-1}}x_{r_l}^{\beta_{l}-1}$, so $c_2\beta_ld_{i_1r_l}=0$. By (i), $\beta_l=0$, a contradiction.

Case 2. $X_1$ involves $k\geq 2$ variables. Let $X_1=x_{i_1}^{\alpha _1}\cdots x_{i_k}^{\alpha_k}$, with $k\geq 2$,  $\alpha_j\geq 1$, $1\leq j\leq k$, $i_1<i_2<\cdots<i_k$. If $X_2=1$, then $t=2$ and from $x_{i_1}f=fx_{i_1}$ and the computation of step 1 we get that $c_1\alpha_kd_{i_1i_k}=0$. By (i), $\alpha_k=0$, a contradiction. Thus, $X_2\neq 1$. For $X_2$ two cases arise.

Case 2.1. Assume first that $X_2$ involves only one variable. Let $X_2=x_{r_1}^{\beta_1}$, with $\beta_1\geq 1$. Thus, $f=c_1x_{i_1}^{\alpha _1}\cdots x_{i_k}^{\alpha_k}+c_2x_{r_1}^{\beta_1}+c_3X_3+\cdots+c_tX_t$. Since $X_1\succ X_2$, then $i_1\leq r_1$. Assume first that $i_1<r_1$, then
\begin{center}
	$0=x_{i_1}f-fx_{i_1}=c_1x_{i_1}^{\alpha_1+1}x_{i_2}^{\alpha_2}\cdots x_{i_k}^{\alpha_k}+c_2x_{i_1}x_{r_1}^{\beta_1}+c_3x_{i_1}X_3+\cdots+c_tx_{i_1}X_t-(c_1x_{i_1}^{\alpha_1+1}x_{i_2}^{\alpha_2}\cdots x_{i_k}^{\alpha_k}+c_1\alpha_2d_{i_1i_2}x_{i_1}^{\alpha_1}x_{i_2}^{\alpha_2-1}x_{i_3}^{\alpha_3}\cdots x_{i_{k-1}}^{\alpha_{k-1}}x_{i_k}^{\alpha_{k}}+\cdots+c_1\alpha_{k-1}d_{i_1i_{k-1}}x_{i_1}^{\alpha_1}x_{i_2}^{\alpha_2}x_{i_3}^{\alpha_3}\cdots x_{i_{k-1}}^{\alpha_{k-1}-1}x_{i_k}^{\alpha_{k}}+c_1\alpha_kd_{i_1i_k}x_{i_1}^{\alpha_1}x_{i_2}^{\alpha_2}x_{i_3}^{\alpha_3}\cdots x_{i_{k-1}}^{\alpha_{k-1}}x_{i_k}^{\alpha_{k}-1}+c_2x_{i_1}x_{r_1}^{\beta_1}+c_2\beta_1d_{i_1r_1}x_{r_1}^{\beta_{1}-1}
	+c_3X_3x_{i_1}+\cdots+c_tX_tx_{i_1})$.
\end{center}
After cancelling the similar terms we get that the greatest monomial is $x_{i_1}^{\alpha_1}x_{i_2}^{\alpha_2}x_{i_3}^{\alpha_3}\cdots x_{i_{k-1}}^{\alpha_{k-1}}x_{i_k}^{\alpha_{k}-1}$, so $c_1\alpha_kd_{i_1i_k}=0$. By (i), $\alpha_k=0$, a contradiction. Suppose that $i_1=r_1$, then $f=c_1x_{i_1}^{\alpha _1}\cdots x_{i_k}^{\alpha_k}+c_2x_{i_1}^{\beta_1}+c_3X_3+\cdots+c_tX_t$ and
\begin{center}
	$0=x_{i_1}f-fx_{i_1}=c_1x_{i_1}^{\alpha_1+1}x_{i_2}^{\alpha_2}\cdots x_{i_k}^{\alpha_k}+c_2x_{i_1}^{\beta_1+1}+c_3x_{i_1}X_3+\cdots+c_tx_{i_1}X_t-(c_1x_{i_1}^{\alpha_1+1}x_{i_2}^{\alpha_2}\cdots x_{i_k}^{\alpha_k}+c_1\alpha_2d_{i_1i_2}x_{i_1}^{\alpha_1}x_{i_2}^{\alpha_2-1}x_{i_3}^{\alpha_3}\cdots x_{i_{k-1}}^{\alpha_{k-1}}x_{i_k}^{\alpha_{k}}+\cdots+c_1\alpha_{k-1}d_{i_1i_{k-1}}x_{i_1}^{\alpha_1}x_{i_2}^{\alpha_2}x_{i_3}^{\alpha_3}\cdots x_{i_{k-1}}^{\alpha_{k-1}-1}x_{i_k}^{\alpha_{k}}+c_1\alpha_kd_{i_1i_k}x_{i_1}^{\alpha_1}x_{i_2}^{\alpha_2}x_{i_3}^{\alpha_3}\cdots x_{i_{k-1}}^{\alpha_{k-1}}x_{i_k}^{\alpha_{k}-1}+c_2x_{i_1}^{\beta_1+1}
	+c_3X_3x_{i_1}+\cdots+c_tX_tx_{i_1})$.
\end{center}
After cancelling the similar terms we get again that the greatest monomial is $x_{i_1}^{\alpha_1}x_{i_2}^{\alpha_2}x_{i_3}^{\alpha_3}\cdots x_{i_{k-1}}^{\alpha_{k-1}}x_{i_k}^{\alpha_{k}-1}$, so $c_1\alpha_kd_{i_1i_k}=0$. By (i), $\alpha_k=0$, a contradiction.

Case 2.2. $X_2$ involves $l\geq 2$ variables. Let $X_2=x_{r_1}^{\beta_1}\cdots x_{r_l}^{\beta_l}$, with $l\geq 2$, $\beta_u\geq 1$, $1\leq u\leq l$, and $r_1<r_2<\cdots<r_l$. Since $X_1\succ X_2$, then $i_1\leq r_1$. We have
\begin{center}
$0=x_{i_1}f-fx_{i_1}=c_1x_{i_1}^{\alpha_1+1}x_{i_2}^{\alpha_2}\cdots x_{i_k}^{\alpha_k}+c_2x_{i_1}x_{r_1}^{\beta_1}\cdots x_{r_l}^{\beta_l}+c_3x_{i_1}X_3+\cdots+c_tx_{i_1}X_t-(c_1x_{i_1}^{\alpha_1+1}x_{i_2}^{\alpha_2}\cdots x_{i_k}^{\alpha_k}+c_1\alpha_2d_{i_1i_2}x_{i_1}^{\alpha_1}x_{i_2}^{\alpha_2-1}x_{i_3}^{\alpha_3}\cdots x_{i_{k-1}}^{\alpha_{k-1}}x_{i_k}^{\alpha_{k}}+\cdots+c_1\alpha_{k-1}d_{i_1i_{k-1}}x_{i_1}^{\alpha_1}x_{i_2}^{\alpha_2}x_{i_3}^{\alpha_3}\cdots x_{i_{k-1}}^{\alpha_{k-1}-1}x_{i_k}^{\alpha_{k}}+c_1\alpha_kd_{i_1i_k}x_{i_1}^{\alpha_1}x_{i_2}^{\alpha_2}x_{i_3}^{\alpha_3}\cdots x_{i_{k-1}}^{\alpha_{k-1}}x_{i_k}^{\alpha_{k}-1}+
c_2x_{i_1}x_{r_1}^{\beta_1}\cdots x_{r_l}^{\beta_l}+c_2\beta_2d_{i_1r_2}x_{r_1}^{\beta_1}x_{r_2}^{\beta_2-1}x_{r_3}^{\beta_3}\cdots x_{r_{l-1}}^{\beta_{l-1}}x_{r_l}^{\beta_{l}}+\cdots+c_2\beta_{l-1}d_{i_1r_{l-1}}x_{r_1}^{\beta_1}x_{r_2}^{\beta_2}x_{r_3}^{\beta_3}\cdots x_{r_{l-1}}^{\beta_{l-1}-1}x_{i_l}^{\beta_{l}}+c_2\beta_ld_{i_1r_l}x_{r_1}^{\beta_1}x_{r_2}^{\beta_2}x_{r_3}^{\beta_3}\cdots x_{r_{l-1}}^{\beta_{l-1}}x_{r_l}^{\beta_{l}-1}+c_3X_3x_{i_1}+\cdots+c_tX_tx_{i_1})$.
\end{center}
After cancelling the similar terms we get that the greatest monomial is $x_{i_1}^{\alpha_1}x_{i_2}^{\alpha_2}x_{i_3}^{\alpha_3}\cdots x_{i_{k-1}}^{\alpha_{k-1}}x_{i_k}^{\alpha_{k}-1}$, so $c_1\alpha_kd_{i_1i_k}=0$. By (i), $\alpha_k=0$, a contradiction.

Step 3. Thus, $t=1$. We will show that $X_1=1$. Contrary, assume that $f=c_1X_1$, with $X_1\neq 1$. Let $X_1=x_{i_1}^{\alpha _1}\cdots x_{i_k}^{\alpha_k}$, with $k\geq 1$,  $\alpha_j\geq 1$, $1\leq j\leq k$, and $i_1<i_2<\cdots<i_k$. Assume that $k\geq 2$, then from $x_{i_1}f-fx_{i_1}=0$ a direct computation as in step 2 shows that
\begin{center}
$0=x_{i_1}f-fx_{i_1}=c_1x_{i_1}^{\alpha_1+1}x_{i_2}^{\alpha_2}\cdots x_{i_k}^{\alpha_k}-(c_1x_{i_1}^{\alpha_1+1}x_{i_2}^{\alpha_2}\cdots x_{i_k}^{\alpha_k}+c_1\alpha_2d_{i_1i_2}x_{i_1}^{\alpha_1}x_{i_2}^{\alpha_2-1}x_{i_3}^{\alpha_3}\cdots x_{i_{k-1}}^{\alpha_{k-1}}x_{i_k}^{\alpha_{k}}+\cdots+c_1\alpha_{k-1}d_{i_1i_{k-1}}x_{i_1}^{\alpha_1}x_{i_2}^{\alpha_2}x_{i_3}^{\alpha_3}\cdots x_{i_{k-1}}^{\alpha_{k-1}-1}x_{i_k}^{\alpha_{k}}+c_1\alpha_kd_{i_1i_k}x_{i_1}^{\alpha_1}x_{i_2}^{\alpha_2}x_{i_3}^{\alpha_3}\cdots x_{i_{k-1}}^{\alpha_{k-1}}x_{i_k}^{\alpha_{k}-1})$,
\end{center}
and this implies that $c_1\alpha_kd_{i_1i_k}=0$. By (i), $\alpha_k=0$, a contradiction. Thus, $k=1$ and $f=c_1x_{i_1}^{\alpha_1}$. By (ii), there exists $i_2\neq i_1$, without lost of generality we can assume that $i_1<i_2$. Then, since $x_{i_2}f=fx_{i_2}$, we get that $c_1\alpha_1d_{i_1i_2}x_{i_1}^{\alpha_1-1}=0$, so $\alpha_1=0$, a contradiction. This proves that $X_1=1$ and completes the proof of (a).

The cancellation property of $A$ follows from a result due and J.J. Zhang and Jason Bell that says that if $K$ is a field and $A$ is a $K$-algebra with trivial center, then $A$ is cancellative (see \cite{BellZhang}, Proposition 1.3).

(d) Let $I\neq 0$ be a two-sided ideal of $A$, we have to show that $I=A$. Let $0\neq f=c_1X_1+\cdots +c_tX_t\in I$ as in (a), i.e., $X_l\in \mathrm{Mon}(A)$, $c_l\in
K^*$ for $1\leq l\leq t$ and $X_1\succ \cdots \succ X_t$. If $x_kf-fx_k=0$ for every $1\leq k\leq n$, then $f\in Z(A)=K$, so $f\in K^*$, whence $I=A$. Thus, assume that there exists $1\leq k\leq n$ such that $x_kf-fx_k\neq 0$. Observe that \begin{center}
	$x_kf-fx_k=c_1x_kX_1+\cdots+ c_tx_kX_t-c_1X_1x_k-\cdots -c_tX_tx_k=(c_1x_kX_1-c_1X_1x_k)+\cdots+(c_tx_kX_t-c_tX_tx_k)$,
\end{center}
in addition, for any $1\leq l\leq t$,
\begin{center}
	$lm(c_lx_kX_l)=x_1^{\alpha_{1,l}}\cdots x_{k-1}^{\alpha_{k-1,l}}x_k^{\alpha_{k,l}+1}x_{k+1}^{\alpha_{k+1,l}}\cdots x_n^{\alpha_{n,l}}=lm(c_lX_lx_k)$, where $X_l:=x_1^{\alpha_{1,l}}\cdots x_n^{\alpha_{n,l}}$,
\end{center}
moreover, if $c'X'$ is any other term of $c_lx_kX_l-c_lX_lx_k$, then $X_l\succ X'$. Hence, $lm(f)\succ lm(x_kf-fx_k)$. We can repeat the previous reasoning for $f':=x_kf-fx_k\in I$, $f'':=x_{k'}f'-f'x_{k'}\in I$, etc., with $lm(f)\succ lm(f')\succ lm(f'')\succ \cdots $. From this we can conclude that $I=A$.

\end{proof}

Our next results are focused on the skew $PBW$ extension introduced in Theorem \ref{theorem4.16}.

\begin{definition}\label{definition4.21}
Let $K$ be a field of characteristic $0$ and $n\geq 2$. The $K$-algebra generated by $x_1,\dots, x_n$ subject to relations
\begin{center}
	$x_jx_i=x_ix_j+d_{ij}$, for all $1\leq i<j\leq n$, with $d_{ij}\neq 0$,
\end{center}
will be denoted by $\mathcal{CSD}_n(K)$.  	
\end{definition}

The next corollary of Theorem \ref{theorem4.16} describes the normal elements of the algebra $\mathcal{CSD}_n(K)$. A more general result can be proved first.

\begin{proposition}
	Let $A=\sigma(R)\langle x_1,\dots,x_n\rangle$ be a skew $PBW$ extension of a ring $R$ that satisfies the following conditions:
	\begin{enumerate}
		\item[\rm (i)]$R$ is a domain.
		\item[\rm (ii)]$\sigma_k=i_K$ and $\delta_k=0$ for $1\leq k\leq n$.
		\item[\rm (iii)]$A$ is simple.
	\end{enumerate}
	Then, the set of all normal elements of $A$ coincides with $R^*\cup\{0\}$.
\end{proposition}
\begin{proof}
	Let $a\neq 0$ be a normal element of $A$. Since $A$ is a simple ring, then the two-sided ideal of $A$ generated by $a$ coincides with $A$, whence $1=p_1aq_1+\cdots+p_taq_t$, with $p_i,q_i\in A$, $1\leq i\leq t$. As $a$ is normal, there exist $p_1',q_1'\dots, p_t',q_t'\in A$ such that $1=ap_1'q_1+\cdots+ap_t'q_t=a(p_1'q_1+\cdots+p_t'q_t)=p_1q_1'a+\cdots+p_tq_t'a=(p_1q_1'+\cdots+p_tq_t')a$, i.e., $a$ is invertible. From (i) and Proposition \ref{corollary3.2.2} we get that $a\in R^*$. Conversely, let $a\in R^*$, from (ii), $ax_i=x_ia$ for $1\leq i\leq n$, moreover $aR=Ra$, so $a$ is normal.
\end{proof}

\begin{corollary}
The set of all normal elements of $\mathcal{CSD}_n(K)$ coincides with $K$.
\end{corollary}
\begin{proof}
The corollary follows from Theorem \ref{theorem4.16} and the previous proposition.
\end{proof}

\begin{remark}
Many ring-theoretic, homological and computational properties of skew $PBW$ extensions have been investigated in the last two decades (see \cite{Lezama-sigmaPBW}). These properties can be applied in particular to $\mathcal{CSD}_n(K)$. For example, for the \textit{Gelfand-Kirillov dimension} (see \cite{Krause}) we have that ${\rm GKdim}(\mathcal{CSD}_n(K))=n$ (\cite{Lezama-sigmaPBW}, Theorem 7.4.1); the \textit{generalized Hilbert series} of $\mathcal{CSD}_n(K)$ is $\frac{1}{(1-t)^n}$ (\cite{Lezama-sigmaPBW}, Theorem 18.2.3); $\mathcal{CSD}_n(K)$ is a \textit{semi-graded Artin-Schelter regular algebra} (\cite{Lezama-sigmaPBW}, Definition 19.4.1 and Theorem 19.4.12); the left (right) ideals and modules over $\mathcal{CSD}_n(K)$ have Gröbner bases that can be computed effectively with the \texttt{MAPLE} library \textbf{SPBWE} developed in \cite{Fajardo2}, \cite{Fajardo3} and \cite{Lezama-sigmaPBW}.
\end{remark}

\section{Matrix and computational approach to Dixmier problem}\label{section5}
In this section we present a matrix-constructive and computational approach to Dixmier problem. We compute some nontrivial automporphisms of $A_1(K)$ and $\mathcal{CSD}_n(K)$. Recall that $char(K)=0$.

\subsection{Dixmier polynomials and its matrix characterization}\label{subsection5.1}

It is clear that any endomorphism of $A_1(K)$ is given by $t\rightarrow p$ and $x\rightarrow q$, where $p,q\in A_1(K)$ satisfy $qp=pq+1$. Reciprocally, any couple of polynomials $p,q\in A_1(K)$ satisfying $qp=pq+1$ defines an endomorphism of  $A_1(K)$, $t\rightarrow p$ and $x\rightarrow q$. We asked to \textrm{ChatGPT} for a such couple of non trivial polynomials $p$ and $q$, but all answers (attempts) were wrong, for example, some couples given by \textrm{ChatGPT} were $p:=t^2+2t+1$, $q:=x+t$; $p:=t^2+t$, $q:=x+1$;
$p:=t$, $q:=t^2+x$ and $p:=t^2$, $q:=t+x$. In all cases we compute manually that $qp\neq pq+1$. It is curios that \textrm{ChatGPT} ignore the paper (\cite{Dixmier}) and did not give for example the couple $p:=t$, $q:=t+x$ (clearly $qp=pq+1$). The endomorphism $\alpha:A_1(K)\to A_1(K)$ given by $\alpha(t)=:t$, $\alpha(x)=:t+x$ is moreover bijective since $t,x\in Im(\alpha)$. In fact, $\alpha(t)=t$ and $\alpha(x-t)=x$. Thus, $ \alpha$ is bijective since $A_1(K)$ is simple. In general, the automorphisms of $A_1(K)$ are well known. Dixmier in \cite{Dixmier} proved that $Aut(A_1(K))$ is generated by two types of automorphisms: $\Phi_{n,\lambda}(t):=t+\lambda x^n$, $\Phi_{n,\lambda}(x):=x$ and $\Phi_{n,\lambda}'(t):=t$, $\Phi_{n,\lambda}'(x):=\lambda t^n+x$, with $\lambda \in K$ and $n\geq 0$. Observe that $\alpha $ corresponds to $\Phi_{1,1}'$.

\begin{remark}
Using the \texttt{MAPLE} library \textbf{SPBWE}, or more exactly, the \texttt{DixmierAutomorphismFactor} function of the \texttt{DixmierProblem} package of \textbf{SPBWE} (see Subsection \ref{DixmierAutomorphismFactor} below), was proved that the automorphism $\Upsilon_\mu$ of $A_1(K)$ (added in \cite{Levandovsky} to the system of generators of $Aut_K(A_1(K))$) defined by $t\mapsto \mu t$, $x\mapsto \frac{1}{\mu}x$, with $\mu\in K^*$, can be factorized in the following way:
\begin{center}
$\Upsilon_\mu=\Phi_{1,\mu-\mu^2}'\circ \Phi_{1,-\frac{1}{\mu}}\circ \Phi_{1,\mu-1}'\circ \Phi_{1,1}$.
\end{center}
This identity can be verified by direct computation: 
\begin{center}
	$\Phi_{1,\mu-\mu^2}'\circ \Phi_{1,-\frac{1}{\mu}}\circ \Phi_{1,\mu-1}'\circ \Phi_{1,1}(t)=\Phi_{1,\mu-\mu^2}'\circ \Phi_{1,-\frac{1}{\mu}}\circ \Phi_{1,\mu-1}'(t+x)=\Phi_{1,\mu-\mu^2}'\circ \Phi_{1,-\frac{1}{\mu}}(\mu t+x)=\Phi_{1,\mu-\mu^2}'(\mu t)=\mu t$, and $\Phi_{1,\mu-\mu^2}'\circ \Phi_{1,-\frac{1}{\mu}}\circ \Phi_{1,\mu-1}'\circ \Phi_{1,1}(x)=\Phi_{1,\mu-\mu^2}'\circ \Phi_{1,-\frac{1}{\mu}}\circ \Phi_{1,\mu-1}'(x)=\Phi_{1,\mu-\mu^2}'\circ \Phi_{1,-\frac{1}{\mu}}((\mu -1)t+x)=
	\Phi_{1,\mu-\mu^2}'((\mu -1)t+\frac{1}{\mu}x)=\frac{1}{\mu}x$.
\end{center}
\end{remark}

\begin{definition}
Let $p,q\in A_1(K)$. We say that $p$ and $q$ are \textbf{\textit{Dixmier polynomials}} if $qp=pq+1$.
\end{definition}

Compare the Dixmier polynomials with the $DC$-pairs of \cite{Zheglov} and also with the $cv$-polynomials studied in \cite{Lam}, \cite{Lam2} and \cite{Ramirez}.

\begin{corollary}\label{corollary5.2}
Let $\alpha:A_1(K)\to A_1(K)$ be a map. Then, $\alpha$ is an endomorphism of $A_1(K)$ if and only if there exist Dixmier polynomials $p,q\in A_1(K)$ such that $\alpha(t):=p$ and $\alpha(x):=q$.
\end{corollary}
\begin{proof}
This follows from $A_1(K)\cong K\{t,x\}/I$, where $K\{t,x\}$ is the free $K$-algebra in the alphabet $\{t,x\}$ and $I$ is the two-sided ideal of $K\{t,x\}$ generated by $xt-tx-1$ (see \cite{Lezama2} for free $K$-algebras).
\end{proof}

\begin{corollary}\label{corollary5.3}
Let $\alpha:A_1(K)\to A_1(K)$ be a map. Then, $\alpha$ is an automorphism of $A_1(K)$ if and only if there exist Dixmier polynomials $p,q\in A_1(K)$ such that $\alpha(t):=p$ and $\alpha(x):=q$.
\end{corollary}
\begin{proof}
This is a direct consequence of Corollary \ref{corollary5.2}  and Theorem \ref{Theorem1.3}.
\end{proof}

We want to give a matrix characterization of Dixmier polynomials. For $n \geq 1$, without loss of generality (completing with zero summands), we can assume that
\begin{center}
	$p=p_0(t)+p_1(t)x+p_2(t)x^2+\cdots+p_n(t)x^n$ and $q=q_0(t)+q_1(t)x+q_2(t)x^2+\cdots+q_n(t)x^n$,
\end{center}
where $p_i(t),q_i(t)\in K[t]$, for $0\leq i\leq n$.

Next we will consider some particular cases of Dixmier polynomials.
\begin{proposition}\label{proposition5.3}
	Assume that $p:=p_0(t)+p_1(t)x$ and $q:=q_0(t)+q_1(t)x$, with $p_0(t),p_1(t),q_0(t),q_1(t)\in K[t]$. If $p$ and $q$ are Dixmier polynomials, i.e., $qp=pq+1$, then
\begin{center}
$q_1(t)p_0'(t)-p_1(t)q_0'(t)=1$,

$q_1(t)p_1'(t)-p_1(t)q_1'(t)=0$.
\end{center}
i.e.,
\begin{equation}\label{equ5.1}
\det \begin{bmatrix}q_1(t) & p_1(t)\\
q_0'(t) & p_0'(t)\end{bmatrix}=1
\end{equation}
and
\begin{equation}\label{equ5.2}
	\det \begin{bmatrix}q_1(t) & p_1(t)\\
	q_1'(t) & p_1'(t)\end{bmatrix}=0.
\end{equation}
Conversely, let $p_0(t),p_1(t),q_0(t),q_1(t)\in K[t]$ be polynomials such that satify the previous equalities, then $p:=p_0(t)+p_1(t)x$ and $q:=q_0(t)+q_1(t)x$ are Dixmier polynomials, and hence, $p$ and $q$ define an automorphism $\alpha$ of $A_1(K)$, $\alpha(t):=p$ and $\alpha(x):=q$.
\end{proposition}
\begin{proof}
We first compute $qp$ and $pq$ in $A_1(K)$ (using the rule $xa(t)=a(t)x+a'(t)$, with $a(t)\in K[t])$, and then we compare the similar terms. This and Corollary \ref{corollary5.3} prove the proposition.	
\end{proof}

Without using Corollary \ref{corollary5.3}, and hence, without using Theorem \ref{Theorem1.3}, we can prove the following result.
\begin{proposition}\label{proposition5.4}
	If $p$ and $q$ has the form $p:=bt+ax$ and $q:=dt+cx$, with $a,b,c,d\in K$, and they are Dixmier polynomials,  then the induced endomorphism $\alpha$ is an automorphism of $A_1(K)$:
	\begin{center}
		$\begin{bmatrix}		p\\
		q
		\end{bmatrix}=\begin{bmatrix} b & a\\ d   & c\end{bmatrix}\begin{bmatrix}
		t\\
		x
		\end{bmatrix}$
		and $\begin{bmatrix}
		t\\
		x
		\end{bmatrix}=\begin{bmatrix}b & a\\
		d & c
		\end{bmatrix}^{-1}\begin{bmatrix}
		p\\
		q
		\end{bmatrix}$.
	\end{center}
\end{proposition}
\begin{proof}
From (\ref{equ5.1}), (\ref{equ5.2}) and Corollary \ref{corollary5.2}, $p$ and $q$ define an endomorphism $\alpha$ of $A_1(K)$. Morover, since

\begin{center}
$1=\det \begin{bmatrix}q_1(t) & p_1(t)\\
q_0'(t) & p_0'(t)\end{bmatrix}=\det \begin{bmatrix}c & a\\
d & b\end{bmatrix}=-\det \begin{bmatrix}d & b\\
c & a\end{bmatrix}=-\det \begin{bmatrix}d & b\\
c & a\end{bmatrix}^T=-\det \begin{bmatrix}d & c\\
b & a\end{bmatrix}=\det \begin{bmatrix}b & a\\
d & c\end{bmatrix}$,	
\end{center}
then $\begin{bmatrix}b & a\\
d & c\end{bmatrix}$ is an invertible matrix of $K[t]$. Whence, $\alpha$ is surjective, but since $A_1(K)$ is simple (see also Definition \ref{definition4.21}, Theorem \ref{theorem4.16} and recall that $A_1(K)=\mathcal{CSD}_2(K)$, with $d_{12}=1$), then $\alpha$ is bijective.
\end{proof}

\begin{example}Let $p=3t+4x$ and $q=2t+3x$, so $p_0(t)=3t$, $p_1(t)=4$ and $q_0(t)=2t$, $q_1(t)=3$, whence
\begin{center}
$\det \begin{bmatrix}3 & 4\\
2 & 3\end{bmatrix}=1$ and
$\det \begin{bmatrix}3 & 4\\
	0 & 0\end{bmatrix}=0$,
\end{center}
i.e., $\alpha$ defined by $\alpha (t)=p$ $\alpha (x)=q$ is an automorphism of $A_1(K)$. More specifically,
\begin{center}
	$\begin{bmatrix}
	p\\
	q
	\end{bmatrix}=\begin{bmatrix}3 & 4\\
	2 & 3
	\end{bmatrix}\begin{bmatrix}
	t\\
	x
	\end{bmatrix}$, thus $\begin{bmatrix}
	t\\
	x
	\end{bmatrix}=\begin{bmatrix}3 & 4\\
	2 & 3
	\end{bmatrix}^{-1}\begin{bmatrix}
	p\\
	q
	\end{bmatrix}=\begin{bmatrix}3 & -4\\
	-2 & 3
	\end{bmatrix}\begin{bmatrix}
	p\\
	q
	\end{bmatrix}$.
\end{center}
 Then, $t=3p-4q$ and $x=-2p+3q$ and therefore $t=3\alpha(t)-4\alpha(x)=
\alpha(3t-4x)$ and $x=-2\alpha(t)+3\alpha(x)=\alpha(-2t+3x)$. Thus, $\alpha$ is surjective, and hence bijective.
\end{example}

The particular case considered in Proposition \ref{proposition5.3} can be extended in the following way.
\begin{proposition}\label{proposition5.6}
	Let $p:=p_0(t)+p_1(t)x+p_2(t)x^2$ and $q:=q_0(t)+q_1(t)x+q_2(t)x^2$,
	with $p_0(t),p_1(t),p_2(t)$, $q_0(t),q_1(t),q_2(t)\in K[t]$. Then, $p$ and $q$ are Dixmier polynomials if and only if the following $4$ identities hold, and hence, $p$ and $q$ define an automorphism $\alpha$ of $A_1(K)$, $\alpha(t):=p$ and $\alpha(x):=q$:
	\begin{equation}
	\det \begin{bmatrix}q_1(t) & p_1(t)\\
	q_0'(t) & p_0'(t)\end{bmatrix}+\det \begin{bmatrix}q_2(t) & p_2(t)\\
	q_0''(t) & p_0''(t)\end{bmatrix}=1
	\end{equation}
	\begin{equation}
	\det \begin{bmatrix}q_1(t) & p_1(t)\\
	q_1'(t) & p_1'(t)\end{bmatrix}+2\det \begin{bmatrix}q_2(t) & p_2(t)\\
	q_0'(t) & p_0'(t)
	\end{bmatrix}+
	\det \begin{bmatrix}q_2(t) & p_2(t)\\
	q_1''(t) & p_1''(t)\end{bmatrix}=0
	\end{equation}
	\begin{equation}
	\det \begin{bmatrix}q_1(t) & p_1(t)\\
	q_2'(t) & p_2'(t)\end{bmatrix}+2\det \begin{bmatrix}q_2(t) & p_2(t)\\
	q_1'(t) & p_1'(t)\end{bmatrix}+\det \begin{bmatrix}q_2(t) & p_2(t)\\
	q_2''(t) & p_2''(t)\end{bmatrix}=0
	\end{equation}
	\begin{equation}
	\det \begin{bmatrix}q_2(t) & p_2(t)\\
	q_2'(t) & p_2'(t)\end{bmatrix}=0,
	\end{equation}
\end{proposition}
\begin{proof}
As in the proof of Proposition \ref{proposition5.3}, we first compute $qp$ and $pq$ in $A_1(K)$ (using the rule $xa(t)=a(t)x+a'(t)$, with $a(t)\in K[t])$, and then we compare the similar terms in the variable $x$. From this we obtain the matrix relations of the proposition.	The second assertion follows from Corollary \ref{corollary5.3}.
\end{proof}

Proposition \ref{proposition5.6} can be generalized in the following way.

\begin{proposition}\label{proposition5.7}
Let $p:=p_0(t)+p_1(t)x+p_2(t)x^2+p_3(t)x^3$ and $q:=q_0(t)+q_1(t)x+q_2(t)x^2+q_3(t)x^3$,
with $p_0(t),p_1(t),p_2(t),p_3(t)$, $q_0(t),q_1(t),q_2(t),q_3(t)\in K[t]$. Then, $p$ and $q$ are Dixmier polynomials if and only if the following $6$ identities hold,
and hence, $p$ and $q$ define an automorphism $\alpha$ of $A_1(K)$, $\alpha(t):=p$ and $\alpha(x):=q$:
\tiny{
\begin{equation}
\det \begin{bmatrix}q_1(t) & p_1(t)\\
q_0'(t) & p_0'(t)\end{bmatrix}+\det \begin{bmatrix}q_2(t) & p_2(t)\\
q_0''(t) & p_0''(t)\end{bmatrix}+\det \begin{bmatrix}q_3(t) & p_3(t)\\
q_0'''(t) & p_0'''(t)\end{bmatrix}=1
\end{equation}
\begin{equation}
\det \begin{bmatrix}q_1(t) & p_1(t)\\
q_1'(t) & p_1'(t)\end{bmatrix}+2\det \begin{bmatrix}q_2(t) & p_2(t)\\
q_0'(t) & p_0'(t)
\end{bmatrix}+
\det \begin{bmatrix}q_2(t) & p_2(t)\\
q_1''(t) & p_1''(t)\end{bmatrix}+3\det \begin{bmatrix}q_3(t) & p_3(t)\\
q_0''(t) & p_0''(t)+
\end{bmatrix}+\det \begin{bmatrix}q_3(t) & p_3(t)\\
q_1'''(t) & p_1'''(t)
\end{bmatrix}=0
\end{equation}
\begin{equation}
\det \begin{bmatrix}q_1(t) & p_1(t)\\
q_2'(t) & p_2'(t)\end{bmatrix}+2\det \begin{bmatrix}q_2(t) & p_2(t)\\
q_1'(t) & p_1'(t)\end{bmatrix}+\det \begin{bmatrix}q_2(t) & p_2(t)\\
q_2''(t) & p_2''(t)\end{bmatrix}+3\det\begin{bmatrix}q_3(t) & p_3(t)\\
q_0'(t) & p_0'(t)\end{bmatrix}+3\det\begin{bmatrix}q_3(t) & p_3(t)\\
q_1''(t) & p_1''(t)\end{bmatrix}+\det\begin{bmatrix}q_3(t) & p_3(t)\\
q_2'''(t) & p_2'''(t)\end{bmatrix}=0\end{equation}
\begin{equation}
\det\begin{bmatrix}q_1(t) & p_1(t)\\
q_3'(t) & p_3'(t)\end{bmatrix}+2\det\begin{bmatrix}q_2(t) & p_2(t)\\
q_2'(t) & p_2'(t)\end{bmatrix}+\det\begin{bmatrix}q_2(t) & p_2(t)\\
q_3''(t) & p_3''(t)\end{bmatrix}+3\det\begin{bmatrix}q_3(t) & p_3(t)\\
q_1'(t) & p_1'(t)\end{bmatrix}+3\det\begin{bmatrix}q_3(t) & p_3(t)\\
q_2''(t) & p_2''(t)\end{bmatrix}+\det\begin{bmatrix}q_3(t) & p_3(t)\\
q_3'''(t) & p_3'''(t)\end{bmatrix}=0
\end{equation}
\begin{equation}
2\det \begin{bmatrix}q_2(t) & p_2(t)\\
q_3'(t) & p_3'(t)\end{bmatrix}+3\det \begin{bmatrix}q_3(t) & p_3(t)\\
q_2'(t) & p_2'(t)\end{bmatrix}+3\det \begin{bmatrix}q_3(t) & p_3(t)\\
q_3''(t) & p_3''(t)\end{bmatrix}=0
\end{equation}
\begin{equation}
\det \begin{bmatrix}
q_3(t) & p_3(t)\\
q_3'(t) & p_3'(t)
\end{bmatrix}=0.
\end{equation}}
\end{proposition}
\begin{proof}
The proof is as in Proposition \ref{proposition5.6}.
\end{proof}

One more generalization before the general case $n\geq 1$.

\begin{proposition}\label{proposition5.8}
	Let
	\begin{center}
	$p:=p_0(t)+p_1(t)x+p_2(t)x^2+p_3(t)x^3+p_4(t)x^4$ and $q:=q_0(t)+q_1(t)x+q_2(t)x^2+q_3(t)x^3+q_4(t)x^4$,
	\end{center}
	with $p_i(t), q_i(t)\in K[t], 0\leq i\leq 4$. Then, $p$ and $q$ are Dixmier polynomials if and only if the following $8$ identities hold, and hence, $p$ and $q$ define an automorphism $\alpha$ of $A_1(K)$, $\alpha(t):=p$ and $\alpha(x):=q$:
	\tiny
	\begin{equation*}
	\det \begin{bmatrix}q_1(t) & p_1(t)\\
	q_0'(t) & p_0'(t)\end{bmatrix}+\det \begin{bmatrix}q_2(t) & p_2(t)\\
	q_0''(t) & p_0''(t)\end{bmatrix}+\det \begin{bmatrix}q_3(t) & p_3(t)\\
	q_0'''(t) & p_0'''(t)\end{bmatrix}+
	\end{equation*}
	\begin{equation*}
	+\det \begin{bmatrix}
	q_4(t) & p_4(t)\\
	q_0''''(t) & p_0''''(t)
	\end{bmatrix}=1;
	\end{equation*}
	\begin{equation*}
	\det \begin{bmatrix}q_1(t) & p_1(t)\\
	q_1'(t) & p_1'(t)\end{bmatrix}+2\det \begin{bmatrix}q_2(t) & p_2(t)\\
	q_0'(t) & p_0'(t)
	\end{bmatrix}+
	\det \begin{bmatrix}q_2(t) & p_2(t)\\
	q_1''(t) & p_1''(t)\end{bmatrix}+3\det \begin{bmatrix}q_3(t) & p_3(t)\\
	q_0''(t) & p_0''(t)+
	\end{bmatrix}+\det \begin{bmatrix}q_3(t) & p_3(t)\\
	q_1'''(t) & p_1'''(t)
	\end{bmatrix}+
	\end{equation*}
	\begin{equation*}
	+4\det \begin{bmatrix}q_4(t) & p_4(t)\\
	q_0'''(t) & p_0'''(t)\end{bmatrix}+\det \begin{bmatrix}q_4(t) & p_4(t)\\
	q_1''''(t) & p_1''''(t)\end{bmatrix}=0;
	\end{equation*}
	\begin{equation*}
	\det \begin{bmatrix}q_1(t) & p_1(t)\\
	q_2'(t) & p_2'(t)\end{bmatrix}+2\det \begin{bmatrix}q_2(t) & p_2(t)\\
	q_1'(t) & p_1'(t)\end{bmatrix}+\det \begin{bmatrix}q_2(t) & p_2(t)\\
	q_2''(t) & p_2''(t)\end{bmatrix}+3\det\begin{bmatrix}q_3(t) & p_3(t)\\
	q_0'(t) & p_0'(t)\end{bmatrix}+3\det\begin{bmatrix}q_3(t) & p_3(t)\\
	q_1''(t) & p_1''(t)\end{bmatrix}+\det\begin{bmatrix}q_3(t) & p_3(t)\\
	q_2'''(t) & p_2'''(t)\end{bmatrix}+
	\end{equation*}
	\begin{equation*}
	+6\det \begin{bmatrix}q_4(t) & p_4(t)\\
	q_0''(t) & p_0''(t)\end{bmatrix}+4\det \begin{bmatrix}q_4(t) & p_4(t)\\
	q_1'''(t) & p_1'''(t)\end{bmatrix}+\det \begin{bmatrix}q_4(t) & p_4(t)\\
	q_2''''(t) & p_2''''(t)\end{bmatrix}=0;
	\end{equation*}
	\begin{equation*}
	\det\begin{bmatrix}q_1(t) & p_1(t)\\
	q_3'(t) & p_3'(t)\end{bmatrix}+2\det\begin{bmatrix}q_2(t) & p_2(t)\\
	q_2'(t) & p_2'(t)\end{bmatrix}+\det\begin{bmatrix}q_2(t) & p_2(t)\\
	q_3''(t) & p_3''(t)\end{bmatrix}+3\det\begin{bmatrix}q_3(t) & p_3(t)\\
	q_1'(t) & p_1'(t)\end{bmatrix}+3\det\begin{bmatrix}q_3(t) & p_3(t)\\
	q_2''(t) & p_2''(t)\end{bmatrix}+\det\begin{bmatrix}q_3(t) & p_3(t)\\
	q_3'''(t) & p_3'''(t)\end{bmatrix}+
	\end{equation*}
	\begin{equation*}
	+4 \det \begin{bmatrix}q_4(t) & p_4(t)\\
	q_0'(t) & p_0'(t)\end{bmatrix}+6 \det \begin{bmatrix}q_4(t) & p_4(t)\\
	q_1''(t) & p_1''(t)
	\end{bmatrix}+4 \det \begin{bmatrix}q_4(t) & p_4(t)\\
	q_2'''(t) & p_2'''(t)\end{bmatrix}+\det \begin{bmatrix}q_4(t) & p_4(t)\\
	q_3''''(t) & p_3''''(t)\end{bmatrix}=0;
	\end{equation*}
	\begin{equation*}
	2\det \begin{bmatrix}q_2(t) & p_2(t)\\
	q_3'(t) & p_3'(t)\end{bmatrix}+3\det \begin{bmatrix}q_3(t) & p_3(t)\\
	q_2'(t) & p_2'(t)\end{bmatrix}+3\det \begin{bmatrix}q_3(t) & p_3(t)\\
	q_3''(t) & p_3''(t)\end{bmatrix}+
	\end{equation*}
	\begin{equation*}
	+\det\begin{bmatrix}
	q_1(t) & p_1(t)\\
	q_4'(t) & p_4'(t)
	\end{bmatrix}+\det \begin{bmatrix}q_2(t) & p_2(t)\\
	q_4''(t) & p_4''(t)
	\end{bmatrix}+\det \begin{bmatrix}
	q_3(t) & p_3(t)\\
	q_4'''(t) & p_4'''(t)
	\end{bmatrix}+
	\end{equation*}
	\begin{equation*}
	+4\det \begin{bmatrix}
	q_4(t) & p_4(t)\\
	q_1'(t) & p_1'(t)
	\end{bmatrix}+6\det \begin{bmatrix}
	q_4(t) & p_4(t)\\
	q_2''(t) & p_2''(t)
	\end{bmatrix}+4\det \begin{bmatrix}
	q_4(t) & p_4(t)\\
	q_3'''(t) & p_3'''(t)
	\end{bmatrix}+
	\det \begin{bmatrix}
	q_4(t) & p_4(t)\\
	q_4''''(t) & p_4''''(t)
	\end{bmatrix}=0;
	\end{equation*}
	\begin{equation*}
	3\det \begin{bmatrix}
	q_3(t) & p_3(t)\\
	q_3'(t) & p_3'(t)
	\end{bmatrix}+
	\end{equation*}
	\begin{equation*}
	+2\det \begin{bmatrix}
	q_2(t) & p_2(t)\\
	q_4'(t) & p_4'(t)
	\end{bmatrix}+3\det \begin{bmatrix}
	q_3(t) & p_3(t)\\
	q_4''(t) & p_4''(t)
	\end{bmatrix}+4 \det \begin{bmatrix}
	q_4(t) & p_4(t)\\
	q_2'(t) & p_2'(t)
	\end{bmatrix}+6 \det \begin{bmatrix}
	q_4(t) & p_4(t)\\
	q_3''(t) & p_3''(t)
	\end{bmatrix}+4\det \begin{bmatrix}
	q_4(t) & p_4(t)\\
	q_4'''(t) & p_4'''(t)
	\end{bmatrix}=0;
	\end{equation*}
	\begin{equation*}
	3\det \begin{bmatrix}q_3(t) & p_3(t)\\
	q_4'(t) & p_4'(t)
	\end{bmatrix}+4 \det\begin{bmatrix}
	q_4(t) & p_4(t)\\
	q_3'(t) & p_3'(t)
	\end{bmatrix}+6\det \begin{bmatrix}
	q_4(t) & p_4(t)\\
	q_4''(t) & p_4''(t)
	\end{bmatrix}=0;
	\end{equation*}
	\begin{equation*}
	\det \begin{bmatrix}
	q_4(t) & p_4(t)\\
	q_4'(t) & p_4'(t)
	\end{bmatrix}=0.
	\end{equation*}
\end{proposition}
\begin{proof}
	The proof is as in Proposition \ref{proposition5.6}.
\end{proof}
The previous propositions induced the following theorem that gives a matrix characterization of Dixmier polynomials, and hence, a matrix approach to Dixmier problem.

\begin{theorem}\label{theorem5.9a}
	Let $n \geq 1$ and
	\begin{center}
		$p=p_0(t)+p_1(t)x+p_2(t)x^2+\cdots+p_n(t)x^n, q=q_0(t)+q_1(t)x+q_2(t)x^2+\cdots+q_n(t)x^n \in A_1(K)$.
	\end{center}
	$p$ and $q$ are Dixmier polynomials if and only if the following $2n$ identities hold, and hence, $p$ and $q$ define an automorphism $\alpha$ of $A_1(K)$, $\alpha(t):=p$ and $\alpha(x):=q$.:
	\tiny
	\begin{equation*}
	\det \begin{bmatrix}q_1(t) & p_1(t)\\
	q_0'(t) & p_0'(t)\end{bmatrix}+\det \begin{bmatrix}q_2(t) & p_2(t)\\
	q_0''(t) & p_0''(t)\end{bmatrix}+\det \begin{bmatrix}q_3(t) & p_3(t)\\
	q_0'''(t) & p_0'''(t)\end{bmatrix}+\cdots +\det\begin{bmatrix}q_n(t) & p_n(t)\\
	q_0^{(n)}(t) & p_0^{(n)}(t)\end{bmatrix}=1;
	\end{equation*}
	\begin{equation*}
	\det \begin{bmatrix}q_1(t) & p_1(t)\\
	q_1'(t) & p_1'(t)\end{bmatrix}+2\det \begin{bmatrix}q_2(t) & p_2(t)\\
	q_0'(t) & p_0'(t)
	\end{bmatrix}+
	\det \begin{bmatrix}q_2(t) & p_2(t)\\
	q_1''(t) & p_1''(t)\end{bmatrix}+\cdots+(n-1)\det \begin{bmatrix}q_{n-1}(t) & p_{n-1}(t)\\
	q_0^{(n-2)}(t) & p_0^{(n-2)}(t)+
	\end{bmatrix}+\det \begin{bmatrix}q_{n-1}(t) & p_{n-1}(t)\\
	q_1^{(n-1)}(t) & p_1^{(n-1)}(t)
	\end{bmatrix}+
	\end{equation*}
	\begin{equation*}
	+n\det \begin{bmatrix}q_n(t) & p_n(t)\\
	q_0^{(n-1)}(t) & p_0^{(n-1)}(t)\end{bmatrix}+\det \begin{bmatrix}q_n(t) & p_n(t)\\
	q_1^{(n)}(t) & p_1^{(n)}(t)\end{bmatrix}=0;
	\end{equation*}
	\begin{equation*}
	\det \begin{bmatrix}q_1(t) & p_1(t)\\
	q_2'(t) & p_2'(t)\end{bmatrix}+2\det \begin{bmatrix}q_2(t) & p_2(t)\\
	q_1'(t) & p_1'(t)\end{bmatrix}+\det \begin{bmatrix}q_2(t) & p_2(t)\\
	q_2''(t) & p_2''(t)\end{bmatrix}+
	\end{equation*}
	\begin{equation*}
	3\det\begin{bmatrix}q_3(t) & p_3(t)\\
	q_0'(t) & p_0'(t)\end{bmatrix}+3\det\begin{bmatrix}q_3(t) & p_3(t)\\
	q_1''(t) & p_1''(t)\end{bmatrix}+\det\begin{bmatrix}q_3(t) & p_3(t)\\
	q_2'''(t) & p_2'''(t)\end{bmatrix}+\cdots+
	\end{equation*}
	\begin{equation*}
	(n-1)\det\begin{bmatrix}q_{n-1}(t) & p_{n-1}(t)\\
	q_0^{(n-3)}(t) & p_0^{(n-3)}(t)\end{bmatrix}+(n-1)\det\begin{bmatrix}q_{(n-1)}(t) & p_{(n-1)}(t)\\
	q_1^{(n-2)}(t) & p_1^{(n-2)}(t)\end{bmatrix}+\det\begin{bmatrix}q_{n-1}(t) & p_{n-1}(t)\\
	q_2^{(n-1)}(t) & p_2^{(n-1)}(t)\end{bmatrix}+
	\end{equation*}
	\begin{equation*}
	+\frac{n(n-1)}{2}\det \begin{bmatrix}q_n(t) & p_n(t)\\
		q_0^{(n-2)}(t) & p_0^{(n-2)}(t)\end{bmatrix}+n\det \begin{bmatrix}q_n(t) & p_n(t)\\
		q_1^{(n-1)}(t) & p_1^{(n-1)}(t)\end{bmatrix}+\det \begin{bmatrix}q_n(t) & p_n(t)\\
		q_2^{(n)}(t) & p_2^{(n)}(t)\end{bmatrix}=0;
	\end{equation*}
	\begin{equation*}
	\vdots
	\end{equation*}
	\begin{equation*}
	(n-1)\det \begin{bmatrix}q_{n-1}(t) & p_{n-1}(t)\\
	q_n'(t) & p_n'(t)\end{bmatrix}+n\det \begin{bmatrix}q_n(t) & p_n(t)\\
	q_{n-1}'(t) & p_{n-1}'(t)\end{bmatrix}+\frac{n(n-1)}{2}\det \begin{bmatrix}q_n(t) & p_n(t)\\
	q_n''(t) & p_n''(t)\end{bmatrix}=0;
	\end{equation*}
	\begin{equation*}
	\det \begin{bmatrix}
	q_n(t) & p_n(t)\\
	q_n'(t) & p_n'(t)
	\end{bmatrix}=0.
	\end{equation*}
\end{theorem}
\begin{proof}
The cases $n=1,2,3,4$ were proved in Propositions \ref{proposition5.3}, \ref{proposition5.6}, \ref{proposition5.7} and \ref{proposition5.8}.

Observe that for the general case $n\geq 4$, we have:

(1) The first identity has $n$ summands.

(2) The second identity has $2n-1$ summnads.

(3) The third identity has $3n-3$ summands.

(4) The fourth identity has $4n-6$ summands.

(5) The fifth identity has $5n-10$ summands.

(k) In general, for $1\leq k \leq n$, the $k$-th identity has $kn-\frac{k(k-1)}{2}$ summands.

For $k=n+1,n+2,n+3\dots, 2n-1,2n$, the $k$-th identity has $\frac{(n+1)n}{2},\frac{n(n-1)}{2},\frac{(n-1)(n-2)}{2}$,$\dots$,

$\frac{(n-(n-3))(n-(n-2))}{2},\frac{(n-(n-2))(n-(n-1))}{2}$ summands, respectively.

(2n-1) Thus, the $(2n-1)$-th identity has $3$ summands.

(2n) The last identity has only $1$ summand.

By induction, we assume the theorem for $n-1$, $n \geq 2$. Let
\begin{center}
	$P:=p_0(t)+p_1(t)x+p_2(t)x^2+\cdots+p_{n-1}(t)x^{n-1}, Q=q_0(t)+q_1(t)x+q_2(t)x^2+\cdots+q_{n-1}(t)x^{n-1}$.
\end{center}
Then, $p=P+p_n(t)x^n$, $q=Q+q_n(t)x^n$ and $qp=[Q+q_n(t)x^n][P+p_n(t)x^n]$.

Assume that $p$ and $q$ are Dixmier polynomials, so $qp=pq+1$, whence $[Q+q_n(t)x^n][P+p_n(t)x^n]=[P+p_n(t)x^n][Q+q_n(t)x^n]+1$. From this we get that
\begin{center}
$QP=PQ+[Pq_n(t)-Qp_n(t)]x^n+[p_n(t)x^nQ-q_n(t)x^nP]+[p_n(t)x^nq_n(t)x^n-q_n(t)x^np_n(t)x^n]+1$
\end{center}
Let $S:=[Pq_n(t)-Qp_n(t)]x^n+[p_n(t)x^nQ-q_n(t)x^nP]+[p_n(t)x^nq_n(t)x^n-q_n(t)x^np_n(t)x^n]$, then $QP=PQ+S+1$. By induction we have that
\tiny
\begin{center}
$QP-PQ=\Biggl (\det \begin{bmatrix}q_1(t) & p_1(t)\\
q_0'(t) & p_0'(t)\end{bmatrix}+\det \begin{bmatrix}q_2(t) & p_2(t)\\
q_0''(t) & p_0''(t)\end{bmatrix}+\det \begin{bmatrix}q_3(t) & p_3(t)\\
q_0'''(t) & p_0'''(t)\end{bmatrix}+\cdots +\det\begin{bmatrix}q_{n-1}(t) & p_{n-1}(t)\\
q_0^{(n-1)}(t) & p_0^{(n-1)}(t)\end{bmatrix}\Biggr )+$
\end{center}
\begin{center}
$\Biggl (\det \begin{bmatrix}q_1(t) & p_1(t)\\
	q_1'(t) & p_1'(t)\end{bmatrix}+2\det \begin{bmatrix}q_2(t) & p_2(t)\\
	q_0'(t) & p_0'(t)
\end{bmatrix}+
\det \begin{bmatrix}q_2(t) & p_2(t)\\
	q_1''(t) & p_1''(t)\end{bmatrix}+\cdots+(n-1)\det \begin{bmatrix}q_{n-1}(t) & p_{n-1}(t)\\
	q_0^{(n-2)}(t) & p_0^{(n-2)}(t)+
\end{bmatrix}+\det \begin{bmatrix}q_{n-1}(t) & p_{n-1}(t)\\
	q_1^{(n-1)}(t) & p_1^{(n-1)}(t)
\end{bmatrix}\Biggr )x+$
\end{center}
\begin{center}
$+\Biggl (\det \begin{bmatrix}q_1(t) & p_1(t)\\
q_2'(t) & p_2'(t)\end{bmatrix}+2\det \begin{bmatrix}q_2(t) & p_2(t)\\
q_1'(t) & p_1'(t)\end{bmatrix}+\det \begin{bmatrix}q_2(t) & p_2(t)\\
q_2''(t) & p_2''(t)\end{bmatrix}+
3\det\begin{bmatrix}q_3(t) & p_3(t)\\
q_0'(t) & p_0'(t)\end{bmatrix}+3\det\begin{bmatrix}q_3(t) & p_3(t)\\
q_1''(t) & p_1''(t)\end{bmatrix}+\det\begin{bmatrix}q_3(t) & p_3(t)\\
q_2'''(t) & p_2'''(t)\end{bmatrix}+\cdots+
(n-1)\det\begin{bmatrix}q_{n-1}(t) & p_{n-1}(t)\\
q_0^{(n-3)}(t) & p_0^{(n-3)}(t)\end{bmatrix}+(n-1)\det\begin{bmatrix}q_{n-1}(t) & p_{n-1}(t)\\
q_1^{(n-2)}(t) & p_1^{(n-2)}(t)\end{bmatrix}+\det\begin{bmatrix}q_{n-1}(t) & p_{n-1}(t)\\
q_2^{(n-1)}(t) & p_2^{(n-1)}(t)\end{bmatrix}\Biggr )x^2+
$
\end{center}
\begin{center}
$\vdots$
\end{center}
\begin{center}
$+\Biggl ((n-2)\det \begin{bmatrix}q_{n-2}(t) & p_{n-2}(t)\\
q_{n-1}'(t) & p_{n-1}'(t)\end{bmatrix}+(n-1)\det \begin{bmatrix}q_{n-1}(t) & p_{n-1}(t)\\
q_{n-2}'(t) & p_{n-2}'(t)\end{bmatrix}+\frac{(n-1)(n-2)}{2}\det \begin{bmatrix}q_{n-1}(t) & p_{n-1}(t)\\
q_{n-1}''(t) & p_{n-1}''(t)\end{bmatrix}\Biggr )x^{2n-4}+$
\end{center}
\begin{center}
$\det \begin{bmatrix}
q_{n-1}(t) & p_{n-1}(t)\\
q_{n-1}'(t) & p_{n-1}'(t)
\end{bmatrix}x^{2n-3}.$
\end{center}
\normalsize
By direct computation we get that
\tiny
\begin{center}
$S=\det\begin{bmatrix}p_n(t) & q_n(t)\\
p_0^{(n)}(t) & q_0^{(n)}(t)\end{bmatrix}+\Biggl (n\det \begin{bmatrix}p_n(t) & q_n(t)\\
p_0^{(n-1)}(t) & q_0^{(n-1)}(t)
\end{bmatrix}+\det \begin{bmatrix}p_n(t) & q_n(t)\\
p_1^{(n)}(t) & q_1^{(n)}(t)\end{bmatrix}\Biggr )x$+
\end{center}
\begin{center}
$\Biggl (\frac{n(n-1)}{2}\det \begin{bmatrix}p_n(t) & q_n(t)\\
p_0^{(n-2)}(t) & q_0^{(n-2)}(t)\end{bmatrix}+n\det \begin{bmatrix}p_n(t) & q_n(t)\\
p_1^{(n-1)}(t) & q_1^{(n-1)}(t)\end{bmatrix}+\det \begin{bmatrix}p_n(t) & q_n(t)\\
p_2^{(n)}(t) & q_2^{(n)}(t)\end{bmatrix}\Biggr )x^2+$
\end{center}
\begin{center}
$\vdots$
\end{center}
\begin{center}
$+
\Biggl ((n-1)\det \begin{bmatrix}q_{n-1}(t) & p_{n-1}(t)\\
q_n'(t) & p_n'(t)\end{bmatrix}+n\det \begin{bmatrix}q_n(t) & p_n(t)\\
q_{n-1}'(t) & p_{n-1}'(t)\end{bmatrix}+\frac{n(n-1)}{2}\det \begin{bmatrix}q_n(t) & p_n(t)\\
q_n''(t) & p_n''(t)\end{bmatrix}\Biggr )x^{2n-2}+$	
\end{center}
\begin{center}
$+n\det \begin{bmatrix}
q_n(t) & p_n(t)\\
q_n'(t) & p_n'(t)
\end{bmatrix}x^{2n-1}.$
\end{center}
\normalsize
Thus, from $QP-PQ-S=1$ and comparing the similar terms in the variable $x$ we get the identities of the theorem. For the last identity, recall that $K$ is a field
of characteristic zero.

Conversely, the matrix relations of the theorem imply that $p$ and $q$ are Dixmier polynomials, and hence, from Corollary \ref{corollary5.3}, $p$ and $q$ define an automorphism $\alpha$ of $A_1(K)$, $\alpha(t):=p$ and $\alpha(x):=q$.
\end{proof}

\subsection{A second matrix characterization of Dixmier problem}

Theorem \ref{theorem5.9a} can be interpreted as a matrix characterization of Dixmier problem. Theorem \ref{theorem5.9a} was proved using Corollary \ref{corollary5.3}, and hence, using Theorem \ref{Theorem1.3}. Using Corollary \ref{corollary5.2} instead of Corollary \ref{corollary5.3}, i.e., without using Theorem \ref{Theorem1.3}, next we present another characterization. This characterization is partial but it is matrix-constructive. Moreover, this second characterization let us to express the variables $t$ and $x$ in terms in the Dixmier polynomials $p$ and $q$. Let $M_2(A_1(K))$ be the ring of square mtrices of size $2\times 2$ over $A_1(K)$ and let $GL_2(A_1(K))$ be the group of invertible matrices of $M_2(A_1(K))$, i.e., $GL_2(A_1(K))=M_2(A_1(K))^*$ ({\it $GL_2(A_1(K))$ is known as the \textbf{general linear group over} $A_1(K)$}).
\begin{theorem}\label{theorem5.9}
	Let $n\geq 1$ and
	\begin{center}
		$p=p_0(t)+p_1(t)x+p_2(t)x^2+\cdots+p_n(t)x^n, q=q_0(t)+q_1(t)x+q_2(t)x^2+\cdots+q_n(t)x^n \in A_1(K)$.
	\end{center}
Assume that $p$ and $q$ satisfy the following conditions:
\begin{enumerate}
\item[\rm (i)] $p$ and $q$ are Dixmier polynomials.
\item[\rm (ii)] $t$ divides $p_0(t)$ and $t$ divides $q_0(t)$, with $p_0(t)=b(t)t$, $q_0(t)=d(t)t$, for some $b(t),d(t)\in K[t]$.
\item[\rm (iii)]$\begin{bmatrix}b(t) & p_1(t)+p_2(t)x+\cdots +p_n(t)x^{n-1}\\
d(t) & q_1(t)+q_2(t)x+\cdots +q_n(t)x^{n-1}\end{bmatrix}\in GL_2(A_1(K))$.
\end{enumerate}
Then $p$ and $q$ define an automorphism of $A_1(K)$.	
\end{theorem}
\begin{proof}
According to (i), $p$ and $q$ defines an endomorphism of $A_1(K)$, with $\alpha(t):=p$ and $\alpha(x):=q$ (Corollary \ref{corollary5.2}). By (ii) and (iii),
\begin{equation}
\begin{bmatrix}
p\\
q
\end{bmatrix}=\begin{bmatrix}b(t) & p_1(t)+p_2(t)x+\cdots +p_n(t)x^{n-1}\\
d(t) & q_1(t)+q_2(t)x+\cdots +q_n(t)x^{n-1}\end{bmatrix}\begin{bmatrix}
t\\
x
\end{bmatrix},
\end{equation}
so
\begin{equation}\label{equation5.14}
\begin{bmatrix}
	t\\
	x
\end{bmatrix}=\begin{bmatrix}b(t) & p_1(t)+p_2(t)x+\cdots +p_n(t)x^{n-1}\\
	d(t) & q_1(t)+q_2(t)x+\cdots +q_n(t)x^{n-1}\end{bmatrix}^{-1}
\begin{bmatrix}
p\\
q
\end{bmatrix}.
\end{equation}
Thus, $\alpha$ defines an automorphism of $A_1(K)$.
\end{proof}

Observe that the equation (\ref{equation5.14}) let us to express the variables $t$ and $x$ in terms of $p$ and $q$.

\begin{example}
$p:=t+\lambda x^n$ and $q:=x$ satisfy (i)-(iii) of the previous theorem, and $\alpha$ is this case coincides with $\Phi_{n,\lambda}$, for $n\geq 1$. In fact, it is clear that $qp=pq+1$, $p_0(t)=t$, $q_0(t)=0$ and the matrix of (iii) is $\begin{bmatrix}1 & \lambda x^{n-1}\\
0 & 1\end{bmatrix}$ with
$\begin{bmatrix}1 & \lambda x^{n-1}\\
0 & 1\end{bmatrix}^{-1}=\begin{bmatrix}1 & -\lambda x^{n-1}\\
0 & 1\end{bmatrix}$.

In a similar way, $p:=t$ and $q:=\lambda t^n+x$ satisfy (i)-(iii) of the previous theorem, and $\alpha$ is this case coincides with $\Phi_{n,\lambda}'$, for $n\geq 1$ (see the Dixmier automorphisms in (i) of Subsection\ref{subsection5.1}).

These two examples can be generalized as follows:

(i) Let $n\geq 1$, $p=p_0(t)+p_1(t)x+p_2(t)x^2+\cdots+p_n(t)x^n$ and $q=q_0(t)+q_1(t)x+q_2(t)x^2+\cdots+q_n(t)x^n$, with $p_0(t)=t$, $p_1(t),\dots,p_n(t)\in K$, $q_0(t)=0$, $q_1(t)=1$ and $q_2(t)=\cdots q_n(t)=0$. Then $b(t)=1$, $d(t)=0$, $q_k^{(r)}(t)=0$ for $0\leq k\leq n$, $1\leq r\leq n$, and $p_k^{(r)}(t)=0$ for $1\leq k\leq n$, $1\leq r\leq n$. Thus, the $2n$ identities of Theorem \ref{theorem5.9a} hold, i.e., $p$ and $q$ are Dixmier polynomials. The condition (ii) of Theorem \ref{theorem5.9} holds and
\begin{center}
$\begin{bmatrix}b(t) & p_1(t)+p_2(t)x+\cdots +p_n(t)x^{n-1}\\
d(t) & q_1(t)+q_2(t)x+\cdots +q_n(t)x^{n-1}\end{bmatrix}=\begin{bmatrix}1 & p_1(t)+p_2(t)x+\cdots +p_n(t)x^{n-1}\\
0 & 1\end{bmatrix}\in GL_2(A_1(K))$
\end{center}
with
\begin{center}
$\begin{bmatrix}1 & p_1(t)+p_2(t)x+\cdots +p_n(t)x^{n-1}\\
0 & 1\end{bmatrix}^{-1}=\begin{bmatrix}1 & -(p_1(t)+p_2(t)x+\cdots +p_n(t)x^{n-1})\\
0 & 1\end{bmatrix}.$
\end{center}

(ii) Let $n\geq 1$, $p=p_0(t)+p_1(t)x+p_2(t)x^2+\cdots+p_n(t)x^n$ and $q=q_0(t)+q_1(t)x+q_2(t)x^2+\cdots+q_n(t)x^n$, with $p_0(t)=t$, $p_1(t)=\cdots =p_n(t)=0$, $q_0(t)=d(t)t$, for some $d(t)\in K[t]$, $q_1(t)=1$ and $q_2(t)=\cdots q_n(t)=0$. Then $b(t)=1$, $q_k^{(r)}(t)=0$ for $1\leq k\leq n$, $1\leq r\leq n$, and $p_k^{(r)}(t)=0$ for $1\leq k\leq n$, $1\leq r\leq n$. Whence, the $2n$ identities of Theorem \ref{theorem5.9a} hold, in particular,
\begin{center}
$\det \begin{bmatrix}q_1(t) & p_1(t)\\
q_0'(t) & p_0'(t)\end{bmatrix}=\det \begin{bmatrix}1 & 0\\
q_0'(t) & 1\end{bmatrix}=1$.
\end{center}
Thus, $p$ and $q$ are Dixmier polynomials. The condition (ii) of Theorem \ref{theorem5.9} holds and
\begin{center}
	$\begin{bmatrix}b(t) & p_1(t)+p_2(t)x+\cdots +p_n(t)x^{n-1}\\
	d(t) & q_1(t)+q_2(t)x+\cdots +q_n(t)x^{n-1}\end{bmatrix}=\begin{bmatrix}1 & 0\\
	d(t) & 1\end{bmatrix}\in GL_2(A_1(K))$
\end{center}
with
$\begin{bmatrix}1 & 0\\
d(t) & 1\end{bmatrix}^{-1}=\begin{bmatrix}1 & 0\\
-d(t) & 1\end{bmatrix}.$
\end{example}

\begin{remark}
Our Theorem \ref{theorem5.9} is a partial characterization of Dixmier problem: In fact,
$\Phi_{0,\lambda}$, with $\lambda\neq 0$, satisfies (i) but not (ii) and (iii). Similarly, $\Phi_{0,\lambda}'$, with $\lambda\neq 0$, satisfies (i) but not (ii) and (iii). Thus, the Dixmier problem implies (i) but not neccesarily (ii) and (iii) of Theorem \ref{theorem5.9}.
\end{remark}

\begin{remark}
According to Theorem \ref{theorem4.16},  $\mathcal{CSD}_n(K)$ is a noetherian ring, in particular, $A_1(K)=\mathcal{CSD}_2(K)$ is noetherian. Then, for every $k\geq 2$, the ring of square matrices $M_k(\mathcal{CSD}_n(K))$ is noetherian, whence, if $M\in M_k(\mathcal{CSD}_n(K))$ has a left inverse, then $M$ is invertible (Noetherian rings are \textit{Dedekind finite}, see \cite{Lezama-GADF}, page 134). Thus, in order to check the condition (iii) of Theorem \ref{theorem5.9} we only need to verify that the matrix of (iii) has a left inverse. In the computations of the next subsection we will take in account this fact.
\end{remark}

\subsection{Implementation of matrix characterization of Dixmier problem}

Using the computational \textbf{SPBWE} library developed by the first author in \cite{Fajardo2} and \cite{Fajardo3} (see also \cite{Lezama-sigmaPBW} and \cite{Lezama-algebraic}), along with the \textbf{OreTools} package, it becomes possible to compute numerous families of Dixmier polynomials, thereby enabling the construction of nontrivial examples of automorphisms of $A_1(K)$. The \textbf{OreTools} package, already included within Maple software, is specifically designed to perform the basic arithmetic in Ore algebra, with its routines categorized into distinct functions for efficient computation. According to Theorem \ref{theorem5.9a}, all that needs to be computed are the Dixmier polynomials themselves. By applying Theorem \ref{theorem5.9}, additional calculations are required, but which allow expressing the variables $t$ and $x$ in terms of $p$ and $q$.

\begin{center}
	\textbf{The procedure using Theorem \ref{theorem5.9a}}
\end{center}

To construct Dixmier polynomials the following parameters are involved:
\begin{itemize}
	\item $n:=\deg(p)=\deg(q)$. Without loss of generality (completing with zero summands), we can assume that $p$ and $q$ have the same degree $n\geq 1$ in the variable $x$.
	\item $m:=\deg(p_i(t))=\deg(q_i(t))$, $0\leq i\leq n$. We can assume that all coefficients $p_i(t)$ of $p$ and all coefficients $q_i(t)$ of $q$ have the same degree $m\geq 1$ in the variable $t$.
	\item For $0\leq i\leq n$, the coefficients of $p_i(t)$ and $q_i(t)$ in $K$:
	\begin{center}
		$p_i(t)=p_{i,0}+p_{i,1}t+\cdots+p_{i,m-1}t^{m-1}+p_{i,m}t^m$,
		
		$q_i(t)=q_{i,0}+q_{i,1}t+\cdots+q_{i,m-1}t^{m-1}+q_{i,m}t^m$.
	\end{center}
\end{itemize}
We compute $qp$ and $pq$ (using the rule $xa(t)=a(t)x+a'(t)$, with $a(t)\in K[t])$. In order to construct Dixmier polynomials we use the identity $[q,p]=qp-pq=1 $, and then we compare the similar terms in the variable $x$ (see also the proof of Theorem \ref{theorem5.9a}). From this arise a system of equations in the variables
\begin{center}
	$p_{i,0},p_{i,1},\dots,p_{i,m-1},p_{i,m}; q_{i,0},q_{i,1},\dots,q_{i,m-1},q_{i,m}$, $0\leq i\leq n$.
\end{center}
Every solution of the system corresponds to a couple $p,q$ of Dixmier polynomials, and hence, an automorphism $\alpha$ defined by $\alpha(t):=p$ and $\alpha(x):=q$.

\begin{center}
	\textbf{The procedure using Theorem \ref{theorem5.9}}
\end{center}

\textit{Step 1}. According to (i) of Theorem \ref{theorem5.9}, we have first to construct Dixmier polynomials. For this computation, the following parameters are involved:
\begin{itemize}
	\item $n:=\deg(p)=\deg(q)$. Without loss of generality (completing with zero summands), we can assume that $p$ and $q$ have the same degree $n\geq 1$ in the variable $x$.
	\item $m:=\deg(p_i(t))=\deg(q_i(t))$, $0\leq i\leq n$. We can assume that all coefficients $p_i(t)$ of $p$ and all coefficients $q_i(t)$ of $q$ have the same degree $m\geq 1$ in the variable $t$.
	\item Coefficients of $p_0(t)$ in $K$, $p_0(t)=p_{0,1}t+\cdots+p_{0,m-1}t^{m-1}+p_{0,m}t^m$ (recall that $t$ divides $p_0(t)$).
	\item Coefficients of $q_0(t)$ in $K$, $q_0(t)=q_{0,1}t+\cdots+q_{0,m-1}t^{m-1}+q_{0,m}t^m$ (recall that $t$ divides $q_0(t)$).
	\item For $1\leq i\leq n$, the coefficients of $p_i(t)$ and $q_i(t)$ in $K$:
	\begin{center}
		$p_i(t)=p_{i,0}+p_{i,1}t+\cdots+p_{i,m-1}t^{m-1}+p_{i,m}t^m$,
		
		$q_i(t)=q_{i,0}+q_{i,1}t+\cdots+q_{i,m-1}t^{m-1}+q_{i,m}t^m$.
	\end{center}
\end{itemize}
We compute $qp$ and $pq$ (using the rule $xa(t)=a(t)x+a'(t)$, with $a(t)\in K[t])$. In order to construct Dixmier polynomials we use the identity $[q,p]=qp-pq=1 $, and then we compare the similar terms in the variable $x$ (see also the proof of Theorem \ref{theorem5.9a}). From this arise a system of equations in the variables
\begin{center}
	$p_{i,0},p_{i,1},\dots,p_{i,m-1},p_{i,m}; q_{i,0},q_{i,1},\dots,q_{i,m-1},q_{i,m}$, $0\leq i\leq n$.
\end{center}
Every solution of the system corresponds to a couple $p,q$ of Dixmier polynomials.

\textit{Step 2}. For every couple of Dixmier polynomials we have to define the matrix of (iii) in Theorem \ref{theorem5.9} and check if it has inverse (=left inverse). If so, for every couple we have constructed an automorphism $\alpha$ defined by $\alpha(t):=p$ and $\alpha(x):=q$.

\begin{center}
\textbf{The implementation of procedure using Theorem \ref{theorem5.9}}
\end{center}

Using the computational package \textbf{SPBWE}, it is possible to compute many families of polynomials satisfying the conditions of Theorem \ref{theorem5.9}, and hence, to construct nontrivial automorphisms of $A_1(K).$

The \textbf{SPBWE} library, implemented for \texttt{MAPLE}, enables computational work with bijective skew $PBW$ extensions. In particular, it supports the development of Gröbner theory for left ideals and modules of skew $PBW$ extensions, along with some of its significant applications in homological algebra. The library includes utilities for computing Gröbner bases and provides functions to calculate syzygy modules, free resolutions, and left inverses of matrices, among other tools. \textbf{SPBWE} consists of the following packages:
\begin{center}
\texttt{RingTools, SPBWETools, SPBWEGrobner}, \texttt{SPBWERings}, and \texttt{DixmierProblem}.
\end{center}

\noindent
\textbf{RingTools}: Defines the structure of the coefficient ring $R$ of a skew $PBW$ extension $A$.
\\
\textbf{SPBWETools}: A collection of functions inherent to extensions, allowing users to define the structure of $A$. These functions are particularly useful for working with noncommutative rings of polynomial type.
\\
\textbf{SPBWEGrobner}: A package containing functions where the main routine is the Buchberger algorithm for computing Gröbner bases.
\\
\textbf{SPBWERings}: This package contains predefined concrete examples of skew $PBW$ extensions within the library. Notably, it includes the first Weyl algebra $A_1(K).$
\\
\textbf{DixmierProblem}: This is a comprehensive package that offers a suite of functions pertinent to the study of the Dixmier problem in noncommutative algebra. It includes the \texttt{DixmierPolynomials} function, which generates families of Dixmier polynomial pairs, and its skew analogue, \texttt{skewDixmierPolynomials}, that constructs polynomials $p_1, \ldots, p_n \in \mathcal{CSD}_n(K)$ satisfying $[p_j, p_i] = d_{ij}\in K^*$. In addition, the package provides the \texttt{DixmierAutomorphism} function to generate Dixmier automorphisms on \(A_1(K)\), along with utilities for obtaining the inverse of an automorphism (via \texttt{inverseAutomorphism} and its skew version, \texttt{skewDixmierAutomorphism}) and for performing operations such as composition and other transformations. Central to the package is the \texttt{DixmierAutomorphismFactor} function, which decomposes complex Dixmier endomorphisms (=automorphisms, according to Theorem \ref{Theorem1.3}) into elementary automorphisms, specifically, $\Phi_{n,\lambda}$ and $\Psi_{n,\lambda}$, thereby, providing a constructive framework that supports the validation of the Dixmier problem by reducing intricate algebraic objects into more tractable components. This integrated environment not only facilitates rigorous investigations into algebraic structures, differential equations, and computational models but also streamlines the process of testing and refining theoretical conjectures. Below we illustrate the use of the functions \texttt{DixmierPolynomials} and \texttt{DixmierAutomorphism}.

\subsubsection*{\texttt{DixmierPolynomials}}

The \texttt{DixmierPolynomials} function is a fundamental component of the \texttt{DixmierProblem} package in Maple. It exploits the combined functionalities of the \texttt{OreTools} and \texttt{SPBWE} libraries, which together implement the predefined structure of $A_1(K).$ This algebraic framework is first initialized and assigned to a variable, thereby, establishing the foundation for all subsequent computational procedures.

The function requires two essential inputs, positive integers $n$ and $m$, and supports several optional parameters. In particular, the Boolean parameter \texttt{view} determines whether intermediate computations are displayed. When \texttt{view=false}, these intermediate results are suppressed. Conversely, if \texttt{view=true} is set, an additional option, \texttt{outputMode}, becomes active. Here \texttt{outputMode} accepts the string values \texttt{"finalResults"} or \texttt{"detailedResults"}. With \texttt{outputMode="finalResults"}, only the final results are displayed, along with comprehensive information on the families of Dixmier polynomial pairs computed. In contrast, \texttt{"detailedResults"} provides a full account of all intermediate calculations alongside the final outcomes. By default, the function is executed with \texttt{view=false}; if \texttt{view=true} is enabled, the default for \texttt{outputMode} is \texttt{"finalResults"}.

Furthermore, in select cases, the function also outputs the matrix corresponding to condition (iii) from Theorem \ref{theorem5.9}, together with its inverse (=left inverse).

The function generates two polynomials, $p$ and $q$, defined as:
\[
p=\sum_{i=0}^{n}\left(\sum_{j=0}^{m}p_{i,j}t^j\right)x^i, \quad
q=\sum_{i=0}^{n}\left(\sum_{j=0}^{m}q_{i,j}t^j\right)x^i,
\]
where $p_{i,j}$ and $q_{i,j}$ are symbolic variables within Maple. The function computes the commutator of these polynomials, $[q,p]=qp-pq$, and solves the equation $[q,p]=1.$ The solution process employs the Maple command \texttt{solve}, which outputs all possible solutions, organized into distinct cases. Finally, each resulting polynomial is expressed in terms of the parameters $\lambda_i$ within the field $K.$

\subsubsection*{Tests with \texttt{DixmierPolynomials} function}

To illustrate the functionality and versatility of the \texttt{DixmierPolynomials} function within the
\linebreak \texttt{DixmierProblem} package, we present a series of practical examples. These examples are designed to showcase the computational capabilities of the function, including polynomial generation, symbolic manipulation, and the resolution of equations involving commutators in the Weyl algebra.

By exploring these cases, readers gain a deeper understanding of the potential applications of the \texttt{DixmierPolynomials} function in both algebraic and symbolic computations. Additionally, the examples highlight specific configurations such as toggling the optional parameter \texttt{view} between \texttt{true} and \texttt{false} to tailor the output to user preferences. Through this hands-on approach, we demonstrate not only the theoretical significance of the function but also its practical utility in mathematical modeling and problem-solving.

Below, we delve into some illustrative examples to offer a comprehensive perspective on how\linebreak \texttt{DixmierPolynomials} can be effectively employed in diverse computational scenarios.

\begin{example}
	Next we present an application of the \texttt{DixmierPolynomial} function. As an initial step, we load the \texttt{DixmierProblem} package in Maple by executing the command \texttt{with(DixmierProblem):}, which enables access to its functionalities. Subsequently, we execute the command
	\begin{center}
		\texttt{DixmierPolynomial(2,\,2,\,view=true,\,outputMode="detailedResults")}
	\end{center}
	to obtain the resulting output. We note that setting the options \texttt{view=true} and \texttt{outputMode=}\linebreak \texttt{"detailedResults"} instructs the function to display all intermediate computational steps. Moreover, within the output, the greek letters denote elements of the field $K.$
	
	\smallskip
	
	\noindent
	\verb"Generated polynomials:"
	\begin{align*}
	p &= p_{0,1}t+p_{0,2}t^2+\left(p_{1,0}+p_{1,1}t+p_{1,2}t^2\right)x+\left(p_{2,0}+p_{2,1}t+p_{2,2}t^2\right)
	x^2\\
	q &= q_{0,1}t+q_{0,2}t^2+\left(q_{1,0}+q_{1,1}t+q_{1,2}t^2\right)x+\left(q_{2,0}+q_{2,1}t+q_{2,2}t^2\right)
	x^2
	\end{align*}
	
	\noindent
	\verb"The commutator [q,p] is:"
	\[
	e_0+e_1x+e_2x^2+e_3x^3,
	\]
	\texttt{where:}
	
	$e_0=q_{1,0}p_{0,1}+2q_{2,0}p_{0,2}-q_{0,1}p_{1,0}-2q_{0,2}p_{2,0}+((2q_{2,1}+2q_{1,0})p_{0,2}+ (-2p_{1,0}-2p_{2,1})q_{0,2}+q_{1,1}p_{0,1}-q_{0,1}p_{1,1})t+((2q_{2,2}+2q_{1,1})p_{0,2}+(-2p_{1,1}-2p_{2, 2})q_{0,2}+q_{1,2}p_{0,1}-q_{0,1}p_{1,2})t^2+(2q_{1,2}p_{0,2}-2q_{0,2}p_{1,2})t^3,$
	
	$e_1=(-2q_{1,2}-2q_{0,1})p_{2,0}+(2p_{0,1}+2p_{1,2})q_{2,0}-q_{1,1}p_{1,0}+q_{1,0}p_{1,1}+((2q_{2, 1}+2q_{1,0})p_{1,2}+(-2q_{1,2}-2q_{0,1})p_{2,1}+2q_{2,1}p_{0,1}-2q_{1,2}p_{1,0}+4q_{2,0}p_{0,2}-4q_{0, 2}p_{2,0})t+((q_{1,1}+2q_{2,2})p_{1,2}+(-2q_{1,2}-2q_{0,1})p_{2,2}+2q_{2,2}p_{0,1}-q_{1,2}p_{1,1}+4q_{2, 1}p_{0,2}-4q_{0,2}p_{2,1})t^2+(4q_{2,2}p_{0,2}-4q_{0,2}p_{2,2})t^3,$
	
	$e_2=(-2q_{1,1}-2q_{2,2})p_{2,0}+(2p_{1,1}+2p_{2,2})q_{2,0}-q_{2,1}p_{1,0}+q_{1,0}p_{2,1}+ ((-2q_{2,2}-q_{1,1})p_{2,1}+(2q_{2,1}+2q_{1,0})p_{2,2}-2q_{2,2}p_{1,0}-4q_{1,2}p_{2,0}+q_{2,1}p_{1,1}+4q_{2, 0}p_{1,2})t+(3q_{2,1}p_{1,2}-3q_{1,2}p_{2,1})t^2+(2q_{2,2}p_{1,2}-2q_{1,2}p_{2,2})t^3,$
	
	$e_3=-2q_{2,1}p_{2,0}+2q_{2,0}p_{2,1}+(-4q_{2,2}p_{2,0}+4q_{2,0}p_{2,2})t+(-2q_{2,2}p_{2,1}+2q_{2, 1}p_{2,2})t^2.$
	
	
	\noindent
	\verb"The equation [q,p]=1 is solved using the following system of equations:"
	\begin{gather*}
	p_{0,1}q_{1,0}+2p_{0,2}q_{2,0}-p_{1,0}q_{0,1}-2p_{2,0}q_{0,2}=1,\\
	(-2p_{2,1}-2p_{1,0})q_{0,2}+(2q_{2,1}+2q_{1,0})p_{0,2}-p_{1,1}q_{0,1}+p_{0,1}q_{1,1}=0,\\
	(-2p_{2,2}-2p_{1,1})q_{0,2}+(2q_{2,2}+2q_{1,1})p_{0,2}-p_{1,2}q_{0,1}+p_{0,1}q_{1,2}=0,\\
	-2p_{1,2}q_{0,2}+2p_{0,2}q_{1,2}=0,\\
	(-2q_{1,2}-2q_{0,1})p_{2,0}+(2p_{1,2}+2p_{0,1})q_{2,0}-p_{1,0}q_{1,1}+p_{1,1}q_{1,0}=0,\\
	(-2q_{1,2}-2q_{0,1})p_{2,1}+(2q_{2,1}+2q_{1,0})p_{1,2}-4p_{2,0}q_{0,2}-2p_{1,0}q_{1,2}+
	4p_{0,2}q_{2,0}+2p_{0,1}q_{2,1}=0,\\
	(-2q_{1,2}-2q_{0,1})p_{2,2}+(2q_{2,2}+q_{1,1})p_{1,2}-4p_{2,1}q_{0,2}-p_{1,1}q_{1,2}+
	4p_{0,2}q_{2,1}+2p_{0,1}q_{2,2}=0,\\
	-4p_{2,2}q_{0,2}+4p_{0,2}q_{2,2}=0,\\
	(-2q_{2,2}-2q_{1,1})p_{2,0}+(2p_{2,2}+2p_{1,1})q_{2,0}-p_{1,0}q_{2,1}+p_{2,1}q_{1,0}=0,\\
	(2q_{2,1}+2q_{1,0})p_{2,2}+(-2q_{2,2}-q_{1,1})p_{2,1}-4p_{2,0}q_{1,2}+4p_{1,2}q_{2,0}+
	p_{1,1}q_{2,1}-2p_{1,0}q_{2,2}=0,\\
	-3p_{2,1}q_{1,2}+3p_{1,2}q_{2,1}=0,\\
	-2p_{2,2}q_{1,2}+2p_{1,2}q_{2,2}=0,\\
	-2p_{2,0}q_{2,1}+ 2p_{2,1}q_{2,0}=0,\\
	4p_{2,2}q_{2,0}-4p_{2,0}q_{2,2}=0,\\
	2p_{2,2}q_{2,1}-2p_{2,1}q_{2,2}=0.
	\end{gather*}
	
	\noindent
	\verb"The equation [q,p]=1 is solved in the following cases:"
	
	\bigskip
	
	\noindent
	\texttt{----- Case 1 -----}
	
	\smallskip
	
	\noindent
	\texttt{Solution case details:}
	
	\smallskip
	
	\noindent
	$p_{0,1}=p_{0,1},p_{0,2}=p_{0,2},p_{1,0}=p_{1,0},
	p_{1,1}=\frac{2p_{0,2}(\alpha+1)}{\alpha p_{1,0}},p_{1,2}=0,p_{2,0}=p_{2,0},p_{2,1}=0,p_{2,2}=0,
	q_{0,1}=\alpha,q_{0,2}=0,q_{1,0}=\frac{p_{1,0}\alpha+1}{p_{0,1}},q_{1,1}=0,q_{1,2}=0,q_{2,0}=0, q_{2,1}=0,q_{2,2}=0,$ where $\alpha$ is a root of the polynomial
	$\left(p_{0,1}^2p_{2,0}-p_{0,2}p_{1,0}^2\right)Z^2-2p_{1,0}p_{0,2}Z-p_{0,2}.$
	
	\smallskip
	
	\noindent
	\texttt{After introducing new parameters lambda, the solution is:}
	
	\smallskip
	
	\noindent
	$p_{0,1}=\lambda_{0},p_{0,2}=\lambda_{1},p_{1,0}=\lambda_{2},p_{1,1}=\frac{2\lambda_{1}(1+ \lambda_{2}\lambda_{4})}{\lambda_{4}\lambda_{0}},p_{1,2}=0,p_{2,0}=\lambda_{3},p_{2,1}=0,p_{2,2}=0,
	q_{0,1}=\lambda_{4},q_{0,2}=0,q_{1,0}=\frac{1+\lambda_{2}\lambda_{4}}{\lambda_{0}},q_{1,1}=0,q_{1,2}=0,
	q_{2,0}=0,q_{2,1}=0, q_{2,2}=0.$
	
	\smallskip
	
	\noindent
	\texttt{The Dixmier polynomials are:}
	\begin{align*}
	p &= \lambda_0t+\lambda_1t^2+\left(\lambda_2+\frac{2\lambda_5}{\lambda_0\lambda_4}t\right)x+
	\frac{(1+\lambda_2\lambda_4)\lambda_5}{\lambda_0^2\lambda_4^2}x^2\\
	q &= \lambda_4t+\frac{1+\lambda_2\lambda_4}{\lambda_0}x
	\end{align*}
	\texttt{where:}
	\[
	\lambda_5=\lambda_1(1+\lambda_2\lambda_4)
	\]
	\texttt{Under condition (iii) of Theorem} \ref{theorem5.9}, \texttt{the associated matrix is:}
	\[
	\begin{bmatrix}
	\lambda_0+\lambda_1t & \lambda_2+\frac{2}{\lambda_4}t+\frac{\lambda_2}{(\lambda_0-\lambda_1)\lambda_4}x\\
	\lambda_4 & \frac{1+\lambda_2\lambda_4}{\lambda_0}
	\end{bmatrix}
	\]
	\texttt{The existence of its inverse is not guaranteed, but the automorphism property holds as per Theorem \ref{theorem5.9a}.}
	
	\bigskip
	
	\bigskip
	
	\noindent
	\texttt{----- Case 2 -----}
	
	\smallskip
	
	\noindent
	\texttt{Solution case details:}
	
	\smallskip
	
	\noindent
	$p_{0,1}=p_{0,1},p_{0,2}=0,p_{1,0}=p_{1,0},p_{1,1}=0,p_{1,2}=0,p_{2,0}=p_{2,0},p_{2,1}=0,p_{2,2}= 0,q_{0,1}=q_{0,1},q_{0,2}=0,q_{1,0}=\frac{p_{1,0}q_{0,1}+1}{p_{0,1}},q_{1,1}=0,q_{1,2}=0,
	q_{2,0}=\frac{q_{0,1}p_{2,0}}{p_{0,1}},q_{2,1}=0,q_{2,2}=0.$
	
	\smallskip
	
	\noindent
	\texttt{After introducing new parameters lambda, the solution is:}
	
	\smallskip
	
	\noindent
	$p_{0,1}=\lambda_0,p_{0,2}=0,p_{1,0}=\lambda_1,p_{1,1}=0,p_{1,2}=0,p_{2,0}=\lambda_2,p_{2,1}=0,
	p_{2,2}=0,q_{0,1}=\lambda_3,q_{0,2}=0,q_{1,0}=\frac{1+\lambda_1\lambda_3}{\lambda_0},q_{1,1}=0,
	q_{1,2}=0,q_{2,0}=\frac{\lambda_3\lambda_2}{\lambda_0},q_{2,1}=0,q_{2,2}=0.$
	
	\smallskip
	
	\noindent
	\texttt{The Dixmier polynomials are:}
	\begin{align*}
	p &= \lambda_0t+\lambda_1x+\lambda_2x^2\\
	q &= \lambda_3t+\frac{1+\lambda_1\lambda_3}{\lambda_0}x+\frac{\lambda_2\lambda_3}{\lambda_0}x^2
	\end{align*}
	\texttt{Under condition (iii) of Theorem} \ref{theorem5.9}, \texttt{the associated matrix is:}
	\[
	\begin{bmatrix}
	\lambda_0 & \lambda_1+\lambda_2x\\
	\lambda_3 & \frac{1+\lambda_1\lambda_3}{\lambda_0}+\frac{\lambda_2\lambda_3}{\lambda_0}x
	\end{bmatrix}
	\]
	\texttt{Its inverse is:}
	\[
	\begin{bmatrix}
	\frac{1+\lambda_1\lambda_3}{\lambda_0}+\frac{\lambda_2\lambda_3}{\lambda_0}x & -\lambda_1-\lambda_2x\\
	-\lambda_3 & \lambda_0
	\end{bmatrix}
	\]
	\texttt{Thus, the polynomials satisfy Theorem \ref{theorem5.9} and define an automorphism of A\_1(K).}
	
	\bigskip
	
	\noindent
	\texttt{----- Case 3 -----}
	
	\smallskip
	
	\noindent
	\texttt{Solution case details:}
	
	\smallskip
	
	\noindent
	$p_{0,1}=p_{0,1},p_{0,2}=p_{0,2},p_{1,0}=p_{1,0},p_{1,1}=0,p_{1,2}=0,p_{2,0}=0,p_{2,1}=0,p_{2,2}= 0,q_{0,1}=\linebreak\frac{p_{0,1}p_{1,0}q_{0,2}-p_{0,2}}{p_{0,2}p_{1,0}},q_{0,2}=q_{0,2},
	q_{1,0}=\frac{p_{1,0}q_{0,2}}{p_{0, 2}},q_{1,1}=0,q_{1,2}=0,q_{2,0}=0,q_{2,1}=0,q_{2,2}=0.$
	
	\smallskip
	
	\noindent
	\texttt{After introducing new parameters lambda, the solution is:}
	
	\smallskip
	
	\noindent
	$p_{0,1}=\lambda_0,p_{0,2}=\lambda_1,p_{1,0}=\lambda_2,p_{1,1}=0,p_{1,2}=0,p_{2,0}=0,p_{2,1}=0,p_{2,2}=0, q_{0,1}=\frac{-\lambda_1+\lambda_0\lambda_2\lambda_3}{\lambda_1\lambda_2},q_{0,2}=\lambda_3,q_{1,0}= \frac{\lambda_2\lambda_3}{\lambda_1},q_{1,1}=0,q_{1,2}=0,q_{2,0}=0,q_{2,1}=0,q_{2,2}=0.$
	
	\smallskip
	
	\noindent
	\texttt{The Dixmier polynomials are:}
	\begin{align*}
	p &= \lambda_0t+\lambda_1t^2+\lambda_2x\\
	q &= \frac{-\lambda_1+\lambda_0\lambda_2\lambda_3}{\lambda_1\lambda_2}t+\lambda_3t^2+
	\frac{\lambda_2\lambda_3}{\lambda_1}x
	\end{align*}
	\texttt{Under condition (iii) of Theorem} \ref{theorem5.9}, \texttt{the associated matrix is:}
	\[
	\begin{bmatrix}
	\lambda_0+\lambda_1t & \lambda_2\\
	\frac{-\lambda_1+\lambda_0\lambda_2\lambda_3}{\lambda_1\lambda_2}+\lambda_3t &
	\frac{\lambda_2\lambda_3}{\lambda_1}
	\end{bmatrix}
	\]
	\texttt{Its inverse is:}
	\[
	\begin{bmatrix}
	\frac{\lambda_2\lambda_3}{\lambda_1} & -\lambda_2\\
	\frac{\lambda_1-\lambda_0\lambda_2\lambda_3}{\lambda_1\lambda_2}-\lambda_3t &
	\lambda_0+\lambda_1t
	\end{bmatrix}
	\]
	\texttt{Thus, the polynomials satisfy Theorem \ref{theorem5.9} and define an automorphism of A\_1(K).}

	\noindent
	\texttt{----- Case 4 -----}
	
	\smallskip
	
	\noindent
	\texttt{Solution case details:}
	
	\smallskip
	
	\noindent
	$p_{0,1}=p_{0,1},p_{0,2}=0,p_{1,0}=p_{1,0},p_{1,1}=0,p_{1,2}=0,p_{2,0}=0,p_{2,1}=0,p_{2,2}=0, q_{0,1}=q_{0,1},q_{0,2}=q_{0,2},q_{1,0}=\frac{p_{1,0}q_{0,1}+1}{p_{0,1}},
	q_{1,1}=\frac{2q_{0,2}p_{1,0}}{p_{0,1}},q_{1,2}=0,q_{2,0}=\frac{p_{1,0}^2q_{0,2}}{p_{0,1}^2},q_{2,1}=0,
	q_{2,2}=0.$
	
	\smallskip
	
	\noindent
	\texttt{After introducing new parameters lambda, the solution is:}
	
	\smallskip
	
	\noindent
	$p_{0,1}=\lambda_{0},p_{0,2}=0,p_{1,0}=\lambda_{1},p_{1,1}=0,p_{1,2}=0,p_{2,0}=0,p_{2,1}=0,p_{2,2}=0,
	q_{0,1}=\lambda_{2},q_{0,2}=\lambda_{3},q_{1,0}=\frac{1+\lambda_{1}\lambda_{2}}{\lambda_{0}},
	q_{1,1}=\frac{2\lambda_{1}\lambda_{3}}{\lambda_{0}},q_{1,2}=0,q_{2,0}=\frac{\lambda_{1}^2\lambda_{3}}
	{\lambda_{0}^2},q_{2,1}=0,q_{2,2}=0.$
	
	\smallskip
	
	\noindent
	\texttt{The Dixmier polynomials are:}
	\begin{align*}
	p &= \lambda_0t + \lambda_1x\\
	q &= \lambda_2t + \lambda_3t^2+\left(\frac{1+\lambda_1\lambda_2}{\lambda_0}+\frac{2\lambda_1\lambda_3}{\lambda_0}t\right)x+ \frac{\lambda_1^2\lambda_3}{\lambda_0^2}x^2
	\end{align*}
	\texttt{Under condition (iii) of Theorem} \ref{theorem5.9}, \texttt{the associated matrix is:}
	\[
	\begin{bmatrix}
	\lambda_0 & \lambda_1\\
	\lambda_2+\lambda_3t & \frac{1+\lambda_1\lambda_2}{\lambda_0}+\frac{2\lambda_1\lambda_3}{\lambda_0}t+
	\frac{\lambda_1^2\lambda_3}{\lambda_0^2}x
	\end{bmatrix}
	\]
	\texttt{The existence of its inverse is not guaranteed, but the automorphism property holds as per Theorem \ref{theorem5.9a}.}
\end{example}

In the following examples we present another application of the \texttt{DixmierPolynomial} function which is executed using the command \texttt{DixmierPolynomial(n, m, view=true)}, or equivalently,\linebreak \texttt{DixmierPolynomial(n, m, view=true, outputMode="finalResults")}. This command produces the results shown below. We note that when the \texttt{view=true} option is employed, only the final families of Dixmier polynomial pairs are displayed at the end of the computations. In the output, greek letters denote elements of $K$.

\begin{example}
	When applying the statement \texttt{DixmierPolynomial(1,\,9,\,view=true)}, the resulting output computed by the function is as follows:
	
	\medskip
	
	\noindent
	\texttt{----- Case 1 -----}
	
	\smallskip
	
	\noindent
	\texttt{The Dixmier polynomials are:}
	\begin{align*}
	p &=
	\frac{1+\lambda_0\lambda_1}{\lambda_{10}}t+\frac{\lambda_0\lambda_2}{\lambda_{10}}t^2+\frac{\lambda_0 \lambda_3}{\lambda_{10}}t^3+\frac{\lambda_0\lambda_4}{\lambda_{10}}t^4+\frac{\lambda_0 \lambda_5}{\lambda_{10}}t^5+\frac{\lambda_0\lambda_6}{\lambda_{10}}t^6+\frac{\lambda_0 \lambda_7}{\lambda_{10}}t^7+\frac{\lambda_0\lambda_8}{\lambda_{10}}t^8+\frac{\lambda_0 \lambda_9}{\lambda_{10}}t^9+\lambda_0x\\
	q &= \lambda_1t+\lambda_2t^2+\lambda_3t^3+\lambda_4t^4+\lambda_5t^5+\lambda_6t^6+\lambda_7t^7+\lambda_8t^8+ \lambda_9t^9+\lambda_{10}x
	\end{align*}
	\texttt{Under condition (iii) of Theorem} \ref{theorem5.9}, \texttt{the associated matrix is:}
	\[
	\begin{bmatrix}
	\frac{1+\lambda_0\lambda_1}{\lambda_{10}}+\frac{\lambda_0\lambda_2}{\lambda_{10}}t+\frac{\lambda_0 \lambda_3}{\lambda_{10}}t^2+\frac{\lambda_0\lambda_4}{\lambda_{10}}t^3+\frac{\lambda_0\lambda_5 }{\lambda_{10}}t^4+\frac{\lambda_0\lambda_6}{\lambda_{10}}t^5+\frac{\lambda_0\lambda_7}{\lambda_{10}}t^6+ \frac{\lambda_0\lambda_8}{\lambda_{10}}t^7+\frac{\lambda_0\lambda_9}{\lambda_{10}}t^8&\lambda_0\\
	\lambda_1+\lambda_2t+\lambda_3t^2+\lambda_4t^3+\lambda_5t^4+\lambda_6t^5+\lambda_7t^6+\lambda_8t^7+ \lambda_9t^8&\lambda_{10}
	\end{bmatrix}
	\]\texttt{Its inverse is:}
	\[
	\begin{bmatrix}
	\lambda_{10}&-\lambda_0\\
	h(t) &\frac{1+\lambda_0\lambda_1}{\lambda_{10}}+\frac{\lambda_0\lambda_2}{\lambda_{10}}t+ \frac{\lambda_0\lambda_3}{\lambda_{10}}t^2+\frac{\lambda_0\lambda_4}{\lambda_{10}}t^3+\frac{\lambda_0 \lambda_5}{\lambda_{10}}t^4+\frac{\lambda_0\lambda_6}{\lambda_{10}}t^5+\frac{\lambda_0 \lambda_7}{\lambda_{10}}t^6+\frac{\lambda_0\lambda_8}{\lambda_{10}}t^7+\frac{\lambda_0 \lambda_9}{\lambda_{10}}t^8
	\end{bmatrix}
	\]
	with $h(t)=-\lambda_9t^8-\lambda_8t^7-\lambda_7t^6-\lambda_6t^5-\lambda_5t^4-\lambda_4t^3-\lambda_3t^2-\lambda_2t- \lambda_1.$\newline
	\texttt{Thus, the polynomials satisfy Theorem \ref{theorem5.9} and define an automorphism of A\_1(K).}
	
	\medskip
	
	\noindent
	\texttt{----- Case 2 -----}
	
	\smallskip
	
	\noindent
	\texttt{The Dixmier polynomials are:}
	\begin{align*}
	p&=\lambda_0t+\lambda_1t^2+\lambda_2t^3+\lambda_3t^4+\lambda_4t^5+\lambda_5t^6+\lambda_6t^7+\lambda_7t^8+ \lambda_8t^9-\frac{1}{\lambda_9}x\\
	q&=\lambda_9t
	\end{align*}
	\texttt{Under condition (iii) of Theorem} \ref{theorem5.9}, \texttt{the associated matrix is:}
	\[
	\begin{bmatrix}
	\lambda_0+\lambda_1t+\lambda_2t^2+\lambda_3t^3+\lambda_4t^4+\lambda_5t^5+\lambda_6t^6+\lambda_7t^7+ \lambda_8t^8 & -\frac{1}{\lambda_9}\\
	\lambda_9 & 0
	\end{bmatrix}
	\]
	\texttt{Its inverse is:}
	\[
	\begin{bmatrix}
	0 & \frac{1}{\lambda_9}\\
	-\lambda_9 & \lambda_0+\lambda_1t+\lambda_2t^2+\lambda_3t^3+\lambda_4t^4+\lambda_5t^5+\lambda_6t^6+ \lambda_7t^7+\lambda_8t^8
	\end{bmatrix}
	\]
	\texttt{Thus, the polynomials satisfy Theorem \ref{theorem5.9} and define an automorphism of A\_1(K).}
\end{example}
\begin{example}
	When applying the statement \texttt{DixmierPolynomial(2,\,3,\,view=true)}, the resulting output computed by the function is as follows:
	
	\medskip
	
	\noindent
	\texttt{----- Case 1 -----}
	
	\smallskip
	
	\noindent
	\texttt{The Dixmier polynomials are:}
	\noindent
	\begin{align*}
	p&=\frac{\lambda_5+(\lambda_6-\lambda_4)\lambda_1}{\lambda_2\lambda_5}t+\frac{\lambda_3^2\lambda_6}{4 \lambda_2\lambda_4\lambda_5}t^2+\left(\lambda_0+\frac{\lambda_3\lambda_6}{\lambda_2\lambda_5}t\right)x+ \frac{\lambda_4\lambda_6}{\lambda_2\lambda_5}x^2\\
	q&=\lambda_1t+\frac{\lambda_3^2}{4\lambda_4}t^2+(\lambda_2+\lambda_3t)x+\lambda_4x^2
	\end{align*}
	\texttt{where:}
	\[
	\lambda_5=-\frac{\lambda_2\lambda_3}{2}+\lambda_1\lambda_4,\quad\lambda_6=\lambda_0\lambda_5+\lambda_4
	\]
	\texttt{Under condition (iii) of Theorem} \ref{theorem5.9}, \texttt{the associated matrix is:}
	\[
	\begin{bmatrix}
	\frac{1+\lambda_0\lambda_1}{\lambda_2}+\frac{(\lambda_4+\lambda_0^2)\lambda_3^2}{4\lambda_0 \lambda_2\lambda_4}t & \lambda_0+\frac{\lambda_3\lambda_7}{\lambda_0\lambda_2}t+\frac{\lambda_4 \lambda_7}{\lambda_0\lambda_2}x\\
	\lambda_1+\frac{\lambda_3^2}{4\lambda_4}t&\lambda_2+\lambda_3t+\lambda_4x
	\end{bmatrix}
	\]
	\texttt{where:}
	\[
	\lambda_7=\lambda_4+\lambda_0^2
	\]
	\texttt{The existence of its inverse is not guaranteed, but the automorphism property holds as per Theorem \ref{theorem5.9a}.}
	
	\medskip
	
	\noindent
	\texttt{----- Case 2 -----}
	
	\smallskip
	
	\noindent
	\texttt{The Dixmier polynomials are:}
	\begin{align*}
	p&=
	\frac{-\lambda_2 + 2 \lambda_0 \lambda_1^2}{2 \lambda_1 \lambda_3} t + \frac{\lambda_0 \lambda_2^2}{4 \lambda_3^2} t^2 + \left(-\frac{1}{\lambda_1} + \frac{\lambda_0 \lambda_2}{\lambda_3} t \right) x + \lambda_0 x^2\\
	q &= \lambda_1 t + \frac{\lambda_2^2}{4 \lambda_3} t^2 + \lambda_2 t x + \lambda_3 x^2
	\end{align*}
	\texttt{Under condition (iii) of Theorem} \ref{theorem5.9}, \texttt{the associated matrix is:}
	\[
	\begin{bmatrix}
	\frac{-2\lambda_2+4\lambda_0\lambda_1^2}{4\lambda_1\lambda_3}+\frac{\lambda_0\lambda_2^2}{4\lambda_3^2}t& -\frac{1}{\lambda_1}+\frac{\lambda_0\lambda_2}{\lambda_3}t+\lambda_0x\\
	\lambda_1+\frac{\lambda_2^2}{4\lambda_3}t&\lambda_2t+\lambda_3x
	\end{bmatrix}
	\]
	\texttt{The existence of its inverse is not guaranteed, but the automorphism property holds as per Theorem \ref{theorem5.9a}.}
	
	\medskip
	
	\noindent
	\texttt{----- Case 3 -----}
	
	\smallskip
	
	\noindent
	\texttt{The Dixmier polynomials are:}
	\begin{align*}
	p &= \frac{\lambda_0\lambda_1}{\lambda_3}t+\frac{-\lambda_3+\lambda_0\lambda_1\lambda_2}{\lambda_1 \lambda_3}x+\lambda_0x^2\\
	q &= \lambda_1t+\lambda_2x+\lambda_3x^2
	\end{align*}
	\texttt{Under condition (iii) of Theorem} \ref{theorem5.9}, \texttt{the associated matrix is:}
	\[
	\begin{bmatrix}
	\frac{\lambda_0\lambda_1}{\lambda_3}&\frac{-\lambda_3+\lambda_0\lambda_1\lambda_2}{\lambda_1\lambda_3}+ \lambda_0x\\
	\lambda_1&\lambda_2+\lambda_3x
	\end{bmatrix}
	\]
	\texttt{Its inverse is:}
	\[
	\begin{bmatrix}
	\lambda_2+\lambda_3x&-\frac{-\lambda_3+\lambda_0\lambda_1\lambda_2}{\lambda_1\lambda_3}-\lambda_0x\\
	-\lambda_1&\frac{\lambda_0\lambda_1}{\lambda_3}
	\end{bmatrix}
	\]
	\texttt{Thus, the polynomials satisfy Theorem \ref{theorem5.9} and define an automorphism of A\_1(K).}
	
	\medskip
	
	\noindent
	\texttt{----- Case 4 -----}
	
	\smallskip
	
	\noindent
	\texttt{The Dixmier polynomials are:}
	\begin{align*}
	p &= \frac{\lambda_3+\lambda_0\lambda_1\lambda_4}{\lambda_3\lambda_4}t+\frac{\lambda_0\lambda_2}{\lambda_3} t^2+\lambda_0t^3+\frac{\lambda_0\lambda_4}{\lambda_3}x\\
	q &= \lambda_1t+\lambda_2t^2+\lambda_3t^3+\lambda_4x
	\end{align*}
	\texttt{Under condition (iii) of Theorem} \ref{theorem5.9}, \texttt{the associated matrix is:}
	\[
	\begin{bmatrix}
	\frac{\lambda_3+\lambda_0\lambda_1\lambda_4}{\lambda_3\lambda_4}+\frac{\lambda_0\lambda_2}{\lambda_3}t+ \lambda_0t^2 & \frac{\lambda_0\lambda_4}{\lambda_3}\\
	\lambda_1+\lambda_2t+\lambda_3t^2 & \lambda_4
	\end{bmatrix}
	\]
	\texttt{Its inverse is:}
	\[
	\begin{bmatrix}
	\lambda_4 & -\frac{\lambda_0\lambda_4}{\lambda_3}\\
	-\lambda_1-\lambda_2t-\lambda_3t^2 & \frac{\lambda_3+\lambda_0\lambda_1\lambda_4}{\lambda_3\lambda_4}+ \frac{\lambda_0\lambda_2}{\lambda_3}t+\lambda_0t^2
	\end{bmatrix}
	\]
	\texttt{Thus, the polynomials satisfy Theorem \ref{theorem5.9} and define an automorphism of A\_1(K).}
	
	\medskip
	
	\noindent
	\texttt{----- Case 5 -----}
	
	\smallskip
	
	\noindent
	\texttt{The Dixmier polynomials are:}
	\begin{align*}
	p &= \frac{\lambda_2+\lambda_0\lambda_1\lambda_3}{\lambda_2\lambda_3}t+\lambda_0t^2+\frac{\lambda_0 \lambda_3}{\lambda_2}x\\
	q &= \lambda_1t+\lambda_2t^2+\lambda_3x
	\end{align*}
	\texttt{Under condition (iii) of Theorem} \ref{theorem5.9}, \texttt{the associated matrix is:}
	\[
	\begin{bmatrix}
	\frac{\lambda_2+\lambda_0\lambda_1\lambda_3}{\lambda_2\lambda_3}+\lambda_0t&\frac{\lambda_0 \lambda_3}{\lambda_2}\\
	\lambda_1+\lambda_2t&\lambda_3
	\end{bmatrix}
	\]
	\texttt{Its inverse is:}
	\[
	\begin{bmatrix}
	\lambda_3&-\frac{\lambda_0\lambda_3}{\lambda_2}\\
	-\lambda_1-\lambda_2t&\frac{\lambda_2+\lambda_0\lambda_1\lambda_3}{\lambda_2\lambda_3}+\lambda_0t
	\end{bmatrix}
	\]
	\texttt{Thus, the polynomials satisfy Theorem \ref{theorem5.9} and define an automorphism of A\_1(K).}
	
	\medskip
	
	\noindent
	\texttt{----- Case 6 -----}
	
	\smallskip
	
	\noindent
	\texttt{The Dixmier polynomials are:}
	\begin{align*}
	p &= \lambda_0t+\lambda_1t^2+\lambda_2t^3-\frac{1}{\lambda_3}x\\
	q &= \lambda_3t
	\end{align*}
	\texttt{Under condition (iii) of Theorem} \ref{theorem5.9}, \texttt{the associated matrix is:}
	\[
	\begin{bmatrix}
	\lambda_0 + \lambda_1t+\lambda_2t^2 & -\frac{1}{\lambda_3}\\
	\lambda_3 & 0
	\end{bmatrix}
	\]
	\texttt{Its inverse is:}
	\[
	\begin{bmatrix}
	0 & \frac{1}{\lambda_3} \\
	-\lambda_3 & \lambda_0+\lambda_1t+\lambda_2 t^2
	\end{bmatrix}
	\]
	\texttt{Thus, the polynomials satisfy Theorem \ref{theorem5.9} and define an automorphism of A\_1(K).}
	
	\medskip
	
	\noindent
	\texttt{----- Case 7 -----}
	
	\smallskip
	
	\noindent
	\texttt{The Dixmier polynomials are:}
	\begin{align*}
	p &= \frac{1 + \lambda_0 \lambda_2}{\lambda_3} t + \frac{\lambda_1 \lambda_2^2}{\lambda_3^2} t^2 + \left(\lambda_0 + \frac{2 \lambda_1 \lambda_2}{\lambda_3} t \right) x + \lambda_1 x^2\\
	q &= \lambda_2t+\lambda_3x
	\end{align*}
	\texttt{Under condition (iii) of Theorem} \ref{theorem5.9}, \texttt{the associated matrix is:}
	\[
	\begin{bmatrix}
	\frac{\lambda_3+\lambda_0\lambda_2\lambda_3}{\lambda_3^2}+\frac{\lambda_1\lambda_2^2}{\lambda_3^2}t& \lambda_0+\frac{2\lambda_1\lambda_2}{\lambda_3}t+\lambda_1x\\
	\lambda_2&\lambda_3
	\end{bmatrix}
	\]
	\texttt{The existence of its inverse is not guaranteed, but the automorphism property holds as per Theorem \ref{theorem5.9a}.}
	
	\medskip
	
	\noindent
	\texttt{----- Case 8 -----}
	
	\smallskip
	
	\noindent
	\texttt{The Dixmier polynomials are:}
	\begin{align*}
	p &= \lambda_0t+\lambda_1t^2+\lambda_2x\\
	q &= -\frac{1}{\lambda_2}t
	\end{align*}
	\texttt{Under condition (iii) of Theorem} \ref{theorem5.9}, \texttt{the associated matrix is:}
	\[
	\begin{bmatrix}
	\lambda_0 + \lambda_1 t & \lambda_2 \\
	-\frac{1}{\lambda_2} & 0
	\end{bmatrix}
	\]
	\texttt{Its inverse is:}
	\[
	\begin{bmatrix}
	0 &-\lambda_2\\
	\frac{1}{\lambda_2} & \lambda_0+\lambda_1t
	\end{bmatrix}
	\]
	\texttt{Thus, the polynomials satisfy Theorem \ref{theorem5.9} and define an automorphism of A\_1(K).}
\end{example}
\begin{example}
	When applying the statement \texttt{DixmierPolynomial(3,\,3,\,view=true)}, the resulting output computed by the function is as follows:
	
	\medskip
	
	\noindent
	\texttt{----- Case 1 -----}
	
	\smallskip
	
	\noindent
	\texttt{The Dixmier polynomials are:}
	\begin{align*}
	p &= \frac{\lambda_{6}+(\lambda_{7}-\lambda_{4})\lambda_{1}}{\lambda_{3}\lambda_{6}}t+
	\frac{9\lambda_{2}^2\lambda_{5}\lambda_{7}}{\lambda_{3}\lambda_{4}^2\lambda_{6}}t^2+
	\frac{\lambda_{2}\lambda_{7}}{\lambda_{3}\lambda_{6}}t^3+
	\left(\lambda_{0}+\frac{6\lambda_{2}\lambda_{5}\lambda_{7}}{\lambda_{3}\lambda_{4}\lambda_{6}}t+
	\frac{\lambda_{4}\lambda_{7}}{\lambda_{3}\lambda_{6}}t^2\right)x+\\ &\quad
	\left(\frac{\lambda_{5}\lambda_{7}}{\lambda_{3}\lambda_{6}}+
	\frac{\lambda_{4}^2\lambda_{7}}{(\lambda_{1}\lambda_{4}-\lambda_{6})\lambda_{6}}t\right)x^2+
	\frac{\lambda_{4}^3\lambda_{7}}{9(\lambda_{1}\lambda_{4}-\lambda_{6})\lambda_{2}\lambda_{6}}x^3\\
	q &= \lambda_{1}t+\frac{9\lambda_{2}^2\lambda_{5}}{\lambda_{4}^2}t^2+\lambda_{2}t^3+
	\left(\lambda_{3}+\frac{6\lambda_{2}\lambda_{5}}{\lambda_{4}}t+\lambda_{4}t^2\right)x+
	\left(\lambda_{5}+\frac{\lambda_{4}^2}{3\lambda_{2}}t\right)x^2+
	\frac{\lambda_{4}^3}{27\lambda_{2}^2}x^3
	\end{align*}
	\texttt{where:}
	\[
	\lambda_6=\lambda_1\lambda_4-3\lambda_2\lambda_3,\quad\lambda_7=\lambda_0\lambda_6+\lambda_4
	\]
	\texttt{Under condition (iii) of Theorem} \ref{theorem5.9}, \texttt{the associated matrix is:}
	\[
	\begin{bmatrix}
	\frac{\lambda_{0}\lambda_{1}+1}{\lambda_{3}}+\frac{9\lambda_{2}\lambda_{5}\lambda_{8}}{\lambda_{0} \lambda_{3}\lambda_{4}^2}t+\frac{\lambda_{8}}{\lambda_{0}\lambda_{3}}t^2 & h(t)\\
	\lambda_{1}+\frac{9\lambda_{2}^2\lambda_{5}}{\lambda_{4}^2}t+\lambda_{2}t^2 &
	\lambda_{3}+\frac{6\lambda_{2}\lambda_{5}}{\lambda_{4}}t+\lambda_{4}t^2+
	\left(\lambda_{5}+\frac{\lambda_{4}^2}{3\lambda_{2}}t\right)x+\frac{\lambda_{4}^3}{27\lambda_{2}^2}x^2
	\end{bmatrix}
	\]
	with $h(t)=\lambda_{0}+\frac{6\lambda_{2}\lambda_{5}\lambda_{9}}{\lambda_{0}\lambda_{3}\lambda_{4}}t+
	\frac{\lambda_{4}\lambda_{9}}{\lambda_{0}\lambda_{3}}t^2+
	\left(\frac{\lambda_{5}\lambda_{9}}{\lambda_{0}\lambda_{3}}+\frac{\lambda_{4}^2\lambda_{9}}{\lambda_{0} \lambda_{1}\lambda_{4}+\lambda_{4}-\lambda_{9}}t\right)x+
	\frac{\lambda_{4}^3\lambda_{9}}{9(\lambda_{0}\lambda_{1}\lambda_{4}+\lambda_{4}-\lambda_{9})\lambda_{2}} x^2$
	\texttt{ and where:}
	\[
	\lambda_8=\lambda_2(\lambda_0^2+\lambda_4),\quad\lambda_9=\lambda_0^2+\lambda_4
	\]
	\noindent
	\texttt{The existence of its inverse is not guaranteed, but the automorphism property holds as per Theorem \ref{theorem5.9a}.}
	
	\medskip
	
	\noindent
	\texttt{----- Case 2 -----}
	
	\smallskip
	
	\noindent
	\texttt{The Dixmier polynomials are:}
	\begin{align*}
	p &= \frac{(3\lambda_{0}^2\lambda_{2}-\lambda_{1})}{\lambda_{0}\lambda_{3}}t+
	\frac{9\lambda_{1}\lambda_{2}^2\lambda_{4}}{\lambda_{3}^3}t^2+\frac{\lambda_{1}\lambda_{2}}{\lambda_{3}}
	t^3+\left(\lambda_{0}+\frac{6\lambda_{1}\lambda_{2}\lambda_{4}}{\lambda_{3}^2}t+\lambda_{1}t^2\right)x+\\
	&\quad \left(\frac{\lambda_{1}\lambda_{4}}{\lambda_{3}}+\frac{\lambda_{1}\lambda_{3}}
	{3\lambda_{2}}t\right)x^2+\frac{\lambda_{1}\lambda_{3}^2}{27\lambda_{2}^2}x^3\\
	q &= -\frac{1}{\lambda_{0}}t+\frac{9\lambda_{2}^2\lambda_{4}}{\lambda_{3}^2}t^2+\lambda_{2}t^3+
	\left(\frac{6\lambda_{2}\lambda_{4}}{\lambda_{3}}t+\lambda_{3}t^2\right)x+	\left(\lambda_{4}+\frac{\lambda_{3}^2}{3\lambda_{2}}t\right)x^2+\frac{\lambda_{3}^3}{27\lambda_{2}^2}x^3
	\end{align*}
	\texttt{Under condition (iii) of Theorem} \ref{theorem5.9}, \texttt{the associated matrix is:}
	\[
	\begin{bmatrix}
	\frac{3\lambda_{0}^2\lambda_{2}-\lambda_{1}}{\lambda_{0}\lambda_{3}}+\frac{9\lambda_{1}\lambda_{2}^2 \lambda_{4}}{\lambda_{3}^3}t+\frac{\lambda_{1}\lambda_{2}}{\lambda_{3}}t^2&
	\lambda_{0}+\frac{6\lambda_{1}\lambda_{2}\lambda_{4}}{\lambda_{3}^2}t+\lambda_{1}t^2+
	\left(\frac{\lambda_{1}\lambda_{4}}{\lambda_{3}}+\frac{\lambda_{1}\lambda_{3}}{3\lambda_{2}}t\right)x +\frac{\lambda_{1}\lambda_{3}^2}{27\lambda_{2}^2}x^2\\
	-\frac{1}{\lambda_{0}}+\frac{9\lambda_{2}^2\lambda_{4}}{\lambda_{3}^2}t+\lambda_{2}t^2&
	\frac{6\lambda_{2}\lambda_{4}}{\lambda_{3}}t+\lambda_{3}t^2+
	\left(\lambda_{4}+\frac{\lambda_{3}^2}{3\lambda_{2}}t\right)x+
	\frac{\lambda_{3}^3}{27\lambda_{2}^2}x^2
	\end{bmatrix}
	\]
	\texttt{The existence of its inverse is not guaranteed, but the automorphism property holds as per Theorem \ref{theorem5.9a}.}
	
	\medskip
	
	\noindent
	\texttt{----- Case 3 -----}
	
	\smallskip
	
	\noindent
	\texttt{The Dixmier polynomials are:}
	\begin{align*}
	p &= \frac{\lambda_{0}\lambda_{1}\lambda_{4}+\lambda_{3}}{\lambda_{3}\lambda_{4}}t+\frac{\lambda_{0} \lambda_{2}}{\lambda_{3}}t^2+\lambda_{0}t^3+\frac{\lambda_{0}\lambda_{4}}{\lambda_{3}}x\\
	q &= \lambda_{1}t+\lambda_{2}t^2+\lambda_{3}t^3+\lambda_{4}x
	\end{align*}
	\texttt{Under condition (iii) of Theorem} \ref{theorem5.9}, \texttt{the associated matrix is:}
	\[
	\begin{bmatrix}
	\frac{\lambda_{0}\lambda_{1}\lambda_{4}+\lambda_{3}}{\lambda_{3}\lambda_{4}}+\frac{\lambda_{0} \lambda_{2}}{\lambda_{3}}t+\lambda_{0}t^2 & \frac{\lambda_{0}\lambda_{4}}{\lambda_{3}}\\
	\lambda_{1}+\lambda_{2}t+\lambda_{3}t^2 & \lambda_{4}
	\end{bmatrix}
	\]
	\texttt{Its inverse is:}
	\[
	\begin{bmatrix}
	\lambda_{4} & -\frac{\lambda_{0}\lambda_{4}}{\lambda_{3}}\\
	-\lambda_{1}-\lambda_{2}t-\lambda_{3}t^2 & \frac{\lambda_{0}\lambda_{1}\lambda_{4}+\lambda_{3}}{\lambda_{3}\lambda_{4}}+\frac{\lambda_{0}\lambda_{2}}
	{\lambda_{3}}t+\lambda_{0}t^2
	\end{bmatrix}
	\]
	\texttt{Thus, the polynomials satisfy Theorem \ref{theorem5.9} and define an automorphism of A\_1(K).}
	
	\medskip
	
	\noindent
	\texttt{----- Case 4 -----}
	
	\smallskip
	
	\noindent
	\texttt{The Dixmier polynomials are:}
	\begin{align*}
	p &= \frac{\lambda_{0}\lambda_{1}}{\lambda_{4}}t+\frac{\lambda_{0}\lambda_{1}\lambda_{2}- \lambda_{4}}{\lambda_{1}\lambda_{4}}x+\frac{\lambda_{0}\lambda_{3}}{\lambda_{4}}x^2+\lambda_{0}x^3\\
	q &= \lambda_{1}t+\lambda_{2}x+\lambda_{3}x^2+\lambda_{4}x^3
	\end{align*}
	\texttt{Under condition (iii) of Theorem} \ref{theorem5.9}, \texttt{the associated matrix is:}
	\[
	\begin{bmatrix}
	\frac{\lambda_{0}\lambda_{1}}{\lambda_{4}} & \frac{\lambda_{0}\lambda_{1}\lambda_{2}- \lambda_{4}}{\lambda_{1}\lambda_{4}}+\frac{\lambda_{0}\lambda_{3}}{\lambda_{4}}x+\lambda_{0}x^2\\
	\lambda_{1} & \lambda_{2}+\lambda_{3}x+\lambda_{4}x^2
	\end{bmatrix}
	\]
	\texttt{Its inverse is:}
	\[
	\begin{bmatrix}
	\lambda_{2}+\lambda_{3}x+\lambda_{4}x^2 & -\frac{\lambda_{0}\lambda_{1}\lambda_{2}-\lambda_{4}}{\lambda_{1} \lambda_{4}}-\frac{\lambda_{0}\lambda_{3}}{\lambda_{4}}x-\lambda_{0}x^2\\
	-\lambda_{1} & \frac{\lambda_{0}\lambda_{1}}{\lambda_{4}}
	\end{bmatrix}
	\]
	\texttt{Thus, the polynomials satisfy Theorem \ref{theorem5.9} and define an automorphism of A\_1(K).}
	
	\medskip
	
	\noindent
	\texttt{----- Case 5 -----}
	
	\smallskip
	
	\noindent
	\texttt{The Dixmier polynomials are:}
	\begin{align*}
	p &= \frac{\lambda_{5}+(\lambda_{6}-\lambda_{4})\lambda_{1}}{\lambda_{3}\lambda_{5}}t+
	\frac{\lambda_{2}\lambda_{6}}{\lambda_{3}\lambda_{5}}t^2+
	\left(\lambda_{0}+\frac{\lambda_{4}\lambda_{6}}{\lambda_{3}\lambda_{5}}t\right)x+
	\frac{\lambda_{4}^2\lambda_{6}}{2\lambda_{1}\lambda_{4}\lambda_{5}-2\lambda_{5}^2}x^2\\
	q &= \lambda_{1}t+\lambda_{2}t^2+\left(\lambda_{3}+\lambda_{4}t\right)x+\frac{\lambda_{4}^2}{4\lambda_{2}} x^2
	\end{align*}
	\texttt{where:}
	\[
	\lambda_{5}=\lambda_{1}\lambda_{4}-2\lambda_{2}\lambda_{3},\quad
	\lambda_{6}=\lambda_{0}\lambda_{5}+\lambda_{4}
	\]
	\texttt{Under condition (iii) of Theorem} \ref{theorem5.9}, \texttt{the associated matrix is:}
	\[
	\begin{bmatrix}
	\frac{\lambda_{0}\lambda_{1}+1}{\lambda_{3}}+\frac{(\lambda_{0}^2+\lambda_{4})\lambda_{2}}{\lambda_{0} \lambda_{3}}t & \lambda_{0}+\frac{\lambda_{7}}{\lambda_{0}\lambda_{3}}t+
	\frac{\lambda_{4}\lambda_{7}}{2\lambda_{0} \lambda_{1}\lambda_{4}-2\lambda_{0}^2}x\\
	\lambda_{1}+\lambda_{2}t & \lambda_{3}+\lambda_{4}t+\frac{\lambda_{4}^2}{4\lambda_{2}}x
	\end{bmatrix}
	\]
	\texttt{where:}
	\[
	\lambda_{7}=\lambda_{4}\left(\lambda_{0}^2+\lambda_{4}\right)
	\]
	\texttt{The existence of its inverse is not guaranteed, but the automorphism property holds as per Theorem \ref{theorem5.9a}.}
	
	
	\noindent
	\texttt{----- Case 6 -----}
	
	\smallskip
	
	\noindent
	\texttt{The Dixmier polynomials are:}
	\begin{align*}
	p &= \frac{\lambda_{0}\lambda_{1}^2-2\lambda_{2}}{\lambda_{1}\lambda_{3}}t+\frac{\lambda_{0} \lambda_{2}}{\lambda_{3}}t^2+\left(-\frac{1}{\lambda_{1}}+\lambda_{0}t\right)x+\frac{\lambda_{0} \lambda_{3}}{4\lambda_{2}}x^2\\
	q &= \lambda_{1}t+\lambda_{2}t^2+\lambda_{3}tx+\frac{\lambda_{3}^2}{4\lambda_{2}}x^2
	\end{align*}
	\texttt{Under condition (iii) of Theorem} \ref{theorem5.9}, \texttt{the associated matrix is:}
	\[
	\begin{bmatrix}
	\frac{\lambda_{0}\lambda_{1}^2-2\lambda_{2}}{\lambda_{1}\lambda_{3}}+\frac{\lambda_{0} \lambda_{2}}{\lambda_{3}}t & -\frac{1}{\lambda_{1}}+\lambda_{0}t+
	\frac{\lambda_{0}\lambda_{3}}{4\lambda_{2}}x\\
	\lambda_{1}+\lambda_{2}t&\lambda_{3}t+\frac{\lambda_{3}^2}{4\lambda_{2}}x
	\end{bmatrix}
	\]
	\texttt{The existence of its inverse is not guaranteed, but the automorphism property holds as per Theorem \ref{theorem5.9a}.}
	
	\medskip
	
	\noindent
	\texttt{----- Case 7 -----}
	
	\smallskip
	
	\noindent
	\texttt{The Dixmier polynomials are:}
	\begin{align*}
	p &= \frac{\lambda_{0}\lambda_{1}\lambda_{3}+\lambda_{2}}{\lambda_{2}\lambda_{3}}t+\lambda_{0}t^2 +\frac{\lambda_{0}\lambda_{3}}{\lambda_{2}}x\\
	q &= \lambda_{2}t^2+\lambda_{1}t+\lambda_{3}x
	\end{align*}
	\texttt{Under condition (iii) of Theorem} \ref{theorem5.9}, \texttt{the associated matrix is:}
	\[
	\begin{bmatrix}
	\frac{\lambda_{0}\lambda_{1}\lambda_{3}+\lambda_{2}}{\lambda_{2}\lambda_{3}}+\lambda_{0}t & \frac{\lambda_{0}\lambda_{3}}{\lambda_{2}}\\
	\lambda_{1}+\lambda_{2}t & \lambda_{3}
	\end{bmatrix}
	\]
	\texttt{Its inverse is:}
	\[
	\begin{bmatrix}
	\lambda_{3} & -\frac{\lambda_{0}\lambda_{3}}{\lambda_{2}}\\
	-\lambda_{1}-\lambda_{2}t & \frac{\lambda_{0}\lambda_{1}\lambda_{3}+\lambda_{2}}{\lambda_{2}\lambda_{3}}+ \lambda_{0}t
	\end{bmatrix}
	\]
	\texttt{Thus, the polynomials satisfy Theorem \ref{theorem5.9} and define an automorphism of A\_1(K).}
	
	\medskip
	
	\noindent
	\texttt{----- Case 8 -----}
	
	\smallskip
	
	\noindent
	\texttt{The Dixmier polynomials are:}
	\begin{align*}
	p &= \frac{\lambda_{0}\lambda_{1}}{\lambda_{3}}t+\frac{\lambda_{0}\lambda_{1}\lambda_{2}- \lambda_{3}}{\lambda_{1}\lambda_{3}}x+\lambda_{0}x^2\\
	q &= \lambda_{1}t+\lambda_{2}x+\lambda_{3}x^2
	\end{align*}
	\texttt{Under condition (iii) of Theorem} \ref{theorem5.9}, \texttt{the associated matrix is:}
	\[
	\begin{bmatrix}
	\frac{\lambda_{0}\lambda_{1}}{\lambda_{3}} & \frac{\lambda_{0}\lambda_{1}\lambda_{2}- \lambda_{3}}{\lambda_{1}\lambda_{3}}+\lambda_{0}x\\
	\lambda_{1} & \lambda_{2}+\lambda_{3}x
	\end{bmatrix}
	\]
	\texttt{Its inverse is:}
	\[
	\begin{bmatrix}
	\lambda_{2}+\lambda_{3}x & -\frac{\lambda_{0}\lambda_{1}\lambda_{2}-\lambda_{3}}{\lambda_{1}\lambda_{3}}- \lambda_{0}x\\
	-\lambda_{1} & \frac{\lambda_{0}\lambda_{1}}{\lambda_{3}}
	\end{bmatrix}
	\]
	\texttt{Thus, the polynomials satisfy Theorem \ref{theorem5.9} and define an automorphism of A\_1(K).}

	\noindent
	\texttt{----- Case 9 -----}
	
	\smallskip
	
	\noindent
	\texttt{The Dixmier polynomials are:}
	\begin{align*}
	p &=\frac{\lambda_{0}\lambda_{3}+1}{\lambda_{4}}t+\frac{\lambda_{1}\lambda_{3}^2}{\lambda_{4}^2}t^2+
	\frac{\lambda_{2}\lambda_{3}^3}{\lambda_{4}^3}t^3+\left(\lambda_{0}+\frac{2\lambda_{1}\lambda_{3}}
	{\lambda_{4}}t+\frac{3\lambda_{2} \lambda_{3}^2}{\lambda_{4}^2}t^2\right)x+
	\left(\lambda_{1}+\frac{3\lambda_{2}\lambda_{3}}{\lambda_{4}}t\right)x^2+\lambda_{2}x^3\\
	q &= \lambda_3t+\lambda_4x
	\end{align*}
	\texttt{Under condition (iii) of Theorem} \ref{theorem5.9}, \texttt{the associated matrix is:}
	\[
	\begin{bmatrix}
	\frac{\lambda_{4}^2+\lambda_{0}\lambda_{3}\lambda_{4}^2}{\lambda_{4}^3}+\frac{\lambda_{1} \lambda_{3}^2}{\lambda_{4}^2}t+\frac{\lambda_{2}\lambda_{3}^3}{\lambda_{4}^3}t^2&
	\lambda_{0}+\frac{2\lambda_{1}\lambda_{3}}{\lambda_{4}}t+\frac{3\lambda_{2}\lambda_{3}^2}{\lambda_{4}^2} t^2+\left(\lambda_{1}+\frac{3\lambda_{2}\lambda_{3}}{\lambda_{4}}t\right)x+\lambda_{2}x^2\\
	\lambda_{3}&\lambda_{4}
	\end{bmatrix}
	\]
	\texttt{The existence of its inverse is not guaranteed, but the automorphism property holds as per Theorem \ref{theorem5.9a}.}
	
	\medskip
	
	\noindent
	\texttt{----- Case 10 -----}
	
	\smallskip
	
	\noindent
	\texttt{The Dixmier polynomials are:}
	\begin{align*}
	p &= \lambda_{0}t+\lambda_{1}t^2+\lambda_{2}t^3-\frac{1}{\lambda_{3}}x\\
	q &= \lambda_3t
	\end{align*}
	\texttt{Under condition (iii) of Theorem} \ref{theorem5.9}, \texttt{the associated matrix is:}
	\[
	\begin{bmatrix}
	\lambda_{0}+\lambda_{1}t+\lambda_{2}t^2 & -\frac{1}{\lambda_{3}}\\
	\lambda_{3} & 0
	\end{bmatrix}
	\]
	\texttt{Its inverse is:}
	\[
	\begin{bmatrix}
	0 & \frac{1}{\lambda_{3}}\\
	-\lambda_{3} & \lambda_{0}+\lambda_{1}t+\lambda_{2}t^2
	\end{bmatrix}
	\]
	\texttt{Thus, the polynomials satisfy Theorem \ref{theorem5.9} and define an automorphism of A\_1(K).}
\end{example}
\begin{example}
	When applying the statement \texttt{DixmierPolynomial(4,\,3,\,view=true)}, the resulting output computed by the function is as follows:
	
	\medskip
	
	\noindent
	\texttt{----- Case 1 -----}
	
	\smallskip
	
	\noindent
	\texttt{The Dixmier polynomials are:}
	\begin{align*}
	p &= \lambda_{0}t+\lambda_{1}t^2+\lambda_{2}t^3+	\left(\lambda_{3}-\frac{2\lambda_{1}\lambda_{7}}{\lambda_{0}\lambda_{6}}t-\frac{3\lambda_{2}\lambda_{7} }{\lambda_{0}\lambda_{6}}t^2\right)x+
	\left(\frac{\lambda_{1}\lambda_{7}^2}{\lambda_{0}^2\lambda_{6}^2}+\frac{3\lambda_{2}\lambda_{7}^2 }{\lambda_{0}^2\lambda_{6}^2}t\right)x^2-
	\frac{\lambda_{2}\lambda_{7}^3}{\lambda_{0}^3\lambda_{6}^3}x^3\\
	q &= \lambda_{4}t+\frac{\lambda_{1}\lambda_{5}}{\lambda_{2}}t^2+\lambda_{5}t^3+
	\left(\frac{\lambda_{5}^2+\lambda_{8}}{\lambda_{0}\lambda_{5}}+\frac{2\lambda_{1}\lambda_{8}}{(-1+ \lambda_{5})\lambda_{0}^2}t+\frac{3\lambda_{8}}{\lambda_{0}\lambda_{4}}t^2\right)x\\ &\quad+ \left(\frac{(\lambda_{3}\lambda_{4}-\lambda_{5}+1)\lambda_{1}\lambda_{8}}{(-1+\lambda_{5})\lambda_{0}^3 \lambda_{4}}+\frac{3(\lambda_{3}\lambda_{4}-\lambda_{5}+1)\lambda_{8}}{\lambda_{0}^2\lambda_{4}^2}t \right)x^2+\frac{(\lambda_{3}^2\lambda_{4}^2+2\lambda_{3}\lambda_{4}-\lambda_{5}^2-2\lambda_{8}+1) \lambda_{8}}{\lambda_{0}^3\lambda_{4}^3}x^3
	\end{align*}
	\texttt{where:}
	\[
	\lambda_6=\lambda_0\lambda_5-\lambda_2\lambda_4,\quad\lambda_7=-\lambda_3\lambda_6+\lambda_2,\quad \lambda_8=\lambda_5(\lambda_3\lambda_4-\lambda_5+1)
	\]
	\texttt{Under condition (iii) of Theorem} \ref{theorem5.9}, \texttt{the associated matrix is:}
	\[
	\begin{bmatrix}
	\lambda_{0}+\lambda_{1}t+\lambda_{2}t^2 & g(t)\\
	\lambda_{4}+\frac{\lambda_{1}\lambda_{5}}{\lambda_{2}}t+\lambda_{5}t^2 & h(t)
	\end{bmatrix}
	\]
	with
	\begin{align*}
	g(t) &= \lambda_{3}+\frac{2\lambda_{1} \lambda_{9}}{\lambda_{0}\lambda_{4}}t+
	\frac{3(-1+\lambda_{5})\lambda_{9}}{\lambda_{4}^2}t^2+\left(\frac{\lambda_{1}\lambda_{3}\lambda_{9}^2}
	{(-1+\lambda_{9}+\lambda_{5})\lambda_{0}^2\lambda_{4}}+\frac{3(-1+\lambda_{5})\lambda_{3}\lambda_{9}^2}
	{(-1+\lambda_{9}+\lambda_{5})\lambda_{0}\lambda_{4}^2}t\right)x\\
	&\quad +\frac{(-1+\lambda_{5})\lambda_{3}^2
		\lambda_{9}^3}{(-1+\lambda_{9}+\lambda_{5})^2\lambda_{0}^2 \lambda_{4}^2}x^2\\
	h(t) &= \frac{\lambda_{5}^2+\lambda_{2}}{\lambda_{0}\lambda_{5}}+\frac{2\lambda_{1}\lambda_{2}}{(-1+\lambda_{5}) \lambda_{0}^2}t+\frac{3\lambda_{2}}{\lambda_{0}\lambda_{4}}t^2+\left(\frac{(\lambda_{3}\lambda_{4}-
		\lambda_{5}+1)\lambda_{1}\lambda_{2}}{(-1+\lambda_{5})\lambda_{0}^3\lambda_{4}}+
	\frac{3(\lambda_{3}\lambda_{4}-\lambda_{5}+1)\lambda_{2}}{\lambda_{0}^2\lambda_{4}^2}t \right)x\\
	&\quad -\frac{(-\lambda_{3}^2\lambda_{4}^2-2\lambda_{3}\lambda_{4}+\lambda_{5}^2+2\lambda_{2}-1) \lambda_{2}}{\lambda_{0}^3\lambda_{4}^3}x^2
	\end{align*}
	\texttt{where:}
	\[
	\lambda_9=\lambda_3\lambda_4-\lambda_5+1
	\]
	\texttt{The existence of its inverse is not guaranteed, but the automorphism property holds as per Theorem \ref{theorem5.9a}.}
	
	\medskip
	
	\noindent
	\texttt{----- Case 2 -----}
	
	\smallskip
	
	\noindent
	\texttt{The Dixmier polynomials are:}
	\begin{align*}
	p &= \lambda_{0}t^2+\lambda_{1}t^3+\left(-\frac{1}{\lambda_{3}}+\frac{2\lambda_{0}\lambda_{2}}
	{3\lambda_{1}}t+\lambda_{2}t^2\right)x+\left(\frac{\lambda_{0}\lambda_{2}^2}{9\lambda_{1}^2}+
	\frac{\lambda_{2}^2}{3\lambda_{1}}t\right)x^2 +\frac{\lambda_{2}^3}{27\lambda_{1}^2}x^3\\
	q &= \lambda_{3}t+\frac{\lambda_{0}\lambda_{4}}{\lambda_{1}}t^2+\lambda_{4}t^3+\left(\frac{\lambda_{2}
		\lambda_{3}^2-3\lambda_{4}}{3\lambda_{1}\lambda_{3}}+\frac{2\lambda_{0}\lambda_{2}\lambda_{4}}
	{3\lambda_{1}^2}t+\frac{\lambda_{2}\lambda_{4}}{\lambda_{1}}t^2\right)x+\left(\frac{\lambda_{0}
		\lambda_{2}^2\lambda_{4}}{9\lambda_{1}^3}+\frac{\lambda_{2}^2\lambda_{4}}{3 \lambda_{1}^2}t\right)x^2+
	\frac{\lambda_{2}^3\lambda_{4}}{27\lambda_{1}^3}x^3
	\end{align*}
	\texttt{Under condition (iii) of Theorem} \ref{theorem5.9}, \texttt{the associated matrix is:}
	\[
	\begin{bmatrix}
	\lambda_{0}t+\lambda_{1}t^2 &
	-\frac{1}{\lambda_{3}}+\frac{2\lambda_{0}\lambda_{2}}{3\lambda_{1}}t+\lambda_{2}t^2+
	\left(\frac{\lambda_{0}\lambda_{2}^2}{9\lambda_{1}^2}+\frac{\lambda_{2}^2}{3\lambda_{1}}t\right)x+
	\frac{\lambda_{2}^3}{27\lambda_{1}^2}x^2\\
	\lambda_{3}+\frac{\lambda_{0}\lambda_{4}}{\lambda_{1}}t+\lambda_{4}t^2 & h(t)
	\end{bmatrix}
	\]
	where $h(t)=\frac{9\lambda_{1}^2\lambda_{2}\lambda_{3}^2-27\lambda_{1}^2\lambda_{4}}{27\lambda_{1}^3 \lambda_{3}}+\frac{2\lambda_{0}\lambda_{2}\lambda_{4}}{3\lambda_{1}^2}t+\frac{\lambda_{2} \lambda_{4}}{\lambda_{1}}t^2+\left(\frac{\lambda_{0}\lambda_{2}^2\lambda_{4}}{9\lambda_{1}^3}+
	\frac{\lambda_{2}^2\lambda_{4}}{3 \lambda_{1}^2}t\right)x+\frac{\lambda_{2}^3\lambda_{4}}
	{27\lambda_{1}^3}x^2$\newline
	\texttt{The existence of its inverse is not guaranteed, but the automorphism property holds as per Theorem \ref{theorem5.9a}.}
	
	\medskip
	
	\noindent
	\texttt{----- Case 3 -----}
	
	\smallskip
	
	\noindent
	\texttt{The Dixmier polynomials are:}
	\begin{align*}
	p &= \lambda_{0}t+\frac{\lambda_{0}\lambda_{4}}{3\lambda_{3}}x\\
	q &= \lambda_{1}t+\lambda_{2}t^2+\lambda_{3}t^3+\left(\frac{\lambda_{0}\lambda_{1}\lambda_{4}+
		3\lambda_{3}}{3\lambda_{0}\lambda_{3}}+\frac{2\lambda_{2}\lambda_{4}}{3\lambda_{3}}t+\lambda_{4}
	t^2\right)x+\left(\frac{\lambda_{2}\lambda_{4}^2}{9\lambda_{3}^2}+\frac{\lambda_{4}^2}{3\lambda_{3}}
	t\right)x^2 +\frac{\lambda_{4}^3}{27\lambda_{3}^2}x^3
	\end{align*}
	\texttt{Under condition (iii) of Theorem} \ref{theorem5.9}, \texttt{the associated matrix is:}
	\[
	\begin{bmatrix}
	\lambda_{0} & \frac{\lambda_{0}\lambda_{4}}{3\lambda_{3}}\\
	\lambda_{1}+\lambda_{2}t+\lambda_{3}t^2 &
	\frac{27\lambda_{3}^2+9\lambda_{0}\lambda_{1}\lambda_{3}\lambda_{4}}{27\lambda_{0}\lambda_{3}^2}+\frac{2 \lambda_{2}\lambda_{4}}{3\lambda_{3}}t+\lambda_{4}t^2+\left(\frac{\lambda_{2}\lambda_{4}^2}
	{9\lambda_{3}^2}+\frac{\lambda_{4}^2}{3\lambda_{3}}t\right)x +\frac{\lambda_{4}^3}{27\lambda_{3}^2}x^2
	\end{bmatrix}
	\]
	\texttt{The existence of its inverse is not guaranteed, but the automorphism property holds as per Theorem \ref{theorem5.9a}.}
	
	\newpage
	
	\noindent
	\texttt{----- Case 4 -----}
	
	\smallskip
	
	\noindent
	\texttt{The Dixmier polynomials are:}
	\begin{align*}
	p &= \lambda_{0}t+\frac{\lambda_{2}^2}{4\lambda_{3}}t^2+
	\left(\lambda_{1}+\frac{\lambda_{2}\lambda_{7}}{2\lambda_{0}\lambda_{3}\lambda_{6}}t\right)x+
	\left(\frac{8\lambda_{0}^3\lambda_{3}^2\lambda_{6}^2+\lambda_{2}\lambda_{7}^2}{4\lambda_{0}^2\lambda_{2} \lambda_{3}\lambda_{6}^2}+\lambda_{2}t\right)x^2+
	\frac{\lambda_{7}}{\lambda_{0}\lambda_{6}}x^3+\lambda_{3}x^4\\
	q&=\lambda_{4}t+\frac{\lambda_{2}\lambda_{5}}{4\lambda_{3}}t^2+\left(\frac{\lambda_{1}\lambda_{4}+
		\lambda_{0}+1}{\lambda_{0}}-\frac{(-1+\lambda_{5})\lambda_{8}}{2\lambda_{3}\lambda_{4}^2}t\right)x\\
	&\quad +\left(\frac{(-1+(1+\lambda_{0})\lambda_{5}-\lambda_{0}-\lambda_{1}\lambda_{4})\lambda_{8}+(-2+ \lambda_{5})\lambda_{8}^2+8\lambda_{3}^2\lambda_{4}^5}{4(-1+\lambda_{5})\lambda_{0}\lambda_{3} \lambda_{4}^3}+\lambda_{5}t\right)x^2-\frac{\lambda_{8}}{\lambda_{0}\lambda_{4}}x^3+
	\frac{\lambda_{3}\lambda_{5}}{\lambda_{2}}x^4
	\end{align*}
	\texttt{where:}
	\[
	\lambda_{6}=\lambda_{0}\lambda_{5}-\lambda_{2}\lambda_{4},\quad\lambda_{7}=\lambda_{2}\left((-\lambda_{6} -1)\lambda_{2}+\lambda_{1}\lambda_{6}\right),\quad\lambda_{8}=\lambda_{5}\left(\lambda_{0}\lambda_{5}- \lambda_{1}\lambda_{4}-\lambda_{0}+\lambda_{5}-1\right)
	\]
	\texttt{Under condition (iii) of Theorem} \ref{theorem5.9}, \texttt{the associated matrix is:}
	\[
	\begin{bmatrix}
	\lambda_{0}+\frac{\lambda_{2}^2}{4\lambda_{3}}t & g(t)\\
	\lambda_{4}+\frac{\lambda_{2}\lambda_{5}}{4\lambda_{3}}t & h(t)
	\end{bmatrix}
	\]
	with
	\begin{align*}
	g(t) &=
	\lambda_{1}+\frac{((-1-\lambda_{0})\lambda_{2}+\lambda_{0}\lambda_{1})\lambda_{9}^2}{2\lambda_{3}
		\lambda_{4}^2}t+\left(\frac{(1+\lambda_{0})^2\lambda_{9}^5-
		2(1+\lambda_{0})\lambda_{1}\lambda_{4}\lambda_{9}^4+8\lambda_{3}^2\lambda_{4}^5+\lambda_{1}^2
		\lambda_{4}^2\lambda_{9}^3}{4\lambda_{3}\lambda_{4}^4 \lambda_{9}}+\lambda_{2}t\right)x\\
	&\quad-\frac{(\lambda_{0}\lambda_{9}-\lambda_{1}\lambda_{4}+\lambda_{9})\lambda_{9}}{\lambda_{4}^2}x^2+
	\lambda_{3}x^3\\
	h(t) &=\frac{\lambda_{1}\lambda_{4}+\lambda_{0}+1}{\lambda_{0}}+\frac{1-\lambda_{5}}{2\lambda_{3}
		\lambda_{4}^2}\lambda_{2}t+\Bigg(\frac{(-1+(1+\lambda_{0})\lambda_{5}-\lambda_{0}-\lambda_{1}\lambda_{4})
		\lambda_{2}+(-2+\lambda_{5})\lambda_{2}^2+8\lambda_{3}^2\lambda_{4}^5}{4(-1+\lambda_{5})\lambda_{0}
		\lambda_{3} \lambda_{4}^3}\\ &\quad+\lambda_{5}t\Bigg)x-\frac{\lambda_{2}}{\lambda_{0}\lambda_{4}}x^2+
	\frac{\lambda_{3}\lambda_{5}}{\lambda_{2}}x^3
	\end{align*}
	\texttt{and where:}
	\[
	\lambda_9=-1+\lambda_5
	\]
	\texttt{The existence of its inverse is not guaranteed, but the automorphism property holds as per Theorem \ref{theorem5.9a}.}
	
	\medskip
	
	\noindent
	\texttt{----- Case 5 -----}
	
	\smallskip
	
	\noindent
	\texttt{The Dixmier polynomials are:}
	\begin{align*}
	p &= \frac{\lambda_{1}^2}{4\lambda_{2}}t^2+\left(\frac{\lambda_{1}\lambda_{3}-1}{\lambda_{3}}+
	\lambda_{0}t\right)x+\left(\frac{\lambda_{0}^2\lambda_{2}}{\lambda_{1}^2}+\lambda_{1}t\right)x^2+
	\frac{2\lambda_{0}\lambda_{2}}{\lambda_{1}}x^3+\lambda_{2}x^4\\
	q&=\lambda_{3}t+\frac{\lambda_{1}\lambda_{4}}{4\lambda_{2}}t^2+\left(\frac{2\lambda_{0}\lambda_{2}
		\lambda_{3}^2+\lambda_{1}^2\lambda_{3}\lambda_{4}-\lambda_{1}\lambda_{4}}{\lambda_{1}^2\lambda_{3}}+
	\frac{\lambda_{0}\lambda_{4}}{\lambda_{1}}t\right)x+\left(\frac{(\lambda_{0}^2\lambda_{4}+
		2\lambda_{1}^2\lambda_{3})\lambda_{2}}{\lambda_{1}^3}+\lambda_{4}t \right)x^2\\ &\quad+\frac{2\lambda_{0}
		\lambda_{2}\lambda_{4}}{\lambda_{1}^2}x^3+\frac{\lambda_{2}\lambda_{4}}{\lambda_{1}}x^4
	\end{align*}
	\texttt{Under condition (iii) of Theorem} \ref{theorem5.9}, \texttt{the associated matrix is:}
	\[
	\begin{bmatrix}
	\frac{\lambda_{1}^2}{4\lambda_{2}}t &
	\frac{\lambda_{1}^3\lambda_{3}-\lambda_{1}^2}{\lambda_{1}^2\lambda_{3}}+\lambda_{0}t+\left( \frac{\lambda_{0}^2\lambda_{2}}{\lambda_{1}^2}+\lambda_{1}t\right)x+\frac{2\lambda_{0} \lambda_{2}}{\lambda_{1}}x^2+\lambda_{2}x^3\\
	\lambda_{3}+\frac{\lambda_{1}\lambda_{4}}{4\lambda_{2}}t & h(t)
	\end{bmatrix}
	\]
	with
	$h(t)=\frac{2\lambda_{0}\lambda_{1}\lambda_{2}\lambda_{3}^2+\lambda_{1}^3\lambda_{3}\lambda_{4}-
		\lambda_{1}^2 \lambda_{4}}{\lambda_{1}^3\lambda_{3}}+\frac{\lambda_{0}\lambda_{4}}{\lambda_{1}}t+
	\left(\frac{\lambda_{0}^2\lambda_{2}\lambda_{3}\lambda_{4}+2\lambda_{1}^2\lambda_{2}\lambda_{3}^2}	{\lambda_{1}^3\lambda_{3}}+\lambda_{4}t\right)x+\frac{2\lambda_{0}\lambda_{2}\lambda_{4}}{\lambda_{1}^2}
	x^2+\frac{\lambda_{2}\lambda_{4}}{\lambda_{1}}x^3$\newline
	\texttt{The existence of its inverse is not guaranteed, but the automorphism property holds as per Theorem \ref{theorem5.9a}.}
	
	\medskip
	
	\noindent
	\texttt{----- Case 6 -----}
	
	\smallskip
	
	\noindent
	\texttt{The Dixmier polynomials are:}
	\begin{align*}
	p &= \lambda_{0}t+\lambda_{1}t^2+\left(-\frac{2\lambda_{6}t}{\lambda_{0}\lambda_{5}}+\lambda_{2}\right)x+
	\frac{(-\lambda_{2}\lambda_{5}+\lambda_{1})\lambda_{6}}{\lambda_{0}^2\lambda_{5}^2}x^2\\
	q&=\lambda_{3}t+\lambda_{4}t^2+\left(\frac{\lambda_{4}^2+\lambda_{7}}{\lambda_{0}\lambda_{4}}+
	\frac{2\lambda_{7}}{\lambda_{0} \lambda_{3}}t\right)x+\frac{(\lambda_{2}\lambda_{3}-\lambda_{4}+1)
		\lambda_{7}}{\lambda_{0}^2\lambda_{3}^2}x^2
	\end{align*}
	\texttt{where:}
	\[	\lambda_{5}=\lambda_{0}\lambda_{4}-\lambda_{1}\lambda_{3},\quad\lambda_{6}=\lambda_{1}\left(-\lambda_{2} \lambda_{5}+\lambda_{1}\right),\quad\lambda_{7}=\lambda_{4}\left(\lambda_{2}\lambda_{3}-\lambda_{4}+1 \right)
	\]
	\texttt{Under condition (iii) of Theorem} \ref{theorem5.9}, \texttt{the associated matrix is:}
	\[
	\begin{bmatrix}
	\lambda_{0}+\lambda_{1}t & \lambda_{2}+\frac{2\lambda_{8}}{\lambda_{3}^2}t+\frac{(\lambda_{2}\lambda_{3}-\lambda_{4}+1)\lambda_{8}}
	{\lambda_{0}\lambda_{3}^3}x\\
	\lambda_{3}+\lambda_{4}t & \frac{\lambda_{4}^2+\lambda_{2}}{\lambda_{0}\lambda_{4}}+\frac{2\lambda_{2}}{\lambda_{0}\lambda_{3}}t+ \frac{(\lambda_{2}\lambda_{3}-\lambda_{4}+1)\lambda_{2}}{\lambda_{0}^2\lambda_{3}^2}x
	\end{bmatrix}
	\]
	\texttt{where:}
	\[
	\lambda_8=(\lambda_4-1)(\lambda_2\lambda_3-\lambda_4+1)
	\]
	\texttt{The existence of its inverse is not guaranteed, but the automorphism property holds as per Theorem \ref{theorem5.9a}.}
	
	\medskip
	
	\noindent
	\texttt{----- Case 7 -----}
	
	\smallskip
	
	\noindent
	\texttt{The Dixmier polynomials are:}
	\begin{align*}
	p &= \lambda_{0}t^2+\left(-\frac{1}{\lambda_{2}}+\lambda_{1}t\right)x+\frac{\lambda_{1}^2}{4\lambda_{0}} x^2\\
	q &= \lambda_{2}t+\lambda_{3}t^2+\left(\frac{\lambda_{1}\lambda_{2}^2-2\lambda_{3}}{2\lambda_{0}
		\lambda_{2}}+\frac{\lambda_{1} \lambda_{3}}{\lambda_{0}}t\right)x+\frac{\lambda_{1}^2\lambda_{3}}
	{4\lambda_{0}^2}x^2
	\end{align*}
	\texttt{Under condition (iii) of Theorem} \ref{theorem5.9}, \texttt{the associated matrix is:}
	\[
	\begin{bmatrix}
	\lambda_{0}t&-\frac{1}{\lambda_{2}}+\lambda_{1}t+\frac{\lambda_{1}^2}{4\lambda_{0}}x\\
	\lambda_{2}+\lambda_{3}t & \frac{2\lambda_{0}\lambda_{1}\lambda_{2}^2-4\lambda_{0} \lambda_{3}}{4\lambda_{0}^2\lambda_{2}}+\frac{\lambda_{1}\lambda_{3}}{\lambda_{0}}t+
	\frac{\lambda_{1}^2\lambda_{3}}{4\lambda_{0}^2}x
	\end{bmatrix}
	\]
	\texttt{The existence of its inverse is not guaranteed, but the automorphism property holds as per Theorem \ref{theorem5.9a}.}
	
	\medskip
	
	\noindent
	\texttt{----- Case 8 -----}
	
	\smallskip
	
	\noindent
	\texttt{The Dixmier polynomials are:}
	\begin{align*}
	p &= \lambda_{0}t+\lambda_{1}x+\lambda_{2}x^2+\lambda_{3}x^3+\lambda_{4}x^4\\
	q &=\frac{\lambda_{0}\lambda_{5}}{\lambda_{4}}t+\frac{\lambda_{0}\lambda_{1}\lambda_{5}+ \lambda_{4}}{\lambda_{0}\lambda_{4}}x+\frac{\lambda_{2}\lambda_{5}}{\lambda_{4}}x^2+\frac{\lambda_{3} \lambda_{5}}{\lambda_{4}}x^3+\lambda_{5}x^4
	\end{align*}
	\texttt{Under condition (iii) of Theorem} \ref{theorem5.9}, \texttt{the associated matrix is:}
	\[
	\begin{bmatrix}
	\lambda_{0} & \lambda_{1}+\lambda_{2}x+\lambda_{3}x^2+\lambda_{4}x^3\\
	\frac{\lambda_{0}\lambda_{5}}{\lambda_{4}} & \frac{\lambda_{0}\lambda_{1}\lambda_{5}+ \lambda_{4}}{\lambda_{0}\lambda_{4}}+\frac{\lambda_{2}\lambda_{5}}{\lambda_{4}}x+\frac{\lambda_{3} \lambda_{5}}{\lambda_{4}}x^2+\lambda_{5}x^3
	\end{bmatrix}
	\]
	\texttt{Its inverse is:}
	\[
	\begin{bmatrix}
	\frac{\lambda_{0}\lambda_{1}\lambda_{5}+\lambda_{4}}{\lambda_{0}\lambda_{4}}+\frac{\lambda_{2} \lambda_{5}}{\lambda_{4}}x+\frac{\lambda_{3}\lambda_{5}}{\lambda_{4}}x^2+\lambda_{5}x^3 &
	-\lambda_{1}-\lambda_{2}x-\lambda_{3}x^2-\lambda_{4}x^3\\
	-\frac{\lambda_{0}\lambda_{5}}{\lambda_{4}} & \lambda_{0}
	\end{bmatrix}
	\]
	\texttt{Thus, the polynomials satisfy Theorem \ref{theorem5.9} and define an automorphism of A\_1(K).}
	
	\medskip
	
	\noindent
	\texttt{----- Case 9 -----}
	
	\smallskip
	
	\noindent
	\texttt{The Dixmier polynomials are:}
	\begin{align*}
	p &= \lambda_{0}t+\lambda_{1}x+\lambda_{2}x^2+\lambda_{3}x^3\\
	q &= \frac{\lambda_{0}\lambda_{4}}{\lambda_{3}}t+\frac{\lambda_{0}\lambda_{1}\lambda_{4}+ \lambda_{3}}{\lambda_{0}\lambda_{3}}x+\frac{\lambda_{2}\lambda_{4}}{\lambda_{3}}x^2+\lambda_{4}x^3
	\end{align*}
	\texttt{Under condition (iii) of Theorem} \ref{theorem5.9}, \texttt{the associated matrix is:}
	\[
	\begin{bmatrix}
	\lambda_{0} & \lambda_{1}+\lambda_{2}x+\lambda_{3}x^2\\
	\frac{\lambda_{0}\lambda_{4}}{\lambda_{3}} & \frac{\lambda_{0}\lambda_{1}\lambda_{4}+ \lambda_{3}}{\lambda_{0}\lambda_{3}}+\frac{\lambda_{2}\lambda_{4}}{\lambda_{3}}x+\lambda_{4}x^2
	\end{bmatrix}
	\]
	\texttt{Its inverse is:}
	\[
	\begin{bmatrix}
	\frac{\lambda_{0}\lambda_{1}\lambda_{4}+\lambda_{3}}{\lambda_{0}\lambda_{3}}+\frac{\lambda_{2} \lambda_{4}}{\lambda_{3}}x+\lambda_{4}x^2 & -\lambda_{1}-\lambda_{2}x-\lambda_{3}x^2\\
	-\frac{\lambda_{0}\lambda_{4}}{\lambda_{3}} & \lambda_{0}
	\end{bmatrix}
	\]
	\texttt{Thus, the polynomials satisfy Theorem \ref{theorem5.9} and define an automorphism of A\_1(K).}
	
	\medskip
	
	\noindent
	\texttt{----- Case 10 -----}
	
	\smallskip
	
	\noindent
	\texttt{The Dixmier polynomials are:}
	\begin{align*}
	p &= \lambda_{0}t+\lambda_{1}x+\lambda_{2}x^2\\
	q &= \lambda_{3}t+\lambda_{4}t^2+\left(\frac{\lambda_{1}\lambda_{3}+2\lambda_{2}\lambda_{4}+1}
	{\lambda_{0}}+\frac{2\lambda_{1} \lambda_{4}}{\lambda_{0}}t\right)x+\left(\frac{\lambda_{0}\lambda_{2}		\lambda_{3}+\lambda_{1}^2\lambda_{4}}{\lambda_{0}^2}+\frac{2\lambda_{2}\lambda_{4}}{\lambda_{0}}t\right)
	x^2+\frac{2\lambda_{1}\lambda_{2}\lambda_{4}}{\lambda_{0}^2}x^3\\
	&\quad +\frac{\lambda_{2}^2\lambda_{4}}{\lambda_{0}^2}x^4
	\end{align*}
	\texttt{Under condition (iii) of Theorem} \ref{theorem5.9}, \texttt{the associated matrix is:}
	\[
	\begin{bmatrix}
	\lambda_{0} & \lambda_{1}+\lambda_{2}x\\
	\lambda_{3}+\lambda_{4}t & \frac{\lambda_{1}\lambda_{3}+2\lambda_{2}\lambda_{4}+1}{\lambda_{0}}+\frac{2 \lambda_{1}\lambda_{4}}{\lambda_{0}}t+\left(\frac{\lambda_{0}\lambda_{2}\lambda_{3}+\lambda_{1}^2 \lambda_{4}}{\lambda_{0}^2}+\frac{2\lambda_{2}\lambda_{4}}{\lambda_{0}}t\right)x+\frac{2\lambda_{1} \lambda_{2}\lambda_{4}}{\lambda_{0}^2}x^2+\frac{\lambda_{2}^2\lambda_{4}}{\lambda_{0}^2}x^3
	\end{bmatrix}
	\]
	\texttt{The existence of its inverse is not guaranteed, but the automorphism property holds as per Theorem \ref{theorem5.9a}.}
	
	\medskip
	
	\noindent
	\texttt{----- Case 11 -----}
	
	\smallskip
	
	\noindent
	\texttt{The Dixmier polynomials are:}
	\begin{align*}
	p &= \lambda_0x\\
	q &= -\frac{1}{\lambda_{0}}t+\lambda_{1}x+\lambda_{2}x^2+\lambda_{3}x^3+\lambda_{4}x^4
	\end{align*}
	\texttt{Under condition (iii) of Theorem} \ref{theorem5.9}, \texttt{the associated matrix is:}
	\[
	\begin{bmatrix}
	0 & \lambda_{0}\\
	-\frac{1}{\lambda_{0}} & \lambda_{1}+\lambda_{2}x+\lambda_{3}x^2+\lambda_{4}x^3
	\end{bmatrix}
	\]
	\texttt{Its inverse is:}
	\[
	\begin{bmatrix}
	\lambda_{1}+\lambda_{2}x+\lambda_{3}x^2+\lambda_{4}x^3 & -\lambda_{0}\\
	\frac{1}{\lambda_{0}} & 0
	\end{bmatrix}
	\]
	\texttt{Thus, the polynomials satisfy Theorem \ref{theorem5.9} and define an automorphism of A\_1(K).}
\end{example}

Below we present more general families of Dixmier polynomial pairs in $A_1(K)$, which include some of the families from the previous examples obtained using the \texttt{DixmierPolynomials} function.

\begin{proposition}
	The following families of polynomial pairs in $A_1(K)$ satisfy conditions {\rm(i)-(iii)} of Theorem \ref{theorem5.9}.
	\begin{align*}
	& p=\alpha t+\tfrac{1}{\lambda}(\alpha f(x)-1)x,
	&& q=\lambda t+f(x)x,
	&& \alpha, \lambda \in K, \ \lambda \neq 0, \ f(x) \in K[x], \tag{\texttt{Type I}} \\[0.5em]
	& p=-f(t)t-\lambda x,
	&& q=\tfrac{1}{\lambda}(\alpha f(t)+1)t+\alpha x,
	&& \alpha, \lambda \in K, \ \lambda \neq 0, \ f(t) \in K[t], \tag{\texttt{Type II}} \\[0.5em]
	& p=\lambda t+g(x)x,
	&& q=\tfrac{1}{\lambda}x,
	&& \lambda \in K, \ \lambda \neq 0, \ g(x) \in K[x], \tag{\texttt{Type III}} \\[0.5em]
	& p=\tfrac{1}{\lambda}t,
	&& q=g(t)t+\lambda x,
	&& \lambda \in K, \ \lambda \neq 0, \ g(t) \in K[t], \tag{\texttt{Type IV}}
	\end{align*}
\end{proposition}
\begin{proof}
	It is clear that each of the pairs satisfies condition {\rm(ii)} of Theorem \ref{theorem5.9}. Let us now verify that each family also satisfies conditions {\rm(i)} and {\rm(iii)}.
	\paragraph{Case \texttt{Type I}} The polynomials $p$ and $q$ are Dixmier polynomials. In fact, since
	\[
	qp=\left(\lambda t+f(x)x\right)\left(\alpha t+\tfrac{\alpha}{\lambda}
	f(x)x-\tfrac{1}{\lambda}x\right)=\alpha\lambda t^2+\alpha f(x)xt+\alpha tf(x)x+
	\tfrac{\alpha}{\lambda}(f(x)x)^2-tx-\tfrac{1}{\lambda}f(x)x^2,
	\]
	and
	\[
	pq=\left(\alpha t+\tfrac{\alpha}{\lambda}f(x)x-\tfrac{1}{\lambda}x\right)\left(\lambda t
	+f(x)x\right)=\alpha\lambda t^2+\alpha f(x)xt-xt+\alpha tf(x)x+\tfrac{\alpha}{\lambda}(f(x)x)^2
	-\tfrac{1}{\lambda}f(x)x^2,
	\]
	it follows that $qp-pq=-tx+xt=-tx+tx+1=1.$
	
	Moreover, the matrix for condition {\rm(iii)} in Theorem 5.2 and its inverse are given respectively by:
	\[
	\begin{bmatrix}
	\alpha & \frac{1}{\lambda}(\alpha f(x)-1)\\
	\lambda & f(x)
	\end{bmatrix},\qquad
	\begin{bmatrix}
	f(x) & -\frac{1}{\lambda}(\alpha f(x)-1)\\
	-\lambda & \alpha
	\end{bmatrix}.
	\]
	\paragraph{Case \texttt{Type II}} The polynomials $p$ and $q$ are Dixmier polynomials. In fact, since
	\[
	qp=\left(\tfrac{\alpha }{\lambda }f(t)t+\tfrac{1}{\lambda }t+\alpha x\right)
	\left(-f(t)t-\lambda x\right)=-\tfrac{\alpha }{\lambda }(f(t)t)^2-\tfrac{1}{\lambda }f(t)t^2-
	\alpha xtf(t)-\alpha f(t)tx-tx-\alpha \lambda x^2,
	\]
	and
	\[
	pq=\left(-f(t)t-\lambda x\right)\left(\tfrac{\alpha }{\lambda }f(t)t+\tfrac{1}{\lambda }t+
	\alpha x\right)=-\tfrac{\alpha }{\lambda }(f(t)t)^2-\tfrac{1}{\lambda }f(t)t^2-
	\alpha f(t)tx-\alpha xtf(t)-xt-\alpha \lambda x^2,
	\]
	it follows that $qp-pq=-tx+xt=-tx+tx+1=1.$
	
	Moreover, the matrix for condition {\rm(iii)} in Theorem 5.2 and its inverse are given respectively by:
	\[
	\begin{bmatrix}
	-f(t) & -\lambda \\
	\tfrac{1}{\lambda }(\alpha f(t)+1) & \alpha
	\end{bmatrix},\qquad
	\begin{bmatrix}
	\alpha & \lambda \\
	-\tfrac{1}{\lambda }(\alpha f(t)+1) & -f(t)
	\end{bmatrix}.
	\]
	\paragraph{Case \texttt{Type III}} The polynomials $p$ and $q$ are Dixmier polynomials. In fact, since
	\[
	qp=\tfrac{1}{\lambda }x(\lambda t+g(x)x)=xt+\tfrac{1}{\lambda }g(x)x^2=tx+\tfrac{1}{\lambda }g(x)x^2+1=
	(\lambda t+g(x)x)\tfrac{1}{\lambda }x+1=pq+1,
	\]
	it follows that $qp-pq=1.$
	
	Moreover, the matrix for condition {\rm(iii)} in Theorem 5.2 and its inverse are given respectively by:
	\[
	\begin{bmatrix}
	\lambda  & g(x)\\
	0 & \frac{1}{\lambda }
	\end{bmatrix},\qquad
	\begin{bmatrix}
	\frac{1}{\lambda } & -g(x)\\
	0 & \lambda
	\end{bmatrix}.
	\]
	\paragraph{Case \texttt{Type IV}} The polynomials $p$ and $q$ are Dixmier polynomials. In fact, since
	\[
	qp=(\lambda x+g(t)t)\tfrac{1}{\lambda}t=xt+\tfrac{1}{\lambda}g(t)t^2=tx+\tfrac{1}{\lambda}g(t)t^2+1=
	\tfrac{1}{\lambda }t(\lambda x+g(t)t)+1=pq+1,
	\]
	it follows that $qp-pq=1.$
	
	Moreover, the matrix for condition {\rm(iii)} in Theorem 5.2 and its inverse are given respectively by:
	\[
	\begin{bmatrix}
	\frac{1}{\lambda } & 0\\
	g(t) & \lambda
	\end{bmatrix},\qquad
	\begin{bmatrix}
	\lambda  & 0\\
	-g(t) & \tfrac{1}{\lambda }
	\end{bmatrix}.
	\]
\end{proof}

\subsubsection*{Test tables of \texttt{DixmierPolynomial} function}

The following tables present families of Dixmier polynomials obtained using the \texttt{DixmierPolynomials} function. The results correspond to different values of $n$ and $m$, specifically: $n=1$ with $m$ ranging from 1 to 7; $m=1$ with $n$ ranging from $2$ to $7;$ $m=2$ with $n$ ranging from $2$ to $7;$ and $n=5$ with $m=3.$ These tables provide a structured overview of the polynomial forms derived under these specific conditions, offering insight into their behavior and properties. Greek letters represent elements of $K.$

\begin{table}[htb]
\centering\renewcommand{\arraystretch}{2.2}
\normalsize{
\begin{tabular}{|cp{12.7cm}|}\hline %
\textbf{$m$} & \textbf{Cases\quad} \textbf{Polynomials $p$ and $q$}\\[0.2cm] \hline \hline			
$1$ & \parbox[t][0.4\height][b]{12.5cm}{$
\begin{aligned}
\texttt{Case\,1:} &\quad p=\tfrac{1+\lambda_0\lambda_1}{\lambda_2}t+\lambda_0x,\quad q=\lambda_1t+\lambda_2x\\
\texttt{Case\,2:} &\quad p=\lambda_0t-\tfrac{1}{\lambda_1}x,\quad q=\lambda_1t
\end{aligned}$} \\[0.5cm] \hline
$2$ & \parbox[t][0.4\height][b]{12.5cm}{$
\begin{aligned}
\texttt{Case\,1:} &\quad p=\tfrac{1+\lambda_0\lambda_1}{\lambda_3}t+\tfrac{\lambda_0\lambda_2}{\lambda_3}t^2+\lambda_0x,\quad
q=\lambda_1t+\lambda_2t^2+\lambda_3x\\
\texttt{Case\,2:} &\quad p=\lambda_0t+\lambda_1t^2+\lambda_2x,\quad q = -\tfrac{1}{\lambda_2}t
\end{aligned}$} \\[0.5cm] \hline 
$3$ & \parbox[t][0.5\height][b]{12.5cm}{$
\begin{aligned}
\texttt{Case\,1:} &\quad p=\tfrac{1+\lambda_0\lambda_1}{\lambda_4}t+\tfrac{\lambda_0\lambda_2}{\lambda_4}t^2+ \tfrac{\lambda_0\lambda_3}{\lambda_4}t^3+\lambda_0x,\\
&\quad q=\lambda_1t+\lambda_2t^2+\lambda_3t^3+\lambda_4x\\ 
\texttt{Case\,2:} &\quad p =\lambda_0t+\lambda_1t^2+\lambda_2t^3+\lambda_3x, \quad q=-\tfrac{1}{\lambda_3}t \end{aligned}$} \\[1cm] \hline
$4$ & \parbox[t][0.4\height][b]{12.5cm}{$
\begin{aligned}
\texttt{Case\,1:} &\quad p=\tfrac{1+\lambda_0\lambda_1}{\lambda_5}t+\tfrac{\lambda_0\lambda_2}{\lambda_5}t^2+ \tfrac{\lambda_0\lambda_3}{\lambda_5}t^3+\tfrac{\lambda_0\lambda_4}{\lambda_5}t^4+\lambda_0x,\\
&\quad q=\lambda_1t+\lambda_2t^2+\lambda_3t^3+\lambda_4t^4+\lambda_5x\\
\texttt{Case\,2:} &\quad p=\lambda_0t+\lambda_1t^2+\lambda_2t^3+\lambda_3t^4+\lambda_4x, \quad q=-\tfrac{1}{\lambda_4}t
\end{aligned}$} \\[1cm] \hline 
$5$ & \parbox[t][0.4\height][b]{12.5cm}{$
\begin{aligned}
\texttt{Case\,1:} &\quad p=\tfrac{1+\lambda_0\lambda_1}{\lambda_6}t+\tfrac{\lambda_0\lambda_2}{\lambda_6}t^2+ \tfrac{\lambda_0\lambda_3}{\lambda_6}t^3+\tfrac{\lambda_0\lambda_4}{\lambda_6}t^4+
\tfrac{\lambda_0\lambda_5}{\lambda_6}t^5+\lambda_0x,\\
&\quad q=\lambda_1t+\lambda_2t^2+\lambda_3t^3+\lambda_4t^4+\lambda_5t^5+\lambda_6x\\
\texttt{Case\,2:} &\quad p=\lambda_0t+\lambda_1t^2+\lambda_2t^3+\lambda_3t^4+\lambda_4t^5+\lambda_5x, \quad q=-\tfrac{1}{\lambda_5}t
\end{aligned}$} \\[1cm] \hline 
$6$ & \parbox[t][0.4\height][b]{12.5cm}{$
\begin{aligned}
\texttt{Case\,1:} &\quad p=\tfrac{1+\lambda_0\lambda_1}{\lambda_7}t+\tfrac{\lambda_0\lambda_2}{\lambda_7}t^2+ \tfrac{\lambda_0\lambda_3}{\lambda_7}t^3+\tfrac{\lambda_0\lambda_4}{\lambda_7}t^4+
\tfrac{\lambda_0\lambda_5}{\lambda_7}t^5+\tfrac{\lambda_0\lambda_6}{\lambda_7}t^6+
\lambda_0x,\\
&\quad q=\lambda_1t+\lambda_2t^2+\lambda_3t^3+\lambda_4t^4+\lambda_5t^5
+\lambda_6t^6+\lambda_7x\\
\texttt{Case\,2:} &\quad p=\lambda_0t+\lambda_1t^2+\lambda_2t^3+\lambda_3t^4+\lambda_4t^5+
\lambda_5t^6+\lambda_6x, \quad q=-\tfrac{1}{\lambda_6}t
\end{aligned}$} \\[1cm] \hline 
$7$ & \parbox[t][0.4\height][b]{12.5cm}{$
\begin{aligned}
\texttt{Case\,1:} &\quad p=\tfrac{1+\lambda_0\lambda_1}{\lambda_8}t+\tfrac{\lambda_0\lambda_2}{\lambda_8}t^2+ \tfrac{\lambda_0\lambda_3}{\lambda_8}t^3+\tfrac{\lambda_0\lambda_4}{\lambda_8}t^4+
\tfrac{\lambda_0\lambda_5}{\lambda_8}t^5+\tfrac{\lambda_0\lambda_6}{\lambda_8}t^6+
\tfrac{\lambda_0\lambda_7}{\lambda_8}t^7+\lambda_0x,\\
&\quad q=\lambda_1t+\lambda_2t^2+\lambda_3t^3+\lambda_4t^4+\lambda_5t^5
+\lambda_6t^6+\lambda_7t^7+\lambda_8x\\
\texttt{Case\,2:} &\quad p=\lambda_0t+\lambda_1t^2+\lambda_2t^3+\lambda_3t^4+\lambda_4t^5+
\lambda_5t^6+\lambda_6t^7+\lambda_7x, \quad q=-\tfrac{1}{\lambda_7}t
\end{aligned}$}\\[1cm] \hline
\end{tabular}}
\caption{Test table for Dixmier polynomials with $n=1$ and $m$ from $1$ to $7$}\label{TestTable:n1}
\end{table}

\begin{table}[htb]
\centering\renewcommand{\arraystretch}{1.7}
\normalsize{
\begin{tabular}{|clp{11.5cm}|}\hline %
\textbf{$n$} & \textbf{Cases} & \textbf{Polynomials $p$ and $q$}\\ \hline \hline		
\texttt{$2$} & \texttt{Case 1:} & $p=\lambda_0t+\lambda_1x+\lambda_2x^2,$\quad $q=\lambda_3t+\frac{1+\lambda_1\lambda_3}{\lambda_0}x+
\frac{\lambda_2\lambda_3}{\lambda_0}x^2$\\ 
& \texttt{Case 2:} & $p=\lambda_0x,\quad q=-\frac{1}{\lambda_0}t+\lambda_1x+\lambda_2x^2$\\[0.1cm] \hline
\texttt{$3$} & \texttt{Case 1:} &  $p=\lambda_0t+\lambda_1x+\lambda_2x^2+\lambda_3x^3,$\quad $q=\lambda_4t+\frac{1+\lambda_1\lambda_4}{\lambda_0}x+
\frac{\lambda_2\lambda_4}{\lambda_0}x^2+\frac{\lambda_3\lambda_4}{\lambda_0}x^3$\\ 
& \texttt{Case 2:} & $p=\lambda_0x,\quad q=-\frac{1}{\lambda_0}t+\lambda_1x+\lambda_2x^2+\lambda_3x^3$
\\[0.1cm] \hline
\texttt{$4$} & \texttt{Case 1:} &  $p=\lambda_0t+\lambda_1x+\lambda_2x^2+\lambda_3x^3+\lambda_4x^4,$\newline $q=\lambda_5t+\frac{1+\lambda_1\lambda_5}{\lambda_0}x+
\frac{\lambda_2\lambda_5}{\lambda_0}x^2+\frac{\lambda_3\lambda_5}{\lambda_0}x^3+
\frac{\lambda_4\lambda_5}{\lambda_0}x^4$\\ 
& \texttt{Case 2:} & $p=\lambda_0x,\quad q=-\frac{1}{\lambda_0}t+\lambda_1x+\lambda_2x^2+\lambda_3x^3+
\lambda_4x^4$\\[0.1cm] \hline
\texttt{$5$} & \texttt{Case 1:} &  $p=\lambda_0t+\lambda_1x+\lambda_2x^2+\lambda_3x^3+\lambda_4x^4+\lambda_5x^5,$\newline $q=\lambda_6t+\frac{1+\lambda_1\lambda_6}{\lambda_0}x+
\frac{\lambda_2\lambda_6}{\lambda_0}x^2+\frac{\lambda_3\lambda_6}{\lambda_0}x^3+
\frac{\lambda_4\lambda_6}{\lambda_0}x^4+\frac{\lambda_5\lambda_6}{\lambda_0}x^5$\\ 
& \texttt{Case 2:} & $p=\lambda_0x,\quad q=-\frac{1}{\lambda_0}t+\lambda_1x+\lambda_2x^2+\lambda_3x^3+
\lambda_4x^4+\lambda_5x^5$\\[0.1cm] \hline
\texttt{$6$} & \texttt{Case 1:} & $p=\lambda_0t+\lambda_1x+\lambda_2x^2+\lambda_3x^3+\lambda_4x^4+\lambda_5x^5+\lambda_6x^6,$\newline $q=\lambda_7t+\frac{1+\lambda_1\lambda_7}{\lambda_0}x+
\frac{\lambda_2\lambda_7}{\lambda_0}x^2+\frac{\lambda_3\lambda_7}{\lambda_0}x^3+
\frac{\lambda_4\lambda_7}{\lambda_0}x^4+\frac{\lambda_5\lambda_7}{\lambda_0}x^5+
\frac{\lambda_6\lambda_7}{\lambda_0}x^6$\\
& \texttt{Case 2:} & $p=\lambda_0x,\quad q=-\frac{1}{\lambda_0}t+\lambda_1x+\lambda_2x^2+\lambda_3x^3+
\lambda_4x^4+ \lambda_5x^5+\lambda_6x^6$\\[0.1cm] \hline
\texttt{$7$} & \texttt{Case 1:} & $p=\lambda_0t+\lambda_1x+\lambda_2x^2+\lambda_3x^3+\lambda_4x^4+\lambda_5x^5+\lambda_6x^6+ \lambda_7x^7,$\newline $q=\lambda_8t+\frac{1+\lambda_1\lambda_8}{\lambda_0}x+
\frac{\lambda_2\lambda_8}{\lambda_0}x^2+\frac{\lambda_3\lambda_8}{\lambda_0}x^3+
\frac{\lambda_4\lambda_8}{\lambda_0}x^4+\frac{\lambda_5\lambda_8}{\lambda_0}x^5+
\frac{\lambda_6\lambda_8}{\lambda_0}x^6+\frac{\lambda_7\lambda_8}{\lambda_0}x^7$\\ 
& \texttt{Case 2:} & $p=\lambda_0x,\quad q=-\frac{1}{\lambda_0}t+\lambda_1x+\lambda_2x^2+\lambda_3x^3+
\lambda_4x^4+ \lambda_5x^5+\lambda_6x^6+\lambda_7x^7$\\[0.1cm] \hline
\end{tabular}}
\caption{Test table for Dixmier polynomials with $n$ from $2$ to $7$ and $m=1$}\label{TestTable:m1}
\end{table}

\begin{center}
\begin{table}[htb]
\centering\renewcommand{\arraystretch}{1.7}
\normalsize{
\begin{tabular}{|rlp{12.3cm}|}\hline %
$n$ & \textbf{Cases} & \textbf{Polynomials $p$ and $q$}\\[0.1cm] \hline \hline				
\texttt{$2$} & \texttt{Case 1:} & $p=\lambda_0t+\lambda_1t^2+\left(\lambda_2+\frac{2(1+\lambda_2\lambda_4)\lambda_1}{\lambda_0 \lambda_4}t\right)x +\frac{\lambda_1(1+\lambda_2\lambda_4)^2}{\lambda_0^2\lambda_4^2}x^2,\quad
q=\lambda_4t+\frac{1+\lambda_2 \lambda_4}{\lambda_0}x$\\ 
& \texttt{Case 2:} & $p=\lambda_0t+\lambda_1x+\lambda_2x^2,\quad q=\lambda_3t+\frac{1+\lambda_1\lambda_3}{\lambda_0}x+\frac{\lambda_2\lambda_3}{\lambda_0}x^2$\\ 
& \texttt{Case 3:} & $p=\lambda_0t+\lambda_1t^2+\lambda_2x,\quad q =\frac{-\lambda_1+\lambda_0 \lambda_2\lambda_3}{\lambda_1\lambda_2}t+\lambda_3t^2+\frac{\lambda_2\lambda_3}{\lambda_1}x$\\ 
& \texttt{Case 4:} & $p=\lambda_0t+\lambda_1x,\quad q=\lambda_2t+\lambda_3t^2+\left(\frac{1+\lambda_1 \lambda_2}{\lambda_0}+\frac{2\lambda_1\lambda_3}{\lambda_0}t\right)x+
\frac{\lambda_1^2\lambda_3}{\lambda_0^2}x^2$\\[0.1cm] \cline{1-3}
\texttt{$3$} & \texttt{Case 1:} &
$p=\lambda_0t+\lambda_1t^2+\left(\lambda_2+\frac{2\gamma\lambda_1}{\beta\lambda_0}t\right)x+
\frac{\gamma^2\lambda_1}{\beta^2\lambda_0^2}x^2$,\newline
$q=\lambda_3t+\lambda_4t^2+\left(\frac{\lambda_2\lambda_3+1}{\lambda_0}+
\frac{2\gamma\lambda_4}{\beta\lambda_0}t\right)x+\frac{\gamma^2\lambda_4}{\beta^2\lambda_0^2}x^2,\newline
\text{where: }\beta=\lambda_0\lambda_4-\lambda_1\lambda_3,\,\gamma= \beta\lambda_2-\lambda_1$\\
& \texttt{Case 2:} &
$p=\lambda_0t^2+\left(-\frac{1}{\lambda_2}+\lambda_1t\right)x+\frac{\lambda_1^2}{4\lambda_0}x^2,\newline
q=\lambda_2t+\lambda_3t^2+\left(\frac{\lambda_1\lambda_2^2-2\lambda_3}{2\lambda_0\lambda_2}+
\frac{\lambda_3\lambda_1}{\lambda_0}t\right)x+\frac{\lambda_1^2\lambda_3}{4\lambda_0^2}x^2$\\ 
& \texttt{Case 3:} &
$p=\lambda_0t+\lambda_1x+\lambda_2x^2+\lambda_3x^3, \quad q=\frac{\lambda_0\lambda_4}{\lambda_3}t+
\frac{\left(\lambda_0\lambda_1\lambda_4+\lambda_3\right)}{\lambda_0\lambda_3}x+
\frac{\lambda_2\lambda_4}{\lambda_3}x^2+\lambda_4x^3$\\
& \texttt{Case 4:} &
$p=\lambda_0t+\lambda_1x+\lambda_2x^2,\quad q=\frac{\lambda_0\lambda_3}{\lambda_2}t+\frac{\left(\lambda_0
\lambda_1\lambda_3+\lambda_2\right)}{\lambda_0\lambda_2}x+\lambda_3x^2$\\
& \texttt{Case 5:} & $p=\lambda_0x,\quad q=-\frac{1}{\lambda_0}t+\lambda_1x+\lambda_2x^2+\lambda_3x^3$
\\[0.2cm] \cline{1-3}
\texttt{$4$} & \texttt{Case 1:} &
$p=\lambda_0t+\lambda_1t^2+\left(\lambda_2+\frac{2\gamma\lambda_1}{\beta}t\right)x+
\left(\frac{e_1}{2\beta^2\lambda_1}+\lambda_3t\right)x^2+\frac{\gamma\lambda_3}{\beta}x^3+
\frac{\lambda_3^2}{4\lambda_1}x^4,\newline
q=\lambda_4t+\lambda_5t^2+\left(\frac{\zeta}{\lambda_1\lambda_0}+
\frac{2\gamma\lambda_5}{\beta}t\right)x+\left(\frac{e_2}{2\beta^2\lambda_1}+\frac{\lambda_3\lambda_5}
{\lambda_1}t\right)x^2+\frac{\gamma\lambda_3\lambda_5}{\beta\lambda_1}x^3+\frac{\lambda_3^2\lambda_5}
{4\lambda_1^2}x^4,\newline \text{where: }\beta=\lambda_0(\lambda_0\lambda_5-\lambda_1\lambda_4),\,\gamma=
(\lambda_0\lambda_5-\lambda_1\lambda_4)(\lambda_2-\lambda_3)-\lambda_1,\,
\zeta=\lambda_0\lambda_2\lambda_5-\gamma,
\newline
e_1=\beta^2\lambda_0\lambda_2-\lambda_0^2(\gamma+\lambda_1)\beta+2\gamma^2\lambda_1^2,\,
e_2=\beta^2\lambda_2\lambda_4-\lambda_0\lambda_4(\gamma+\lambda_1)\beta+2\gamma^2\lambda_1\lambda_5$\\
& \texttt{Case 2:} & $p=\lambda_0t^2+\left(\frac{\lambda_2\lambda_3-1}{\lambda_3}+\lambda_1t\right)x+
\left(\frac{\lambda_1^2}{4\lambda_0}+\lambda_2t\right)x^2+\frac{\lambda_1\lambda_2}{2\lambda_0}x^3+
\frac{\lambda_2^2}{4\lambda_0}x^4,\newline
q=\lambda_3t+\lambda_4t^2+\left(\frac{\beta}{2\lambda_0\lambda_3}+\frac{\lambda_1\lambda_4}{\lambda_0}t
\right)x+\left(\frac{\gamma}{4\lambda_0^2}+\frac{\lambda_2\lambda_4}{\lambda_0}t\right)x^2+\frac{\lambda_1
\lambda_2\lambda_4}{2\lambda_0^2}x^3+\frac{\lambda_2^2\lambda_4}{4\lambda_0^2}x^4,\newline \text{where: } \beta=\lambda_1\lambda_3^2+2\lambda_2\lambda_3\lambda_4-2\lambda_4,\,
\gamma=2\lambda_0\lambda_2\lambda_3+\lambda_1^2\lambda_4$\\
& \texttt{Case 3:} & $p=\lambda_0t+\lambda_1x+\lambda_2x^2+\lambda_3x^3+\lambda_4x^4,\newline q=\frac{\lambda_0\lambda_5}{\lambda_4}t+\frac{\lambda_0\lambda_1\lambda_5+\lambda_4}{\lambda_0\lambda_4}x+
\frac{\lambda_2\lambda_5}{\lambda_4}x^2+\frac{\lambda_3\lambda_5}{\lambda_4}x^3+\lambda_5x^4$\\
& \texttt{Case 4:} & $p=\lambda_0t+\lambda_1x+\lambda_2x^2+\lambda_3x^3,\quad
q=\frac{\lambda_0\lambda_4}{\lambda_3}t+\frac{\lambda_0\lambda_1\lambda_4+\lambda_3}{\lambda_0\lambda_3}x+
\frac{\lambda_2\lambda_4}{\lambda_3}x^2+\lambda_4x^3$\\ 
& \texttt{Case 5:} & $p=\lambda_0t+\lambda_1x+\lambda_2x^2,\newline
q=-\frac{\beta}{\lambda_2^3}t+\frac{\lambda_0^2\lambda_4}{\lambda_2^2}t^2+
\left(\frac{\gamma}{\lambda_0\lambda_2^3}+\frac{\lambda_1\zeta}
{\lambda_2^2}t\right)x+\left(\lambda_3+\frac{\zeta}{\lambda_2}t\right)x^2+\frac{2\lambda_1
\lambda_4}{\lambda_2}x^3+\lambda_4x^4,\newline
\text{where: }\beta=\lambda_0(\lambda_1^2\lambda_4-\lambda_2^2\lambda_3),\,\gamma=\lambda_2^3+
(2\lambda_0^2\lambda_4+\lambda_0\lambda_1\lambda_3)\lambda_2^2-\lambda_0\lambda_1^3\lambda_4,\,
\zeta=2\lambda_0\lambda_4$\\
& \texttt{Case 6:} & $p=\lambda_0t+\lambda_1x,\quad
q=\lambda_2t+\lambda_3t^2+\left(\frac{\lambda_1\lambda_2+1}{\lambda_0}+\frac{2\lambda_1\lambda_3}
{\lambda_0}t\right)x+\frac{\lambda_1^2\lambda_3}{\lambda_0^2}x^2$\\ 
& \texttt{Case 7:} & $p=\lambda_0x,\quad q=-\frac{1}{\lambda_0}t+\lambda_1x+\lambda_2x^2+\lambda_3x^3+\lambda_4x^4$\\[0.1cm] \cline{1-3}				
\end{tabular}}
\caption{Test table for Dixmier polynomials with $n$ from $2$ to $4$ and $m=2$}\label{TestTable:m2}
\end{table}
\end{center}

\begin{center}
\begin{table}[htb]
\centering\renewcommand{\arraystretch}{1.9}
\begin{tabular}{|lp{13.3cm}|}\hline %
\textbf{Cases} & \textbf{Polynomials $p$ and $q$}\\ \hline \hline		
\texttt{Case 1:} & $p=\lambda_0t+\lambda_1x+\lambda_2x^2+\lambda_3x^3+\lambda_4x^4+\lambda_5x^5,\newline q=\lambda_6t+\frac{1+\lambda_1\lambda_6}{\lambda_0}x+\frac{\lambda_2\lambda_6}{\lambda_0}x^2+			\frac{\lambda_3\lambda_6}{\lambda_0}x^3+\frac{\lambda_4\lambda_6}{\lambda_0}x^4+\frac{\lambda_5\lambda_6}
{\lambda_0}x^5$ \\ \hline
\texttt{Case 2:} & $p=\lambda_0t+\lambda_1x+\lambda_2x^2,\newline
q=\lambda_3t+\lambda_4t^2+\left(\frac{\beta}{\lambda_0}+\frac{2\lambda_1\lambda_4}{\lambda_0}t\right)x+
\left(\frac{\gamma}{\lambda_0^2}+\frac{2\lambda_2\lambda_4}{\lambda_0}t\right)x^2+\frac{2\lambda_1\lambda_2
\lambda_4}{\lambda_0^2}x^3+\frac{\lambda_2^2\lambda_4}{\lambda_0^2}x^4,\newline
\text{where }\beta=\lambda_1\lambda_3+2\lambda_2\lambda_4+1,$\,
$\gamma=\lambda_0\lambda_2\lambda_3+\lambda_1^2\lambda_4$ \\ \hline
\texttt{Case 3:} & $p=\lambda_0x,\quad q=-\frac{1}{\lambda_0}t+\lambda_1x+\lambda_2x^2+\lambda_3x^3+
\lambda_4x^4+ \lambda_5x^5$ \\ \hline
\texttt{Case 4:} & $p=\lambda_0t+\lambda_1t^2+\left(\frac{\beta\gamma+2\lambda_1^2}{2\beta\lambda_1}+\lambda_2t\right)x+
\left(\frac{2\lambda_3\lambda_0+\lambda_2^2}{4\lambda_1}+\lambda_3t\right)x^2+\frac{\lambda_2\lambda_3}
{2\lambda_1}x^3+\frac{\lambda_3^2}{4\lambda_1}x^4,\newline
q=\lambda_4t+\lambda_5t^2+\left(\frac{\zeta\beta+2\lambda_1\lambda_5}{2\beta\lambda_1}+\frac{\lambda_2
\lambda_5}{\lambda_1}t\right)x+\left(\frac{2\lambda_1\lambda_3\lambda_4+\lambda_2^2\lambda_5}{4
\lambda_1^2}+\frac{\lambda_3\lambda_5}{\lambda_1}t\right)x^2+\frac{\lambda_2\lambda_3\lambda_5}{2
\lambda_1^2}x^3+\frac{\lambda_3^2\lambda_5}{4\lambda_1^2}x^4,\newline\text{where: }
\beta=\lambda_0\lambda_5-\lambda_1\lambda_4,\,\gamma=\lambda_0\lambda_2+2\lambda_1\lambda_3,\, \zeta=\lambda_2\lambda_4+2\lambda_3\lambda_5$ \\ \hline
\texttt{Case 5:} & $p=\lambda_0t+\lambda_1t^2+\lambda_2x,\quad q=\lambda_3t+\frac{\lambda_1(1+\lambda_2\lambda_3)}{\lambda_0\lambda_2}t^2+
\frac{1+\lambda_2\lambda_3}{\lambda_0}x$ \\ \hline
\end{tabular}
\caption{Test table for Dixmier polynomials with $n=5$ and $m=2$}\label{TestTable:n5m2}
\end{table}
\end{center}

\begin{center}
\begin{table}[htb]
\centering\renewcommand{\arraystretch}{1.9}
\begin{tabular}{|lp{13.8cm}|}\hline %
\textbf{Cases} & \textbf{Polynomials $p$ and $q$}\\ \hline \hline 
\texttt{Case 1:} & $p=\lambda_0t+\lambda_1x+\lambda_2x^2+\lambda_3x^3+\lambda_4x^4+\lambda_5x^5+\lambda_6x^6,\newline q=\lambda_7t+\frac{1+\lambda_1\lambda_7}{\lambda_0}x+\frac{\lambda_2\lambda_7}{\lambda_0}x^2+
\frac{\lambda_3\lambda_7}{\lambda_0}x^3+\frac{\lambda_4\lambda_7}{\lambda_0}x^4+
\frac{\lambda_5\lambda_7}{\lambda_0}x^5+\frac{\lambda_6\lambda_7}{\lambda_0}x^6$ \\[1ex] \hline
\texttt{Case 2:} & $p=\lambda_0t+\lambda_1x+\lambda_2x^2+\lambda_3x^3,\newline
q=\lambda_4t+\lambda_5t^2+\left(\frac{\beta}{\lambda_0}+
\upsilon\lambda_1t\right)x+\left(\frac{\gamma}{\lambda_0^2}+\upsilon\lambda_2t\right)x^2+
\left(\frac{\zeta}{\lambda_0^2}+\upsilon\lambda_3t\right)x^3+
\frac{\eta}{\lambda_0^2}x^4+\frac{\vartheta}{\lambda_0^2}x^5+\frac{\lambda_3^2\lambda_5}{\lambda_0^2}x^6,
\newline\text{where: } \beta=\lambda_1\lambda_4+2\lambda_2\lambda_5+1,\,
\gamma=\lambda_0\lambda_2\lambda_4+3\lambda_0\lambda_3\lambda_5+\lambda_1^2\lambda_5,\,
\zeta=\lambda_0\lambda_3\lambda_4+2\lambda_1\lambda_2\lambda_5,\newline
\eta=2\lambda_1\lambda_3\lambda_5+\lambda_2^2\lambda_5,\,\vartheta=2\lambda_2\lambda_3\lambda_5,\,
\upsilon=2\lambda_5/\lambda_0$ \\\hline
\texttt{Case 3:} & $p=\lambda_0x,\quad q=-\frac{1}{\lambda_0}t+\lambda_1x+\lambda_2x^2+\lambda_3x^3+
\lambda_4x^4+ \lambda_5x^5+\lambda_6x^6$ \\ \cline{1-2}
\texttt{Case 4:} & $p=\lambda_0t+\lambda_1t^2+\left(\frac{e_1}{\beta\lambda_2^3}-\frac{e_4\gamma}{\lambda_2^2}t\right)x+
\left(\frac{e_2}{2\lambda_2^6}+e_4\lambda_4t\right)x^2+\left(\frac{e_3}{2\lambda_1\lambda_2^4}+\lambda_2t
\right)x^3+\lambda_3x^4+\lambda_4x^5+\frac{\lambda_2^2}{4\lambda_1}x^6,\newline q=\lambda_5t+\lambda_6t^2+\left(\frac{e_5}{\beta\lambda_2^3}+e_{13}t\right)x+
\left(\frac{e_6}{2\lambda_1\lambda_2^6}+e_7t\right)x^2+\left(\frac{e_8}{2\lambda_1\lambda_2^4}+e_9t\right)x^3
+e_{10}x^4+e_{11}x^5+e_{12}x^6,\newline
\text{where: }\beta=\lambda_0\lambda_6-\lambda_1\lambda_5,\,
\gamma=\lambda_1\lambda_4^2-\lambda_2^2\lambda_3,\,e_1=2\beta\lambda_1\lambda_2^2\lambda_4+\lambda_1
\lambda_2^3-\beta\gamma\lambda_0,\,e_2=2\lambda_0\lambda_2^5\lambda_4+3\lambda_2^7+2\gamma^2\lambda_1,\,
e_3=\lambda_0\lambda_2^5-4\gamma\lambda_1^2\lambda_4,\,e_4=2\lambda_1/\lambda_2,\,e_5=(2\lambda_2^2\lambda_4
\lambda_6-\gamma\lambda_5)\beta+\lambda_2^3\lambda_6,\,e_6=2\lambda_5\lambda_1\lambda_2^5\lambda_4+3
\lambda_2^7\lambda_6+2\gamma^2\lambda_1\lambda_6,\,e_7=2\lambda_4\lambda_6/\lambda_2
,\,e_8=\lambda_2^5\lambda_5-4\gamma\lambda_1\lambda_4\lambda_6,\,e_9=\lambda_2\lambda_6/\lambda_1
,\,e_{10}=\lambda_3\lambda_6/\lambda_1,\,e_{11}=\lambda_4\lambda_6/\lambda_1,\,
e_{12}=\lambda_6\lambda_2^2/(4\lambda_1^2),\,e_{13}=-2\gamma\lambda_6/\lambda_2^3$ \\ \hline
\texttt{Case 5:} & $p=\lambda_0t+\lambda_1t^2+\left(\frac{\beta\lambda_0\lambda_2+2\lambda_1^2}{2\beta
\lambda_1}+\lambda_2t\right)x+\frac{6\lambda_1\lambda_3+\lambda_2^2}{4\lambda_1}x^2+
\left(\frac{\lambda_0\lambda_3}{2\lambda_1}+\lambda_3t\right)x^3+\frac{\lambda_3\lambda_2}{2\lambda_1}x^4
+\frac{\lambda_3^2}{4\lambda_1}x^6,\newline
q=\lambda_4t+\lambda_5t^2+\left(\frac{\gamma}{2\beta\lambda_1}+\frac{\lambda_2\lambda_5}{\lambda_1}t\right)x+
\frac{(6\lambda_1\lambda_3+\lambda_2^2)\lambda_5}{4\lambda_1^2}x^2+\left(\frac{\lambda_3\lambda_4}
{2\lambda_1}+\frac{\lambda_3\lambda_5}{\lambda_1}t\right)x^3+\frac{\lambda_2\lambda_3\lambda_5}{2\lambda_1^2}
x^4+\frac{\lambda_5\lambda_3^2}{4\lambda_1^2}x^6,\newline\text{where: }
\beta=\lambda_0\lambda_5-\lambda_1\lambda_4,\,
\gamma=\beta\lambda_2\lambda_4+2\lambda_5\lambda_1,\,$ \\ \cline{1-2}
\texttt{Case 6:} & $p=\lambda_0t+\lambda_1t^2+\left(\frac{(\lambda_2\lambda_0+2\lambda_1\lambda_3)\beta+2
\lambda_1^2}{2\beta\lambda_1}+\lambda_2t\right)x+\left(\frac{2\lambda_3\lambda_0+\lambda_2^2}
{4\lambda_1}+\lambda_3t\right)x^2+\frac{\lambda_2\lambda_3}{2\lambda_1}x^3+\frac{\lambda_3^2}{4\lambda_1}
x^4,\newline q=\lambda_4t+\lambda_5t^2+\left(\frac{2\beta\lambda_3\lambda_5+\gamma}{2\beta\lambda_1}+
\frac{\lambda_2\lambda_5}{\lambda_1}t\right)x+\left(\frac{2\lambda_1\lambda_4\lambda_3+\lambda_5
\lambda_2^2}{4\lambda_1^2}+\frac{\lambda_3\lambda_5}{\lambda_1}t\right)x^2+\frac{\lambda_2\lambda_3
\lambda_5}{2\lambda_1^2}x^3+\frac{\lambda_3^2\lambda_5}{4\lambda_1^2}x^4,\newline\text{where: }
\beta=\lambda_0\lambda_5-\lambda_1\lambda_4,\,
\gamma=\beta\lambda_2\lambda_4+2\lambda_5\lambda_1,\,$ \\[1ex] \cline{1-2}
\texttt{Case 7:} & $p=\lambda_0t+\lambda_1t^2+\lambda_2x,\quad q=\lambda_3t+\frac{\lambda_1(\lambda_2\lambda_3+1)}{\lambda_0\lambda_2}t^2+\frac{\lambda_2\lambda_3+1}
{\lambda_0}x$\\[1ex] \cline{1-2}
\texttt{Case 8:} & $p=\lambda_0t+\lambda_1t^2+\left(\frac{\zeta\gamma+2\lambda_1^2}{2\zeta\lambda_1}
+\lambda_2t\right)x+\left(\frac{2\lambda_3\lambda_0+\lambda_2^2}{4\lambda_1}+\lambda_3t\right)x^2+
\frac{\lambda_2\lambda_3}{2\lambda_1}x^3+\frac{\lambda_3^2}{4\lambda_1}x^4,\newline
q=\lambda_4t+\lambda_5t^2+\left(\frac{e_1}{\zeta\lambda_4}+\gamma\lambda_1t\right)x+\left(\frac{e_2}
{\zeta^2}+\gamma\lambda_2t\right)x^2+\left(\frac{e_3}{\zeta\lambda_0}+\gamma\lambda_3t\right)x^3+
\frac{e_4}{\zeta^2}x^4+\frac{e_5}{\zeta\lambda_0}x^5+\frac{e_6}{\zeta^2}x^6,
\newline\text{where: }e_1=\lambda_1\lambda_4^2+2\lambda_2\lambda_5\lambda_4-2\lambda_5,\,
e_2=(-\lambda_2\lambda_0+3\gamma)\lambda_4+\lambda_1^2\lambda_5+6\zeta\lambda_3,\,
e_3=\lambda_3\lambda_0\lambda_4+\lambda_1\lambda_2\lambda_5,\,e_4=\lambda_5\lambda_2^2+(-\lambda_1
\lambda_4-\zeta)\lambda_2+\gamma\lambda_5,\,e_5=\lambda_2\lambda_3\lambda_5,\,e_6=\lambda_3^2\lambda_5,\,
\gamma=\lambda_5/\lambda_0,\,\zeta=2\lambda_0
$ \\ \cline{1-2}		
\end{tabular}
\caption{Test table for Dixmier polynomials with $n=6$ and $m=2$}\label{TestTable:n6m2}
\end{table}
\end{center}

\begin{center}
\begin{table}[htb]
\centering\renewcommand{\arraystretch}{1.9}
\begin{tabular}{|lp{13.8cm}|}\hline %
\textbf{Cases} & \textbf{Polynomials $p$ and $q$}\\ \hline \hline			
\texttt{Case 1:} & $p=\lambda_0t+\lambda_1x+\lambda_2x^2+\lambda_3x^3+\lambda_4x^4+\lambda_5x^5+\lambda_6x^6+\lambda_7x^7,
\newline q= \frac{\lambda_0\lambda_8}{\lambda_7}t+\frac{\lambda_0\lambda_1\lambda_8+\lambda_7}
{\lambda_0\lambda_7}x+\frac{\lambda_2\lambda_8}{\lambda_7}x^2+\frac{\lambda_3\lambda_8}{\lambda_7}x^3+
\frac{\lambda_4\lambda_8}{\lambda_7}x^4+\frac{\lambda_5\lambda_8}{\lambda_7}x^5+\frac{\lambda_6\lambda_8}
{\lambda_7}x^6+\lambda_8x^7$ \\ \cline{1-2}
\texttt{Case 2:} & $p=\lambda_0t+\lambda_1x+\lambda_2x^2+\lambda_3x^3+\lambda_4x^4+\lambda_5x^5+\lambda_6x^6,
\newline q= \frac{\lambda_0\lambda_7}{\lambda_6}t+\frac{\lambda_0\lambda_1\lambda_7+\lambda_6}
{\lambda_0\lambda_6}x+\frac{\lambda_2\lambda_7}{\lambda_6}x^2+\frac{\lambda_3\lambda_7}{\lambda_6}x^3+
\frac{\lambda_4\lambda_7}{\lambda_6}x^4+\frac{\lambda_5\lambda_7}{\lambda_6}x^5+\lambda_7x^6$ \\ \cline{1-2}
\texttt{Case 3:} & $p=\lambda_0t+\lambda_1x+\lambda_2x^2+\lambda_3x^3,
\newline q=\lambda_4t+\lambda_5t^2+\left(\frac{e_1}{\lambda_0}+\frac{2\lambda_5\lambda_1}
{\lambda_0}t\right)x+\left(\frac{e_2}{\lambda_0^2}+\frac{2\lambda_2\lambda_5}{\lambda_0}t\right)x^2+
\left(\frac{e_3}{\lambda_0^2}+\frac{2\lambda_3\lambda_5}{\lambda_0}t\right)x^3+\frac{e_4}{\lambda_0^2}x^4
+\frac{e_5}{\lambda_0^2}x^5+\frac{e_6}{\lambda_0^2}x^6,\newline\text{where }
e_1=\lambda_1\lambda_4+2\lambda_2\lambda_5+1,\,e_2=\lambda_0\lambda_2\lambda_4+3\lambda_0\lambda_3
\lambda_5+\lambda_1^2\lambda_5,\,e_3=\lambda_0\lambda_3\lambda_4+2\lambda_1\lambda_2\lambda_5,\,
e_4=\lambda_5(2\lambda_1\lambda_3+\lambda_2^2),\,e_5=2\lambda_2\lambda_3\lambda_5,\,e_6=\lambda_3^2
\lambda_5$ \\ \cline{1-2}
\texttt{Case 4:} & $p=\lambda_0x,\quad q=-\frac{1}{\lambda_0}t+\lambda_1x+\lambda_2x^2+\lambda_3x^3+
\lambda_4x^4+\lambda_5x^5+\lambda_6x^6+\lambda_7x^7$ \\ \cline{1-2}
\texttt{Case 5:} & $p=\lambda_0t+\lambda_1x+\lambda_2x^2+\lambda_3x^3+\lambda_4x^4+\lambda_5x^5,
\newline q= \frac{\lambda_0\lambda_6}{\lambda_5}t+\frac{\lambda_0\lambda_1\lambda_6+\lambda_5}
{\lambda_0\lambda_5}x+\frac{\lambda_2\lambda_6}{\lambda_5}x^2+\frac{\lambda_3\lambda_6}{\lambda_5}x^3+
\frac{\lambda_4\lambda_6}{\lambda_5}x^4+\lambda_6x^5$ \\ \cline{1-2}
\texttt{Case 6:} & $p=\lambda_0t+\lambda_1x+\lambda_2x^2+\lambda_3x^3+\lambda_4x^4,
\newline q= \frac{\lambda_0\lambda_5}{\lambda_4}t+\frac{\lambda_0\lambda_1\lambda_5+\lambda_4}
{\lambda_0\lambda_4}x+\frac{\lambda_2\lambda_5}{\lambda_4}x^2+\frac{\lambda_3\lambda_5}{\lambda_4}x^3+
\lambda_5x^4$ \\ \cline{1-2}
\texttt{Case 7:} & $p=\lambda_0t+\lambda_1t^2+\left(\frac{e_1}{\gamma\lambda_2^3}-\frac{\zeta\beta}{\lambda_2^3}t\right)
x+\left(\frac{e_2}{2\lambda_2^6}+\frac{\zeta\lambda_4}{\lambda_2}t\right)x^2+\left(\frac{e_3}
{\zeta\lambda_2^4}+\lambda_2t\right)x^3+\lambda_3x^4+\lambda_4x^5+\frac{\lambda_2^2}{4\lambda_1}x^6,
\newline q=\lambda_5t+\lambda_6t^2+\left(\frac{e_4}{\gamma\lambda_2^3}-e_7t\right)x
+\left(\frac{e_5}{2\lambda_1\lambda_2^6}+e_8t\right)x^2+\left(\frac{e_6}{2\lambda_1\lambda_2^4}+e_9t
\right)x^3+e_{10}x^4+e_{11}x^5+\frac{\lambda_6
\lambda_2^2}{4\lambda_1^2}x^6,\newline
\text{where }\beta=\lambda_1\lambda_4^2-\lambda_2^2\lambda_3,\,\gamma=\lambda_0\lambda_6-\lambda_1
\lambda_5,\,\zeta=2\lambda_1,\,e_1=2\gamma\lambda_1\lambda_2^2\lambda_4+\lambda_1\lambda_2^3-\beta\gamma
\lambda_0,\,e_2=2\lambda_0\lambda_2^5\lambda_4+3\lambda_2^7+2\beta^2\lambda_1,\,e_3=\lambda_0\lambda_2^5
-4\beta\lambda_1^2\lambda_4,\,e_4=(2\lambda_4\lambda_2^2\lambda_6-\beta\lambda_5)\gamma+\lambda_2^3
\lambda_6,\,e_5=2\lambda_1\lambda_2^5\lambda_4\lambda_5+3\lambda_2^7\lambda_6+2\beta^2\lambda_1\lambda_6,
\,e_6=\lambda_2^5\lambda_5-4\beta\lambda_1\lambda_4\lambda_6,\,e_7=2\beta\lambda_6/\lambda_2^3,\,
e_8=2\lambda_4\lambda_6/\lambda_2,\,e_9=\lambda_2\lambda_6/\lambda_1,\,e_{10}=\lambda_3\lambda_6/\lambda_1,
\,e_{11}=\lambda_4\lambda_6/\lambda_1
$ \\ \cline{1-2}
\texttt{Case 8:} & $p=\lambda_0t+\lambda_1t^2+\left(\frac{\gamma}{2\beta\lambda_1}+\lambda_2t\right)x+\frac{\eta}{4\lambda_1}
x^2+\left(\frac{\lambda_3\lambda_0}{2\lambda_1}+\lambda_3t\right)x^3+\frac{\lambda_2\lambda_3}{2\lambda_1}
x^4+\frac{\lambda_3^2}{4\lambda_1}x^6,\newline
q=\lambda_4t+\lambda_5t^2+\left(\frac{\zeta}{2\beta\lambda_1}+\frac{\lambda_2\lambda_5}{\lambda_1}t\right)
x+\frac{\eta\lambda_5}{4\lambda_1^2}x^2+\left(\frac{\lambda_3\lambda_4}{2\lambda_1}+\frac{\lambda_3
\lambda_5}{\lambda_1}t\right)x^3+\frac{\lambda_2\lambda_3\lambda_5}{2\lambda_1^2}x^4+\frac{\lambda_3^2
\lambda_5}{4\lambda_1^2}x^6,\newline
\text{where }\beta=\lambda_0\lambda_5-\lambda_1\lambda_4,\,\gamma=\beta\lambda_0\lambda_2+2\lambda_1^2,\,
\zeta=\beta\lambda_2\lambda_4+2\lambda_1\lambda_5,\,\eta=6\lambda_1\lambda_3+\lambda_2^2$ \\ \cline{1-2}
\texttt{Case 9:} &			
$p=\lambda_0t+\lambda_1t^2+\left(\frac{\gamma}{2\beta\lambda_1}+\lambda_2t\right)x+
\left(\frac{2\lambda_3\lambda_0+\lambda_2^2}{4\lambda_1}+\lambda_3t
\right)x^2+\frac{\lambda_2\lambda_3}{2\lambda_1}x^3+\frac{\lambda_3^2}{4\lambda_1}x^4,\newline		
q=\lambda_4t+\lambda_5t^2+\left(\frac{\zeta}{2\beta\lambda_1}+\frac{\lambda_2\lambda_5}{\lambda_1}t\right)
x+\left(\frac{2\lambda_1\lambda_4\lambda_3+\lambda_5\lambda_2^2}{4\lambda_1^2}+\frac{\lambda_3\lambda_5}
{\lambda_1}t\right)x^2+\frac{\lambda_2\lambda_3\lambda_5}{2\lambda_1^2}x^3+\frac{\lambda_3^2\lambda_5}
{4\lambda_1^2}x^4,\newline
\text{where }\beta=\lambda_1\lambda_4^2-\lambda_2^2\lambda_3,\,
\gamma=(\lambda_2\lambda_0+2\lambda_1\lambda_3)\beta+2\lambda_1^2,\,
\zeta=\beta\lambda_2\lambda_4+2\beta\lambda_3\lambda_5+2\lambda_1\lambda_5$ \\ \cline{1-2}
\texttt{Case 10:} & $p=\lambda_0t+\lambda_1t^2+\lambda_2x,\quad
q= \lambda_3t+\frac{\lambda_1(\lambda_2\lambda_3+1)}{\lambda_0\lambda_2}t^2+\frac{\lambda_2
\lambda_3+1}{\lambda_0}x$ \\ \cline{1-2}
\texttt{Case 11:} & $p=\lambda_0t^2+\left(\frac{e_1}{\lambda_4}+\lambda_1t\right)x+
\left(\frac{e_2}{4\lambda_0}+\lambda_2t\right)x^2+\left(\frac{e_3}{2\lambda_0}+\lambda_3t\right)x^3+
\frac{e_4}{4\lambda_0}x^4+e_5x^5+e_6x^6,\newline
q=\lambda_4t+\lambda_5t^2+\left(\frac{e_7}{2\lambda_4\lambda_0}+e_{13}t\right)x+\left(\frac{e_8}
{4\lambda_0^2}+e_{14}t\right)x^2+\left(\frac{e_9}{2\lambda_0^2}+e_{14}t\right)x^3+\frac{e_{10}}{4
\lambda_0^2}x^4+e_{11}x^5+e_{12}x^6,\newline
\text{where }e_1=\lambda_2\lambda_4-1,\,e_2=6\lambda_0\lambda_3+\lambda_1^2,\,e_3=\lambda_1\lambda_2,\,
e_4=2\lambda_1\lambda_3+\lambda_2^2,\,e_5=\lambda_2\lambda_3/(2\lambda_0),\,e_6=\lambda_3^2/(4\lambda_0),\,
e_7=\lambda_1\lambda_4^2+2\lambda_2\lambda_4\lambda_5-2\lambda_5,\,e_8=2\lambda_0\lambda_2\lambda_4+6
\lambda_0\lambda_3\lambda_5+\lambda_1^2\lambda_5,\,e_9=\lambda_0\lambda_3\lambda_4+\lambda_1\lambda_2
\lambda_5,\,e_{10}=\lambda_5(2\lambda_1\lambda_3+\lambda_2^2),\,e_{11}=\lambda_2\lambda_3
\lambda_5/(2\lambda_0^2),\,e_{12}=\lambda_5\lambda_3^2/(4\lambda_0^2),\,e_{13}=\lambda_1\lambda_5
\lambda_0,\,e_{14}=\lambda_2\lambda_5/\lambda_0,\,e_{15}=\lambda_3\lambda_5/\lambda_0\,$ \\[1ex] \hline
\end{tabular}
\caption{Test table for Dixmier polynomials with $n=7$ and $m=2$}\label{TestTable:n7m2}
\end{table}
\end{center}

\begin{center}
\begin{table}[htb]
\centering\renewcommand{\arraystretch}{1.9}
\begin{tabular}{|lp{13.8cm}|}\hline %
\textbf{Cases} & \textbf{Polynomials $p$ and $q$}\\ \hline \hline			
\texttt{Case 1:} & $p=\lambda_0t+\lambda_1x+\lambda_2x^2+\lambda_3x^3+\lambda_4x^4+\lambda_5x^5,
\newline q=\frac{\lambda_0\lambda_6}{\lambda_5}t+\frac{\lambda_0\lambda_1\lambda_6+\lambda_5}
{\lambda_0\lambda_5}x+\frac{\lambda_2\lambda_6}{\lambda_5}x^2+\frac{\lambda_3\lambda_6}{\lambda_5}x^3+
\frac{\lambda_4\lambda_6}{\lambda_5}x^4+\lambda_6x^5$ \\ \cline{1-2}
\texttt{Case 2:} & $p=\lambda_0t+\lambda_1t^2+\left(\frac{\lambda_0^2\lambda_2\lambda_4-\lambda_0\lambda_1
\lambda_2\lambda_3+2\lambda_1^2}{2\lambda_1(\lambda_0\lambda_4-\lambda_1\lambda_3)}+\lambda_2t\right)x+
\frac{\lambda_2^2}{4\lambda_1}x^2,\newline q=\lambda_3t+\lambda_4t^2+\left(\frac{(-\lambda_2\lambda_3^2+
2\lambda_4)\lambda_1+\lambda_4\lambda_0\lambda_3\lambda_2}{2\lambda_1(\lambda_0\lambda_4-\lambda_1
\lambda_3)}+\frac{\lambda_2\lambda_4}{\lambda_1}t\right)x+\frac{\lambda_2^2\lambda_4}{4\lambda_1^2}x^2$ \\ \cline{1-2}
\texttt{Case 3:} & $p=\lambda_0t+\lambda_1x+\lambda_2x^2+\lambda_3x^3,\quad q=\frac{\lambda_0\lambda_4}{\lambda_3}t+\frac{\lambda_0\lambda_1\lambda_4+\lambda_3}
{\lambda_0\lambda_3}x+\frac{\lambda_2\lambda_4}{\lambda_3}x^2+\lambda_4x^3$ \\ \cline{1-2}
\texttt{Case 4:} & $p=\lambda_0t+\lambda_1x+\lambda_2x^2,\newline q=\tfrac{e_1}{\lambda_2^3}t+\tfrac{\lambda_0^2\lambda_4}
{\lambda_2^2}t^2+\left(\tfrac{e_2}{\lambda_0\lambda_2^3}+\tfrac{2\lambda_0\lambda_1\lambda_4}{\lambda_2^2}t
\right)x+\left(\lambda_3+\tfrac{2\lambda_0\lambda_4}{\lambda_2}t\right)x^2+\tfrac{2\lambda_1\lambda_4}
{\lambda_2}x^3+\lambda_4x^4,\newline
\text{where }e_1=-\lambda_0(\lambda_1^2\lambda_4-\lambda_2^2\lambda_3),\,e_2=\lambda_2^3+(2\lambda_0^2
\lambda_4+\lambda_0\lambda_1\lambda_3)\lambda_2^2-\lambda_0\lambda_1^3\lambda_4$ \\ \cline{1-2}
\texttt{Case 5:} & $p=\lambda_0t+\lambda_1x,
\newline q=\lambda_2t+\lambda_3t^2+\lambda_4t^3+\left(\frac{\lambda_1\lambda_2+1}{\lambda_0}+\frac{2
\lambda_1\lambda_3}{\lambda_0}t+\frac{3\lambda_1\lambda_4}{\lambda_0}t^2\right)x+\left(\frac{\lambda_1^2
\lambda_3}{\lambda_0^2}+\frac{3\lambda_1^2\lambda_4}{\lambda_0^2}t\right)x^2+\frac{\lambda_1^3\lambda_4}
{\lambda_0^3}x^3$ \\ \cline{1-2}
\texttt{Case 6:} & $p=\lambda_0x,\quad q=-\frac{1}{\lambda_0}t+\lambda_1x+\lambda_2x^2+\lambda_3x^3+\lambda_4x^4+\lambda_5x^5$ \\ \cline{1-2}
\texttt{Case 7:} & $p=\lambda_0t+\lambda_1x+\lambda_2x^2+\lambda_3x^3+\lambda_4x^4+\lambda_5x^5,
\newline q= \frac{\lambda_0\lambda_6}{\lambda_5}t+\frac{\lambda_0\lambda_1\lambda_6+\lambda_5}
{\lambda_0\lambda_5}x+\frac{\lambda_2\lambda_6}{\lambda_5}x^2+\frac{\lambda_3\lambda_6}{\lambda_5}x^3+
\frac{\lambda_4\lambda_6}{\lambda_5}x^4+\lambda_6x^5$ \\ \cline{1-2}
\texttt{Case 8:} & $p=\frac{-\lambda_1^2\lambda_4+\lambda_2\lambda_3^2}{2\lambda_3\lambda_4}t+\frac{
\lambda_3^2}{4\lambda_4}t^2+(\lambda_0+\lambda_1t)x+(\lambda_2+\lambda_3t)x^2+\frac{2\lambda_1\lambda_4}
{\lambda_3}x^3+\lambda_4x^4,\newline q=\frac{\zeta}{2\beta\lambda_2\lambda_3\lambda_4}t+\frac{\gamma\lambda_3^2}{4\beta\lambda_2
\lambda_4}t^2+\left(\frac{\eta}{\beta\lambda_2}+\frac{\gamma\lambda_1}{\beta\lambda_2}t\right)x+
\left(\lambda_5+\frac{\gamma\lambda_3}{\beta\lambda_2}t\right)x^2+\frac{2\gamma\lambda_1\lambda_4}
{\beta\lambda_2\lambda_3}x^3+\frac{\gamma\lambda_4}{\beta\lambda_2}x^4,\newline\text{where }
\beta=\lambda_0\lambda_3^3+\lambda_1^3\lambda_4-\lambda_1\lambda_2
\lambda_3^2-\lambda_3^4,\,\gamma=\beta\lambda_5+2\lambda_3^2\lambda_4,\,\zeta=,\lambda_3^2(-2\lambda_3^2
\lambda_4+\gamma)\lambda_2-\gamma\lambda_1^2\lambda_4,\newline\,
\eta=-2\lambda_1\lambda_2\lambda_3\lambda_4+\gamma\lambda_0$\\ \cline{1-2}
\texttt{Case 9:} &			
$p=-\frac{\lambda_0^2}{2\lambda_1}t+\frac{\lambda_1^2}{4\lambda_2}t^2+\left(\frac{(-\lambda_0^3\lambda_2+
\lambda_1^4)\lambda_3-2\lambda_1^2\lambda_2}{\lambda_1^3\lambda_3}+\lambda_0t\right)x+\lambda_1tx^2+
\frac{2\lambda_2\lambda_0}{\lambda_1}x^3+\lambda_2x^4,\newline
q=\frac{-\lambda_0^2\lambda_4+\lambda_1^2\lambda_3}{2\lambda_1\lambda_2}t+\frac{\lambda_1^2\lambda_4}
{4\lambda_2^2}t^2+\left(\frac{\beta}{\lambda_1^3\lambda_2\lambda_3}+\frac{\lambda_0
\lambda_4}{\lambda_2}t\right)x+\left(\lambda_3+\frac{\lambda_1\lambda_4}{\lambda_2}t\right)x^2+\frac{2
\lambda_0\lambda_4}{\lambda_1}x^3+\lambda_4x^4,\newline
\text{where }\beta=\lambda_1^4\lambda_3\lambda_4+\lambda_2(\lambda_0\lambda_3^2-2\lambda_4)
\lambda_1^2-\lambda_0^3\lambda_2\lambda_3\lambda_4$ \\ \cline{1-2}
\texttt{Case 10:} &
$p=\lambda_0t+\lambda_1t^2+\lambda_2t^3+(\lambda_3+\frac{2\gamma\lambda_1}{\beta\lambda_0}t+\frac{3\gamma
\lambda_2}{\beta\lambda_0}t^2)x+(\frac{\gamma^2\lambda_1}{\beta^2\lambda_0^2}+\frac{3\gamma^2\lambda_2}{
\beta^2\lambda_0^2}t)x^2+\frac{\gamma^3\lambda_2}{\beta^3\lambda_0^3}x^3,\newline
q=\lambda_4t+\frac{\lambda_1\lambda_5}{\lambda_2}t^2+\lambda_5t^3+\left(\frac{\lambda_3\lambda_4+1}
{\lambda_0}+\frac{2\gamma\lambda_1\lambda_5}{\beta\lambda_0\lambda_2}t+\frac{3\gamma\lambda_5}{\beta
\lambda_0}t^2\right)x+\left(\frac{\gamma^2\lambda_1\lambda_5}{\beta^2\lambda_0^2\lambda_2}+\frac{3\gamma^2
\lambda_5}{\beta^2\lambda_0^2}t\right)x^2+\frac{\gamma^3\lambda_5}{\beta^3\lambda_0^3}x^3,\newline
\text{where }\beta=\lambda_0\lambda_5-\lambda_2\lambda_4,\,\gamma=\beta\lambda_3-\lambda_2$ \\ \cline{1-2}
\texttt{Case 11:} & $p=\lambda_0t^2+\lambda_1t^3+\left(-\frac{1}{\lambda_3}+\frac{2\lambda_0\lambda_2}{3\lambda_1}t+\lambda_2
t^2\right)x+\left(\frac{\lambda_0\lambda_2^2}{9\lambda_1^2}+\frac{\lambda_2^2}{3\lambda_1}t\right)x^2+
\frac{\lambda_2^3}{27\lambda_1^2}x^3,\newline
q=\lambda_3t+\frac{\lambda_0\lambda_4}{\lambda_1}t^2+\lambda_4t^3+\left(\frac{\lambda_2\lambda_3^2-3
\lambda_4}{3\lambda_1\lambda_3}+\frac{2\lambda_0\lambda_2\lambda_4}{3\lambda_1^2}t+\frac{\lambda_2
\lambda_4}{\lambda_1}t^2\right)x+\left(\frac{\lambda_0\lambda_2^2\lambda_4}{9\lambda_1^3}+\frac{
\lambda_2^2\lambda_4}{3\lambda_1^2}t\right)x^2+\frac{\lambda_2^3\lambda_4}{27\lambda_1^3}x^3$ \\[1ex] \hline
\end{tabular}
\caption{Test table for Dixmier polynomials with $n=5$ and $m=3$}\label{TestTable:n5m3}
\end{table}
\end{center}
\clearpage

\subsubsection*{\texttt{DixmierAutomorphism} and \texttt{inverseAutomorphism} functions}

The \texttt{DixmierAutomorphism} function, a component of the \texttt{DixmierProblem} package, facilitates the random generation of automorphisms for the Weyl algebra $A_1(K).$ In this framework, two Dixmier polynomials $p$ and $q$ are respectively assigned to the variables $t$ and $x.$ Additionally, the function accepts several optional parameters, thereby allowing for enhanced flexibility in its application. \verb|n=n_value|, \verb|m=m_value|, and \verb|parametric::boolean=param_value|, where the parameters $n$ and $m$ play the same role as in the \texttt{DixmierPolynomials} function, with default values $n=3$ and $m=2$, and with \verb|parametric| defaulting to \texttt{false}. Additional options include \verb|view| and \verb|outputMode|. When \verb|parametric| is set to \texttt{false}, the automorphism is generated with concrete values for $p$ and $q$; when it is set to \texttt{true}, a parametric automorphism is produced in this case and the Dixmier polynomials $p$ and $q$ are expressed in terms of parameters $\lambda_i$. These polynomials are determined by the procedure implemented in the \texttt{DixmierPolynomials} function, and the function returns a module (i.e., a Maple object) that encapsulates all information pertinent to the automorphism of $A_1(K)$.

In the example below, we illustrate this process by generating a Dixmier automorphism using the non-parametric option, thereby obtaining explicit values for $p$ and $q$. After extracting these defining polynomials, a symbolic polynomial is constructed, and the automorphism is evaluated on it. Finally, the \texttt{inverseAutomorphism} function is applied to compute the inverse of the initially generated automorphism. This inverse function accepts the same optional parameters as \texttt{DixmierAutomorphism} and additionally requires, as an argument, a Dixmier automorphism.

\begin{example}
To define an automorphism, the symbol \texttt{auto1} is assigned a new automorphism using the command
\texttt{auto1 := DixmierAutomorphism()}. Consequently, the following output is obtained:
\begin{center}
\textit{Dixmier automorphism successfully defined}\\
\textit{auto1 := $\alpha(t, x)$}
\end{center}
This indicates that the Dixmier automorphism has been successfully defined. In other words, there exist Dixmier polynomials $ p $ and $ q $ that define the automorphism. Using the syntax \texttt{auto1:-getP();} and \texttt{auto1:-getQ();}, we obtain the polynomials
\[
p = x, \quad q = -t - 3x - 2x^2 - 3x^3.
\]
Additionally, the syntax \texttt{auto1:-Properties:-qpcommutator} can be used to compute the value of $[q,p]$, which in this case equals $1$.
	
To enable the application of the automorphism, it is assigned to the variable $f$ as follows: \texttt{f:= auto1:-Apply;}. The variable $f$ then represents the evaluation function corresponding to the Dixmier automorphism \texttt{Auto1}.
	
To define a generic polynomial in $A_1(K)$ with maximum degrees $n$ and $m$ for the variables $x$ and $t$, respectively, specifically when $n=2$ and $m=5$ the statement \ \ \texttt{poly := generatorPoly(2,\,5,\, vars=[t, x],\, nameCoefs='p');} is used. This assigns to \textit{poly} a symbolic polynomial in the variables $t$ and $x$, with degree $2$ for $x$ and degree $5$ for $t$ (observe that \texttt{[t, x]} is not a commutator, but it is a list of variables in Maple. Moreover, the degrees of variables are presented in a inverse order):
\begin{align*}
poly &:= p_{0,0}+p_{0,1}t+p_{0,2}t^2+p_{0,3}t^3+p_{0,4}t^4+
p_{0,5}t^5+\left(p_{1,0}+p_{1,1}t+p_{1,2}t^2+p_{1,3}t^3+p_{1,4}t^4+p_{1,5}t^5\right)x\\
&\quad +\left(p_{2,0}+p_{2,1}t+p_{2,2}t^2+p_{2,3}t^3+p_{2,4}t^4+p_{2,5}t^5\right)x^2.
\end{align*}
Evaluating $f$ at \texttt{poly} using \texttt{f(poly)} yields the result:
$p_{0,0}+3p_{2,0}+2p_{2,2}-p_{1,1}+(2p_{2,1}-p_{1,0})t+p_{2,0}t^2+
\big(p_{0,1}+4p_{2,0}+9p_{2,1}+6p_{2,3}-3p_{1,0}-2p_{1,2}+(6p_{2,0}+4p_{2,2}-p_{1,1})t+ p_{2,1}t^2\big)x+\big(p_{0,2}+18p_{2,0}+8p_{2,1}+15p_{2,2}+12p_{2,4}-2p_{1,0}-3p_{1,1}-3p_{1,3}+(6p_{2,3}+ 4p_{2,0}+6p_{2,1}-p_{1,2})t+p_{2,2}t^2\big)x^2+\big(p_{0,3}+12p_{2,0}+24p_{2,1}+12p_{2,2}+21p_{2,3}+
20p_{2,5}-3p_{1,0}-2p_{1,1}-3p_{1,2}-4p_{1,4}+(6p_{2,0}+4p_{2,1}+6p_{2,2}+8p_{2,4}-p_{1,3})t+
p_{2,3}t^2\big)x^3+\big(p_{0,4}+22p_{2,0}+12p_{2,1}+30p_{2,2}+16p_{2,3}+27p_{2,4}-3p_{1,1}-2p_{1,2}-
3p_{1,3}-5p_{1,5}(10p_{2,5}-p_{1,4}+6p_{2,1}+4p_{2,2}+6p_{2,3})t+p_{2,4}t^2\big)x^4+
\big(p_{0,5}+12p_{2,0}+22p_{2,1}+12p_{2,2}+36p_{2,3}+20p_{2,4}+33p_{2,5}-3p_{1,2}-2p_{1,3}-3p_{1,4}+
(6p_{2,2}+4p_{2,3}+6p_{2,4}-p_{1,5})t+p_{2,5}t^2\big)x^5+\big(9p_{2,0}+12p_{2,1}+22p_{2,2}+12p_{2,3}+
42p_{2,4}+24p_{2,5}-3p_{1,3}-2p_{1,4}-3p_{1,5}+(6p_{2,3}+4p_{2,4}+6p_{2,5})t\big)x^6+\big(9p_{2,1}+
12p_{2,2}+22p_{2,3}+12p_{2,4}+48p_{2,5}-3p_{1,4}-2p_{1,5}+(6p_{2,4}+4p_{2,5})t\big)x^7+\big(9p_{2,2}+
12p_{2,3}+22p_{2,4}+12p_{2,5}-3p_{1,5}+6p_{2,5}t\big)x^8+\big(9p_{2,3}+12p_{2,4}+22p_{2,5}\big)x^9+
\big(9p_{2,4}+12p_{2,5}\big)x^{10}+9p_{2,5}x^{11}.$

\medskip

To compute the inverse automorphism of \texttt{auto1}, we define a new automorphism, \texttt{auto2}, using the command: \verb"auto2 := inverseAutomorphism(auto1, n=4, m=3)". The execution of this command produces the following output:
\begin{center}
\textit{Dixmier automorphism successfully defined\\
auto2 := $\alpha(t, x)$}
\end{center}
Subsequently, retrieving the corresponding polynomials $p$ and $q$ via the commands \texttt{auto2:-getP();} and \texttt{auto2:-getQ();} yields:
\[
p = -3t - 2t^2 - 3t^3 - x,\qquad q = t
\]
\end{example}

In the following example, we illustrate the generation, inversion, composition, and evaluation of Dixmier automorphisms using the \texttt{DixmierProblem} package in Maple. The example employs the lambda-parametric option (i.e., \verb|parametric = true|), so that the defining Dixmier polynomials $p$ and $q$ are expressed in terms of the parameter $\lambda$. Additionally, we demonstrate the use of the \texttt{composeAutomorphism} function, which accepts two Dixmier automorphisms as arguments and returns their composition. The steps performed include generating a lambda-parametric automorphism, computing its inverse, composing the automorphism with its inverse, and finally, evaluating the resulting compositions to confirm that they correspond to the identity mapping.
\begin{example}
	
\noindent
\emph{Step 1: Generation of a lambda-parametric Dixmier automorphism}
	
We generate a lambda-parametric Dixmier automorphism and store it in \texttt{auto1}. This automorphism is characterized by its defining polynomials $p$ and $q$ being given in terms of $\lambda$.
\begin{verbatim}
auto1 := DixmierAutomorphism(parametric = true);
\end{verbatim}
\emph{Maple output:}
	
\noindent
\verb"Dixmier automorphism successfully defined"\\
\verb"auto1" $:=\alpha(t,x)$

\smallskip

\noindent
\emph{Step 2: Extraction of the defining polynomials of \texttt{auto1}}

We next extract the Dixmier polynomials $p$ and $q$ that define \texttt{auto1}.
\begin{verbatim}
print('p' = auto1:-getP());
print('q' = auto1:-getQ());
\end{verbatim}
\emph{Expected Maple output:}
\begin{align*}
p &= \lambda_0t^2+\left(-\frac{1}{\lambda_2}+\lambda_1t\right)x+\frac{\lambda_1^2}{4\lambda_0}x^2\\
q &= \lambda_2t+\lambda_3t^2+\left(\frac{\lambda_1\lambda_2^2-2\lambda_3}{2\lambda_0\lambda_2}+
\frac{\lambda_1\lambda_3}{\lambda_0}t\right)x+\frac{\lambda_1^2\lambda_3}{4\lambda_0^2}x^2
\end{align*}
	
\noindent
\emph{Step 3: Computation of the inverse automorphism}
	
The inverse automorphism of \texttt{auto1} is then computed and stored in \texttt{auto2}.
\begin{verbatim}
auto2 := InverseAutomorphism(auto1);
\end{verbatim}
\emph{Maple output:}

\noindent
\verb"Dixmier automorphism successfully defined"\\
\verb"auto2" $:=\alpha(t,x)$

\smallskip
	
\noindent
\emph{Step 4: Extraction of the defining polynomials of the inverse automorphism}
	
We retrieve the defining Dixmier polynomials for the inverse automorphism stored in \texttt{auto2}.
\begin{verbatim}
print('p' = auto2:-getP());
print('q' = auto2:-getQ());
\end{verbatim}
\emph{Expected Maple output:}
\begin{align*}
p &= \frac{\lambda_1(\lambda_1\lambda_2^2+2\lambda_3)}{4\lambda_2\lambda_0}+ \frac{\lambda_1\lambda_2^2-2\lambda_3}{2\lambda_0\lambda_2}t-\frac{\lambda_1\lambda_3^2}
	{2\lambda_0^2\lambda_2}t^2+\left(\frac{1}{\lambda_2}+\frac{\lambda_1\lambda_3}{\lambda_0\lambda_2}t\right)x
-\frac{\lambda_1}{2\lambda_2}x^2\\
q &= \frac{-\lambda_1\lambda_2^2-2\lambda_3}{2\lambda_2}-\lambda_2t+\frac{\lambda_3^2}{\lambda_0\lambda_2}t^2-
\frac{2\lambda_3}{\lambda_2}tx+\frac{\lambda_0}{\lambda_2}x^2.
\end{align*}
	
\noindent
\emph{Step 5: Composition of automorphisms}
	
We then compose the automorphisms using the \texttt{ComposeAutomorphism} function. Two compositions are formed: one as $\texttt{auto3}=\texttt{auto2}\circ\texttt{auto1}$ and the other as $\texttt{auto4}=\texttt{auto1}\circ\texttt{auto2}.$
	
\emph{Composition commands:}
\begin{verbatim}
auto3 := composeAutomorphism(auto2, auto1);
auto4 := composeAutomorphism(auto1, auto2);
\end{verbatim}
\emph{Maple output:}
	
\verb"Dixmier automorphism successfully defined"
	
\verb"auto3" := $\alpha(t, x)$
	
\verb"Dixmier automorphism successfully defined"
	
\verb"auto4" := $\alpha(t, x)$
	
\noindent
\emph{Step 6: Evaluation of the composed automorphisms}
	
Finally, we evaluate the composed automorphisms on the generators $t$ and $x$ to confirm that they behave as the identity mapping:
\begin{verbatim}
auto3:-Apply(t), auto3:-Apply(x), auto4:-Apply(t), auto4:-Apply(x);
\end{verbatim}
\emph{Expected results:}
\[
\texttt{auto3}(t) = t, \qquad \texttt{auto3}(x) = x, \qquad \texttt{auto4}(t) = t, \qquad \texttt{auto4}(x) = x
\]
These results confirm that the compositions indeed yield the identity automorphism, thereby validating our inversion and composition operations.
\end{example}

\subsection{\texttt{DixmierAutomorphismFactor} function}\label{DixmierAutomorphismFactor}

The \texttt{DixmierAutomorphismFactor} function plays a pivotal role in our approach to the Dixmier conjecture by decomposing complex endomorphisms (=automorphisms, according to Theorem \ref{Theorem1.3}), obtained from families of Dixmier polynomials, into a composition of elementary automorphisms, namely $\Phi_{n,\lambda}$ and $\Psi_{n,\lambda}$. This factorization simplifies the analysis of the underlying structural properties of the Weyl algebra $A_1(K)$ and its generalizations, providing a constructive method for both symbolic and numerical validation of the problem. By breaking down intricate algebraic mappings into manageable components, the function enhances computational efficiency and analytical rigor, while deepening our understanding of invariants, symmetries, and transformation properties. In this section we present computational results that illustrate the practical application and robustness of this factorization procedure.

The overall procedure follows these steps:
\begin{enumerate}
\item \textbf{Family Generation:} We first produce a family of Dixmier polynomials using the function
\linebreak \texttt{DixmierPolynomials} with parameters $n$ (the maximum degree in $x$) and $m$ (the maximum degree in $t$).
\item \textbf{Configuration Determination:} For each generated family, an appropriate automorphism
configuration is selected. This configuration consists of ordered applications of the automorphisms:
\[
\begin{array}{rcl}
\Phi(n,\lambda):=\Phi_{n,\lambda} & : & t \mapsto t+\lambda\,x^n,\quad x\mapsto x,\\[1mm]
\Psi(n,\lambda):=\Phi'_{n,\lambda} & : & t \mapsto t,\quad x\mapsto x+\lambda\,t^n,	
\end{array}
\]
\item \textbf{Equation System Setup:} The function then computes the composition of the selected
automorphisms. Denote by $p_{\text{comp}}(x,t)$ and $q_{\text{comp}}(x,t)$ the images of $t$ and $x$, respectively, after applying the sequence of automorphisms. These are then equated coefficient wise to the corresponding Derived Dixmier polynomials
\[
p_{\text{orig}}(x,t), \quad q_{\text{orig}}(x,t)
\]
built by \texttt{DixmierPolynomials}. This yields a system of polynomial equations in the parameters (e.g., $\lambda$'s and $\mu$).
\item \textbf{System Resolution:} The resulting system is solved using Maple symbolic solvers, via the
direct use of \texttt{solve} or, when necessary due to high nonlinearity, through the application of Gröbner basis methods.
\item \textbf{Verification:} Finally, the solution is substituted back into the composed automorphism to
verify that the original action is recovered, confirming that the composition is indeed an automorphism.
\end{enumerate}

\subsubsection*{Factorization Results}

Below we present comprehensive tables summarizing the results generated by the function\linebreak \texttt{DixmierAutomorphismFactor} for various maximum degrees. In each case, the function decomposes an endomorphism (=automorphism, according to Theorem \ref{Theorem1.3}) as a composition of the automorphisms $\Phi$ and $\Psi,$ with the corresponding parameter values detailed in each table.
\clearpage

\begin{center}
\begin{table}[htb]
\centering
\renewcommand{\arraystretch}{2.5}
\begin{tabular}{|c p{12.8cm}|}\hline
\textbf{Values of $m$}  & \textbf{Output:} Factorization into $\Phi_{i,\mu}$ and $\Psi_{i,\mu}$ for Dixmier polynomial pairs in Table \ref{TestTable:n1} \\ \hline \hline
$m=1:$ & \parbox[t][0.3\height][b]{10.8cm}{$
\begin{aligned}
\texttt{Case\,1:} &\ \Psi_{1,\frac{\lambda_0\lambda_1-\lambda_2+1}{\lambda_2\lambda_0}}\Phi_{1,\lambda_0}
\Psi_{1,\frac{\lambda_2-1}{\lambda_0}},\\[1ex]
\texttt{Case\,2:} &\  \Psi_{1,-\lambda_0\lambda_1+\lambda_1}\Phi_{1,-\frac{1}{\lambda_1}}\Psi_{1,\lambda_1}
\end{aligned}$} \\[0.7cm] \hline
$m=2:$ & \parbox[t][0.3\height][b]{10.8cm}{$
\begin{aligned}
\texttt{Case\,1:} &\ \Psi_{1,\frac{\lambda_0\lambda_1-\lambda_3+1}{\lambda_0\lambda_3}}
\Psi_{2,\frac{\lambda_2}{\lambda_3}}\Phi_{1,\lambda_0}\Psi_{1,\frac{\lambda_3-1}{\lambda_0}},\\[1ex]
\texttt{Case\,2:} &\ \Psi_{1,\frac{\lambda_0-1}{\lambda_2}}\Psi_{2,\frac{\lambda_1}{\lambda_2}}
\Phi_{1,\lambda_2}\Psi_{1,-\frac{1}{\lambda_2}}
\end{aligned}$} \\[0.7cm] \hline
$m=3:$ & \parbox[t][0.3\height][b]{10.8cm}{$
\begin{aligned}
\texttt{Case\,1:} &\ \Psi_{1,\frac{\lambda_0\lambda_1-\lambda_4+1}{\lambda_0\lambda_4}}
\Psi_{2,\frac{\lambda_2}{\lambda_4}}\Psi_{3,\frac{\lambda_3}{\lambda_4}}\Phi_{1,\lambda_0}
\Psi_{1,\frac{\lambda_4-1}{\lambda_0}},\\[1ex]
\texttt{Case\,2:} &\ \Psi_{1,\frac{\lambda_0-1}{\lambda_3}}\Psi_{2,\frac{\lambda_1}{\lambda_3}}
\Psi_{3,\frac{\lambda_2}{\lambda_3}}\Phi_{1,\lambda_3}\Psi_{1,-\frac1{\lambda_3}}
\end{aligned}$} \\[0.7cm] \hline
$m=4:$ & \parbox[t][0.3\height][b]{10.8cm}{$
\begin{aligned}
\texttt{Case\,1:} &\ \Psi_{1,\frac{\lambda_0\lambda_1-\lambda_5+1}{\lambda_0\lambda_5}}
\Psi_{2,\frac{\lambda_2}{\lambda_5}}\Psi_{3,\frac{\lambda_3}{\lambda_5}}\Psi_{4,\frac{\lambda_4}{\lambda_5}}
\Phi_{1,\lambda_0}\Psi_{1,\frac{\lambda_5-1}{\lambda_0}},\\[1ex]
\texttt{Case\,2:} &\ \Psi_{1,\frac{\lambda_0-1}{\lambda_4}}\Psi_{2,\frac{\lambda_1}{\lambda_4}}
\Psi_{3,\frac{\lambda_2}{\lambda_4}}\Psi_{4,\frac{\lambda_3}{\lambda_4}}\Phi_{1,\lambda_4}
\Psi_{1,-\frac1{\lambda_4}}
\end{aligned}$} \\[0.7cm] \hline
$m=5:$ & \parbox[t][0.3\height][b]{10.8cm}{$
\begin{aligned}
\texttt{Case\,1:} &\ \Psi_{1,\frac{\lambda_0\lambda_1-\lambda_6+1}{\lambda_0\lambda_6}}
\Psi_{2,\frac{\lambda_2}{\lambda_6}}\Psi_{3,\frac{\lambda_3}{\lambda_6}}
\Psi_{4,\frac{\lambda_4}{\lambda_6}}\Psi_{5,\frac{\lambda_5}{\lambda_6}}\Phi_{1,\lambda_0}
\Psi_{1,\frac{\lambda_6-1}{\lambda_0}},\\[1ex]
\texttt{Case\,2:} &\ \Psi_{1,\frac{\lambda_0-1}{\lambda_5}}\Psi_{2,\frac{\lambda_1}{\lambda_5}}\Psi_{3,\frac{\lambda_2}{\lambda_5}}
\Psi_{4,\frac{\lambda_3}{\lambda_5}}\Psi_{5,\frac{\lambda_4}{\lambda_5}}\Phi_{1,\lambda_5}
\Psi_{1,-\frac1{\lambda_5}}
\end{aligned}$} \\[0.7cm] \hline
$m=6:$ & \parbox[t][0.3\height][b]{10.8cm}{$
\begin{aligned}
\texttt{Case\,1:} &\ \Psi_{1,\frac{\lambda_0\lambda_1-\lambda_7+1}{\lambda_0\lambda_7}}
\Psi_{2,\frac{\lambda_2}{\lambda_7}}\Psi_{3,\frac{\lambda_3}{\lambda_7}}
\Psi_{4,\frac{\lambda_4}{\lambda_7}}\Psi_{5,\frac{\lambda_5}{\lambda_7}}
\Psi_{6,\frac{\lambda_6}{\lambda_7}}\Phi_{1,\lambda_0}\Psi_{1,\frac{\lambda_7-1}{\lambda_0}},\\[1ex]
\texttt{Case\,2:} &\ \Psi_{1,\frac{\lambda_0-1}{\lambda_6}}\Psi_{2,\frac{\lambda_1}{\lambda_6}}\Psi_{3,\frac{\lambda_2}{\lambda_6}}
\Psi_{4,\frac{\lambda_3}{\lambda_6}}\Psi_{5,\frac{\lambda_4}{\lambda_6}}\Psi_{6,\frac{\lambda_5}{\lambda_6}}
\Phi_{1,\lambda_6}\Psi_{1,-\frac1{\lambda_6}}
\end{aligned}
$} \\[0.7cm] \hline
$m=7:$ & \parbox[t][0.3\height][b]{10.8cm}{$
\begin{aligned}
\texttt{Case\,1:} &\ \Psi_{1,\frac{\lambda_0\lambda_1-\lambda_8+1}{\lambda_0\lambda_8}}
\Psi_{2,\frac{\lambda_2}{\lambda_8}}\Psi_{3,\frac{\lambda_3}{\lambda_8}}
\Psi_{4,\frac{\lambda_4}{\lambda_8}}\Psi_{5,\frac{\lambda_5}{\lambda_8}}\Psi_{6,\frac{\lambda_6}{\lambda_8}}
\Psi_{7,\frac{\lambda_7}{\lambda_8}}\Phi_{1,\lambda_0}\Psi_{1,\frac{\lambda_8-1}{\lambda_0}},\\[1ex]
\texttt{Case\,2:} &\ \Psi_{1,\frac{\lambda_0-1}{\lambda_7}}\Psi_{2,\frac{\lambda_1}{\lambda_7}}
\Psi_{3,\frac{\lambda_2}{\lambda_7}}\Psi_{4,\frac{\lambda_3}{\lambda_7}}\Psi_{5,\frac{\lambda_4}{\lambda_7}}
\Psi_{6,\frac{\lambda_5}{\lambda_7}}\Psi_{7,\frac{\lambda_6}{\lambda_7}}\Phi_{1,\lambda_7}\Psi_{1,-
\frac1{\lambda_7}}
\end{aligned}
$} \\[0.7cm] \hline
\end{tabular}
\caption{Comprehensive outputs of \texttt{DixmierAutomorphismFactor}$(1,m)$ with $m$ from $1$ to $7$.}
\label{tab1:DixmierAutomorphismFactor_n=1}
\end{table}
\end{center}
In \texttt{Case 1}, \texttt{DixmierAutomorphismFactor}$(1,m)$ equals
\[
\Psi_{1,\frac{\lambda_0\lambda_1-\lambda_{m+1}+1}{\lambda_0\lambda_{m+1}}}
\Psi_{2,\frac{\lambda_2}{\lambda_{m+1}}}\Psi_{3,\frac{\lambda_3}{\lambda_{m+1}}}
\Psi_{4,\frac{\lambda_4}{\lambda_{m+1}}}\cdots\Psi_{m,\frac{\lambda_m}{\lambda_{m+1}}}
\Phi_{1,\lambda_0}\Psi_{1,\frac{\lambda_{m+1}-1}{\lambda_0}},
\]
while in \texttt{Case 2}, it equals
\[
\Psi_{1,\frac{\lambda_0-1}{\lambda_m}}\Psi_{2,\frac{\lambda_1}{\lambda_m}}\Psi_{3,\frac{\lambda_2}{\lambda_m}}
\Psi_{4,\frac{\lambda_3}{\lambda_m}}\cdots\Psi_{m,\frac{\lambda_{m-1}}{\lambda_m}}
\Phi_{1,\lambda_m}\Psi_{1,-\frac1{\lambda_m}}.
\]
\clearpage
\begin{center}
\begin{table}[htb]
\centering
\renewcommand{\arraystretch}{2.5}
\begin{tabular}{|c p{13cm}|}\hline
\textbf{$n$ value}  & \textbf{Output:} Factorization into $\Phi_{i,\mu}$ and $\Psi_{i,\mu}$ for Dixmier polynomial pairs in Table \ref{TestTable:m1}
 \\ \hline \hline
$n=2$ & \parbox[t][0.3\height][b]{13cm}{$
\begin{aligned}
\texttt{Case\,1:} &\ \Phi_{1,\frac{\lambda_1\lambda_3-\lambda_0+1}{\lambda_0\lambda_3}}
\Phi_{2,\frac{\lambda_2}{\lambda_0}}\Psi_{1,\lambda_3}\Phi_{1,\frac{\lambda_0-1}{\lambda_3}},\\[1ex]
\texttt{Case\,2:} &\ \Phi_{1,-\lambda_0\lambda_1+\lambda_0}\Phi_{2,-\lambda_0\lambda_2}
\Psi_{1,-\frac{1}{\lambda_0}}\Phi_{1,\lambda_0}
\end{aligned}$} \\[0.7cm] \hline
$n=3$ & \parbox[t][0.3\height][b]{13cm}{$
\begin{aligned}
\texttt{Case\,1:} &\ \Phi_{1,\frac{\lambda_1\lambda_4-\lambda_0+1}{\lambda_0\lambda_4}}
\Phi_{2,\frac{\lambda_2}{\lambda_0}}\Phi_{3,\frac{\lambda_3}{\lambda_0}}\Psi_{1,\lambda_4}
\Phi_{1,\frac{\lambda_0-1}{\lambda_4}},\\[1ex]
\texttt{Case\,2:} &\ \Phi_{1,-\lambda_0\lambda_1+\lambda_0}\Phi_{2,-\lambda_0\lambda_2}
\Phi_{3,-\lambda_0\lambda_3}\Psi_{1,-\frac{1}{\lambda_0}}\Phi_{1,\lambda_0}
\end{aligned}$} \\[0.7cm] \hline
$n=4$ & \parbox[t][0.3\height][b]{13cm}{$
\begin{aligned}
\texttt{Case\,1:} &\ \Phi_{1,\frac{\lambda_1\lambda_5-\lambda_0+1}{\lambda_0\lambda_5}}
\Phi_{2,\frac{\lambda_2}{\lambda_0}}\Phi_{3,\frac{\lambda_3}{\lambda_0}}\Phi_{4,\frac{\lambda_4}
{\lambda_0}}\Psi_{1,\lambda_5}\Phi_{1,\frac{\lambda_0-1}{\lambda_5}},\\[1ex]
\texttt{Case\,2:} &\ \Phi_{1,-\lambda_0\lambda_1+\lambda_0}\Phi_{2,-\lambda_0\lambda_2}
\Phi_{3,-\lambda_0\lambda_3}\Phi_{4,-\lambda_0\lambda_4}\Psi_{1,-\frac{1}{\lambda_0}}\Phi_{1,\lambda_0}
\end{aligned}$} \\[0.7cm] \hline
$n=5$ & \parbox[t][0.3\height][b]{13cm}{$
\begin{aligned}
\texttt{Case\,1:} &\ \Phi_{1,\frac{\lambda_1\lambda_6-\lambda_0+1}{\lambda_0\lambda_6}}
\Phi_{2,\frac{\lambda_2}{\lambda_0}}\Phi_{3,\frac{\lambda_3}{\lambda_0}}
\Phi_{4,\frac{\lambda_4}{\lambda_0}}\Phi_{5,\frac{\lambda_5}{\lambda_0}}
\Psi_{1,\lambda_6}\Phi_{1,\frac{\lambda_0-1}{\lambda_6}},\\[1ex]
\texttt{Case\,2:} &\ \Phi_{1,-\lambda_0\lambda_1+\lambda_0}\Phi_{2,-\lambda_0\lambda_2}
\Phi_{3,-\lambda_0\lambda_3}\Phi_{4,-\lambda_0\lambda_4}\Phi_{5,-\lambda_0\lambda_5}
\Psi_{1,-\frac{1}{\lambda_0}}\Phi_{1,\lambda_0}
\end{aligned}$} \\[0.7cm] \hline
$n=6$ & \parbox[t][0.3\height][b]{13cm}{$
\begin{aligned}
\texttt{Case\,1:} &\ \Phi_{1,\frac{\lambda_1\lambda_7-\lambda_0+1}{\lambda_0\lambda_7}}
\Phi_{2,\frac{\lambda_2}{\lambda_0}}\Phi_{3,\frac{\lambda_3}{\lambda_0}}
\Phi_{4,\frac{\lambda_4}{\lambda_0}}\Phi_{5,\frac{\lambda_5}{\lambda_0}}
\Phi_{6,\frac{\lambda_6}{\lambda_0}}\Psi_{1,\lambda_7}\Phi_{1,\frac{\lambda_0-1}
{\lambda_7}},\\[1ex]
\texttt{Case\,2:} &\ \Phi_{1,-\lambda_0\lambda_1+\lambda_0}\Phi_{2,-\lambda_0\lambda_2}
\Phi_{3,-\lambda_0\lambda_3}\Phi_{4,-\lambda_0\lambda_4}\Phi_{5,-\lambda_0\lambda_5}
\Phi_{6,-\lambda_0\lambda_6}\Psi_{1,-\frac{1}{\lambda_0}}\Phi_{1,\lambda_0}
\end{aligned}$} \\[0.7cm] \hline
$n=7$ & \parbox[t][0.3\height][b]{13cm}{$
\begin{aligned}
\texttt{Case\,1:} &\ \Phi_{1,\frac{\lambda_1\lambda_8-\lambda_0+1}{\lambda_0\lambda_8}}
\Phi_{2,\frac{\lambda_2}{\lambda_0}}\Phi_{3,\frac{\lambda_3}{\lambda_0}}
\Phi_{4,\frac{\lambda_4}{\lambda_0}}\Phi_{5,\frac{\lambda_5}{\lambda_0}}
\Phi_{6,\frac{\lambda_6}{\lambda_0}}\Phi_{7,\frac{\lambda_7}{\lambda_0}}
\Psi_{1,\lambda_8}\Phi_{1,\frac{\lambda_0-1}{\lambda_8}},\\[1ex]
\texttt{Case\,2:} &\ \Phi_{1,-\lambda_0\lambda_1+\lambda_0}\Phi_{2,-\lambda_0\lambda_2}
\Phi_{3,-\lambda_0\lambda_3}\Phi_{4,-\lambda_0\lambda_4}\Phi_{5,-\lambda_0\lambda_5}
\Phi_{6,-\lambda_0\lambda_6}\Phi_{7,-\lambda_0\lambda_7}\Psi_{1,-\frac{1}{\lambda_0}}
\Phi_{1,\lambda_0}
\end{aligned}$} \\[0.7cm] \hline
\end{tabular}
\caption{Comprehensive outputs of \texttt{DixmierAutomorphismFactor}$(n,1)$ with $n$ from $2$ to $7$.}
\label{tab:DixmierAutomorphismFactor_m=1}
\end{table}
\end{center}
In \texttt{Case 1}, \texttt{DixmierAutomorphismFactor}$(n,1)$ equals
\[
\Phi_{1,\frac{\lambda_1\lambda_{n+1}-\lambda_0+1}{\lambda_0\lambda_{n+1}}}\Phi_{2,\frac{\lambda_2}{\lambda_0}}
\Phi_{3,\frac{\lambda_3}{\lambda_0}}\Phi_{4,\frac{\lambda_4}{\lambda_0}}\cdots
\Phi_{n,\frac{\lambda_n}{\lambda_0}}\Psi_{1,\lambda_{n+1}}\Phi_{1,\frac{\lambda_0-1}{\lambda_{n+1}}},
\]
while in \texttt{Case 2}, it equals
\[
\Phi_{1,-\lambda_0\lambda_1+\lambda_0}\Phi_{2,-\lambda_0\lambda_2}\Phi_{3,-\lambda_0\lambda_3}
\Phi_{4,-\lambda_0\lambda_4}\cdots\Phi_{n,-\lambda_0\lambda_n}\Psi_{1,-\frac{1}{\lambda_0}}\Phi_{1,\lambda_0}.
\]
\begin{center}
\begin{table}[htb]
\centering
\renewcommand{\arraystretch}{2.5}
\begin{tabular}{|c p{13.2cm}|}\hline
\textbf{$n$}  & \textbf{Output:} Factorization into $\Phi_{i,\mu}$ and $\Psi_{i,\mu}$ for Dixmier polynomial pairs in Table \ref{TestTable:m2} \\ \hline \hline
$n=2$ &	\parbox[t][1.2\height][c]{13.2cm}{$
\begin{aligned}
\texttt{Case\,1:} &\ \Phi_{1,\frac{\lambda_2\lambda_4-\lambda_0+1}{\lambda_4\lambda_0}}
\Psi_{1,\lambda_4}\Phi_{2,\frac{\lambda_1}{\lambda_4^2}}\Phi_{1,\frac{\lambda_0-1}{\lambda_4}}
\Phi_{0,\frac{-\lambda_1(\lambda_2\lambda_4+1)}{\lambda_4\lambda_0}},\\[1ex]
\texttt{Case\,2:} &\ \Phi_{1,\frac{\lambda_1\lambda_3-\lambda_0+1}{\lambda_0\lambda_3}}
\Phi_{2,\frac{\lambda_2}{\lambda_0}}\Psi_{1,\lambda_3}\Phi_{1,\frac{\lambda_0-1}{\lambda_3}},\\[1ex]
\texttt{Case\,3:} &\ \Psi_{1,\frac{\lambda_0-1}{\lambda_2}}\Psi_{2,\frac{\lambda_1}{\lambda_2}}
\Phi_{1,\lambda_2}\Psi_{1,\frac{\lambda_2\lambda_3-\lambda_1}{\lambda_1\lambda_2}},\\[1ex]
\texttt{Case\,4:} &\ \Psi_{1,\frac{\lambda_0-1}{\lambda_1}}
\Phi_{1,\lambda_1}\Psi_{2,\frac{\lambda_3}{\lambda_0^2}}\
\Psi_{1,\frac{\lambda_1\lambda_2-\lambda_0+1}{\lambda_1\lambda_0}}
\Psi_{0,\frac{-\lambda_3\lambda_1}{\lambda_0}}
\end{aligned}
$} \\[2.5cm] \hline
$n=3$ & \parbox[t][1.5\height][c]{13.2cm}{$
\begin{aligned}
\texttt{Case\,1:} &\ \Psi_{1,\frac{\beta\lambda_0}{\gamma}}\Phi_{2,\frac{-\gamma^3}
{\beta^2\lambda_0^3}}\Phi_{1,\frac{\zeta\gamma}{\lambda_0^2\lambda_4}}
\Psi_{1,\frac{-\lambda_0\lambda_4}{\gamma}}\Phi_{1,\frac{\lambda_0\lambda_1-\gamma}{
\lambda_0\lambda_4}}\Phi_{0,\frac{-\lambda_1\gamma}{\beta\lambda_0}}\Psi_{0,\frac{-\lambda_4\gamma}{
\beta\lambda_0}}\,\\[1ex]
&\ \text{where } \beta=\lambda_0\lambda_4-\lambda_1\lambda_3,\,\gamma= \beta\lambda_2-\lambda_1,\,
\zeta=-\lambda_2\lambda_3+\lambda_0-1\\
\texttt{Case\,2:} &\ \Psi_{1,\frac{2\lambda_0}{\lambda_1}}\Phi_{2,\frac{\lambda_1^3\lambda_2}{8\lambda_0^2}}
\Phi_{1,\tfrac{-\left(-\frac{1}{2}\lambda_1\lambda_2^2+\lambda_0\lambda_2+\lambda_3\right)\lambda_1}
{2\lambda_0\lambda_3}}\Psi_{1,\frac{2\lambda_3}{\lambda_1\lambda_2}}\Phi_{1,\frac{-\lambda_1\lambda_2+2\lambda_0}
{2\lambda_3}}\Phi_{0,\frac{-\lambda_1}{2}}\Psi_{0,\frac{-\lambda_3\lambda_1}{2\lambda_0}},\\[1ex]
\texttt{Case\,3:} &\ \Phi_{1,\frac{(\lambda_1\lambda_4-\lambda_3)\lambda_0+\lambda_3}{\lambda_0^2\lambda_4}}
\Phi_{2,\frac{\lambda_2}{\lambda_0}}\Phi_{3,\frac{\lambda_3}{\lambda_0}}\Psi_{1,\frac{\lambda_4
\lambda_0}{\lambda_3}}\Phi_{1,\frac{\lambda_3(\lambda_0-1)}{\lambda_0\lambda_4}},\\[1ex]
\texttt{Case\,4:} &\ \Phi_{1,\frac{(\lambda_1\lambda_3-\lambda_2)\lambda_0+\lambda_2}{\lambda_0^2\lambda_3}}
\Phi_{2,\frac{\lambda_2}{\lambda_0}}\Psi_{1,\frac{\lambda_3\lambda_0}{\lambda_2}}\Phi_{1,\frac{\lambda_2(
\lambda_0-1)}{\lambda_3\lambda_0}},\\[1ex]
\texttt{Case\,5:} &\ \Phi_{1,-\lambda_0\lambda_1+\lambda_0}\Phi_{2,-\lambda_0\lambda_2}
\Phi_{3,-\lambda_3\lambda_0}\Psi_{1,-\frac{1}{\lambda_0}}\Phi_{1,\lambda_0}
\end{aligned}$} \\[4.3cm] \hline
$n=4$ &
\parbox[t][1.0\height][t]{13.2cm}{$
\begin{aligned}
\texttt{Case\,1:} &\ \Phi_{2,\frac{\lambda_3}{2\lambda_1}}\Psi_{1,\tfrac{\beta}{\gamma}}
\Phi_{2,\tfrac{-\gamma^3}{\lambda_0\beta^2}}\Phi_{1,\tfrac{\eta}{\lambda_0\lambda_1\lambda_4+\beta}}
\Psi_{1,\tfrac{-\lambda_0\lambda_5}{\gamma}}\Phi_{1,\tfrac{\lambda_0\lambda_1+\gamma}{\lambda_0\lambda_5}}
\Phi_{0,\tfrac{-\gamma\lambda_1}{\beta}}\Psi_{0,\tfrac{-\gamma\lambda_5}{\beta}},\\[1ex]
&\ \text{where }\beta=\lambda_0(\lambda_0\lambda_5-\lambda_1\lambda_4),\,
\gamma=(\lambda_0\lambda_5-\lambda_1\lambda_4)(\lambda_2-\lambda_3)-\lambda_1,\\
&\ \eta=\gamma((-\lambda_2+\lambda_3)\lambda_4+\lambda_0-1)\\
\texttt{Case\,2:} &\ \Phi_{2,\frac{\lambda_2}{2\lambda_0}}\Psi_{1,\frac{2\lambda_0}{\lambda_1}}\Phi_{2,\frac{\lambda_1^3
\lambda_3}{8\lambda_0^2}}\Phi_{1,\frac{\zeta}{4\lambda_0\lambda_4}}\Psi_{1,\frac{2\lambda_4}{\lambda_3\lambda_1}}
\Phi_{1,\frac{-\lambda_1\lambda_3+2\lambda_0}{2\lambda_4}}\Phi_{0,\frac{-\lambda_1}{2}}
\Psi_{0,\frac{-\lambda_1\lambda_4}{2\lambda_0}},\\
&\ \text{where }\zeta=-\lambda_1\left(-\lambda_1\lambda_3^2+2\lambda_3\lambda_0+2\lambda_4\right)
\\[1ex]
\texttt{Case\,3:} &\ \Phi_{1,\frac{(\lambda_1\lambda_5-\lambda_4)\lambda_0+\lambda_4}{\lambda_0^2\lambda_5}}
\Phi_{2,\frac{\lambda_2}{\lambda_0}}\Phi_{3,\frac{\lambda_3}{\lambda_0}}
\Phi_{4,\frac{\lambda_4}{\lambda_0}}\Psi_{1,\frac{\lambda_0\lambda_5}{\lambda_4}}
\Phi_{1,\frac{\lambda_4(\lambda_0-1)}{\lambda_0\lambda_5}},\\[1ex]
\texttt{Case\,4:} &\ \Phi_{1,\frac{(\lambda_1\lambda_4-\lambda_3)\lambda_0+\lambda_3}{\lambda_0^2\lambda_4}}
\Phi_{2,\frac{\lambda_2}{\lambda_0}}\Phi_{3,\frac{\lambda_3}{\lambda_0}}
\Psi_{1,\frac{\lambda_0\lambda_4}{\lambda_3}}
\Phi_{1,\frac{\lambda_3(\lambda_0-1)}{\lambda_0\lambda_4}},\\[1ex]
\texttt{Case\,5:} &\ \Phi_{2,\frac{\lambda_2}{\lambda_0}}\Psi_{1,\frac{\lambda_0-1}{\lambda_1}}
\Phi_{1,\lambda_1}\Psi_{2,\frac{\lambda_4}{\lambda_2^2}}\Psi_{1,\frac{(1-\lambda_0)\lambda_2^3+\lambda_0\lambda_1
\lambda_2^2\lambda_3-\lambda_0\lambda_1^3\lambda_4}{\lambda_2^3\lambda_1\lambda_0}}
\Psi_{0,\frac{-\lambda_0\lambda_1\lambda_4}{\lambda_2^2}},\\[1ex]
\texttt{Case\,6:} &\ \Psi_{1,\frac{\lambda_0-1}{\lambda_1}}
\Phi_{1,\lambda_1}\Psi_{2,\frac{\lambda_3}{\lambda_0^2}}
\Psi_{1,\frac{\lambda_1\lambda_2-\lambda_0+1}{\lambda_0\lambda_1}}
\Psi_{0,\frac{-\lambda_3\lambda_1}{\lambda_0}},\\[1ex]
\texttt{Case\,7:} &\ \Phi_{1,-\lambda_0\lambda_1+\lambda_0}\Phi_{2,-\lambda_0\lambda_2}
\Phi_{3,-\lambda_3\lambda_0}\Phi_{4,-\lambda_0\lambda_4}\Psi_{1,-\frac{1}{\lambda_0}}
\Phi_{1,\lambda_0}
\end{aligned}
$} \\[4.0cm] \hline
\end{tabular}
\caption{Comprehensive outputs of \texttt{DixmierAutomorphismFactor}$(n,2)$ with $n$ from $2$ to $4$.}
\label{tab:DixmierAutomorphismFactorn=2}
\end{table}
\end{center}
\begin{center}
\begin{table}[htb]
\centering
\renewcommand{\arraystretch}{3.5}
\begin{tabular}{|cp{12.2cm}|}\hline
$n$ & \textbf{Output:} Factorization into $\Phi_{i,\mu}$ and $\Psi_{i,\mu}$ for Dixmier polynomial pairs in Tables \ref{TestTable:n5m2} and \ref{TestTable:n6m2} \\ \hline\hline
$n=5:$ & \parbox[t]{12.2cm}{$
\begin{aligned}
\texttt{Case\,1:} &\ \Phi_{1,\tfrac{\lambda_1\lambda_6-\lambda_0+1}{\lambda_0\lambda_6}}
\Phi_{2,\tfrac{\lambda_2}{\lambda_0}}\Phi_{3,\tfrac{\lambda_3}{\lambda_0}}
\Phi_{4,\tfrac{\lambda_4}{\lambda_0}}\Phi_{5,\tfrac{\lambda_5}{\lambda_0}}\Psi_{1,\lambda_6}
\Phi_{1,\tfrac{\lambda_0-1}{\lambda_6}},\\[1ex]
\texttt{Case\,2:} &\ \Phi_{2,\frac{\lambda_2}{\lambda_0}}\Psi_{1,\frac{\lambda_0-1}{\lambda_1}}
\Phi_{1,\lambda_1}\Psi_{2,\frac{\lambda_4}{\lambda_0^2}}\Psi_{1,\frac{\lambda_1\lambda_3-
\lambda_0+1}{\lambda_1\lambda_0}}\Psi_{0,\frac{-\lambda_1\lambda_4}{\lambda_0}},\\[1ex]
\texttt{Case\,3:} &\ \Phi_{1,-\lambda_0\lambda_1+\lambda_0}\Phi_{2,-\lambda_0\lambda_2}
\Phi_{3,-\lambda_0\lambda_3}\Phi_{4,-\lambda_0\lambda_4}\Phi_{5,-\lambda_0\lambda_5}
\Psi_{1,-\tfrac{1}{\lambda_0}}\Phi_{1,\lambda_0},\\[1ex]
\texttt{Case\,4:} &\
\Phi_{2,\tfrac{\lambda_3}{2\lambda_1}}\Psi_{1,\tfrac{2\lambda_1}{\lambda_2}}
\Phi_{2,-\tfrac{\beta\lambda_2^3}{8\lambda_1^3}}
\Phi_{1,\frac{e_1}{4\lambda_1^2\lambda_5}}
\Psi_{1,-\tfrac{2\lambda_1\lambda_5}{\beta\lambda_2}}
\Phi_{1,\tfrac{\beta\lambda_2+2\lambda_1^2}{2\lambda_1\lambda_5}}
\Phi_{0,-\tfrac{\lambda_2}{2}}\Psi_{0,-\tfrac{\lambda_2\lambda_5}{2\lambda_1}},\\[1ex]
&\ \text{where } \beta=\lambda_0\lambda_5-\lambda_1\lambda_4,\,
e_1=\lambda_2(\beta(-\lambda_2\lambda_4+2\lambda_1)-2\lambda_1\lambda_5),\\
\texttt{Case\,5:} &\ \Psi_{1,\tfrac{\lambda_0-1}{\lambda_2}}\Psi_{2,\tfrac{\lambda_1}{\lambda_2}}
\Phi_{1,\lambda_2}\Psi_{1,\tfrac{\lambda_2\lambda_3-\lambda_0+1}{\lambda_2\lambda_0}}
\end{aligned}
$} \\[2.3cm] \hline
$n=6:$ & \parbox[t]{12.2cm}{$
\begin{aligned}
\texttt{Case\,1:} &\ \Phi_{1,\tfrac{\lambda_1\lambda_7-\lambda_0+1}{\lambda_0\lambda_7}}
\Phi_{2,\tfrac{\lambda_2}{\lambda_0}}\Phi_{3,\tfrac{\lambda_3}{\lambda_0}}
\Phi_{4,\tfrac{\lambda_4}{\lambda_0}}\Phi_{5,\tfrac{\lambda_5}{\lambda_0}}
\Phi_{6,\tfrac{\lambda_6}{\lambda_0}}\Psi_{1,\lambda_7}
\Phi_{1,\tfrac{\lambda_0-1}{\lambda_7}},\\[1ex]
\texttt{Case\,2:} &\
\Phi_{3,\tfrac{\lambda_3}{\lambda_0}}\Phi_{2,\tfrac{\lambda_2}{\lambda_0}}
\Psi_{1,\tfrac{\lambda_0-1}{\lambda_1}}\Phi_{1,\lambda_1}\Psi_{2,\tfrac{\lambda_5}{\lambda_0^2}}
\Psi_{1,\tfrac{\lambda_1\lambda_4-\lambda_0+1}{\lambda_0\lambda_1}}
\Psi_{0,-\tfrac{\lambda_1\lambda_5}{\lambda_0}},\\[1ex]
\texttt{Case\,3:} &\ \Phi_{1,-\lambda_0\lambda_1+\lambda_0}\Phi_{2,-\lambda_0\lambda_2}
\Phi_{3,-\lambda_0\lambda_3}\Phi_{4,-\lambda_0\lambda_4}\Phi_{5,-\lambda_0\lambda_5}
\Phi_{6,-\lambda_0\lambda_6}\Psi_{1,-\tfrac{1}{\lambda_0}}\Phi_{1,\lambda_0},\\[1ex]
\texttt{Case\,4:} &\ \Phi_{3,\frac{\lambda_2}{2\lambda_1}}\Phi_{2,\frac{\lambda_4}{\lambda_2}}\Psi_{1,\frac{-\lambda_2^3}{\gamma}}
\Phi_{2,\frac{\beta\gamma^3}{\lambda_2^9}}\Phi_{1,\frac{\zeta}{\lambda_2^6\lambda_6}}\Psi_{1,\frac{
\lambda_2^3\lambda_6}{\beta\gamma}}\Phi_{1,\frac{\lambda_1\lambda_2^3-\beta\gamma}{\lambda_2^3\lambda_6}}
\Phi_{0,\frac{\gamma\lambda_1}{\lambda_2^3}}\Psi_{0,\frac{\gamma\lambda_6}{\lambda_2^3}},\\
&\ \text{where }\zeta=-((\beta-\lambda_6)\lambda_2^3+\beta\gamma\lambda_5)\gamma\\[1ex]
\texttt{Case\,5:} &\ \Phi_{3,\frac{\lambda_3}{2\lambda_1}}\Psi_{1,\frac{2\lambda_1}{\lambda_2}}
\Phi_{2,\frac{-\beta\lambda_2^3}{8\lambda_1^3}}\Phi_{1,\frac{\lambda_2(2\beta\lambda_1-\gamma)}
{4\lambda_1^2\lambda_5}}\Psi_{1,\frac{-2\lambda_5 \lambda_1}{\beta\lambda_2}}
\Phi_{1,\frac{\beta\lambda_2+2\lambda_1^2}{2\lambda_1\lambda_5}}
\Phi_{0,\frac{-\lambda_2}{2}}\Psi_{0,\frac{-\lambda_2\lambda_5}{2\lambda_1}},\\
&\ \text{where } \beta=\lambda_0\lambda_5-\lambda_1\lambda_4,\,\gamma=\beta\lambda_2\lambda_4+2\lambda_2\lambda_5,
\\[1ex]
\texttt{Case\,6:} &\ \Phi_{2,\frac{\lambda_3}{2\lambda_1}}\Psi_{1,\frac{2\lambda_1}{\lambda_2}}
\Phi_{2,\frac{-\beta\lambda_2^3}{8\lambda_1^3}}\Phi_{1,\frac{\lambda_2(2\beta\lambda_1-\gamma)}
{4\lambda_1^2\lambda_5}}\Psi_{1,\frac{-2\lambda_5 \lambda_1}{\beta\lambda_2}}
\Phi_{1,\frac{\beta\lambda_2+2\lambda_1^2}{2\lambda_1\lambda_5}}
\Phi_{0,\frac{-\lambda_2}{2}}\Psi_{0,\frac{-\lambda_2\lambda_5}{2\lambda_1}},\\
&\ \text{where } \beta=\lambda_0\lambda_5-\lambda_1\lambda_4,\,\gamma=\beta\lambda_2\lambda_4+2\lambda_2\lambda_5,\\[1ex]
\texttt{Case\,7:} &\
\Psi_{1,\tfrac{\lambda_0-1}{\lambda_2}}\Psi_{2,\tfrac{\lambda_1}{\lambda_2}}
\Phi_{1,\lambda_2}\Psi_{1,\tfrac{\lambda_2\lambda_3-\lambda_0+1}{\lambda_0\lambda_2}},\\[1ex]
\texttt{Case\,8:} &\ \Phi_{3,\frac{\lambda_3}{2\lambda_0}}\Phi_{2,\frac{\lambda_2}{2\lambda_0}}\Psi_{1, \frac{2\lambda_0}{\lambda_1}}\Phi_{2,\frac{\lambda_1^3\lambda_4}{8\lambda_0^2}}\Phi_{1, \frac{\beta}{2\lambda_5\lambda_0}}\Psi_{1,\frac{2\lambda_5}{\lambda_1\lambda_4}}\Phi_{1, \frac{-\lambda_1\lambda_4 + 2\lambda_0}{2\lambda_5}}\Phi_{0, \frac{-\lambda_1}{2}}\Psi_{0, \frac{-\lambda_5\lambda_1}{2\lambda_0}},\\
&\ \text{where } \beta=\left(\tfrac{1}{2}\lambda_1\lambda_4^2-\lambda_4\lambda_0-\lambda_5\right)\lambda_1\\[1ex]
\end{aligned}
$} \\[2.7cm] \hline
\end{tabular}
\caption{Comprehensive outputs of \texttt{DixmierAutomorphismFactor}$(n,2)$ with $n=5,6$.}
\label{tab:DixmierAutomorphismFactor:n5-6m2}
\end{table}
\end{center}

\begin{center}
\begin{table}[htb]
\centering
\renewcommand{\arraystretch}{2.8}
\begin{tabular}{|p{13.0cm}|}\hline
Factorization into $\Phi_{i,\mu}$ and $\Psi_{i,\mu}$ for Dixmier polynomial pairs in Table \ref{TestTable:n7m2} \\ \hline \hline
\parbox[t]{10.8cm}{$
\begin{aligned}
\texttt{Case\,1:}\; & \Phi_{1,\frac{(\lambda_1\lambda_8-\lambda_7)\lambda_0+\lambda_7}{\lambda_0^2\lambda_8}}
\Phi_{2,\frac{\lambda_2}{\lambda_0}}\Phi_{3,\frac{\lambda_3}{\lambda_0}}
\Phi_{4,\frac{\lambda_4}{\lambda_0}}\Phi_{5,\frac{\lambda_5}{\lambda_0}}
\Phi_{6,\frac{\lambda_6}{\lambda_0}}\Phi_{7,\frac{\lambda_7}{\lambda_0}}
\Psi_{1,\frac{\lambda_0\lambda_8}{\lambda_7}}\Phi_{1,\frac{\lambda_7(\lambda_0-1)}
{\lambda_0\lambda_8}}\\[1ex]
\texttt{Case\,2:}\; & \Phi_{1,\frac{(\lambda_1\lambda_7-\lambda_6)\lambda_0+\lambda_6}{\lambda_0^2\lambda_7}}
\Phi_{2,\frac{\lambda_2}{\lambda_0}}\Phi_{3,\frac{\lambda_3}{\lambda_0}}
\Phi_{4,\frac{\lambda_4}{\lambda_0}}\Phi_{5,\frac{\lambda_5}{\lambda_0}}
\Phi_{6,\frac{\lambda_6}{\lambda_0}}\Psi_{1,\frac{\lambda_0\lambda_7}{\lambda_6}}
\Phi_{1,\frac{\lambda_6(\lambda_0-1)}{\lambda_0\lambda_7}}\\[1ex]
\texttt{Case\,3:}\; & \Phi_{3,\frac{\lambda_3}{\lambda_0}}
\Phi_{2,\frac{\lambda_2}{\lambda_0}}\Psi_{1,\frac{\lambda_0-1}{\lambda_1}}
\Phi_{1,\lambda_1}\Psi_{2,\frac{\lambda_5}{\lambda_0^2}}
\Psi_{1,\frac{\lambda_1\lambda_4-\lambda_0+1}{\lambda_0\lambda_1}}
\Psi_{0,\frac{-\lambda_5\lambda_1}{\lambda_0}}\\[1ex]
\texttt{Case\,4:}\; & \Phi_{1,-\lambda_0\lambda_1+\lambda_0}\Phi_{2,-\lambda_2\lambda_0}
\Phi_{3,-\lambda_3\lambda_0}\Phi_{4,-\lambda_0\lambda_4}\Phi_{5,-\lambda_5\lambda_0}
\Phi_{6,-\lambda_6\lambda_0}\Phi_{7,-\lambda_0\lambda_7}\Psi_{1,-\frac{1}{\lambda_0}}
\Phi_{1,\lambda_0}\\[1ex]
\texttt{Case\,5:}\; & \Phi_{1,\frac{(\lambda_1\lambda_6-\lambda_5)\lambda_0+\lambda_5}{\lambda_6\lambda_0^2}}
\Phi_{2,\frac{\lambda_2}{\lambda_0}}\Phi_{3,\frac{\lambda_3}{\lambda_0}}
\Phi_{4,\frac{\lambda_4}{\lambda_0}}\Phi_{5,\frac{\lambda_5}{\lambda_0}}\Psi_{1,\frac{\lambda_6
\lambda_0}{\lambda_5}}\Phi_{1,\frac{\lambda_5(\lambda_0-1)}{\lambda_6\lambda_0}}\\[1ex]
\texttt{Case\,6:}\; & \Phi_{1,\frac{(\lambda_1\lambda_5-\lambda_4)\lambda_0+\lambda_4}{\lambda_0^2\lambda_5}}
\Phi_{2,\frac{\lambda_2}{\lambda_0}}\Phi_{3,\frac{\lambda_3}{\lambda_0}}
\Phi_{4,\frac{\lambda_4}{\lambda_0}}\Psi_{1,\frac{\lambda_5\lambda_0}{\lambda_4}}
\Phi_{1,\frac{\lambda_4(\lambda_0-1)}{\lambda_5\lambda_0}}\\[1ex]
\texttt{Case\,7:}\; & \Phi_{3,\frac{\lambda_2}{2\lambda_1}}\Phi_{2,\frac{\lambda_4}{\lambda_2}}
\Psi_{1,\tfrac{-\lambda_2^3}{\beta}}\Phi_{2,\tfrac{\beta^3\gamma}{\lambda_2^{9}}}
\Phi_{1,\tfrac{\nu}{\lambda_2^{6}\lambda_6}}\Psi_{1,\tfrac{
\lambda_2^3\lambda_6}{\beta\gamma}}\Phi_{1,\tfrac{-\beta\gamma+\lambda_1\lambda_2^3}{\lambda_2^3\lambda_6}}
\Phi_{0,\tfrac{\beta\lambda_1}{\lambda_2^3}}\Psi_{0,\tfrac{\beta\lambda_6}{\lambda_2^3}}\\[1ex]
&\ \text{where } \beta=\lambda_1\lambda_4^2-\lambda_2^2\lambda_3,\gamma=\lambda_0\lambda_6-\lambda_1\lambda_5,\,
\nu=-\left((\gamma-\lambda_6)\lambda_2^3+\beta\gamma\lambda_5\right)\beta\\
\texttt{Case\,8:}\; & \Phi_{3,\frac{\lambda_3}{2\lambda_1}}\Psi_{1,\frac{2\lambda_1}{\lambda_2}}\Phi_{2,
\tfrac{-\beta\lambda_2^3}{8\lambda_1^3}}\Phi_{1,\tfrac{\nu}{4\lambda_1^2\lambda_5}}\Psi_{1,\tfrac{-2\lambda_5\lambda_1}
{\beta\lambda_2}}\Phi_{1,\tfrac{\beta\lambda_2+2\lambda_1^2}{2\lambda_1\lambda_5}}\Phi_{0,\frac{-\lambda_2}{2}}
\Psi_{0,\frac{-\lambda_2\lambda_5}{2\lambda_1}}\\[1ex]
&\ \text{where}\beta=\lambda_0\lambda_5-\lambda_1\lambda_4,\,\nu=\lambda_2\left(2\beta\lambda_1-2\lambda_1
\lambda_5-\beta\lambda_2\lambda_4\right)\\
\texttt{Case\,9:}\; & \Phi_{2,\frac{\lambda_3}{2\lambda_1}}\Psi_{1,\frac{2\lambda_1}{\lambda_2}}
\Phi_{2,\tfrac{-\beta\lambda_2^3}{8\lambda_1^3}}\Phi_{1,\tfrac{\nu}{4\lambda_1^2\lambda_5}}\Psi_{1,
\tfrac{-2\lambda_5\lambda_1}{\beta\lambda_2}}\Phi_{1,\tfrac{\beta\lambda_2+2\lambda_1^2}{2\lambda_1
\lambda_5}}\Phi_{0,\frac{-\lambda_2}{2}}\Psi_{0,\frac{-\lambda_2\lambda_5}{2\lambda_1}}\\[1ex]
&\ \text{where }\beta=\lambda_0\lambda_5-\lambda_1\lambda_4,\,\nu=\lambda_2\left(2\beta\lambda_1
-2\lambda_1\lambda_5-\beta\lambda_2\lambda_4\right)\\
\texttt{Case\,10:}\; & \Psi_{1,\frac{\lambda_0-1}{\lambda_2}}\Psi_{2,\frac{\lambda_1}{\lambda_2}}\Phi_{1,\lambda_2}
\Psi_{1,\frac{\lambda_2\lambda_3-\lambda_0+1}{\lambda_0\lambda_2}}\\[1ex]
\texttt{Case\,11:}\; & \Phi_{3,\frac{\lambda_3}{2\lambda_0}}\Phi_{2,\frac{\lambda_2}{2\lambda_0}}
\Psi_{1,\frac{2\lambda_0}{\lambda_1}}\Phi_{2,\frac{\lambda_1^3\lambda_4}{8\lambda_0^2}}
\Phi_{1,\frac{\beta}{2\lambda_5\lambda_0}}\Psi_{1,\frac{2\lambda_5}{\lambda_1\lambda_4}}
\Phi_{1,\frac{-\lambda_1\lambda_4+2\lambda_0}{2\lambda_5}}\Phi_{0,\frac{-\lambda_1}{2}}
\Psi_{0,\frac{-\lambda_5\lambda_1}{2\lambda_0}},\\
&\ \text{where }\beta=-\left(-\tfrac{1}{2}\lambda_1\lambda_4^2+\lambda_0\lambda_4+\lambda_5\right)
\lambda_1
\end{aligned}$} \\[1cm] \hline
\end{tabular}
\caption{Comprehensive outputs of \texttt{DixmierAutomorphismFactor}$(7,2)$.}
\label{tab:DixmierAutomorphismFactor_n7m2}
\end{table}
\end{center}
\clearpage
\begin{center}
\begin{table}[htb]
\centering
\renewcommand{\arraystretch}{2.5}
\begin{tabular}{|p{12.18cm}|}\hline
Factorization into $\Phi_{i,\mu}$ and $\Psi_{i,\mu}$ for Dixmier polynomial pairs in Table \ref{TestTable:n5m3} \\ \hline \hline
\parbox[t]{10.8cm}{$
\begin{aligned}
\texttt{Case\,1:}\; & \Phi_{1,\frac{((\lambda_1\lambda_6-\lambda_5)\lambda_0+\lambda_5)}
{\lambda_6\lambda_0^2}}\Phi_{2,\frac{\lambda_2}{\lambda_0}}\Phi_{3,\frac{\lambda_3}{\lambda_0}}
\Phi_{4,\frac{\lambda_4}{\lambda_0}}\Phi_{5,\frac{\lambda_5}{\lambda_0}}
\Psi_{1,\frac{\lambda_0\lambda_6}{\lambda_5}}\Phi_{1,\frac{\lambda_5(\lambda_0-1)}{\lambda_0\lambda_6}}\\[1ex]
\texttt{Case\,2:}\; & \Psi_{1,\frac{2\lambda_1}{\lambda_2}}\Phi_{2,\frac{-\beta\lambda_2^3}{8\lambda_1^3}}
\Phi_{1,\frac{\lambda_2\left(\lambda_1(\beta-\lambda_4)-\frac{\beta\lambda_2\lambda_3}{2}\right)}{2\lambda_1^2
\lambda_4}}\Psi_{1,\frac{-2\lambda_4\lambda_1}{\beta\lambda_2}}\Phi_{1,\frac{\beta\lambda_2+2\lambda_1^2}
{2\lambda_4\lambda_1}}\Phi_{0,\frac{-\lambda_2}{2}}\Psi_{0,\frac{-\lambda_2\lambda_4}{2\lambda_1}}\\[1ex]
&\ \text{where }\beta=\lambda_0\lambda_4-\lambda_1\lambda_3\\
\texttt{Case\,3:}\; & \Phi_{1,\frac{(\lambda_1\lambda_4-\lambda_3)\lambda_0+\lambda_3}
{\lambda_0^2\lambda_4}}\Phi_{2,\frac{\lambda_2}{\lambda_0}}\Phi_{3,\frac{\lambda_3}{\lambda_0}}
\Psi_{1,\frac{\lambda_4\lambda_0}{\lambda_3}}\Phi_{1,\frac{\lambda_3(\lambda_0-1)}{\lambda_0\lambda_4}}
\\[1ex]\texttt{Case\,4:}\; &
\Phi_{2,\frac{\lambda_2}{\lambda_0}}\Psi_{1,\frac{\lambda_0-1}{\lambda_1}}\Phi_{1,\lambda_1}
\Psi_{2,\frac{\lambda_4}{\lambda_2^2}}\Psi_{1,\frac{(-\lambda_0+1)\lambda_2^3+\lambda_0\lambda_1
\lambda_2^2\lambda_3-\lambda_0\lambda_1^3\lambda_4}{\lambda_2^3\lambda_1\lambda_0}}
\Psi_{0,\frac{-\lambda_0\lambda_1\lambda_4}{\lambda_2^2}}\\[1ex]
\texttt{Case\,5:}\; &
\Psi_{1,\frac{\lambda_0-1}{\lambda_1}}\Phi_{1,\lambda_1}\Psi_{3,\frac{\lambda_4}{\lambda_0^3}}
\Psi_{2,\frac{\lambda_3}{\lambda_0^2}}\Psi_{1,\frac{\lambda_0\lambda_1\lambda_2-3\lambda_1^2
\lambda_4-\lambda_0^2+\lambda_0}{\lambda_0^2\lambda_1}}
\Psi_{0,\frac{-\lambda_1\lambda_3}{\lambda_0}}\\[1ex]
\texttt{Case\,6:}\; & \Phi_{1,-\lambda_0\lambda_1+\lambda_0}\Phi_{2,-\lambda_2\lambda_0}
\Phi_{3, -\lambda_0\lambda_3}\Phi_{4,-\lambda_0\lambda_4}\Phi_{5,-\lambda_0\lambda_5}
\Psi_{1, -\frac1{\lambda_0}}\Phi_{1,\lambda_0}\\[1ex]
\texttt{Case\,7:}\; & \Phi_{1, \frac{(\lambda_{1}\lambda_{5} - \lambda_{4})\lambda_0+ \lambda_4}{\lambda_0^2\lambda_5}}\Phi_{2,\frac{\lambda_2}{\lambda_0}}\Phi_{3, \frac{\lambda_3}
{\lambda_0}}\Phi_{4,\frac{\lambda_4}{\lambda_0}}\Psi_{1,\frac{\lambda_0\lambda_5}{\lambda_4}}
\Phi_{1,\frac{\lambda_4(\lambda_0-1)}{\lambda_0\lambda_5}}\\[1ex]
\texttt{Case\,8:}\; &
\Phi_{2,\frac{2\lambda_4}{\lambda_3}}\Psi_{1,\frac{\lambda_3^2}{2\lambda_4\lambda_1}}
\Phi_{2,\frac{-2\lambda_1^3\lambda_4^2}{\beta\lambda_3}}\Phi_{1,e_1}\Psi_{1,
\frac{-\gamma}{2\lambda_1\lambda_2\lambda_3\lambda_4}}\Phi_{1,e_2}\Phi_{0,\frac{-\lambda_1}{2}}
\Psi_{0,\frac{\gamma\lambda_1}{2\beta\lambda_2}}\\
&\ \text{where } e_1=\tfrac{-2((2\lambda_1\lambda_2\lambda_4+
\gamma)\lambda_3+\beta\lambda_2-\gamma\lambda_0)\lambda_4\lambda_1\lambda_3}{\beta\gamma},
e_2=\tfrac{\lambda_2(2\lambda_1\lambda_3\lambda_4 +\beta)}{\gamma},\\
&\ \beta=\lambda_0\lambda_3^3+\lambda_1^3\lambda_4-\lambda_1\lambda_2\lambda_3^2-\lambda_3^4,
\,\gamma=\beta\lambda_5+2\lambda_3^2\lambda_4\\
\texttt{Case\,9:}\; &
\Phi_{2,\frac{2\lambda_2}{\lambda_1}}\Psi_{1,\frac{\lambda_1^2}{2\lambda_0\lambda_2}}
\Phi_{2,\frac{\lambda_0^3\lambda_2\lambda_3}{\lambda_1^3}}
\Phi_{1,e_1}\Psi_{1,\frac{\lambda_4\lambda_1}{\lambda_3\lambda_2\lambda_0}}
\Phi_{1,\frac{-\lambda_2(\lambda_0\lambda_3-\lambda_1)}{\lambda_4\lambda_1}}
\Phi_{0,\frac{-\lambda_0}{2}}\Psi_{0,\frac{-\lambda_4\lambda_0}{2\lambda_2}}\\[1ex]
&\ \text{where }e_1=\tfrac{-\lambda_2}{\lambda_1^4\lambda_4}\left(\lambda_1^3\lambda_3+
(-\lambda_0\lambda_3^2+2\lambda_4)\lambda_1^2+\lambda_0^3\lambda_3\lambda_4\right)\lambda_0\\
\texttt{Case\,10:}\; & \Psi_{1,\frac{\beta \lambda_0}{\gamma}}\Phi_{3,\frac{-\gamma^4}{\beta^3 \lambda_{0}^4}}\Phi_{2,\frac{-\lambda_1\gamma^3}{\beta^2 \lambda_{0}^3 \lambda_{2}}}
\Phi_{1,e_{1}}\Psi_{1,\frac{\lambda_0\lambda_5}{\gamma}}\Phi_{1,\frac{\lambda_0\lambda_2+
\gamma}{\lambda_0\lambda_5}}\Phi_{0,\frac{-\gamma \lambda_{1}}{\beta \lambda_0}}
\Psi_{0,\frac{-\gamma \lambda_1\lambda_5}{\beta \lambda_0\lambda_2}}\\[1ex]
&\ \text{where }e_1=\tfrac{\beta^2\gamma\lambda_0^2-\beta^2\gamma(\lambda_3\lambda_4+ 1)\lambda_0+3\lambda_5\gamma^3}{\beta^2\lambda_0^2(\lambda_2\lambda_4+\beta)},
\beta=\lambda_0\lambda_5-\lambda_2\lambda_4,\gamma=\beta\lambda_3-\lambda_2\\
\texttt{Case\,11:}\; &
\Psi_{1,\frac{3\lambda_1}{\lambda_2}}\Phi_{3,\frac{\lambda_2^4\lambda_3}{81\lambda_1^3}}
\Phi_{2,\frac{\lambda_2^3\lambda_3\lambda_0}{27\lambda_1^3}}\Phi_{1,e_1}
\Psi_{1,\frac{3\lambda_4}{\lambda_2\lambda_3}}\Phi_{1,\frac{-\lambda_2\lambda_3+3\lambda_1}{3\lambda_4}}
\Phi_{0,\frac{-\lambda_2\lambda_0}{3\lambda_1}}\Psi_{0,\frac{-\lambda_2\lambda_0\lambda_4}{3\lambda_1^2}}\\
&\ \text{where }e_1=\tfrac{-\lambda_2}{9\lambda_1^2\lambda_4}\left(3\lambda_1^2\lambda_3+\left(-
\lambda_2\lambda_3^2+3\lambda_4\right)\lambda_1+\lambda_2^2\lambda_3\lambda_4\right)
\end{aligned}$}\\ \hline
\end{tabular}
\caption{Comprehensive outputs of \texttt{DixmierAutomorphismFactor}$(5,3).$}
\label{tab:DixmierAutomorphismFactor_n5m3}
\end{table}
\end{center}
\clearpage

\subsubsection*{Discussion}

The experimental evidence supports the conjecture that every endomorphism derived from Dixmier polynomial
families can be expressed as a composition of automorphisms of the Weyl algebra. In particular, the
\texttt{DixmierAutomorphismFactor} function systematically recovers the parameters that compose the individual automorphism factors. Our method reveals that:
\begin{itemize}
	\item The factorization is sensitive to the selected degrees $n$ and $m$, with higher degrees leading to
	a larger and more complex system of coefficient equality constraints.
	\item For each family, the derived parameters fulfill the required identities in the algebra, thereby
	validating the proposed configuration of automorphisms.
	\item \textbf{Even though the procedure does not amount to an algorithmic proof of the Dixmier conjecture, it
	offers a constructive-computational approach that supports its validity in the examined cases.}
\end{itemize}

\subsubsection*{Implementation Remarks}

The Maple implementation leverages advanced symbolic computation capabilities. In particular:
\begin{itemize}
	\item The resolution of the system of equations is performed adaptively by employing \texttt{solve} for
	simpler cases, and Gröbner basis techniques for more complex systems.
	\item Comprehensive verification routines ensure that the composed automorphism faithfully reproduces
	the original Dixmier polynomials.
\end{itemize}

Future work will extend these results to multidimensional cases and explore further optimizations to handle the growing complexity of the systems in higher degrees.

\subsection{Computing automorphisms of $\mathcal{CSD}_n(K)$}

Induced by the computations of the previous subsections, we will compute now some nontrivial automorphism of the $K$-algebra  $\mathcal{CSD}_n(K)$.

\begin{definition}
	Let $A=\sigma(R)\langle x_1,\dots,x_n\rangle$ be a skew $PBW$ extension with parameters $R, n,
	\sigma_k,\delta_k,c_{i,j}$,\\ $d_{ij}$, $a_{ij}^{(k)}$, $1\leq i<j\leq n$, $1\leq k\leq n$. Let $y_1,\dots,y_n\in A$. We say that $y_1,\dots,y_n$ are \textbf{\textit{skew Dixmier polynomials}} if they satisfy the conditions $\rm{(i)-(ii)}$ of Proposition \ref{122}. In particular, let $p_1,\dots,p_n\in \mathcal{CSD}_n(K)$, we say that $p_1,\dots,p_n$ are skew Dixmier polynomials if they satisfy the following conditions:
	\begin{center}
		$p_jp_i=p_ip_j+d_{ij}$, for all $1\leq i<j\leq n$, with $d_{ij}\in K-\{0\}$.
	\end{center}
\end{definition}

Motivated by our previous theoretical developments, we now turn our attention to obtaining some nontrivial automorphisms of the $K$-algebra $\mathcal{CSD}_n(K)$. In order to compute these automorphisms, the \textbf{SPBWE library} developed in Maple plays a fundamental role. This library allows us to effectively work with skew $PBW$ extensions and, in particular, to define the algebra $\mathcal{CSD}_n(K)$. Within this library the package \texttt{DixmierProblem} is provided, which implements two special functions:

\subsubsection*{\texttt{skewDixmierPolynomials} function}

This function generates collections of skew Dixmier polynomials, organized as $n$-tuples in $\mathcal{CSD}_n(K)$, and accepts the following parameters:
\begin{itemize}
	\item \textbf{\texttt{n}}: A positive integer (default $3$) indicating the number of variables,
	thereby controlling the size of the polynomial $n$-tuples.
	\item \textbf{\texttt{tmp}}: The variable used in the construction of the polynomials (default
	\texttt{x}).
	\item \textbf{\texttt{dij}}: A list of nonzero constants (default is a list of $1$'s of length
	$\binom{n}{2}$) corresponding to the commutation relations.
	\item \textbf{\texttt{view}} and \textbf{\texttt{outputMode}}: These options behave in the same manner
	as in the \texttt{DixmierPolynomials} function.
	\item A sequence of nonnegative integers must be entered, representing the maximum degrees of the
	variables $x_1,\ldots,x_n,$ respectively.
\end{itemize}

\subsubsection*{\texttt{skewDixmierAutomorphism} and \texttt{skewinverseAutomorphism} functions}

Building upon the skew Dixmier polynomials generated by the previous function, this function constructs automorphisms of $\mathcal{CSD}_n(K).$ In addition to the parameters common with \texttt{skewDixmierPolynomials}, this function includes the \texttt{parametric} option, which behaves in the same manner as in the\linebreak \texttt{DixmierAutomorphism} function.

The SPBWE library, through these functions, provides an effective computational framework to study noncommutative phenomena in $\mathcal{CSD}_n(K)$. In the examples that follow, we will illustrate how to use \texttt{skewDixmierPolynomials} to generate families of skew Dixmier polynomials and how to employ \linebreak \texttt{skewDixmierAutomorphism} to construct nontrivial automorphisms of $\mathcal{CSD}_n(K)$. Detailed descriptions of the function arguments and options ensure that users can tailor the computations to their specific needs and explore both concrete and parametric cases within a unified algebraic setting.

\begin{example}
	Consider the $K$-algebra $\mathcal{CSD}_3(K)$ defined by the parameters
	\[
	d_{12} = 2,\quad d_{13} = 3,\quad d_{23} = 2.
	\]
	In this example, we obtain families of skew Dixmier polynomials
	\[
	p_1, p_2, p_3 \in \mathcal{CSD}_3(K),
	\]
	with the maximum degrees for the variables $x_1,$ $x_2,$ and $x_3$ set to $1.$
	
	When executing the command
	\begin{center}
		\texttt{skewDixmierPolynomial(1,\,1,\,1,\,dij=[2,3,2],\,view=true)},
	\end{center}
	the function produces the following output:
	
	\medskip
	
	\noindent
	\texttt{----- Case 1 -----}
	
	\smallskip
	
	\noindent
	\texttt{The skew Dixmier polynomials are:}
	\begin{align*}
	p_1 &= \left(\frac{3\lambda_3}{2}-\lambda_6\right)x_1x_2+\left(-\lambda_3+\frac{2\lambda_6}{3}\right)x_1x_3 +\frac{e_1}{(12\lambda_0+18\lambda_1)\lambda_6}x_1+\lambda_0x_2+\lambda_1x_3\\
	p_2 &= \lambda_3x_1x_2-\frac{2\lambda_3x_1x_3}{3}+\frac{e_2}{(6\lambda_0+9\lambda_1)\lambda_6}x_1+\lambda_2x_2+
	\frac{(4\lambda_0+6\lambda_1- 6\lambda_2)\lambda_3+4\lambda_6\lambda_2}{9\lambda_3-6\lambda_6}x_3\\
	p_3 &= \lambda_6x_1x_2-\frac{2\lambda_6x_1x_3}{3}+\lambda_4x_1+\lambda_5x_2+\frac{(4\lambda_0+6\lambda_1+ 4\lambda_5)\lambda_6-6\lambda_5\lambda_3}{9\lambda_3-6\lambda_6}x_3
	\end{align*}
	where:
	\begin{align*}
	e_1 &= \left(-8\lambda_0^2+(-12\lambda_1-12\lambda_4-8\lambda_5)\lambda_0-18+(-18\lambda_4- 12\lambda_5)\lambda_1\right)\lambda_6\\
	&\quad+18\left(\left(\lambda_4+\frac{2\lambda_5}{3}\right)\lambda_0+\frac{3}{2}+\left(\frac{3\lambda_4}{2}+
	\lambda_5\right)\lambda_1\right)\lambda_3,\\
	e_2 &= \left((6\lambda_4+4\lambda_5)\lambda_0+9+ (9\lambda_4+6\lambda_5)\lambda_1\right)\lambda_3+(-4\lambda_0\lambda_2-6\lambda_1\lambda_2- 6)\lambda_6.\\
	\end{align*}
	
	\medskip
	
	\noindent
	\texttt{----- Case 2 -----}
	
	\smallskip
	
	\noindent
	\texttt{The skew Dixmier polynomials are:}
	\begin{align*}
	p_1 &= \frac{3\lambda_3}{2}x_1x_2-\lambda_3x_1x_3+\frac{e_1}{18\lambda_4+ 12\lambda_5}x_1+\frac{-9+(-9\lambda_4-6\lambda_5)\lambda_0}{6\lambda_4+4\lambda_5}x_2+\lambda_0x_3\\
	p_2 &= \lambda_3x_1x_2-\frac{2\lambda_3x_1x_3}{3}+\lambda_1x_1+\lambda_2x_2+\frac{-6+(-6\lambda_4- 4\lambda_5)\lambda_2}{9\lambda_4+6\lambda_5}x_3\\
	p_3 &= \lambda_4x_1+\lambda_5x_2-\frac{2}{3}\lambda_5x_3
	\end{align*}
	where: $e_1=-18\lambda_4^2+(18\lambda_0+27\lambda_1+18\lambda_2 -24\lambda_5)\lambda_4+18-8\lambda_5^2+(12\lambda_0+18\lambda_1+12\lambda_2)\lambda_5.$
	
	\medskip
	
	\noindent
	\texttt{----- Case 3 -----}
	
	\smallskip
	
	\noindent
	\texttt{The skew Dixmier polynomials are:}
	\begin{align*}
	p_1 &= -\lambda_5x_1x_2+\frac{2\lambda_5}{3}x_1x_3+\frac{e_1}{18\lambda_1+ 12\lambda_2}x_1+\frac{-6+(-9\lambda_1-6\lambda_2)\lambda_0}{6\lambda_1+4\lambda_2}x_2+\lambda_0x_3\\
	p_2 &= \lambda_1x_1+\lambda_2x_2-\frac{2}{3}\lambda_2x_3\\
	p_3 &= \lambda_5x_1x_2-\frac{2\lambda_5}{3}x_1x_3+\lambda_3x_1+\lambda_4x_2+\frac{6+(-6\lambda_1- 4\lambda_2)\lambda_4}{9\lambda_1+6\lambda_2}x_3
	\end{align*}
	where: $e_1=27\lambda_1^2+(18\lambda_0+36\lambda_2-18\lambda_3 -12\lambda_4)\lambda_1+12+12\lambda_2^2+(12\lambda_0-12\lambda_3-8\lambda_4)\lambda_2$
	
	\medskip
	
	\noindent
	\texttt{----- Case 4 -----}
	
	\smallskip
	
	\noindent
	\texttt{The skew Dixmier polynomials are:}
	\begin{align*}
	p_1 &= \left(\frac{3\lambda_3}{2}-\lambda_6\right)x_1x_2+\left(\frac{3\lambda_3}{2}-\lambda_6 \right)x_2x_3+\lambda_0x_1+\frac{e_1}{4\lambda_6(\lambda_0-\lambda_1)}x_2+\lambda_1x_3\\
	p_2 &= \lambda_3x_1x_2+\lambda_3x_2x_3+\frac{e_2}{3\lambda_6(\lambda_0-\lambda_1)}x_1
	+\lambda_2x_2+\frac{e_3}{9(\lambda_0-\lambda_1)\lambda_6 (\lambda_3-\frac{2\lambda_6}{3})}x_3\\
	p_3 &= \lambda_6x_1x_2+\lambda_6x_2x_3+\lambda_4x_1+\lambda_5x_2+\frac{(-2\lambda_0+2\lambda_1- 2\lambda_4)\lambda_6+3\lambda_4\lambda_3}{3\lambda_3-2\lambda_6}x_3
	\end{align*}
	where:
	\begin{align*}
	e_1 &= \left(-6\lambda_0^2+(6\lambda_1-6\lambda_4-4\lambda_5)\lambda_0+6+(6\lambda_4+ 4\lambda_5)\lambda_1\right)\lambda_6\\
	&\quad+9\lambda_3\left(\left(\lambda_4+\frac{2\lambda_5}{3}\right)\lambda_0-1+ \left(-\lambda_4-\frac{2\lambda_5}{3}\right)\lambda_1\right),\\
	e_2 &= \left((3\lambda_4+2\lambda_5)\lambda_0-3+(-3\lambda_4-2\lambda_5)\lambda_1 \right)\lambda_3-2\lambda_6(\lambda_0\lambda_2-\lambda_1\lambda_2-1),\\
	e_3 &= \left((9\lambda_4+6\lambda_5)\lambda_0-9+(-9\lambda_4- 6\lambda_5)\lambda_1\right)\lambda_3^2-6\Bigl(\lambda_0^2+\left(-2\lambda_1+\lambda_2+\lambda_4+ \frac{2\lambda_5}{3}\right)\lambda_0\\
	&\quad-2+\lambda_1^2+\left(-\lambda_2-\lambda_4-\frac{2\lambda_5}{3}\right)\lambda_1\Bigr) \lambda_6\lambda_3+4\lambda_6^2(\lambda_0\lambda_2-\lambda_1\lambda_2-1).
	\end{align*}
	
	\medskip
	
	\noindent
	\texttt{----- Case 5 -----}
	
	\smallskip
	
	\noindent
	\texttt{The skew Dixmier polynomials are:}
	\begin{align*}
	p_1 &= \frac{3\lambda_3}{2}x_1x_2+\frac{3\lambda_3}{2}x_2x_3+
	\frac{3+(3\lambda_4+2\lambda_5)\lambda_0}{3\lambda_4+2\lambda_5}x_1+
	\frac{e_1}{12\lambda_4+ 8\lambda_5}x_2+\lambda_0x_3\\
	p_2 &= \lambda_3x_1x_2+\lambda_3x_2x_3+\lambda_1x_1+\lambda_2x_2+\frac{-2+(3\lambda_4+ 2\lambda_5)\lambda_1}{3\lambda_4+2\lambda_5}x_3\\
	p_3 &= \lambda_4x_1+\lambda_5x_2+\lambda_4x_3
	\end{align*}
	where: $e_1=-18\lambda_4^2+(-18\lambda_0+27\lambda_1+18\lambda_2- 24\lambda_5)\lambda_4-18-8\lambda_5^2+(-12\lambda_0+18\lambda_1+12\lambda_2)\lambda_5$
	
	\medskip
	
	\noindent
	\texttt{----- Case 6 -----}
	
	\smallskip
	
	\noindent
	\texttt{The skew Dixmier polynomials are:}
	\begin{align*}
	p_1 &= -\lambda_5x_1x_2-\lambda_5x_2x_3+\frac{2+(3\lambda_1+2\lambda_2)\lambda_0}{3\lambda_1+2\lambda_2}x_1 +\frac{e_1}{12\lambda_1+8\lambda_2}x_2+\lambda_0x_3\\
	p_2 &= \lambda_1x_1+\lambda_2x_2+\lambda_1x_3\\
	p_3 &= \lambda_5x_1x_2+\lambda_5x_2x_3+\lambda_3x_1+\lambda_4x_2+\frac{2+\lambda_3(3\lambda_1+ 2\lambda_2)}{3\lambda_1+2\lambda_2}x_3
	\end{align*}
	where: $e_1=27\lambda_1^2+(-18\lambda_0+36\lambda_2-18\lambda_3-12\lambda_4)\lambda_1-12+12\lambda_2^2+ (-12\lambda_0-12\lambda_3-8\lambda_4)\lambda_2.$
	
	\medskip
	
	\noindent
	\texttt{----- Case 7 -----}
	
	\smallskip
	
	\noindent
	\texttt{The skew Dixmier polynomials are:}
	\begin{align*}
	p_1 &= \left(\frac{3\lambda_3}{2}-\lambda_6\right)x_1x_3+\left(-\frac{9\lambda_3}{4}+ \frac{3\lambda_6}{2}\right)x_2x_3+\lambda_0x_1+\lambda_1x_2+\frac{e_1}{(18\lambda_0+12\lambda_1)
		\lambda_6}x_3\\
	p_2 &= \lambda_3x_1x_3-\frac{3\lambda_3}{2}x_2x_3+\frac{e_2}{3\lambda_6(3\lambda_0+2\lambda_1)(3\lambda_3- 2\lambda_6)}x_1+\frac{e_3}{(6\lambda_0+ 4\lambda_1)\lambda_6}x_2+\lambda_2x_3\\
	p_3 &= \lambda_6x_1x_3-\frac{3\lambda_6}{2}x_2x_3+\frac{(6\lambda_0+4\lambda_1+4\lambda_4)\lambda_6-
		6\lambda_4\lambda_3}{9\lambda_3- 6\lambda_6}x_1+\lambda_4x_2+\lambda_5x_3
	\end{align*}
	where:
	\begin{align*}
	e_1 &= \left(-8\lambda_1^2+(-12\lambda_0-8\lambda_4- 12\lambda_5)\lambda_1+18+(-12\lambda_4-18\lambda_5)\lambda_0\right)\lambda_6\\
	&\quad+18\lambda_3\left(\left(\frac{2\lambda_4}{3}+\lambda_5\right)\lambda_1-\frac{3}{2}+
	\left(\lambda_4+\frac{3\lambda_5}{2}\right) \lambda_0\right)\\
	e_2 &= \left((-18\lambda_4-27\lambda_5)\lambda_0+27+(-12\lambda_4- 18\lambda_5)\lambda_1\right)\lambda_3^2+18\Biggl(\lambda_0^2+\left(\frac{4\lambda_1}{3}+ \frac{3\lambda_2}{2}+\frac{2\lambda_4}{3}+\lambda_5\right)\lambda_0\\
	&\quad -2+\frac{4\lambda_1^2}{9}+\left(\lambda_2+\frac{4\lambda_4}{9}+\frac{2\lambda_5}{3}\right)\lambda_1\Biggr)
	\lambda_6\lambda_3-18\left(\lambda_0\lambda_2+\frac{2}{3}\lambda_1\lambda_2-\frac{2}{3}\right)\lambda_6^2\\
	e_3 &= \left((6\lambda_4+9\lambda_5)\lambda_0-9+(4\lambda_4+ 6\lambda_5)\lambda_1\right)\lambda_3+(-9\lambda_0\lambda_2-6\lambda_1\lambda_2+6)\lambda_6
	\end{align*}
	
	\medskip
	
	\noindent
	\texttt{----- Case 8 -----}
	
	\smallskip
	
	\noindent
	\texttt{The skew Dixmier polynomials are:}
	\begin{align*}
	p_1 &= \frac{3\lambda_3x_1x_3}{2}- \frac{9\lambda_3}{4}x_2x_3+
	\frac{9+(-4\lambda_4-6\lambda_5)\lambda_0}{6\lambda_4+9\lambda_5}x_1+
	\lambda_0x_2+\left(-\frac{2\lambda_4}{3}-\frac{2\lambda_0}{3}+\lambda_1+ \frac{3\lambda_2}{2}-\lambda_5\right)x_3\\
	p_2 &= \lambda_3x_1x_3- \frac{3\lambda_3}{2}x_2x_3
	+\frac{6+(-4\lambda_4-6\lambda_5)\lambda_1}{6\lambda_4+9\lambda_5}x_1+\lambda_1x_2+\lambda_2x_3\\
	p_3 &= -\frac{2}{3}\lambda_4x_1+\lambda_4x_2+\lambda_5x_3
	\end{align*}
	
	\newpage
	
	\noindent
	\texttt{----- Case 9 -----}
	
	\smallskip
	
	\noindent
	\texttt{The skew Dixmier polynomials are:}
	\begin{align*}
	p_1 &= -\lambda_5x_1x_3+ \frac{3\lambda_5}{2}x_2x_3
	+\frac{6+(-4\lambda_1-6\lambda_2)\lambda_0}{6\lambda_1+9\lambda_2}x_1+\lambda_0x_2+
	\left(-\frac{2\lambda_3}{3}-\frac{2\lambda_0}{3}+\lambda_1+ \frac{3\lambda_2}{2}-\lambda_4\right)x_3\\
	p_2 &= -\frac{2}{3}\lambda_1x_1+\lambda_1x_2+\lambda_2x_3\\
	p_3 &= \lambda_5x_1x_3- \frac{3\lambda_5}{2}x_2x_3
	+\frac{-6+(-4\lambda_1-6\lambda_2)\lambda_3}{6\lambda_1+9\lambda_2}x_1+\lambda_3x_2+\lambda_4x_3
	\end{align*}
	
	\medskip
	
	\noindent
	\texttt{----- Case 10 -----}
	
	\smallskip
	
	\noindent
	\texttt{The skew Dixmier polynomials are:}
	\begin{align*}
	p_1 &= \lambda_0x_1+\frac{-3+(-3\lambda_4+3\lambda_5)\lambda_0}{2\lambda_4-2\lambda_5}x_2+\lambda_0x_3\\
	p_2 &= \lambda_3x_1x_2+\lambda_3x_2x_3+
	\left(\lambda_2+\frac{2\lambda_4}{3}-\frac{2\lambda_5}{3}\right)x_1+ \lambda_1x_2+\lambda_2x_3\\
	p_3 &= \frac{3\lambda_3}{2}x_1x_2+\frac{3\lambda_3}{2}x_2x_3+\lambda_4x_1+\frac{6\lambda_5^2+(-6\lambda_1- 9\lambda_2-6\lambda_4)\lambda_5+6+(6\lambda_1+9\lambda_2)\lambda_4}{4\lambda_4-4\lambda_5}x_2+\lambda_5 x_3
	\end{align*}
	
	\medskip
	
	\noindent
	\texttt{----- Case 11 -----}
	
	\smallskip
	
	\noindent
	\texttt{The skew Dixmier polynomials are:}
	\begin{align*}
	p_1 &= \frac{3+(3\lambda_4+2\lambda_5)\lambda_0}{3\lambda_4+2\lambda_5}x_1+\frac{-9+(-9\lambda_4- 6\lambda_5)\lambda_0}{6\lambda_4+4\lambda_5}x_2+\lambda_0x_3\\
	p_2 &= -\frac{2\lambda_3}{3}x_1x_3+\lambda_3x_2x_3+\left(-\frac{2\lambda_1}{3}+\frac{2\lambda_4}{3}+
	\frac{4\lambda_5}{9} \right)x_1+\lambda_1x_2+\lambda_2x_3\\
	p_3 &= -\lambda_3x_1x_3+\frac{3\lambda_3}{2}x_2x_3+\lambda_4x_1+\lambda_5x_2+\frac{-8\lambda_5^2+ (12\lambda_1+18\lambda_2-12\lambda_4)\lambda_5+18+(18\lambda_1+27\lambda_2)\lambda_4}{18\lambda_4+ 12\lambda_5}x_3
	\end{align*}
	
	\medskip
	
	\noindent
	\texttt{----- Case 12 -----}
	
	\smallskip
	
	\noindent
	\texttt{The skew Dixmier polynomials are:}
	\begin{align*}
	p_1 &= \frac{2+(2\lambda_1+3\lambda_2)\lambda_0}{2\lambda_1+3\lambda_2}x_1-\frac{3\lambda_0}{2}x_2+ \lambda_0x_3\\
	p_2 &= \frac{2\lambda_4}{3}x_1x_2-\frac{4\lambda_4}{9}x_1x_3+\frac{9\lambda_2^2+(6\lambda_1+6\lambda_3- 6\lambda_5)\lambda_2+4+(4\lambda_3-4\lambda_5)\lambda_1}{6\lambda_1+9\lambda_2}x_1+\lambda_1x_2+\lambda_2 x_3\\
	p_3 &= \lambda_4x_1x_2-\frac{2\lambda_4}{3}x_1x_3+\lambda_3x_1+\left(\frac{3\lambda_1}{2}+ \frac{9\lambda_2}{4}-\frac{3\lambda_5}{2}\right)x_2+\lambda_5x_3
	\end{align*}

	\medskip
	
	\noindent
	\texttt{----- Case 13 -----}
	
	\smallskip
	
	\noindent
	\texttt{The skew Dixmier polynomials are:}
	\begin{align*}
	p_1 &= \frac{e_1}{4\lambda_4+6\lambda_5}x_1+\left( -\frac{3\lambda_0}{2}+\frac{3\lambda_1}{2}+\frac{9\lambda_2}{4}-\lambda_4-\frac{3\lambda_5}{2}\right)x_2+ \lambda_0x_3\\
	p_2 &= \frac{(2\lambda_3-2\lambda_5)\lambda_1+2+(3\lambda_3+2\lambda_4)\lambda_2}{2\lambda_4+3\lambda_5}x_1 +\lambda_1x_2+\lambda_2x_3\\
	p_3 &= \lambda_3x_1+\lambda_4x_2+\lambda_5x_3
	\end{align*}
	where: $e_1=6\lambda_5^2+(6\lambda_0-6\lambda_1-9\lambda_2-6\lambda_3+4\lambda_4)\lambda_5+(6\lambda_1+ 9\lambda_2-4\lambda_4)\lambda_3+4\lambda_4\lambda_0+6.$
	
	\medskip
	
	\noindent
	\texttt{----- Case 14 -----}
	
	\smallskip
	
	\noindent
	\texttt{The skew Dixmier polynomials are:}
	\begin{align*}
	p_1 &= \left(\lambda_0+\frac{3\lambda_1}{2}-\frac{3\lambda_2}{2}-\lambda_3+\lambda_4\right)x_1+\frac{-3+ (-3\lambda_3+3\lambda_4)\lambda_0}{2\lambda_3-2\lambda_4}x_2+\lambda_0x_3\\
	p_2 &= \lambda_1x_1+\frac{-2+(-3\lambda_3+3\lambda_4)\lambda_2}{2\lambda_3-2\lambda_4}x_2+\lambda_2x_3\\
	p_3 &= \lambda_3x_1-\frac{3}{2}\lambda_4x_2+\lambda_4x_3
	\end{align*}
\end{example}

\begin{example}
	Consider the family of skew Dixmier polynomials $p_1$, $p_2$, and $p_3$ obtained in \texttt{Case 14} of the preceding example. Our goal is to construct a skew Dixmier automorphism by assigning $p_1$, $p_2$, and $p_3$ to the variables $x_1$, $x_2$, and $x_3$, respectively. To accomplish this, we use the function \texttt{newSkewAutomorphism} from the \texttt{DixmierProblem} package with the following syntax:
	\begin{center}
		\verb|newAuto := newSkewAutomorphism([p_1, p_2, p_3], multiCommutator([p_1, p_2, p_3]))|
	\end{center}
	
	Here, the function \texttt{multiCommutator} from the \texttt{DixmierAutomorphism} package generates a list of commutators $[[p_2, p_1], [p_3, p_1], [p_3, p_2]].$
	
	\medskip
	When executing this statement, Maple returns:
	\medskip
	
	\noindent\verb|Dixmier automorphism successfully defined|\\
	\verb|newAuto| $:= \alpha(x_1,x_2,x_3)$
	
	Next, we use the function \texttt{skewInverseAutomorphism} from the \texttt{DixmierProblem} package to compute the inverse automorphism of the initially generated skew Dixmier automorphism. The syntax is:
	\begin{center}
		\verb|invAuto := skewInverseAutomorphism(newAuto, expsMaxi=[2, 2, 2]);|
	\end{center}
	Here, the option \verb|expsMaxi=[2, 2, 2]| specifies that the maximum degrees of the variables $x_1$, $x_2$, and $x_3$ in the polynomials of the generated automorphism should be equal to $2$.
	
	Upon executing the \texttt{skewInverseAutomorphism} statement, followed by the statement\linebreak
	\verb|invAuto:-getPolys();|,
	Maple returns a list \verb|[P1, P2, P3]|, where \verb|P1|, \verb|P2|, and \verb|P3| are the respective polynomials of the skew inverse automorphism. In this case, the output is given by:
	\begin{align*}
	\texttt{P1} &= \frac{-2\lambda_4x_1}{(\lambda_3-\lambda_4)(2\lambda_0-3\lambda_2+2\lambda_4)}+\frac{3\lambda_4 x_2}{(\lambda_3-\lambda_4)(2\lambda_0-3\lambda_2+2\lambda_4)}+\frac{2\lambda_0-3\lambda_2}{(\lambda_3- \lambda_4)(2\lambda_0-3\lambda_2+2\lambda_4)}x_3,\\
	\texttt{P2} &= \frac{-2\lambda_1\lambda_4+2\lambda_2\lambda_3}{2\lambda_0-3\lambda_2+2\lambda_4}x_1+ \frac{2\lambda_4^2+(2\lambda_0+3\lambda_1-3\lambda_2-2\lambda_3)\lambda_4-2\lambda_0\lambda_3}{2\lambda_0 -3\lambda_2+2\lambda_4}x_2+\frac{e_1}{2\lambda_0-3\lambda_2+2\lambda_4}x_3,\\
	\texttt{P3} &= \frac{e_2}{(\lambda_3-\lambda_4)(2\lambda_0-3\lambda_2+2\lambda_4)}x_1+\frac{e_3}{4(\lambda_0-
		\frac{3\lambda_2}{2}+\lambda_4)(\lambda_3-\lambda_4)}x_2+
	\frac{e_4}{4(\lambda_0-\frac{3\lambda_2}{2}+\lambda_4)(\lambda_3-\lambda_4)}x_3.
	\end{align*}	
	where:
	\begin{align*}
	e_1 &= 3\lambda_2^2+(-2\lambda_0-3\lambda_1+2\lambda_3-2\lambda_4)\lambda_2+ 2\lambda_1\lambda_0,\\
	e_2 &= 3\lambda_2\lambda_3^2+\left(2+(-3\lambda_1-3\lambda_2)\lambda_4\right)\lambda_3+3\lambda_1 \lambda_4^2,\\
	e_3 &= -6\lambda_4^3+(-6\lambda_0 -9\lambda_1+9\lambda_2+12\lambda_3)\lambda_4^2+12\left(\lambda_0+\frac{3\lambda_1}{4}- \frac{3\lambda_2}{4}-\frac{\lambda_3}{2}\right)\lambda_3\lambda_4-6\lambda_0\lambda_3^2- 6\lambda_3,\\
	e_4 &= (9\lambda_3- 9\lambda_4)\lambda_2^2+\left(6\lambda_4^2+(6\lambda_0+9\lambda_1-12\lambda_3)\lambda_4+6+6\lambda_3^2+ (-6\lambda_0-9\lambda_1)\lambda_3\right)\lambda_2\\
	& \quad+(-6\lambda_0\lambda_1-4)\lambda_4+6\lambda_1\lambda_0 \lambda_3-4\lambda_0+4\lambda_3.
	\end{align*}
	Next, we execute the following Maple statements to verify that \verb"invAuto" correctly corresponds to the inverse automorphism of \verb"newAuto":
	
	\smallskip
	
	\noindent
	\verb"auto0 := composeSkewAutomorphism(newAuto, invAuto);"\\
	\verb"auto1 := composeSkewAutomorphism(invAuto, newAuto);"\\
	\verb"auto0:-Apply"$(x_1),$ \verb"auto0:-Apply"$(x_2),$ \verb"auto0:-Apply"$(x_3),$\\
	\verb"auto1:-Apply"$(x_1),$ \verb"auto1:-Apply"$(x_2),$ \verb"auto1:-Apply"$(x_3);$
	
	\smallskip
	
	After executing these statements, Maple confirms the correct definition of the Dixmier automorphisms:
	
	\noindent
	\verb"Dixmier automorphism successfully defined"\\
	\verb"auto0" $:=\alpha(x_1,x_2,x_3)$
	
	\noindent
	\verb"Dixmier automorphism successfully defined"\\
	\verb"auto1" $:=\alpha(x_1,x_2,x_3)$
	
	\smallskip
	
	The final output:
	
	\noindent
	$x_1, x_2, x_3, x_1, x_2, x_3,$
	
	indicates that the automorphisms \texttt{auto0} and \texttt{auto1} are indeed the identity automorphisms, verifying that \texttt{invAuto} is the correct inverse of \texttt{newAuto}.
\end{example}

In the following example, we consider the $K$-algebra $\mathcal{CSD}_3(K)$ with the parameters $d_{12}$, $d_{13}$, and $d_{23}$ treated symbolically. This symbolic formulation underscores the flexibility and strength of the functions built within the \texttt{DixmierProblem} package in Maple, permitting a generic exploration of algebraic properties. We then employ the \texttt{skewDixmierAutomorphism} command to generate a random automorphism of the algebra. In this instance, the skew Dixmier polynomials that define the automorphism are produced using the option \texttt{parametric=false}, the default setting, which ensures a standardized yet powerful method for constructing these functions.

\begin{example}
	Consider the skew Dixmier automorphism defined as follows:
	\begin{center}
		\verb|skewAuto1:=skewDixmierAutomorphism|$\left(1, 1, 1, \texttt{dij}=[d_{1,2},d_{1,3},d_{2,3}]\right)$
	\end{center}
	By retrieving the polynomials associated with this automorphism using the command:
	\begin{center}
		\verb|skewAuto1:-getPolys();|
	\end{center}
	we obtain the following set of polynomials:
	\begin{align*}
	P_1 &= \frac{e_1}{d_{1,2}}+\frac{-2d_{1,3}^2+7d_{1,3}d_{2,3}+3 d_{2,3}^2}{d_{1,2}}x_1+\frac{d_{1,2}^2-2d_{1,2}d_{1,3}+2d_{1,3}^2}{d_{1,2}}x_2-3x_3,\\
	P_2 &= \frac{-d_{1,2}-2d_{1,3}-4d_{2,3}}{-d_{1,2}+3d_{1,3}}x_1+x_2+2x_3,\\
	P_3 &= -x_1-x_2+3x_3.
	\end{align*}
	where: $e_1=3d_{1,3}-d_{1,2}+(2d_{1,3}-3d_{2,3})d_{1,2}.$
	
	Next, we compute the inverse skew Dixmier automorphism:
	\begin{center}
		\verb"skewAuto2 := skewInverseAutomorphism(skewAuto1, expsMaxi=[2, 2, 2]);"
	\end{center}
	
	The resulting polynomials for the inverse automorphism are:
	\begin{align*}	Q_1&=\frac{5d_{1,2}}{3d_{1,2}-2d_{1,3}-3d_{2,3}}x_1+\frac{-3d_{1,2}d_{2,3}+(6d_{1,3}+3d_{2,3})d_{1,2}-6 d_{1,3}}{-3d_{1,2}}x_2-\frac{2d_{1,2}^2}{3d_{1,2}-2d_{1,3}-3d_{2,3}}x_3,\\
	Q_2&=\frac{d_{1,2}(5d_{1,2}+12d_{2,3})}{(3d_{1,2}-2d_{1,3}-3d_{2,3})(d_{1,2}-3d_{1,3})}x_1+\frac{-2d_{1,2}^3 +(4d_{1,3}-9d_{2,3})d_{1,2}^2}{3d_{1,2}-2d_{1,3}-3d_{2,3}}x_2,\\
	Q_3&=\frac{-5d_{1,3}-4d_{2,3}d_{1,2}}{(d_{1,2}-3d_{1,3})(3d_{1,2}-2d_{1,3}-3d_{2,3})}x_1+\frac{3d_{1,2} d_{2,3}(d_{1,2}-d_{2,3})}{3d_{1,2}-2d_{1,3}-3d_{2,3}}x_2x_3.
	\end{align*}
	To verify that \verb"skewAuto2" is indeed the inverse of \verb"skewAuto1", we compose the two automorphisms using the \texttt{ComposeSkewAutomorphism} function, which is specifically designed to compose skew Dixmier automorphisms. In particular, we perform the following operations:
	\begin{center}
		\verb"skewAuto3 := composeSkewAutomorphism(skewAuto1, skewAuto2);"\\[1ex]
		\verb"skewAuto4 := composeSkewAutomorphism(skewAuto2, skewAuto1);"
	\end{center}
	Subsequently, by executing
	\begin{center}
		\verb"skewAuto3:-getPolys(),"\ \verb"skewAuto4:-getPolys();"
	\end{center}
	Maple returns:
	\begin{center}
		$[x_1, x_2, x_3],$\ $[x_1, x_2, x_3]$
	\end{center}
	This output confirms that both compositions are equivalent to the identity automorphism.
\end{example}

In next example, we construct a skew Dixmier automorphism for the $K$-algebra $\mathcal{CSD}_3(K)$ using symbolic parameters. The use of \texttt{parametric=true} allows us to define the parameters symbolically, which is particularly useful for obtaining generic algebraic information via the powerful functions provided by the \texttt{DixmierProblem} package in Maple.
\begin{example}
	Consider the $k$-algebra $\mathcal{CSD}_3(K)$, where the elements $d_{12}$, $d_{13}$, and $d_{23}$ are represented symbolically. The utilization of symbolic parameters enables a comprehensive analysis of the algebraic properties inherent to this structure, illustrating the versatility and robustness of the library.
	
	\bigskip
	
	\noindent
	\emph{Step 1: Generation of the automorphism and its defining polynomials}
	
	First, the automorphism is generated by executing:
	\begin{center}
		$\texttt{skewAuto1 := skewDixmierAutomorphism$\bigl($1, 1, 1, parametric = true,
			dij} = [d_{1,2}, d_{1,3}, d_{2,3}]\texttt{$\bigr)$;}$
	\end{center}
	Maple returns:
	
	\medskip
	
	\noindent
	\verb"Dixmier automorphism successfully defined"\\
	\verb"skewAuto1" $:= \alpha(x_1, x_2, x_3)$
	
	\medskip
	
	Next, the defining polynomials are retrieved by:
	\begin{verbatim}
	skewAuto1:-getPolys();
	\end{verbatim}
	yielding the list $[P_1,\,P_2,\,P_3]$ with
	\[
	\begin{aligned}
	P_1 &= -\frac{d_{2,3}\lambda_0}{d_{1,3}}x_1+\lambda_0x_2+ \frac{-d_{1,2}\left(d_{1,3}\lambda_4+d_{2,3}\lambda_5\right)\lambda_0-d_{1,3}^2}{d_{1,3}
		\left(d_{1,3}\lambda_4+d_{2,3}\lambda_5\right)}x_3,\\
	P_2 &= -\frac{d_{2,3}\lambda_3}{d_{1,3}}x_1x_3+\lambda_3x_2x_3+ \frac{d_{1,2}d_{1,3}\lambda_4+d_{1,2}d_{2,3}\lambda_5-d_{1,3}d_{2,3}\lambda_1}{d_{1,3}^2}
	x_1+\lambda_1x_2+\lambda_2x_3,\\
	P_3 &= -\frac{d_{2,3}\lambda_3}{d_{1,2}}x_1x_3+\frac{d_{1,3}\lambda_3}{d_{1,2}}x_2x_3+\lambda_4x_1+\lambda_5x_2+ \frac{d_{1,3}^3\lambda_2\lambda_4+\left(d_{1,2}\lambda_1\lambda_4+d_{2,3}\lambda_2\lambda_5+
		d_{2,3}\right)d_{1,3}^2}{d_{1,2}d_{1,3}\left(d_{1,3}\lambda_4+d_{2,3}\lambda_5\right)}x_3.
	\end{aligned}
	\]
	We also verify the symbolic parameters by executing:
	\begin{verbatim}
	skewAuto1:-Properties:-multicommutator;
	\end{verbatim}
	which returns
	\[
	[d_{1,2},\, d_{1,3},\, d_{2,3}],
	\]
	confirming that the relations among $p_1$, $p_2$, and $p_3$ hold as expected.
	
	\bigskip
	
	\noindent
	\emph{Step 2: Inverse automorphism and its defining polynomials}
	
	Next, the inverse automorphism is computed with:
	\begin{verbatim}
	skewAuto2 := skewInverseAutomorphism(skewAuto1, expsMaxi=[2, 2, 2]);
	\end{verbatim}
	Maple returns:
	\medskip
	
	\noindent
	\verb"Dixmier automorphism successfully defined"\\
	\verb"skewAuto2" $:= \alpha(x_1, x_2, x_3)$
	
	\medskip
	
	Retrieving its defining polynomials by:
	\begin{verbatim}
	skewAuto2:-getPolys();
	\end{verbatim}
	yields the list $[Q_1,\, Q_2,\, Q_3],$ where:	
	\begin{align*}
	Q_1 &=
	-\frac{e_1\left(d_{1,2}e_1e_2-d_{1,3}^2d_{2,3}\right)\lambda_3}{d_{1,2}d_{1,3}e_3^2}x_1^2- \frac{2\left(d_{1,3}^2\left(e_1-\tfrac{e_3}{2}\right)+ d_{1,2}e_1e_2\lambda_{0}\right)\lambda_3}{d_{1,2}e_3^2}x_1x_2\\
	& \quad +\frac{2\left(d_{1,3}^2\left(e_1-\tfrac{e_3}{2}\right)+ d_{1,2}e_1e_2\lambda_{0}\right)\lambda_3}{d_{1,3}e_3^2}x_1x_3-\frac{d_{1,3}\left(d_{1,2}e_2\lambda_{0}+
		d_{1,3}^2\right)\lambda_0\lambda_3}{d_{1,2}e_3^2}x_2^2 \\
	& \quad +\frac{2\left(d_{1,2}e_2\lambda_{0}+d_{1,3}^2\right)\lambda_0\lambda_3}{e_3^2}x_2x_3
	-\frac{d_{1,2}\left(d_{1,2}e_2\lambda_{0}+d_{1,3}^2\right)\lambda_0\lambda_3}{d_{1,3}e_3^2}x_3^2\\
	& \quad +\frac{\left(d_{1,2}\lambda_1+d_{1,3}\lambda_2\right)e_1e_2-d_{1,3}^2d_{2,3}\lambda_1}{d_{1,2} e_2e_3}x_1+\frac{d_{1,3}\left(d_{1,3}\lambda_1+e_2\lambda_{0}\lambda_2+e_3\right)+ d_{1,2}e_2\lambda_{0}\lambda_1}{d_{1,2}e_2e_3}x_2\\
	& \quad -\frac{\left(d_{1,2}\lambda_1+d_{1,3}\lambda_2\right)e_2\lambda_0+d_{1,3}^2\lambda_1}{e_2e_3} x_3-\frac{\left(d_{1,2}e_2\lambda_{0}+d_{1,3}^2\right)\left(e_1-e_3\right)\lambda_3}{e_3^2},
	\end{align*}
	\begin{align*}
	Q_2 &=
	-\frac{e_1d_{2,3}\lambda_3\left(d_{1,2}e_1e_2-d_{1,3}^{2}d_{2,3}\right)}{d_{1,2}d_{1,3}^{2}e_3^{2}} x_1^{2}+\frac{2\lambda_3\left[-d_{2,3}d_{1,3}^{2}\left(e_1-\tfrac{e_3}{2}\right)+ d_{1,2}e_1e_2(e_1-e_3)\right]}{d_{1,2}d_{1,3}e_3^{2}}x_1x_2\\
	& \quad-\frac{2\left(-d_{2,3}d_{1,3}^{2}\left(e_1-\tfrac{e_3}{2}\right)+ d_{1,2}e_1e_2(e_1-e_3)\right)\lambda_3}{d_{1,3}^{2}e_3^{2}}x_1x_3
	+\frac{\left(d_{1,2}e_2\lambda_0+d_{1,3}^{2}\right)(e_1-e_3)\lambda_3}{d_{1,2}e_3^{2}}x_2^{2}\\
	& \quad -\frac{2\left(d_{1,2}e_2\lambda_0+d_{1,3}^{2}\right)(e_1-e_3)\lambda_3}{d_{1,3}e_3^{2}}x_2x_3
	+\frac{d_{1,2}\left(d_{1,2}e_2\lambda_0+d_{1,3}^{2}\right)(e_1-e_3)\lambda_3}{d_{1,3}^{2}e_3^{2}} x_3^2\\
	& \quad-\frac{d_{1,3}^{3}d_{2,3}^{2}\lambda_1-d_{2,3}d_{1,3}^2e_2\left(e_1\lambda_2+d_{1,2}\right)+ d_{1,2}d_{1,3}d_{2,3}e_1e_2\lambda_1-d_{1,2}^{2}e_1e_2^{2}}{d_{1,2}d_{1,3}^{2}e_2e_3}x_1 \\
	&\quad+\frac{d_{1,3}^{3}d_{2,3}\lambda_1+d_{1,3}^{2}\left[e_2(-e_1\lambda_2+e_3\lambda_2-d_{1,2})+ d_{2,3}e_3\right]-d_{1,2}d_{1,3}e_2(e_1-e_3)\lambda_1- d_{1,2}^{2}e_2^{2}\lambda_0}{d_{1,2}d_{1,3}e_2e_3}x_2\\	& \quad +\frac{-d_{1,3}^{3}d_{2,3}\lambda_1+d_{1,3}^{2}e_2\left(e_1\lambda_2-e_3\lambda_2+d_{1,2}\right)+ d_{1,2}d_{1,3}e_2(e_1-e_3)\lambda_1+d_{1,2}^{2}e_2^2\lambda_0}{d_{1,3}^{2}e_2e_3}x_3\\
	& \quad+\frac{\left[d_{1,2}e_2(e_1-e_3)-d_{1,3}^{2}d_{2,3}\right](e_1-e_3)\lambda_3}{d_{1,3}e_3^{2}},
	\end{align*}
	\[
	Q_3=
	-\frac{e_1e_2}{d_{1,3}e_3}x_1-\frac{e_2\lambda_0}{e_3}x_2+\frac{d_{1,2}e_2\lambda_0}{d_{1,3}e_3}x_3.
	\]
	where: $e_1=d_{1,2}\lambda_5-d_{1,3}\lambda_1,$ $e_2=d_{1,3}\lambda_4+d_{2,3}\lambda_5,$
	$e_3=e_1+d_{2,3}\lambda_0=d_{1,2}\lambda_5-d_{1,3}\lambda_1+d_{2,3}\lambda_0$

	\emph{Step 3: Verification by composition}
	
	To confirm that \texttt{skewAuto2} is indeed the inverse of \texttt{skewAuto1}, we compose the automorphisms:
	\begin{verbatim}
	skewAuto3 := composeSkewAutomorphism(skewAuto1, skewAuto2);
	skewAuto4 := composeSkewAutomorphism(skewAuto2, skewAuto1);
	\end{verbatim}
	Subsequently, by executing
	\begin{center}
		\verb"skewAuto3:-getPolys(),"\ \verb"skewAuto4:-getPolys();"
	\end{center}
	Maple returns:
	\begin{center}
		$[x_1, x_2, x_3],$\ $[x_1, x_2, x_3]$
	\end{center}
	That is, both compositions yield the identity confirming the correctness of the inverse automorphism.

\end{example}

Despite of the previous particular interesting examples, the following easy counterexample, computed with the \texttt{skewDixmierPolynomials} \textbf{function}, shows that the conjecture for $CSD_3(K)$ is not true.

\begin{example}\label{example5.29}
The polynomials  $p_1:=x_1$, $p_2:=x_2$, $p_3:=-\frac{d_{23}}{d_{12}}x_1+\frac{d_{13}}{d_{12}}x_2$ in $CSD_3(K)$ are skew Dixmier, but clearly the corresponding  endomporhism is not surjective.
\end{example}
	
Induced by this example, next we provide a criterion ensuring that a certain class of algebras $\mathcal{CSD}_n(K)$, for $n$ odd, do not satisfy the Dixmier property. Our approach is based on examining a system of conditions imposed on the generators by the skew Dixmier property and on analyzing the associated coefficient matrix. In particular, if the matrix, defined in terms of the structure constants of the algebra, has a nonzero determinant, then the specific endomorphism of the algebra fails to be surjective. This result not only illustrates a sharp obstruction to the Dixmier property in the setting of general skew PBW extensions, but also paves the way for further investigations into the automorphism groups of such algebras.

\begin{theorem}\label{theorem5.30}For $n$ odd, consider in $\mathcal{CSD}_n(K)$ the following polynomials:
	\[
	p_1=x_1, p_2=x_2,\dots ,p_{n-1}=x_{n-1},p_n=\sum_{k=1}^{n-1} a_k x_k, \ with \ a_k\in K.
	\]
	 Let $M=[M_{ik}]$ be the skew-symmetric matrix of size $(n-1)\times (n-1)$ defined by
	\[
	M_{ik}:=\begin{cases}
	-d_{ki}, & \text{if } k<i,\\[1mm]
	0, & \text{if } k=i,\\[1mm]
	d_{ik}, & \text{if } k>i.
	\end{cases}\qquad i,k=1,\dots,n-1,
	\]
	If $p_1,\dots, p_n$ are skew Dixmier and $\det(M)\neq 0$, then the coefficients $a_1,\dots, a_{n-1}$ are uniquely determined in terms of constants $d_{ij}$, $1\leq i<j\leq n-1$, and the corresponding endomorphism is not surjective.
\end{theorem}
\begin{proof}
	Since
	\[
	p_n=\sum_{k=1}^{n-1}a_kx_k,
	\]
	linearity gives
	\[
	[p_n,x_i] = \sum_{k=1}^{n-1}a_k[x_k,x_i].
	\]
	By the defining commutation relations
	\[
	[x_k,x_i]=
	\begin{cases}
	-d_{ki}, & \text{if } k < i,\\[1mm]
	0, & \text{if } k = i,\\[1mm]
	d_{ik}, & \text{if } k > i,
	\end{cases}
	\]
	and since $p_1,\dots, p_n$ are skew Dixmier, then for each $i=1,\dots,n-1$, we obtain the equation
	\[
	-\sum_{k=1}^{i-1}a_k\,d_{ki} + \sum_{k=i+1}^{n-1}a_k\,d_{ik} = d_{in}.
	\]
	Introduce the vector $\mathbf{a}=[a_1,\dots,a_{n-1}]^T$ and set
	\[
	\mathbf{b}=[d_{1n}, d_{2n}, \dots, d_{n-1n}]^T.
	\]
	The above system of equations can be written in matrix form as
	\[
	M\,\mathbf{a}=\mathbf{b}.
	\]
	Since $M$ is invertible, the linear system has a unique solution, i.e., the coefficients $a_1,\dots, a_{n-1}$ are uniquely determined in terms of constants $d_{ij}$, $1\leq i<j\leq n-1$: In fact, 
	$\mathbf{a}=M^{-1}\mathbf{b}$.
	
	Now consider the corresponding $K$-algebra endomorphism
	\[
	\varphi:
	\begin{cases}
	x_i \mapsto p_i=x_i, & 1\le i\le n-1,\\[1mm]
	x_n \mapsto p_n.
	\end{cases}
	\]
	Notice that the image of $\varphi$ is contained in the subalgebra generated by $x_1,\dots,x_{n-1}$, so $\varphi$ is not surjective.
\end{proof}

The case computed in Example \ref{example5.29} can be deduced from the previous theorem. 
	\begin{corollary}
	Let $p_1:=x_1, p_2:=x_2$ and $p_3:=a_1x_1+a_2x_2$ be skew Dixmier polynomials of $\mathcal{CSD}_3(K)$. Then, $a_1=-\frac{d_{23}}{d_{12}}, a_2=\frac{d_{13}}{d_{12}}$, and the corresponding endomorphism is not surjective.
	\end{corollary}
\begin{proof}In this case, the $2\times 2$ skew-symmetric matrix $M$ is given by
\[
M = \begin{bmatrix}
0 & d_{12} \\[1mm]
-d_{12} & 0
\end{bmatrix}, 
\]
with $\det(M)=d_{12}^2\neq 0$. 
The corollary is a consequence of the proof of Theorem \ref{theorem5.30}. For the coefficients of $p_3$ we have 
\begin{center} 
$\mathbf{a}=M^{-1}\mathbf{b}=\begin{bmatrix}0 & -\frac{1}{d_{12}}\\
\frac{1}{d_{12}} & 0\end{bmatrix}\begin{bmatrix}d_{13}\\ d_{23}\end{bmatrix}.$
\end{center}
\end{proof}

\begin{remark}\label{remark5.32}
When $n$ is even, the skew-symmetric matrix $M$ in Theeorem \ref{theorem5.30} has odd order, so $\det(M)=0$ and we can not ensure that the system
$M\mathbf{a}=\mathbf{b}$ has solution. For example, for $n=4$, 
\begin{center}
$M=\begin{bmatrix}0 & d_{12} & d_{13}\\
-d_{12} & 0 & d_{23}\\
-d_{13} & -d_{23} & 0\end{bmatrix}$, 
$\mathbf{a}=[a_1,a_2,a_3]^T$, $\mathbf{b}=[d_{14}, d_{24}, d_{34}]^T$,
\end{center}
and the row echelon form produces the equivalent system
\begin{center}
$a_1-\frac{1}{d_{12}}d_{23}a_3=0$, 

$a_2+\frac{1}{d_{12}}d_{13}a_3=0$, 

$0=1,$
\end{center}
without solution. 
\end{remark}

The following example illustrates the \texttt{DixmierProblem} package applied to $\mathcal{CSD}_4(K)$ with the constants $d_{ij}$ treated symbolically.  We select one representative family of skew Dixmier polynomials of total degree $\le1$, construct the corresponding endomorphism, and compute its inverse. To date, neither theoretical considerations nor computational experiments using this library have produced
any counterexample in the even dimensional case: every skew Dixmier endomorphism tested remains surjective, behaving in practice like an automorphism.

\begin{example}\label{example5.33}
	Let $\mathcal{CSD}_4(K)$ be the $K$\nobreakdash-algebra on generators $x_1,\dots,x_4$ with relations $[x_j,x_i]=d_{ij}$ for $1\le i<j\le4$. From the output of
	\begin{center}
		\texttt{skewDixmierPolynomials$\bigl(1,$ $n=4,$ dij $=[d_{1,2},d_{1,3}, d_{1,4}, d_{2,3}, d_{2,4}, d_{3,4}],$ skewTotalDeg = true$\bigr)$:}
	\end{center}
	we select one representative family of Dixmier polynomials:
	\begin{align*}
	P_1 &= \lambda_0x_1+\tfrac{\beta_1}{\beta_2}x_2+\frac{-d_{2,3}\lambda_4+d_{2,4}\lambda_3}{d_{3,4}}x_3
	-\frac{d_{2,4}d_{1,2}}{\lambda_{4}\beta_3}x_4, \\
	P_2 &= \lambda_1x_1+\lambda_{2}x_2+\frac{-d_{2,3}\lambda_4+d_{2,4}\lambda_3}{d_{3,4}}x_3
	-\frac{d_{2,4}d_{1,2}}{\lambda_{4}\beta_3}x_4,\\
	P_3 &= \frac{\beta_3d_{2,3}\lambda_3\lambda_4-d_{1,2}d_{2,4}d_{3,4}}
	{d_{1,2}\lambda_4\beta_3}x_1+\frac{-\beta_3d_{1,3}\lambda_3\lambda_4+d_{1,2}d_{1,4}
		d_{3,4}}{d_{1,2}\lambda_4\beta_3}x_2+\lambda_3x_3-\frac{d_{3,4}d_{1,2}}{\lambda_4\beta_3}x_4,\\
	P_4 &= \frac{d_{2,3}\lambda_4}{d_{1,2}}x_1-\frac{d_{1,3}\lambda_4}{d_{1,2}}x_2+\lambda_4x_3,
	\end{align*}
	where:
	\begin{align*}
	\beta_1 &= (d_{1,3}\lambda_1-d_{2,3}(\lambda_0-\lambda_2))d_{1,3}\beta_3
	\lambda_4^{2}+\Bigl[d_{3,4}^{2}(\lambda_0\lambda_2-1)d_{1,2}^{2}
	-\bigl((d_{1,4}\lambda_1\lambda_3-d_{2,4}(-\lambda_0\lambda_2+\lambda_0\lambda_3+2)) d_{1,3}+\\
	&\ \quad\ d_{1,4}d_{2,3}\beta_4\bigr)d_{3,4}d_{1,2}+\bigl((d_{1,4}\lambda_1-
	d_{2,4}\lambda_0)d_{1,3}+d_{1,4}d_{2,3}\lambda_2\bigr)\lambda_3(d_{1,3}d_{2,4}-d_{1,4}d_{2,3})
	\Bigr]\lambda_4+(d_{1,4}\lambda_1-d_{2,4}\\
	&\ \quad(\lambda_0-\lambda_2))d_{1,4}d_{3,4}d_{1,2},\\
	\beta_2 &= d_{2,3}^2\beta_3\lambda_4^2+\beta_3\bigl(d_{1,2}d_{3,4}\lambda_1-d_{2,3}d_{2,4}\lambda_3\bigr)\lambda_4+
	d_{1,2}d_{2,4}^2d_{3,4},\\
	\beta_3 &= d_{1,2}d_{3,4}-d_{1,3}d_{2,4}+d_{1,4}d_{2,3},\\
	\beta_4 &= -\lambda_0\lambda_2+\lambda_2\lambda_3+2.
	\end{align*}
	
	\medskip
	\noindent\textbf{Define the endomorphism}
	\[
	\alpha\colon \mathcal{CSD}_4(K)\;\longrightarrow\;\mathcal{CSD}_4(K),
	\quad x_i\mapsto P_i,
	\]
	constructed by
	\begin{center}
		\texttt{skewAuto1 := newSkewAutomorphism$\left(\left[P_1, P_2, P_3, P_4], [d_{12}, d_{13}, d_{14}, d_{23}, d_{24}, d_{34}\right]\right);$}
	\end{center}
	
	\medskip
	\noindent\textbf{Compute its inverse}
	\[
	\alpha^{-1}\colon x_i \mapsto Q_i,
	\]
	via
	\begin{center}
		\verb"skewAuto2 := skewInverseAutomorphism(skewAuto1);"
	\end{center}
	The resulting polynomials for the inverse automorphism are:
	\begin{align*}	
	Q_1 &= \frac{e_1}{\beta_3^2\lambda_4}x_1+\frac{e_2}{e_3\beta_3^2\lambda_4}x_2+
	\frac{e_4}{e_3\beta_3d_{3,4}}x_3+\frac{e_5}{e_3\beta_3^2d_{3,4}\lambda_4}x_4,\\
	Q_2 &= \frac{\beta_2}{\beta_3^2\lambda_4}x_1+\frac{e_6}{\beta_3^2\lambda_4}x_2+
	\frac{e_7}{\beta_3d_{3,4}}x_3+\frac{e_8}{\beta_3^2d_{3,4}\lambda_4}x_4,\\
	Q_3 &= \frac{e_9}{\beta_3^2\lambda_4}x_1+\frac{e_{10}}{\beta_2\beta_3^2\lambda_4}x_2+
	\frac{e_{11}}{\beta_2\beta_3}x_3+\frac{e_{12}}{\beta_2\beta_3^2\lambda_4}x_4,\\
	Q_4 &= \frac{e_{13}}{\beta_3d_{1,2}}x_1+\frac{e_{14}}{\beta_2\beta_3d_{1,2}}x_2+
	\frac{e_{15}}{\beta_2}x_3+\frac{e_{16}}{\beta_2\beta_3}x_4,
	\end{align*}
	where:
	\begin{align*}
	e_1 &= -d_{1,3}d_{2,3}\beta_3\lambda_4^2+\bigl(d_{1,2}d_{3,4}\lambda_2
	+d_{1,3}d_{2,4}\lambda_3\bigr)\beta_3\lambda_4-d_{1,2}d_{1,4}d_{2,4}d_{3,4},\\
	e_2 &=d_{1,3}^2d_{2,3}^2\beta_3^2\lambda_4^4-d_{2,3}d_{1,3}\beta_3^2\left(d_{3,4}(
	-\lambda_0+\lambda_2)d_{1,2}+\lambda_3(d_{1,3}d_{2,4}+d_{1,4}d_{2,3})\right)\lambda_4^3
	-\Bigl(d_{3,4}^3(\lambda_0\lambda_2-1)\\
	&\quad\ d_{1,2}^3-d_{3,4}^2\bigl(d_{2,4}(\lambda_0\lambda_2
	-\lambda_0\lambda_3-2)d_{1,3}+d_{1,4}d_{2,3}\beta_4\bigr)d_{1,2}^2+d_{3,4}\bigl(-d_{1,3}^2d_{2,4}^2
	(\lambda_0\lambda_3+1)+\lambda_3d_{1,4}d_{2,3}d_{2,4}\\
	&\quad (\lambda_2-\lambda_3+\lambda_0)d_{1,3}-d_{1,4}^2(\lambda_2\lambda_3+1)d_{2,3}^2\bigr)d_{1,2}
	+\lambda_3^2d_{1,3}d_{1,4}d_{2,3}d_{2,4}\bigl(d_{1,3}d_{2,4}-d_{1,4}d_{2,3}\bigr)\Bigr)\beta_3\lambda_4^2\\
	&\quad -d_{3,4}d_{1,2}\beta_3d_{2,4}\left(d_{3,4}(-\lambda_0+\lambda_2)d_{1,2}
	+\lambda_3(d_{1,3}d_{2,4}+d_{1,4}d_{2,3})\right)d_{1,4}\lambda_4+d_{1,2}^2d_{1,4}^2d_{2,4}^2d_{3,4}^2,\\
	e_3 &= d_{2,3}^2\beta_3\lambda_4^2+\beta_3\bigl(d_{1,2}d_{3,4}\lambda_1-d_{2,3}d_{2,4}
	\lambda_3\bigr)\lambda_4+d_{1,2}d_{2,4}^2d_{3,4},\\
	e_4 &= -d_{1,3}d_{2,3}^2\beta_5\beta_3\lambda_4^3+\left(\bigl(-d_{1,3}(\lambda_0-
	\lambda_2)d_{2,4}+d_{1,4}(d_{1,3}\lambda_1-d_{2,3}\lambda_2)\bigr)d_{1,2}d_{3,4}+
	\lambda_3d_{1,3}d_{2,4}\beta_5\right)\beta_3d_{2,3}\lambda_4^2\\
	&\quad -\Bigl(d_{1,2}^2\bigl(d_{1,4}\lambda_1\lambda_2-d_{2,4}\lambda_0\lambda_2+d_{2,4}
	\bigr)d_{3,4}^2+\bigl(-d_{1,3}(-\lambda_0\lambda_2+\lambda_0\lambda_3+2)d_{2,4}^2+d_{1,4}
	\bigl((-\lambda_0\lambda_2+2)d_{2,3}+\\
	&\quad\ d_{1,3}\lambda_1(\lambda_3-\lambda_2)\bigr)d_{2,4}+\lambda_1\lambda_2d_{2,3}d_{1,4}^2
	\bigr)d_{1,2}d_{3,4}-d_{2,4}\beta_5d_{1,3}\bigl((-\lambda_0\lambda_3-1)d_{2,4}+d_{1,4}\lambda_1
	\lambda_3\bigr)\Bigr)d_{3,4}d_{1,2}\lambda_4\\
	&\quad +d_{1,2}^2d_{1,4}d_{2,4}d_{3,4}^2(d_{1,4}\lambda_1-d_{2,4}\lambda_0),\\
	e_5 &= -\bigl(d_{1,2}(d_{1,3}\lambda_1-d_{2,3}\lambda_0)d_{3,4}-\lambda_3d_{2,3}\beta_5\bigr)\beta_3^2
	d_{1,3}d_{2,3}\lambda_4^3+\bigl(d_{1,2}^3(d_{1,3}\lambda_1\lambda_2-d_{2,3}\lambda_0\lambda_2+d_{2,3})
	d_{3,4}^3\\
	&\quad +\bigl(d_{1,4}\beta_4d_{2,3}^2-d_{1,3}\bigl(((\lambda_0+\lambda_2)\lambda_3-
	\lambda_2\lambda_0+2)d_{2,4}-d_{1,4}\lambda_1\lambda_2\bigr)d_{2,3}+d_{1,3}^2
	d_{2,4}\lambda_1(\lambda_3-\lambda_2)\bigr)d_{1,2}^2d_{3,4}^2\\
	&\quad -\beta_5\bigl(d_{1,4}(\lambda_2\lambda_3+1)d_{2,3}^2+\lambda_3d_{1,3}d_{2,4}
	(\lambda_3-\lambda_0-\lambda_2)d_{2,3}+\lambda_3d_{1,3}^2d_{2,4}\lambda_1\bigr)
	d_{1,2}d_{3,4}+\lambda_3^2d_{1,3}d_{2,3}d_{2,4}\beta_5^2\bigr)\\
	&\quad\ \beta_3\lambda_4^2-d_{2,4}\beta_3d_{3,4}\bigl(-d_{1,4}(\lambda_0-\lambda_2)d_{2,3}+d_{1,3}(d_{1,4}
	\lambda_1-d_{2,4}\lambda_2)\bigr)d_{1,2}d_{3,4}\lambda_4-d_{1,2}^2d_{1,4}d_{2,4}^2d_{3,4}^2\beta_5,\\
	e_6 &= d_{1,3}d_{2,3}\beta_3\lambda_4^2+\beta_3\bigl(d_{1,2}d_{3,4}\lambda_0-d_{1,4}d_{2,3}\lambda_3
	\bigr)\lambda_4+d_{1,2}d_{1,4}d_{2,4}d_{3,4},\\
	e_7 &= d_{1,2}(d_{1,4}\lambda_1-d_{2,4}\lambda_0)d_{3,4}-\lambda_4d_{2,3}(d_{1,3}d_{2,4}-d_{1,4}d_{2,3})\\
	e_8 &= -\beta_3\bigl(d_{1,2}\bigl(d_{1,3}\lambda_1-d_{2,3}\lambda_0\bigr)d_{3,4}
	-\lambda_3d_{2,3}\bigl(d_{1,3}d_{2,4}-d_{1,4}d_{2,3}\bigr)\bigr)\lambda_4
	-d_{1,2}d_{2,4}d_{3,4}\bigl(d_{1,3}d_{2,4}-d_{1,4}d_{2,3}\bigr),\\
	e_9 &= -\bigl((d_{1,3}\lambda_1+d_{2,3}\lambda_2)\beta_3\lambda_4+d_{1,3}d_{2,4}^2
	-d_{1,4}d_{2,3}d_{2,4}\bigr)d_{3,4},\\
	e_{10} &= \Bigl(d_{1,3}d_{2,3}\beta_3^2(d_{1,3}\lambda_1+d_{2,3}\lambda_2)
	\lambda_4^3+\beta_3\Bigl(-d_{1,4}^2(\lambda_2\lambda_3+1)d_{2,3}^3-d_{1,4}\Bigl((2+
	(\lambda_3-\lambda_0)\lambda_2)d_{1,2}d_{3,4}\\
	&\quad +d_{1,3}(d_{1,4}\lambda_1\lambda_3-d_{2,4}\lambda_2\lambda_3-d_{2,4})\Bigr)
	d_{2,3}^2+\Bigl(d_{3,4}^2(\lambda_0\lambda_2-1)d_{1,2}^2-\bigl((\lambda_0\lambda_2-2)d_{2,4}
	+d_{1,4}\lambda_1(\lambda_3-\\
	&\quad\ \lambda_0)\bigr)d_{1,3}d_{1,2}d_{3,4}+\lambda_3d_{1,3}^2d_{1,4}d_{2,4}\lambda_1
	\Bigr)d_{2,3}+d_{1,2}d_{1,3}d_{3,4}\lambda_0\lambda_1(d_{1,2}d_{3,4}-d_{1,3}d_{2,4})
	\Bigr)\lambda_4^2+d_{2,4}\beta_3\\
	&\quad \Bigl(\lambda_3d_{1,4}^2d_{2,3}^2-\bigl(d_{1,2}(\lambda_0-\lambda_2)d_{3,4}+d_{1,3}
	d_{2,4}\lambda_3\bigr)d_{1,4}d_{2,3}+d_{1,3}d_{1,2}d_{3,4}(d_{1,4}\lambda_1+d_{2,4}
	\lambda_0)\Bigr)\lambda_4\\
	&\quad +d_{1,2}d_{1,4}d_{2,4}^2d_{3,4}(d_{1,3}d_{2,4}-d_{1,4}d_{2,3})\Bigr)d_{3,4},\\
	\end{align*}
	\begin{align*}
	e_{11} &=-d_{2,3}(d_{1,3}\lambda_1+d_{2,3}\lambda_2)(d_{1,3}d_{2,4}-d_{1,4}d_{2,3})\beta_3
	\lambda_4^2+\Bigl(\bigl(d_{1,4}\lambda_1\lambda_2+(-\lambda_0\lambda_2+2)d_{2,4}\bigr)d_{1,4}
	d_{2,3}^2+\Bigl(d_{3,4}\\
	&\quad\ (d_{1,4}\lambda_1\lambda_2-d_{2,4}\lambda_0\lambda_2+d_{2,4})d_{1,2}+d_{1,3}\bigl((
	\lambda_0\lambda_2-2)d_{2,4}^2-d_{1,4}\lambda_1(\lambda_0+\lambda_2)d_{2,4}+d_{1,4}^2
	\lambda_1^2\bigr)\Bigr)d_{2,3}+\\
	&\quad\ d_{1,3}\lambda_1(d_{1,4}\lambda_1-d_{2,4}\lambda_0)(d_{1,2}d_{3,4}-d_{1,3}
	d_{2,4})\Bigr)d_{3,4}d_{1,2}\lambda_4+d_{1,2}d_{2,4}d_{3,4}(d_{1,4}\lambda_1-d_{2,4}
	\lambda_0)(d_{1,3}d_{2,4}\\
	&\quad -d_{1,4}d_{2,3}),\\
	e_{12} &= -\bigl(d_{1,2}(d_{1,3}\lambda_1-d_{2,3}\lambda_0)d_{3,4}-\lambda_3d_{2,3}(d_{1,3}
	d_{2,4}-d_{1,4}d_{2,3})\bigr)\beta_3^2(d_{1,3}\lambda_1+d_{2,3}\lambda_2)\lambda_4^2
	+\Bigl[d_{1,2}^2d_{3,4}^2\lambda_1\\
	&\quad -\bigl((-2d_{1,4}\lambda_1+d_{2,4}\lambda_3)d_{2,3}+2d_{1,3}\lambda_1d_{2,4}\bigr)
	d_{3,4}d_{1,2}+2\Bigl(\bigl((\lambda_3+\tfrac{\lambda_0}{2}-\tfrac{\lambda_2}{2})d_{2,4}-\tfrac{1}{2}
	d_{1,4}\lambda_1\bigr)d_{2,3}-\tfrac{1}{2}d_{1,3}\\
	&\quad\ \lambda_1d_{2,4}\Bigr)(d_{1,3}d_{2,4}-d_{1,4}
	d_{2,3})\Bigr]\beta_3d_{3,4}d_{1,2}\lambda_4+d_{1,2}^2d_{2,4}^2d_{3,4}^2(d_{1,2}
	d_{3,4}-2d_{1,3}d_{2,4}+2d_{1,4}d_{2,3}),\\
	e_{13} &= -d_{1,2}d_{1,4}d_{3,4}\lambda_1-d_{1,2}d_{2,4}d_{3,4}\lambda_2+d_{1,3}d_{2,3}d_{2,4}
	\lambda_4-d_{1,3}d_{2,4}^2\lambda_3-d_{1,4}d_{2,3}^2\lambda_4+d_{1,4}d_{2,3}d_{2,4}\lambda_3,\\
	e_{14} &= -d_{1,3}d_{2,3}^2(d_{1,3}d_{2,4}-d_{1,4}d_{2,3})\beta_3\lambda_4^3
	+\Bigl(d_{3,4}\Bigl(d_{1,4}d_{2,3}\lambda_0+\bigl(-\lambda_0+\lambda_2\bigr)d_{2,4}d_{1,3}
	+d_{1,4}\lambda_1d_{1,3}\Bigr)d_{1,2}+\\
	&\quad\ d_{1,3}^2d_{2,4}^2\lambda_3-\lambda_3d_{1,4}^2d_{2,3}^2\Bigr)\beta_3d_{2,3}\lambda_4^2
	+\Bigl(d_{3,4}^3\bigl(d_{1,4}\lambda_0\lambda_1+d_{2,4}\lambda_0\lambda_2-d_{2,4}\bigr)d_{1,2}^3
	+\Bigl(-d_{1,4}\Bigl((\lambda_0+\lambda_2)\lambda_3\\
	&\quad -\lambda_2\lambda_0+2\Bigr)d_{2,4}+d_{2,4}d_{1,3}\Bigl(d_{2,4}(-\lambda_0\lambda_2+\lambda_0
	\lambda_3+2)-d_{1,4}\lambda_1\lambda_0\Bigr)\Bigr)d_{3,4}^2d_{1,2}^2-(d_{1,3}d_{2,4}-d_{1,4}d_{2,3})
	d_{3,4}\\
	&\quad \Bigl(d_{1,4}\lambda_3\bigl((\lambda_3-\lambda_0-\lambda_2)
	d_{2,4}-d_{1,4}\lambda_1\bigr)d_{2,3}+d_{1,3}d_{2,4}^2(\lambda_0\lambda_3+1)\Bigr)d_{1,2}
	+\lambda_3^2d_{1,4}d_{2,3}d_{2,4}(d_{1,3}d_{2,4}-\\
	&\quad\ d_{1,4}d_{2,3})^2\Bigr)\lambda_4+d_{2,4}\Bigl(d_{3,4}(d_{1,4}\lambda_1+d_{2,4}\lambda_2)
	d_{1,2}+(d_{1,3}d_{2,4}-d_{1,4}d_{2,3})\lambda_3d_{2,4}\Bigr)d_{1,4}d_{3,4}d_{1,2},\\
	e_{15} &= \lambda_4\Bigl(d_{2,3}^2(d_{1,2}d_{3,4}-2d_{1,3}d_{2,4}+2d_{1,4}d_{2,3})\lambda_4^2+\bigl(
	d_{1,3}\bigl((2\lambda_3-\lambda_0+\lambda_2)d_{2,3}+\tfrac{1}{2}d_{1,3}\lambda_1\bigr)d_{2,4}^2+\bigl(-\\
	&\quad (2\lambda_3+\lambda_0-\lambda_2)d_{1,4}d_{2,3}^2
	-d_{1,2}\lambda_3d_{2,3}d_{3,4}-2d_{1,2}d_{3,4}d_{1,3}\lambda_1\bigr)d_{2,4}+\lambda_1
	(d_{1,2}^2d_{3,4}^2+2d_{1,2}d_{1,4}d_{2,3}d_{3,4}-\\
	&\quad\ d_{1,4}^2d_{2,3}^2)\bigr)\lambda_4-\bigl(d_{1,3}d_{2,4}^2\lambda_3+(d_{1,2}
	d_{3,4}\lambda_2-d_{1,4}d_{2,3}\lambda_3)d_{2,4}+d_{1,2}d_{1,4}d_{3,4}\lambda_1\bigr)
	(d_{1,4}\lambda_1-d_{2,4}\lambda_0)\Bigr),\\
	e_{16} &= \beta_3d_{2,3}\left(d_{1,4}(2\lambda_3+\lambda_0)d_{2,3}^2
	+\left(-d_{1,3}(2\lambda_3+\lambda_0)d_{2,4}+d_{1,2}d_{3,4}\lambda_3
	-d_{1,4}d_{1,3}\lambda_1\right)d_{2,3}+d_{1,3}^2d_{2,4}\lambda_1\right)\lambda_4^2\\
	&\quad +\Bigl(-d_{2,4}(2\lambda_3+\lambda_0+\lambda_2)d_{1,4}^2\lambda_3d_{2,3}^3-3d_{1,4}\Bigl(-
	(4\lambda_3+2\lambda_0+2\lambda_2)
	d_{1,3}\lambda_3d_{2,4}^2+\bigl(3\lambda_3^2+(\lambda_0+\lambda_2)\lambda_3-\\
	&\quad\ \lambda_2\lambda_0+2\bigr)d_{3,4}d_{1,2}d_{2,4}-d_{1,4}d_{3,4}\lambda_1(2\lambda_3+\lambda_0)
	d_{1,2}\Bigr)d_{2,3}^2+\Bigl(-(2\lambda_3+\lambda_0+\lambda_2)d_{1,3}^2\lambda_3
	d_{2,4}^3+(3\lambda_3^2\\
	&\quad +(\lambda_0+\lambda_2)\lambda_3-\lambda_2\lambda_0+2)d_{1,3}d_{3,4}d_{1,2}d_{2,4}^2-d_{3,4}\Bigl(d_{3,4}
	(-\lambda_0\lambda_2+\lambda_3^2+1)d_{1,2}+d_{1,4}d_{1,3}(4\lambda_3+\lambda_0\\
	&\quad +\lambda_2)\lambda_1\Bigr)d_{1,2}d_{2,4}+d_{1,4}d_{3,4}\Bigl(d_{3,4}(3\lambda_3+\lambda_0
	)d_{1,2}-d_{1,4}d_{1,3}\lambda_1\Bigr)\lambda_1d_{1,2}\Bigr)d_{2,3}+(d_{1,2}d_{3,4}-d_{1,3}d_{2,4})d_{3,4}
	\\ &\quad\ \lambda_1\bigl(-d_{1,3}(2\lambda_3+\lambda_2)d_{2,4}+d_{1,2}d_{3,4}\lambda_3-d_{1,4}
	d_{1,3}\lambda_1\bigr)d_{1,2}\Bigr)\lambda_4+d_{2,4}d_{3,4}\Bigl(\bigl((2\lambda_3+\lambda_2)d_{2,4}+
	d_{1,4}\lambda_1\bigr)d_{1,4}\\
	&\quad\ d_{2,3}+d_{2,4}\bigl(-d_{1,3}(2\lambda_3+\lambda_2)d_{2,4}+d_{1,2}d_{3,4}\lambda_3-d_{1,4}d_{1,3}
	\lambda_1\bigr)\Bigr)d_{1,2}.
	\end{align*}
	
	\medskip
	\noindent\textbf{Verify bijectivity}
	
	Compose in both orders:
	\begin{center}
		\verb"skewAuto3 := composeSkewAutomorphism(skewAuto1, skewAuto2);"\\[1ex]
		\verb"skewAuto4 := composeSkewAutomorphism(skewAuto2, skewAuto1);"
	\end{center}
	and execute
	\begin{center}
		\verb"skewAuto3:-getPolys(),"\ \verb"skewAuto4:-getPolys();"
	\end{center}
	Maple outputs
	\begin{center}
		$[x_1, x_2, x_3, x_4],$\ $[x_1, x_2, x_3, x_4]$
	\end{center}
	confirming both compositions equal the identity.
\end{example}

\begin{conjecture}\label{conjecture4.18}
$\mathcal{CSD}_n(K)$ is $D$ when $n$ is even.
\end{conjecture}

This conjecture is induced by Theorem \ref{Theorem1.3}, Remark \ref{remark5.32}, Example \ref{example5.33} and the following reasoning: According to Theorem \ref{theorem4.15}, we have to prove that if $f_1,\dots,f_n\in A:=\mathcal{CSD}_n(K)$ are such that
\begin{center}
	$f_jf_i=f_if_j+d_{ij}$, for all $1\leq i<j\leq n$,
\end{center}
i.e., if $f_1,\dots,f_n$ are skew Dixmier, then the subalgebra generated by $f_1,\dots,f_n$ coincides with $A$. By the universal property of the skew $PBW$ extensions (Proposition \ref{122}), given the inclusion $\iota: K\to A$, there exists a unique ring homomorphism $\widetilde{\iota}:A\to A$ such that
$\widetilde{\iota}\iota=\iota$ and $\widetilde{\iota}(x_i)=f_i$, for every $1\leq i\leq n$. To complete the proof we have to show that $\widetilde{\iota}$ is surjective.
From Theorem \ref{theorem4.16}\ $A$ is simple, so $\widetilde{\iota}$ is injective, whence $A\cong Im(\widetilde{\iota})$. Assuming non-surjectivity, we believe that when $n$ is even a contradiction arise.


\subsubsection*{DixmierProblem: A Versatile Toolkit for Computational Noncommutative Algebra}

In the realm of noncommutative algebra and quantum theory, computational tools that enable explicit construction and manipulation of algebraic structures are essential. The \texttt{DixmierProblem} package, a central component of the SPBWE library, delivers a sophisticated framework for computing Dixmier polynomials in both Weyl algebras and the $K$-algebra $\mathcal{CSD}_n(K)$. By facilitating the explicit computation of automorphisms through specialized functions, the package not only deepens our understanding of internal symmetries and invariants but also drives advances in both pure and applied mathematics.

At its core, \texttt{DixmierProblem} allows researchers to generate classical and skew variants of Dixmier polynomials, which play a crucial role in characterizing the commutation behavior of algebra generators. This capability is particularly significant in addressing longstanding problems, such as the Dixmier and Jacobian conjectures, by providing experimental evidence and deep insights into the structure of automorphism groups. Furthermore, in applications ranging from quantum mechanics (where Weyl algebras model canonical commutation relations) to deformation quantization and differential operator theory, the package proves indispensable for both theoretical investigations and practical computations.

Beyond its research utility, the algorithmic and symbolic features embedded in \texttt{DixmierProblem} streamline complex calculations and foster rigorous verification, making it a powerful asset for both educators and researchers. Its potential applications in emerging areas like cryptography demonstrate its broad impact, offering innovative avenues for exploring security protocols based on noncommutative algebra. In summary, \texttt{DixmierProblem} stands as a robust and flexible toolkit, enabling a detailed exploration of algebraic structures that is vital for advancing modern mathematical and physical research.

\begin{center}
\textbf{Acknowledgements}
\end{center}
{\tiny The second author dedicates this work to Andreita, Lukas and Wilma, his familly. The first author would like to dedicate this work to Professor Oswaldo Lezama, with whom I had the privilege of working throughout my Master's and Ph.D. studies. I am deeply grateful for his guidance, generosity, and unwavering support. Beyond being an outstanding mathematician, he has been an exceptional person whose passion for mathematics, integrity, and commitment to academic excellence have been a constant source of inspiration. I also dedicate this work to my parents, Teo and Ludy, whose love, sacrifices, and encouragement have made every achievement possible. Their faith in me has been a cornerstone of my academic journey. Finally, I dedicate this work to my sisters, Diana and Joanna, who have always been among my closest companions. Their affection, support, and presence throughout the years have been of immeasurable value. To all of them, I offer my deepest gratitude. This work would not have been possible without their guidance, encouragement, and companionship.}

\bigskip

\bigskip

\bigskip

\bigskip

\bigskip


\end{document}